\documentclass[10pt,reqno]{amsart} 

\usepackage{float} 

\usepackage{epsfig,amssymb,amsmath,version}
\usepackage{amssymb,version,graphicx,fancybox,mathrsfs,labelfig}

\usepackage{url,hyperref}
\usepackage{subfigure}
\usepackage{color}
\usepackage{stmaryrd}
\usepackage{multirow}
\usepackage{booktabs,siunitx}
\usepackage{booktabs}
\usepackage{multicol,epstopdf}
\usepackage{xypic}
\usepackage[]{graphicx}
\usepackage[all]{xy}
\usepackage{tikz,ifthen}
\usetikzlibrary{positioning, math, decorations.markings, arrows.meta}
\usetikzlibrary{calc,shapes,cd}
\usetikzlibrary{knots} 
\usepgfmodule{decorations}
\allowdisplaybreaks
\tikzset{knotarrow/.pic={ \draw[edge, <-] (0,0) -- +(-.001,0);}}

\tikzset{edge/.style={line width=0.8}}
\tikzset{wall/.style={very thick}}

\tikzset{->-/.style n args={2}{decoration={markings, mark=at position #1 with {\arrow{#2}}}, postaction={decorate}}} 

\ExplSyntaxOn
\cs_new_eq:NN \ifstreqF \str_if_eq:nnF
\cs_new_eq:NN \ifstreqTF \str_if_eq:nnTF
\ExplSyntaxOff

\tikzset{-o-/.code 2 args={\ifstreqF{#2}{} 
{\ifstreqTF{#2}{>}
   {\pgfkeysalso{decoration={markings,mark=at position #1 with {\arrow[scale=0.8]{#2}}}
                    ,postaction={decorate}}
    }
   {\ifstreqTF{#2}{<}
       {\pgfkeysalso{decoration={markings,mark=at position #1 with {\arrow[scale=0.8]{#2}}}
                    ,postaction={decorate}}
        }
       {\pgfkeysalso{decoration={markings,
                    mark=at position #1 with
                    {\draw[black, fill={#2}] circle[radius=2pt];}}
                    ,postaction={decorate}}
        }
     }
  }}}

\allowdisplaybreaks[4]

\newtheorem{theorem}{Theorem}[section]
\newtheorem{lemma}[theorem]{Lemma}
\newtheorem{definition}[theorem]{Definition}
\newtheorem{corollary}[theorem]{Corollary}
\newtheorem{proposition}[theorem]{Proposition}
\newtheorem{rem}[theorem]{Remark}
\newtheorem{conjecture}[theorem]{Conjecture}

\newcommand{\ca}{{\cev{a}  }}

\def\BZ{\mathbb Z}
\def\Id{\mathrm{Id}}
\def\Mat{\mathrm{Mat}}
\def\BN{\mathbb N}

\definecolor{ligreen}{rgb}{0.0, 0.3, 0.0}

\definecolor{darkblue}{rgb}{0.0, 0.0, 0.55}

\definecolor{anti-flashwhite}{rgb}{0.55, 0.57, 0.68}

\def\cF{\mathcal F}

\def\pr{\mathrm{pr}}
\def\cV{\mathcal V}
\def\ot{\otimes}
\def\buu{{\mathbf u}}

\newcommand{\cev}[1]{\reflectbox{\ensuremath{\vec{\reflectbox{\ensuremath{#1}}}}}}

\def\cE{\mathcal E}
\def\fB{\mathfrak B}
\def\cR{\mathcal R}
\def\cY{\mathcal Y}
\def\cS{\mathscr S}

\newcommand{\beq}{\begin{equation}}
\newcommand{\eeq}{\end{equation}}

\graphicspath{{./Figures/} } 

\begin{document}
\bibliographystyle{plain}

\title{On stated $SL(n)$-skein modules}
\author{Zhihao Wang}

\keywords{Skein theory, $SL(n)$, Classical limit, Splitting map,  Frobenius homomorphism, Unicity Theorem}

 \maketitle


\begin{abstract}
We mainly focus on Classical limit, Splitting map, and   Frobenius homomorphism for stated $SL(n)$-skein modules, and
Unicity Theorem for stated $SL(n)$-skein algebras.

Let $(M,N)$ be a marked three manifold. We use $S_n(M,N,v)$ to denote the stated $SL(n)$-skein module of $(M,N)$ where $v$ is a nonzero complex number.
We build a surjective algebra homomorphism from $S_n(M,N,1)$ to the coordinate ring of some algebraic set, and prove it's Kernal consists of all nilpotents. We prove the universal representation algebra of $\pi_1(M)$ is isomorphic to $S_n(M,N,1)$ when $N$ has only one component and $M$ is connected. Furthermore we  show $S_n(M,N^{'},1)$ is isomorphic to
$S_n(M,N,1)\otimes O(SLn)$, where $N\neq \emptyset$, $M$ is connected, and $N^{'}$ is obtained from $N$ by adding one extra marking.
 We also prove the splitting map is injective for any marked three manifold when $v=1$, and show that the splitting map is injective (for general $v$) if there exists at least one component of $N$ such that this component and the boundary of the splitting disk belong to the same component of $\partial M$.

We also establish the Frobenius homomorphism for $SL(n)$, which is map  from $S_n(M,N,1)$ to $S_n(M,N,v)$ when $v$ is a primitive $m$-th root of unity with $m$  being coprime with $2n$ and every component of $M$ contains at least one marking. 
We also show the commutativity between Frobenius homomorphism and splitting map. When $(M,N)$ is the thickening of an essentially bordered pb surface, we prove the Frobenius homomorphism is injective and it's image lives in the center. We prove the stated $SL(n)$-skein algebra $S_n(\Sigma,v)$ is
affine almost Azumaya when $\Sigma$ is an essentially bordered pb surface and $v$ is a primitive $m$-th root of unity with $m$  being coprime with $2n$, which implies the Unicity Theorem for $S_n(\Sigma,v)$.

\end{abstract}

\tableofcontents{}

\section{Introduction}

In this paper we will work with complex number $\mathbb{C}$ with a distinguished nonzero element $v=q^{\frac{1}{2n}}$ where $n$ is a positive integer. When we say algebra, we just mean an algebra over  $\mathbb{C}$. For any other  ring or algebra $A$, we will use $A$-algebra to mean an algebra over $A$. An algebra $A$ is called a {\bf domain}, if, for any two elements $a,b\in A$,
$ab=0$ indicates $a=0$ or $b=0$. For any commutative algebra $A$, we use MaxSpec$(A)$ to denote the set of maximal ideals of $A$.

 A {\bf marked three manifold} is  a pair $(M,N)$, where $M$ is a smooth oriented three manifold, and $N$ is a subset of $\partial M$ consisting of oriented open intervals such that for any two components $\alpha, \beta$ of $N$  there is no intersection between the closure of $\alpha$ and the closure of $\beta$. Note that we allow $N$ to be empty.

L{\^e} and Sikora defined the stated $SL(n)$-skein module for marked three manifolds \cite{le2021stated}, which is a generalization for the classical case when $n=2$ \cite{lestatedsurvery}. Let $(M,N)$ be a marked three manifold, we will use 
$S_n(M,N,v)$ to denote the  stated $SL(n)$-skein module of $(M,N)$.

If $(M,N)$ is the disjoint union of $(M_1,N_1)$ and $(M_2,N_2)$, we can easily get $S_n(M,N,v)=
S_n(M_1,N_1,v)\otimes S_n(M_2,N_2,v)$. Then we can reduce general marked three manifold to $(M,N)$ with $M$ being connected. So sometimes for convenience, we will assume $M$
 is connected (the corresponding results can be easily generalized to general marked three manifolds).

For the classical $SL(2)$-skein theory, Bullock stablished a surjective algebra homomorphism from $S_2(M,\emptyset,v^{2}=-1)$, which is isomorphic to $S_2(M,\emptyset,1)$ \cite{barrett1999skein}, to the coordinate ring of some algebraic set, and showed the Kernal of this map consists of all nilpotents \cite{BL1997rings}. This is an important development since it showed the connection between skein theory and character variety. Moreover it offered a way to interpret $S_2(M,\emptyset,1)$. It is also useful for understanding the  representation theory for $SL(2)$-skein algebra \cite{representation1,representation2,representation3}.
 
There  exists the Chebyshev-Frobenius homomorphism from $S_2(M,N,1)$ to $S_2(M,N,v)$ when $v$
is a root of unity of odd order \cite{bloomquist2020chebyshev,representation1,korinman2019classical}.
When $(M,N)$ is the thickening of a punctured bordered surface, it is injective and it's image lies in the center of the stated skein algebra.
The Chebyshev-Frobenius homomorphism makes $S_2(M,N,1)$  important to understand the center of stated skein algebra and it's representation theory \cite{representation1,representation2,representation3,unicity}.
 When 
  $M$ is the thickening of a closed surface and $N=\emptyset$, 
Bonahon, Wong and Yang introduced a new version of volume conjectuce using the fixed character in MaxSpec($S_2(M,\emptyset,1)$) \cite{bonahon2021asymptotics}. The author generalized this conjecture to periodic case and proved the conjecture for periodic case for once punctured torus \cite{wang2022kauffman}.


We generalize Bullock's work to stated $SL(n)$-skein module for marked three manifolds. 
That is to construct a surjective algebra homomorphism from $S_n(M,N,1)$ to the coordinate ring of some algebraic set and try to calculate the Kernal of this map. Costantino and L{\^e} have constructed this map when $n=2$ and $(M,N)$ is the thickening of an essentially bordered pb surface,  and proved it is an isomorphism \cite{CL2022stated}.
Motivated by Costantino and L{\^e}, we choose the algebraic set to be the homomorphism from the groupoid $\pi_1(M,N)$ to $SL(n,\mathbb{C})$, whose coordinate ring is denoted as $R_n(M,N)$.

\begin{theorem}
Let $(M,N)$ be a marked three manifold. Then
there exists a surjective algebra homomorphism 
$$\Phi^{(M,N)} : S_n(M,N,1)\rightarrow R_n(M,N).$$
\end{theorem}
 We can  show $\Phi^{(M,N)}$ is commutative with the splitting map, and our construction is compatible with  the construction by Costantino and L{\^e}  when $n=2$ and $(M,N)$ is the thickening of an essentially bordered pb surface.

In order to calculate Ker\,$\Phi^{(M,N)}$, we give explicit presentation for $S_n(M,N,1)$. We do this in two steps. First we consider the case when $N$ has only one component.

\begin{theorem}
Let $(M,N)$ be a marked three manifold with $M$ being connected and $N$ consisting of one component. We have $S_n(M,N,1)\simeq \Gamma_n(M)$ where $\Gamma_n(M)$ is the universal presentation algebra of $\pi_1(M)$ (Definition \ref{df3.4}).
\end{theorem}
Then we try to find out the effect on $SL(n)$-stated skein module by adding one extra marking. 

\begin{theorem}
Suppose $(M,N)$ is a marked three manifold with  $M$ being connected and $N\neq\emptyset$, and $N^{'}$ is obtained from $N$ by adding one extra marking. Then we have $S_n(M,N^{'},1)\simeq
S_n(M,N,1)\otimes O(SLn)$, where $ O(SLn)$ is the coordinate ring for $SL(n,\mathbb{C})$.
\end{theorem}

Let $(M,N)$ be any marked three manifold with  $M$ being connected and $N\neq \emptyset$. Combining the above two theorems we have 
$$S_n(M,N,1)\simeq \Gamma_n(M)\otimes O(SLn)^{\otimes(\sharp N-1)}$$
where $\sharp N$ is the number of components of $N$.

For a commutative algebra $A$, we use $\sqrt{0}_A$ to denote the ideal consisting of all nilpotents. We can  omit the subscript for $\sqrt{0}_A$ when there is no confusion with the algebra $A$.
There is a  projection from $\Gamma_n(M)\otimes O(SLn)^{\otimes(\sharp N-1)}$ to $R_n(M,N)$, whose Kernal is $\sqrt{0}$. 
The isomorphism from $S_n(M,N,1)$ to $\Gamma_n(M)\otimes O(SLn)^{\otimes(\sharp N-1)}$ is compatible with $\Phi^{(M,N)}$, that is, the combination of this isomorphism and the projection from $\Gamma_n(M)\otimes O(SLn)^{\otimes(\sharp N-1)}$ to $R_n(M,N)$ is $\Phi^{(M,N)}$. Thus we have 
the following theorem.


\begin{theorem}
For any marked three manifold $(M,N)$, we have Ker\,$\Phi^{(M,N)} = \sqrt{0}$. 
\end{theorem}

L{\^e} and Sikora also introduced the splitting map $\Theta$ for  stated $SL(n)$-skein module of marked three manifolds. They also conjectured the injectivity of the splitting map when $(M,N)$ is the thickening of punctured bordered surface. We will prove this conjecture  when $v=1$ in this paper. 
 We also find the connection between splitting map and adding marking map (Proposition \ref{prop6.3}). After we observed this connection, we found out
L{\^e} and Sikora used this connection to calculate the image of the splitting map when $(M,N)$ is the thickening of a punctured bordered surface \cite{le2021stated}. In this paper, we care more about general marked three manifolds, and will use this connection to prove the injectivity of splitting map for a large family  of marked three manifolds.

\begin{theorem}
Let $(M,N)$ be a marked three manifold. Let $D$ be a properly embedded disk with an oriented open interval $\beta\subset D$. Suppose $\partial D$ is contained in the component $Y$ of $\partial M$. Then we have  $\text{Ker\,}((M,N),Y,v) = \text{Ker}\,\Theta_{(D,\beta)}$. Especially $\text{Ker\,}((M,N),Y,v) = 0$ if and only if $\Theta_{(D,\beta)}$ is injective where $\text{Ker\,}((M,N),Y,v)$ is defined based on the Kernal of adding marking map
(Definition \ref{df6.5}).
\end{theorem}

The above Theorem is useful to prove injectivity of the splitting map. Using this Theorem, we can prove the splitting map is  injective for any marked three manifold when $v=1$. We also have the following Corollary for general $v$.

\begin{corollary}
Let $(M,N)$ be a marked three manifold. Let $D$ be a properly embedded disk with an oriented open interval $\beta\subset D$. Suppose $\partial D$ is contained in the compotent $Y$ of $\partial M$, and 
$N\cap Y\neq \emptyset$. Then $\Theta_{(D,\beta)}$ is injective.
\end{corollary}

We also construct the 
Frobenius homomorphism from $S_n(M,N,1)$ to $S_n(M,N,v)$ when $v$ is a primitive $m$-th root of unity with $m$  being coprime with $2n$ and every component of $M$ contains at least one marking. Before this paper, the Frobenius homomorphism was only  built for $n=2,3$.

\begin{theorem}($n=2$ and $(M,N)$ is the thickening of the closed surface \cite{representation1}, $n=2$ and $(M,N)$ is the thickening of the punctured  bordered surface \cite{korinman2019classical}, $n=2$ and general $(M,N)$ \cite{bloomquist2020chebyshev}, $n=3$ and $(M,N)$ is the thickening of the punctured  bordered  surface with at least one puncture at each component \cite{higgins})\;\;
Let $(M,N)$ be a marked three manifold, where every component of $M$ contains at least one marking, and $v$ be a primitive $m$-th root of unit with $m$ being coprime with $2n$. Then
there exists a unique linear map $\cF:S_n(M,N,1)\rightarrow S_n(M,N,v)$ such that  $\cF(l) = l^{(m)}$
for any stated $n$-web $l$ consisting of stated framed oriented boundary arcs, where $l^{(m)}$ is obtained from $l$ by taking $m$ parallel copies along the framing direction for each stated framed oriented boundary arc.
\end{theorem}

We can also show the Frobenius homomorphism  is commutative with the splitting map.

\begin{theorem}
Suppose $D$ is a  properly embedded disk in a marked three manifold $(M, N )$ and $D$  contains an oriented open interval $\beta$. Let $(M^{'},N^{'})$ be the result of splitting $(M, N )$ along $(D, \beta)$.
Then we have the following commutative diagram:
$$\begin{tikzcd}
S_n(M,N,1)  \arrow[r, "\Theta"]
\arrow[d, "\cF"]  
&  S_n(M^{'},N^{'},1)  \arrow[d, "\cF"] \\
 S_n(M,N,v)  \arrow[r, "\Theta"] 
&  S_n(M^{'},N^{'},v)\\
\end{tikzcd}.$$
\end{theorem}

We  show the image of $\cF$ is transparent, which indicates Im$\cF$ lives in the center when
$(M,N)$ is the thickening of an essentially bordered pb surface. We  also show $\cF$ is injective when $(M,N)$ is  the thickening of an essentially bordered pb surface.

\begin{theorem}
Let $\Sigma$ be an essentially bordered pb surface. Then 
 $\cF:S_n(\Sigma,1)\rightarrow S_n(\Sigma,v)$ is an injective algebra homomorphism. And Im$\cF$ lives in the center of $S_n(\Sigma,v)$.
\end{theorem}

One reason for Frobenius map to be important is because it is very crucial for understanding representation theory for skein algebra. Another useful tool to study representation theory for skein algebra is Unicity Theorem, which was first proved by Frohman, Kania-Bartoszynska, L{\^e} for $SL(2)$-skein algebra \cite{unicity}. Korinman generalized Unicity Theorem to stated $SL(2)$-skein algebra \cite{korinman2021unicity}. The key to prove Unicity Theorem for skein algebra is to prove the skein algebra is affine almost Azumaya, see Section \ref{sss888} for definition.

\begin{theorem}
Let $\Sigma$ be an essentially bordered pb surface, and $v$ be a primitive $m$-th root of unit with $m$ being coprime with $2n$. We have  $S_n(\Sigma,v)$ is affine almost Azumaya.
\end{theorem}

Because of the above Theorem, the Unicity Theorem applies for $S_n(\Sigma,v)$ when $\Sigma$ is an essentially bordered pb surface. Then we have the following Theorem.

\begin{theorem}
Let $\Sigma$ be an essentially bordered pb surface, and $v$ be a primitive $m$-th root of unit with $m$ being coprime with $2n$. Then $S_n(\Sigma,v)$ is finitely generated as a module over it's center.
 Suppose the rank of $S_n(\Sigma,v)$  over it's center is $K$. And  we use $C(S_n(\Sigma,v))$ to denote the center of $S_n(\Sigma,v)$, then we have

(1) any irreducible representation of $S_n(\Sigma,v)$ has dimension at most the square root of $K$;

 (2) the map $ \mathcal{X} :\text{Irrep}\rightarrow  \text{MaxSpec}(C(S_n(\Sigma,v)))$, defined in Remark \ref{rem7.48}, is surjective. 

(3) there exists a Zariski open dense subset $U \subset \text{MaxSpec}(C(S_n(\Sigma,v)))$ such that for any two irreducible representations $V_1$, $V_2$ of $S_n(\Sigma,v)$ with $\mathcal{X}(V_1) = \mathcal{X}(V_2)\in U$, then $V_1$ and $V_2$ are isomorphic and have  dimension the square root of $K$. Moreover any representation sending $C(S_n(\Sigma,v))$ to scalar operators and  whose induced character lies in $U$ is semi-simple.

\end{theorem}

Costantino and L{\^e} also defined the generalized marked three manifold by allowing oriented closed circle in $N$ \cite{CL2022TQFT}. We can obviously define the stated $SL(n)$-skein module for the generalized marked three manifold. For any generalized marked three manifold, we will define $\Gamma_n(M,N)$ as a quotient algebra of $\Gamma_n(M)$ (Definition \ref{df7.3}).

\begin{theorem}
Let $(M,N)$ be a generalized marked three manifold with  $M$ being connected and $N\neq \emptyset$.
Then $S_n(M,N,1)\simeq \Gamma_n(M,N)\otimes O(SLn)^{\otimes(\sharp N - 1)}$.
\end{theorem}

{\bf Acknowledgements}:
The author would like to thank my supervisors Andrew James Kricker and  Roland van der Veen, and my colleague Jeffrey Weenink Andre for helpful discussion. The research is supported by NTU  research scholarship.

\section{Preliminary}

In this section, we will recall some definitions and conclusions in \cite{le2021stated},
 and also introduce some conventions for this paper.

\subsection{Stated $SL(n)$-skein module}
In this paper, we follow the definition in \cite{le2021stated} for stated $SL(n)$-skein modules.  Here we briefly recall the definition.

\begin{definition}(\cite{le2021stated})
 An $n$-web $l$ in a marked three manifold $(M,N)$ is a disjoint union of oriented closed paths and a directed finite  graph properly embedded into $M$. We also have the following requirements:

(1) $l$ only contains $1$-valent or $n$-valent vertices. Each $n$-valent vertex is a source or a  sink. The set of one valent vertices is denoted as $\partial l$, which are  called endpoints of $l$.

 (2) Every edge of the graph is an embedded  closed interval  in $M$. 

(3) $l$ is equipped with a transversal framing. 

(4) The set of half-edges at each $n$-valent vertex is equipped with a  cyclic order. 

(5) $\partial l$ is contained in $N$ and the framing at these endpoints is the velocity  vector of $N$.
\end{definition}



For any two points $a,b\in\partial l$, we say $a$ is higher than $b$ if they belong to a same component $e$ of $N$ and the direction of $e$ is going from $b$ to $a$.

A state of an $n$-web $l$ is a map  $s: \partial l\rightarrow \{1,2,\dots,n\}$. If there is  such a map $s$ for $l$, we say $l$ is stated by $s$. And for any point $a\in\partial l$, we say $a$ is stated by $s(a)$.

Recall that our ground ring is $\mathbb{C}$ with the parameter $v\in\mathbb{C}^{*}$. We set 
$q^{\frac{1}{2n}} = v$, and define the following constants:
\begin{align}
c_i&= (-q)^{n-i} q^{\frac{n-1}{2n}},\quad
t= (-1)^{n-1} q^{\frac{n^2-1}{n}}, \ t^{n/2} =  (-1)^{\frac{(n-1)n}2} q^{\frac{n^2-1}{2}} \label{e.t}\\
a &=   q^{\frac{n+1-2n^2}{4}},\quad 
d_n = (-1)^{n-1}. 
\end{align}

Note that
\begin{equation}\label{e.prodc} 
\prod_{i=1}^n c_i  = t^{n/2}= (-1)^{\frac{(n-1)n}2 } q^\frac{n^2-1}{2} 
\ \text{and}\ c_i\cdot c_{\bar i}=t, \ \text{for}\ i=1,\dots, n.
\end{equation}

We will use $\mathbb{J}$ to denote the set $\{1,2,\dots,n\}$, and use $S_n$ to denote the permutation group on set $\mathbb{J}$.


The stated $SL(n)$-skein module of $(M,N)$, denoted as $S_n(M,N,v)$, is obtained in two steps. We first use all isotopy classes of stated 
$n$-webs in $(M,N)$ as basis to generate a vector space, then  quotient the following relations.

\beq\label{w.cross}
q^{\frac{1}{n}} 
\raisebox{-.20in}{

\begin{tikzpicture}
\tikzset{->-/.style=

{decoration={markings,mark=at position #1 with

{\arrow{latex}}},postaction={decorate}}}
\filldraw[draw=white,fill=gray!20] (-0,-0.2) rectangle (1, 1.2);
\draw [line width =1pt,decoration={markings, mark=at position 0.5 with {\arrow{>}}},postaction={decorate}](0.6,0.6)--(1,1);
\draw [line width =1pt,decoration={markings, mark=at position 0.5 with {\arrow{>}}},postaction={decorate}](0.6,0.4)--(1,0);
\draw[line width =1pt] (0,0)--(0.4,0.4);
\draw[line width =1pt] (0,1)--(0.4,0.6);
\draw[line width =1pt] (0.4,0.6)--(0.6,0.4);
\end{tikzpicture}
}
- q^{-\frac {1}{n}}
\raisebox{-.20in}{
\begin{tikzpicture}
\tikzset{->-/.style=

{decoration={markings,mark=at position #1 with

{\arrow{latex}}},postaction={decorate}}}
\filldraw[draw=white,fill=gray!20] (-0,-0.2) rectangle (1, 1.2);
\draw [line width =1pt,decoration={markings, mark=at position 0.5 with {\arrow{>}}},postaction={decorate}](0.6,0.6)--(1,1);
\draw [line width =1pt,decoration={markings, mark=at position 0.5 with {\arrow{>}}},postaction={decorate}](0.6,0.4)--(1,0);
\draw[line width =1pt] (0,0)--(0.4,0.4);
\draw[line width =1pt] (0,1)--(0.4,0.6);
\draw[line width =1pt] (0.6,0.6)--(0.4,0.4);
\end{tikzpicture}
}
= (q-q^{-1})
\raisebox{-.20in}{

\begin{tikzpicture}
\tikzset{->-/.style=

{decoration={markings,mark=at position #1 with

{\arrow{latex}}},postaction={decorate}}}
\filldraw[draw=white,fill=gray!20] (-0,-0.2) rectangle (1, 1.2);
\draw [line width =1pt,decoration={markings, mark=at position 0.5 with {\arrow{>}}},postaction={decorate}](0,0.8)--(1,0.8);
\draw [line width =1pt,decoration={markings, mark=at position 0.5 with {\arrow{>}}},postaction={decorate}](0,0.2)--(1,0.2);
\end{tikzpicture}
},
\eeq 
\beq\label{w.twist}
\raisebox{-.15in}{
\begin{tikzpicture}
\tikzset{->-/.style=
{decoration={markings,mark=at position #1 with
{\arrow{latex}}},postaction={decorate}}}
\filldraw[draw=white,fill=gray!20] (-1,-0.35) rectangle (0.6, 0.65);
\draw [line width =1pt,decoration={markings, mark=at position 0.5 with {\arrow{>}}},postaction={decorate}](-1,0)--(-0.25,0);
\draw [color = black, line width =1pt](0,0)--(0.6,0);
\draw [color = black, line width =1pt] (0.166 ,0.08) arc (-37:270:0.2);
\end{tikzpicture}}
= t
\raisebox{-.15in}{
\begin{tikzpicture}
\tikzset{->-/.style=
{decoration={markings,mark=at position #1 with
{\arrow{latex}}},postaction={decorate}}}
\filldraw[draw=white,fill=gray!20] (-1,-0.5) rectangle (0.6, 0.5);
\draw [line width =1pt,decoration={markings, mark=at position 0.5 with {\arrow{>}}},postaction={decorate}](-1,0)--(-0.25,0);
\draw [color = black, line width =1pt](-0.25,0)--(0.6,0);
\end{tikzpicture}}
,
\eeq
\beq\label{w.unknot}
\raisebox{-.20in}{
\begin{tikzpicture}
\tikzset{->-/.style=
{decoration={markings,mark=at position #1 with
{\arrow{latex}}},postaction={decorate}}}
\filldraw[draw=white,fill=gray!20] (0,0) rectangle (1,1);
\draw [line width =1pt,decoration={markings, mark=at position 0.5 with {\arrow{>}}},postaction={decorate}](0.45,0.8)--(0.55,0.8);
\draw[line width =1pt] (0.5 ,0.5) circle (0.3);
\end{tikzpicture}}
= (-1)^{n-1} [n]\ 
\raisebox{-.20in}{
\begin{tikzpicture}
\tikzset{->-/.style=
{decoration={markings,mark=at position #1 with
{\arrow{latex}}},postaction={decorate}}}
\filldraw[draw=white,fill=gray!20] (0,0) rectangle (1,1);
\end{tikzpicture}}
,\ \text{where}\ [n]=\frac{q^n-q^{-n}}{q-q^{-1}},
\eeq
\beq\label{wzh.four}
\raisebox{-.30in}{
\begin{tikzpicture}
\tikzset{->-/.style=
{decoration={markings,mark=at position #1 with
{\arrow{latex}}},postaction={decorate}}}
\filldraw[draw=white,fill=gray!20] (-1,-0.7) rectangle (1.2,1.3);
\draw [line width =1pt,decoration={markings, mark=at position 0.5 with {\arrow{>}}},postaction={decorate}](-1,1)--(0,0);
\draw [line width =1pt,decoration={markings, mark=at position 0.5 with {\arrow{>}}},postaction={decorate}](-1,0)--(0,0);
\draw [line width =1pt,decoration={markings, mark=at position 0.5 with {\arrow{>}}},postaction={decorate}](-1,-0.4)--(0,0);
\draw [line width =1pt,decoration={markings, mark=at position 0.5 with {\arrow{<}}},postaction={decorate}](1.2,1)  --(0.2,0);
\draw [line width =1pt,decoration={markings, mark=at position 0.5 with {\arrow{<}}},postaction={decorate}](1.2,0)  --(0.2,0);
\draw [line width =1pt,decoration={markings, mark=at position 0.5 with {\arrow{<}}},postaction={decorate}](1.2,-0.4)--(0.2,0);
\node  at(-0.8,0.5) {$\vdots$};
\node  at(1,0.5) {$\vdots$};
\end{tikzpicture}}=(-q)^{\frac{n(n-1)}{2}}\cdot \sum_{\sigma\in S_n}
(-q^{\frac{1-n}n})^{\ell(\sigma)} \raisebox{-.30in}{
\begin{tikzpicture}
\tikzset{->-/.style=
{decoration={markings,mark=at position #1 with
{\arrow{latex}}},postaction={decorate}}}
\filldraw[draw=white,fill=gray!20] (-1,-0.7) rectangle (1.2,1.3);
\draw [line width =1pt,decoration={markings, mark=at position 0.5 with {\arrow{>}}},postaction={decorate}](-1,1)--(0,0);
\draw [line width =1pt,decoration={markings, mark=at position 0.5 with {\arrow{>}}},postaction={decorate}](-1,0)--(0,0);
\draw [line width =1pt,decoration={markings, mark=at position 0.5 with {\arrow{>}}},postaction={decorate}](-1,-0.4)--(0,0);
\draw [line width =1pt,decoration={markings, mark=at position 0.5 with {\arrow{<}}},postaction={decorate}](1.2,1)  --(0.2,0);
\draw [line width =1pt,decoration={markings, mark=at position 0.5 with {\arrow{<}}},postaction={decorate}](1.2,0)  --(0.2,0);
\draw [line width =1pt,decoration={markings, mark=at position 0.5 with {\arrow{<}}},postaction={decorate}](1.2,-0.4)--(0.2,0);
\node  at(-0.8,0.5) {$\vdots$};
\node  at(1,0.5) {$\vdots$};
\filldraw[draw=black,fill=gray!20,line width =1pt]  (0.1,0.3) ellipse (0.4 and 0.7);
\node  at(0.1,0.3){$\sigma_{+}$};
\end{tikzpicture}},
\eeq
where the ellipse enclosing $\sigma_+$  is the minimum crossing positive braid representing a permutation $\sigma\in S_n$ and $\ell(\sigma)=\mid\{(i,j)\mid 1\leq i<j\leq n, \sigma(i)>\sigma(j)\}|$ is the length of $\sigma\in S_n$.

\beq\label{wzh.five}
   \raisebox{-.30in}{
\begin{tikzpicture}
\tikzset{->-/.style=
{decoration={markings,mark=at position #1 with
{\arrow{latex}}},postaction={decorate}}}
\filldraw[draw=white,fill=gray!20] (-1,-0.7) rectangle (0.2,1.3);
\draw [line width =1pt](-1,1)--(0,0);
\draw [line width =1pt](-1,0)--(0,0);
\draw [line width =1pt](-1,-0.4)--(0,0);
\draw [line width =1.5pt](0.2,1.3)--(0.2,-0.7);
\node  at(-0.8,0.5) {$\vdots$};
\filldraw[fill=white,line width =0.8pt] (-0.5 ,0.5) circle (0.07);
\filldraw[fill=white,line width =0.8pt] (-0.5 ,0) circle (0.07);
\filldraw[fill=white,line width =0.8pt] (-0.5 ,-0.2) circle (0.07);
\end{tikzpicture}}
   = 
   a \sum_{\sigma \in S_n} (-q)^{\ell(\sigma)}\,  \raisebox{-.30in}{
\begin{tikzpicture}
\tikzset{->-/.style=
{decoration={markings,mark=at position #1 with
{\arrow{latex}}},postaction={decorate}}}
\filldraw[draw=white,fill=gray!20] (-1,-0.7) rectangle (0.2,1.3);
\draw [line width =1pt](-1,1)--(0.2,1);
\draw [line width =1pt](-1,0)--(0.2,0);
\draw [line width =1pt](-1,-0.4)--(0.2,-0.4);
\draw [line width =1.5pt,decoration={markings, mark=at position 1 with {\arrow{>}}},postaction={decorate}](0.2,1.3)--(0.2,-0.7);
\node  at(-0.8,0.5) {$\vdots$};
\filldraw[fill=white,line width =0.8pt] (-0.5 ,1) circle (0.07);
\filldraw[fill=white,line width =0.8pt] (-0.5 ,0) circle (0.07);
\filldraw[fill=white,line width =0.8pt] (-0.5 ,-0.4) circle (0.07);
\node [right] at(0.2,1) {$\sigma(n)$};
\node [right] at(0.2,0) {$\sigma(2)$};
\node [right] at(0.2,-0.4){$\sigma(1)$};
\end{tikzpicture}},
\eeq
\beq \label{wzh.six}
\raisebox{-.20in}{
\begin{tikzpicture}
\tikzset{->-/.style=
{decoration={markings,mark=at position #1 with
{\arrow{latex}}},postaction={decorate}}}
\filldraw[draw=white,fill=gray!20] (-0.7,-0.7) rectangle (0,0.7);
\draw [line width =1.5pt,decoration={markings, mark=at position 1 with {\arrow{>}}},postaction={decorate}](0,0.7)--(0,-0.7);
\draw [color = black, line width =1pt] (0 ,0.3) arc (90:270:0.5 and 0.3);
\node [right]  at(0,0.3) {$i$};
\node [right] at(0,-0.3){$j$};
\filldraw[fill=white,line width =0.8pt] (-0.5 ,0) circle (0.07);
\end{tikzpicture}}   = \delta_{\bar j,i }\,  c_i\ \raisebox{-.20in}{
\begin{tikzpicture}
\tikzset{->-/.style=
{decoration={markings,mark=at position #1 with
{\arrow{latex}}},postaction={decorate}}}
\filldraw[draw=white,fill=gray!20] (-0.7,-0.7) rectangle (0,0.7);
\draw [line width =1.5pt](0,0.7)--(0,-0.7);
\end{tikzpicture}},
\eeq
\beq \label{wzh.seven}
\raisebox{-.20in}{
\begin{tikzpicture}
\tikzset{->-/.style=
{decoration={markings,mark=at position #1 with
{\arrow{latex}}},postaction={decorate}}}
\filldraw[draw=white,fill=gray!20] (-0.7,-0.7) rectangle (0,0.7);
\draw [line width =1.5pt](0,0.7)--(0,-0.7);
\draw [color = black, line width =1pt] (-0.7 ,-0.3) arc (-90:90:0.5 and 0.3);
\filldraw[fill=white,line width =0.8pt] (-0.55 ,0.26) circle (0.07);
\end{tikzpicture}}
= \sum_{i=1}^n  (c_{\bar i})^{-1}\, \raisebox{-.20in}{
\begin{tikzpicture}
\tikzset{->-/.style=
{decoration={markings,mark=at position #1 with
{\arrow{latex}}},postaction={decorate}}}
\filldraw[draw=white,fill=gray!20] (-0.7,-0.7) rectangle (0,0.7);
\draw [line width =1.5pt,decoration={markings, mark=at position 1 with {\arrow{>}}},postaction={decorate}](0,0.7)--(0,-0.7);
\draw [line width =1pt](-0.7,0.3)--(0,0.3);
\draw [line width =1pt](-0.7,-0.3)--(0,-0.3);
\filldraw[fill=white,line width =0.8pt] (-0.3 ,0.3) circle (0.07);
\filldraw[fill=black,line width =0.8pt] (-0.3 ,-0.3) circle (0.07);
\node [right]  at(0,0.3) {$i$};
\node [right]  at(0,-0.3) {$\bar{i}$};
\end{tikzpicture}},
\eeq
\beq\label{wzh.eight}
\raisebox{-.20in}{

\begin{tikzpicture}
\tikzset{->-/.style=

{decoration={markings,mark=at position #1 with

{\arrow{latex}}},postaction={decorate}}}
\filldraw[draw=white,fill=gray!20] (-0,-0.2) rectangle (1, 1.2);
\draw [line width =1.5pt,decoration={markings, mark=at position 1 with {\arrow{>}}},postaction={decorate}](1,1.2)--(1,-0.2);
\draw [line width =1pt](0.6,0.6)--(1,1);
\draw [line width =1pt](0.6,0.4)--(1,0);
\draw[line width =1pt] (0,0)--(0.4,0.4);
\draw[line width =1pt] (0,1)--(0.4,0.6);
\draw[line width =1pt] (0.4,0.6)--(0.6,0.4);
\filldraw[fill=white,line width =0.8pt] (0.2 ,0.2) circle (0.07);
\filldraw[fill=white,line width =0.8pt] (0.2 ,0.8) circle (0.07);
\node [right]  at(1,1) {$i$};
\node [right]  at(1,0) {$j$};
\end{tikzpicture}
} =q^{-\frac{1}{n}}\left(\delta_{{j<i} }(q-q^{-1})\raisebox{-.20in}{

\begin{tikzpicture}
\tikzset{->-/.style=

{decoration={markings,mark=at position #1 with

{\arrow{latex}}},postaction={decorate}}}
\filldraw[draw=white,fill=gray!20] (-0,-0.2) rectangle (1, 1.2);
\draw [line width =1.5pt,decoration={markings, mark=at position 1 with {\arrow{>}}},postaction={decorate}](1,1.2)--(1,-0.2);
\draw [line width =1pt](0,0.8)--(1,0.8);
\draw [line width =1pt](0,0.2)--(1,0.2);
\filldraw[fill=white,line width =0.8pt] (0.2 ,0.8) circle (0.07);
\filldraw[fill=white,line width =0.8pt] (0.2 ,0.2) circle (0.07);
\node [right]  at(1,0.8) {$i$};
\node [right]  at(1,0.2) {$j$};
\end{tikzpicture}
}+q^{\delta_{i,j}}\raisebox{-.20in}{

\begin{tikzpicture}
\tikzset{->-/.style=

{decoration={markings,mark=at position #1 with

{\arrow{latex}}},postaction={decorate}}}
\filldraw[draw=white,fill=gray!20] (-0,-0.2) rectangle (1, 1.2);
\draw [line width =1.5pt,decoration={markings, mark=at position 1 with {\arrow{>}}},postaction={decorate}](1,1.2)--(1,-0.2);
\draw [line width =1pt](0,0.8)--(1,0.8);
\draw [line width =1pt](0,0.2)--(1,0.2);
\filldraw[fill=white,line width =0.8pt] (0.2 ,0.8) circle (0.07);
\filldraw[fill=white,line width =0.8pt] (0.2 ,0.2) circle (0.07);
\node [right]  at(1,0.8) {$j$};
\node [right]  at(1,0.2) {$i$};
\end{tikzpicture}
}\right),
\eeq

where   
$\delta_{j<i}= \left \{
 \begin{array}{rr}
     1,                    & j<i\\
     0,                                 & \text{otherwise}
 \end{array}
 \right.,
\delta_{i,j}= \left \{
 \begin{array}{rr}
     1,                    & i=j\\
     0,                                 & \text{otherwise}
 \end{array}
 \right.$. Each shaded rectangle in the above relations is the projection of a cube in $M$. The lines contained in the shaded rectangle represent parts of stated $n$-webs with framing  pointing to  readers. The thick line in the edge of shaded rectangle is part of the marking. For detailed explanation for the above relations, please refer to \cite{le2021stated}.

Suppose $(M, N )$ is a marked 3-manifold. Let $Y$ be any subset of $\partial M-N$. Then
we have $(M-Y,N)$
has the same skein theory with $(M,N)$, see Example 4.3 in \cite{le2021stated}. We call such a subset $Y$  an unimportant part on $\partial M$.
 So sometimes we do not have to distinguish between $(M-Y,N)$ and $(M,N)$.

\subsection{Functoriality}

For any two marked three manifolds $(M,N),(M^{'}N^{'})$, if an orientation preserving  embedding $f:M\rightarrow M^{'}$ maps 
$N$ to $N^{'}$ and preserves  orientations between $N$ and $N^{'}$, we call $f$ an embedding from $(M,N)$ to $(M^{'},N^{'})$. Clearly $f$ induces a linear map $f_{*}:S_n(M,N,v)\rightarrow S_n(M,N,v)$ \cite{le2021stated}.

\subsection{Splitting map}

Let $(M,N)$ be any marked three manifold, and $D$ be a properly embedded disk in $M$ such  that there is no intersection between $D$ and the closure of $N$.
 After removing a collar open neighborhood of $D$, we get a new three manifold $M^{'}$. And $\partial M^{'}$ contains two copies $D_1$ and 
$D_2$ of $D$ such that we can get $M$ from $M^{'}$ by gluing $D_1$ and $D_2$. We use pr to denote the obvious projection from $M^{'}$ to $M$.

Let $\alpha\subset D$ be an oriented open interval. Suppose $\text{pr}^{-1}(\alpha) = \alpha_{1}\cup \alpha_2$ with
   $\alpha_1\in D_1$ and $\alpha_2\in D_2$. We cut $(M,N)$ along $(D,\alpha)$ to  obtain a new marked three manifold   $(M^{'}, N^{'})$, where $N^{'} = N\cup \alpha_1\cup \alpha_2$. We will  denote   $(M^{'}, N^{'})$ as  Cut$_{(D,\beta)}(M, N )$.
 It is easy to see that Cut$_{(D,\alpha)}(M,N)$ is defined up to isomorphism.
If $\alpha^{'}$ is another oriented open interval in $D$, obviously we have Cut$_{(D,\alpha)}(M,N)$ is isomorphic to Cut$_{(D,\alpha^{'})}(M,N)$.
There is a linear homomorphism, called splitting map, $\Theta_{(D,\alpha)}:S_n(M,N,v)\rightarrow S_n(M^{'},N^{'},v)$ \cite{le2021stated}.  When there is no confusion we can omit the subscript for $\Theta_{(D,\alpha)}$.

\subsection{Reversing orientations of $n$-webs}

Let $\cev{\alpha}$ denote an $n$-web $\alpha$ with its orientation reversed (and unchanged framing). 
\begin{corollary}\label{c.orient-rev}(\cite{le2021stated})
$$\cev {\,\cdot\,}: S_n(M,N,v)\to S_n(M,N,v)$$
is a well defined linear automorphism.
\end{corollary}

\subsection{Punctured bordered surfaces and stated $SL(n)$-skein algebras}\label{subb2.4}

All the definitions in this subsection come from  \cite{le2021stated, leY}. 

A {\bf punctured bordered surface} $\Sigma$ is $\overline{\Sigma}-U$ where $\overline{\Sigma}$ is a compact oriented surface and $U$ is a finite set of $\overline{\Sigma}$ such that every component of $\partial \overline{\Sigma}$ intersects $U$.  
 For simplicity, we will call a punctured bordered surface a {\bf pb surface}.

The points in $U$ are called ideal points. An embedded smooth curve in $\overline{\Sigma}$ connecting the two points in $U$ (these two points could be the same point) is called an {\bf ideal arc}.  

An {\bf essentially bordered pb surface} is just a pb surface with nonempty boundary.

The stated $SL(n)$-skein algebra, denoted as $S_n(\Sigma,v)$, of a pb surface $\Sigma $ is defined as following: For every  component $c$ of $\partial \Sigma$, we choose a point $x_c$. Let $M = \Sigma\times[-1,1]$ and $N=\cup_{c}\, x_c \times (-1,1)$  where $c$ goes over all components of $\partial \Sigma$. Then we define $S_n(\Sigma,v) $ to be $S_n(M,N,v)$. We will call $(M,N)$  the thickening of the pb surface $\Sigma$.
Obviously $S_n(\Sigma,v) $ admits an algebra structure. For any two stated $n$-webs $l_1$ and $l_2$ in  the thickening of $\Sigma$, we define $l_1l_2\in S_n({\Sigma},v)$ to be the result of stacking $l_1$ above $l_2$.  With this algebra structure, $\cev {\,\cdot\,}: S_n(\Sigma,v)\rightarrow S_n(\Sigma,v)$ becomes an algebraic isomorphism. Any (stated) $n$-web in the thickening of $\Sigma$ can be represented by a diagram in $\Sigma$ \cite{le2021stated}.

Let $f:\Sigma_1\rightarrow \Sigma_2$ be a proper embedding for pb surfaces. Note that it is possible that $f$ maps
more than one boundary components of $\Sigma_1$ into one boundary component of $\Sigma_2$. For a boundary component $c$ of $\Sigma_2$, we can give a linear order on the set of boundary components of $\Sigma_1$ that are mapped into $b$ under $f$. We call such a linear order for $c$ a {\bf $c$-order}. If for each boundary component $c$ of $\Sigma_2$, there exists such a $c$-order, we call $f$ a  {\bf height ordered embedding}, which induces a linear map  
$f_*: S_n(\Sigma_1,v) \to S_n(\Sigma_2,v)$ \cite{le2021stated}.

\def\fT{{\mathfrak T}}

We call the closed triangle with it's three vertices removed 
  the {\bf standard ideal triangle}, denoted as  $\fT$.

An ideal {\bf triangulation} $\cE$ of a pb surface $\Sigma$ is
 (1) a  collection of  ideal arcs in ${\Sigma}$, (2) any arc in this collection does not bound a disk, and 
 any two arcs are pairwise non-isotopic and pairwise disjoint,
(3) this collection is maximal under condition (2). The ideal arcs in $\cE$ not isotopic to boundary components split ${\Sigma}$ into ideal triangles. We use $tri(\cE)$ to denote the set of these ideal triangles. Define $Int(\cE)=
\{e\in\cE\mid e\text{ is not isotopic to a boundary component}\}$.

\subsection{The splitting map for punctured bordered surface}\label{sub27}

Let $c$ be ideal arc of a pb surface $\Sigma$ such that it is contained in the interior of $\Sigma$ . After cutting $\Sigma$ along $c$, we get a new pb surface $\text{Cut}_c\,{\Sigma}$, which has two copies $c_1,c_2$ for $c$ such that 
${\Sigma}= \text{Cut}_c\,{\Sigma}/(c_1=c_2)$. We use $\pr$ to denote the projection from $\text{Cut}_c\,{\Sigma}$ to $\Sigma$.  Suppose $\alpha$ is a diagram for  a stated $n$-web in $\Sigma$, which is transverse to $c$.
Let $s$ be a map from $c\cap\alpha$ to $\mathbb{J}$, and $h$ is a linear order on $c\cap\alpha$. Then there is lift diagram $\alpha(h,s)$ for a stated $n$-web in $\text{Cut}_c\,{\Sigma}$. The hights of the newly created endpoints of $\alpha(h,s)$ are induced by $h$ (via $\pr$), and the states of the newly created endpoints of $\alpha(h,s)$ are induced by $s$ (via $\pr$).
Then the splitting map is defined by 
$$\Theta_c(\alpha) =\sum_{s: \alpha \cap c \to \{1,\dots, n\}} \alpha(h, s),$$
furthermore $\Theta_c$ is an algebra homomorphism \cite{le2021stated}. When there is no confusion we can omit the subscript for $\Theta_c$.

Suppose $c$ is an  ideal arc of a pb surface $\Sigma$ such that it is contained in the interior of $\Sigma$. Let $\cV\subset c$ be a subset of $c$ consisting of finite points, and let $\hat{\Sigma} = \Sigma\setminus \cV$.
  Then $c\setminus \cV =\cup_{i=1}^k c_i$ is the disjoint union of ideal arcs $c_i$ of $\hat{\Sigma}$. Let $\Sigma^{'}$ be the result of cutting $\Sigma$ along $c$, and $\hat{\Sigma}^{'}$ be the result of splitting $\hat{\Sigma}$ along all $c_i$. We have natural embeddings $\iota: \hat{\Sigma} \rightarrow \Sigma$ and $\iota: \hat{\Sigma}^{'} \rightarrow \Sigma^{'}$. 
A linear order $h$ on set $\{c_i\mid 1\leq i\leq k\}$ induces a height ordered embedding $\iota_h: \hat{\Sigma}^{'} \rightarrow \Sigma^{'}$. 

From the commutativity of the splitting maps, the composition of all  the splitting homomorphisms $\Theta_{c_i}$ can be taken in any order. We also denote this composition by $\Theta_c$.

\begin{lemma}(\cite{bloomquist2020chebyshev})\label{llmm2.2}
 With the above notations, we have the following commutative diagram:

\[\begin{tikzcd}
S_n(\Sigma,v)\arrow[d,"\Theta_c"'] & \arrow[l,"\iota_*"']S_n(\hat{\Sigma},v)\arrow[d,"\Theta_c"] \\
S_n(\Sigma^{'},v) & \arrow[l,"(\iota_h)_*"]S_n(\hat{\Sigma}^{'},v)\\
\end{tikzcd}\]

\end{lemma}

The above Lemma is a generalization for Lemma 3.6 in \cite{bloomquist2020chebyshev}, in which $h$ is the linear order induced by an orientation of $c$.

\subsection{Bigon and $O_q(SLn)$}

We refer to  \cite{KS,le2021stated,leY} for the definitions for Bigon and $O_q(SLn)$.

The {\bf  bigon} $\fB$ is obtained from a closed disk $D$ by removing two points in $\partial D$. We can label the two boundary components of a {\bf  bigon}  by $e_l$ and $e_r$. A bigon with this labeling is called  a {\bf directed bigon},
 see an example
$
\raisebox{-.20in}{

\begin{tikzpicture}
\tikzset{->-/.style=

{decoration={markings,mark=at position #1 with

{\arrow{latex}}},postaction={decorate}}}

\filldraw[draw=black,fill=gray!20] (0.5 ,0.5) circle (0.5);
\filldraw[draw=black,fill=white] (0.5,0) circle (0.05);
\filldraw[draw=black,fill=white] (0.5,1) circle (0.05);
\node [left] at(0,0.5) {$e_{l}$};
\node [right] at(1,0.5) {$e_{r}$};
\end{tikzpicture}
}
$. 
We can draw $\fB$ like $
\raisebox{-.20in}{

\begin{tikzpicture}
\tikzset{->-/.style=

{decoration={markings,mark=at position #1 with

{\arrow{latex}}},postaction={decorate}}}

\filldraw[draw=white,fill=gray!20] (0,0) rectangle (1, 1);
\draw[line width =1pt] (0,0)--(0,1);
\draw[line width =1pt] (1,0)--(1,1);
\end{tikzpicture}
}
$. 
We  use $a^i_j$
to denote 
$
\raisebox{-.20in}{

\begin{tikzpicture}
\tikzset{->-/.style=

{decoration={markings,mark=at position #1 with

{\arrow{latex}}},postaction={decorate}}}

\filldraw[draw=white,fill=gray!20] (0,0) rectangle (1, 1);
\draw [line width =1pt,decoration={markings, mark=at position 0.5 with {\arrow{>}}},postaction={decorate}](0,0.5)--(1,0.5);
\draw[line width =1pt] (0,0)--(0,1);
\draw[line width =1pt] (1,0)--(1,1);
\node [left] at(0,0.5) {$i$};
\node [right] at(1,0.5) {$j$};
\end{tikzpicture}
}
$, and use $\ca^i_j$ to denote $
\raisebox{-.20in}{

\begin{tikzpicture}
\tikzset{->-/.style=

{decoration={markings,mark=at position #1 with

{\arrow{latex}}},postaction={decorate}}}

\filldraw[draw=white,fill=gray!20] (0,0) rectangle (1, 1);
\draw [line width =1pt,decoration={markings, mark=at position 0.5 with {\arrow{<}}},postaction={decorate}](0,0.5)--(1,0.5);
\draw[line width =1pt] (0,0)--(0,1);
\draw[line width =1pt] (1,0)--(1,1);
\node [left] at(0,0.5) {$i$};
\node [right] at(1,0.5) {$j$};
\end{tikzpicture}
}
$.




We have the following coefficients
\beq
 \cR^{ij}_{lk} = q^{-\frac 1n} \left(    q^{ \delta_{i,j}} \delta_{j,k} \delta_{i,l} + (q-q^{-1})
    \delta_{j<i} \delta_{j,l} \delta_{i,k}\right),
 \label{R}
\eeq
where $\delta_{j<i}=1$ if $j<i$ and $\delta_{j<i}=0$ otherwise.

Let $O_q(M(n))$ be the associative algebra generated by  $u_{i,j}$,   $i,j\in\mathbb{J},$
subject to the relations 
\beq
(\buu \ot \buu) \cR = \cR (\buu \ot \buu),  
\eeq
where $\cR$ is the $n^2\times n^2$ matrix given by equation \eqref{R}, and $\buu \ot \buu$ is the $n^2\times n^2$ matrix with entries $(\buu \ot \buu)^{ik}_{jl} = u_{i,j} u_{k,l}$ for $i,j,k,l\in \mathbb{J}$. 
Define  the element 
$$ {\det}_q(\buu)\triangleq \sum_{\sigma\in S_n} (-q)^{\ell(\sigma)}u_{1,\sigma(1)}\cdots u_{n,\sigma(n)} = \sum_{\sigma\in S_n} (-q)^{\ell(\sigma)}u_{\sigma(1),1}\cdots u_{\sigma(n),n}.$$

Define $O_q(SLn)$ to be  $O_q(M(n))/(\det_q \buu-1).$ Then 
$O_q(SLn)$ is a Hopf algebra with the Hopf algebra structure given by
\begin{align*}
\Delta(u_{i,j}) & = \sum_{k=1}^n u_{i,k} \ot u_{k,j}, \quad  \epsilon(u_{i,j})= \delta_{i,j}.\label{eq.Deltave}\\
S({u}_{i,j} )&= (\buu^!)_{i,j} = (-q)^{i-j} {\det}_q(\buu^{j,i}).
\end{align*}
Here $\buu^{j,i}$ is the result of removing the $j$-th row and $i$-th column from $\buu$.

\begin{theorem}(\cite{le2021stated})\label{Hopf} 
(a)  $S_n(\fB,v)$ is a Hopf algebra over $\mathbb{C}$.

(b)We have a unique  Hopf algebra isomorphism  $g_{big}: O_q(SLn) \rightarrow S_n(\fB,v) $  defined by
  $ g_{big} (u_{i,j}) = a^i_j$.
\end{theorem}

\begin{lemma}\label{bigon}
The map $ \cev {\,\cdot\,}:S_n(\fB,v) \rightarrow S_n(\fB,v)$ is  a Hopf algebra isomorphism.  
\end{lemma}
\begin{proof}
Since $\cev {\,\cdot\,}$ is already an algebra homomorphism. We only need to show $\cev {\,\cdot\,}$ respects comultiplication, counit, and antipode. We only need to check this for $a^{i}_j$ because $\cev {\,\cdot\,},\Delta, \epsilon$ are algebra homomorphisms and $S$ is an anti-algebra homomorphism.

From equations (63) and (65) in \cite{le2021stated}, we know 
$$\Delta\circ \cev {\,\cdot\,} =  \cev {\,\cdot\,}\circ \Delta,\;\epsilon\circ \cev {\,\cdot\,} =  \cev {\,\cdot\,}\circ \epsilon.$$

For any $i,j\in \mathbb{J}$, we have 
$$\cev {\,\cdot\,}(S(a^i_j)) = (-q)^{i-j} (\cev {\,\cdot\,}( \cev {\,a\,}^{\bar{j}}_{\bar{i}}))
= (-q)^{i-j}( a^{\bar{j}}_{\bar{i}}),\; 
S(\cev {\,\cdot\,}(a^i_j)) = S(\cev {\,a\,}^{i}_{j}) = (-q)^{i-j} ( a^{\bar{j}}_{\bar{i}}).$$
\end{proof}

For any $i,j\in\mathbb{J}$,
We use $b_{i,j}$ to denote $\cev {\,a\,}^{i}_{j}$. Combine Theorem \ref{Hopf}  and Lemma \ref{bigon} , we have the following Theorem. The reason why we prefer $b_{i,j}$ than $a^i_j$ is because $b_{i,j}$ coincides with our later notation, which requires the orientation of $b_{i,j}$ is from $j$ to $i$.

\begin{theorem}\label{t.Hopf} 

We have a unique Hopf algebra isomorphism  $f_{big}: O_q(SLn) \rightarrow S_n(\fB,v) $  defined by
  $ f_{big} (u_{i,j}) = b_{i,j}$.

\end{theorem}

\begin{lemma}(\cite{PW})\label{mmm2.4}
Suppose $q$ is a primitive $m$-th root of unity with $m$ being odd. Then in $O_q(SLn)$, for any $1\leq i,j\leq n$, we have

(a) 
\begin{align*}
&\sum_{\sigma\in S_n} (-1)^{\ell(\sigma)} (u_{1,\sigma(1)})^{m}(u_{2,\sigma(2)})^{m}\dots (u_{n,\sigma(n)})^{m}\\=&
\sum_{\sigma\in S_n} (-1)^{\ell(\sigma)} (u_{\sigma(1),1})^{m}(u_{\sigma(2),2})^{m}\dots (u_{\sigma(n),n})^{m} = 1,
\end{align*}

(b) $$\Delta((u_{i,j})^{m} )= \sum_{1\leq k\leq n} (u_{i,k})^{m}\otimes (u_{k,j})^{m},$$

(c)  $(u_{i,j})^{m}$ is in the center of $O_q(SLn)$.
\end{lemma}

\subsection{The injectivity of the splitting map for pb surfaces}

\begin{proposition}(\cite{le2021stated})\label{inj}
Let $\Sigma$ be an essentially bordered pb surface, and   $c$ be  any  ideal arc in the interior of $\Sigma$. Then the splitting map $\Theta_c: S_n(\Sigma,v) \to S_n(\text{Cut}_c\,\Sigma,v)$ is injective.
\end{proposition}

Suppose $\Sigma$ is a pb surface, and $\cE$ is an ideal trangulation of $\Sigma$. After cutting  $\Sigma$  into a collection of ideal triangles using $\cE$, for each $e\in Int(\cE)$, $\text{Cut}_{\cE}(\Sigma)$ has a left coaction ${}_e\Delta$ and a right coaction $\Delta_e$ over the Hopf algebra $S_n(\fB,v)$, see subsection 7.1 in \cite{le2021stated}. Because of the commutativity of splitting maps, we can combine   ${}_e\Delta$ (respectively  $ \Delta_e)$ for all $e\in Int(\cE)$ in any order, which we denote as ${}_{Int(\cE)}\Delta$ (respectively  $\Delta_{( Int(\cE))}$).

We use $fl$ to denote the transposition
$$fl: (\otimes_{e\in Int(\cE)}S_n(\fB,v))\otimes ( \otimes_{\fT\in tri(\cE)} S_n(\fT,v))   \rightarrow ( \otimes_{\fT\in tri(\cE)} S_n(\fT,v))\otimes (\otimes_{e\in Int(\cE)}S_n(\fB,v))$$
defined by $fl(a\otimes b) = b\otimes a$ where $a\in \otimes_{e\in Int(\cE)}S_n(\fB,v)$ and $b\in \otimes_{\fT\in tri(\cE)} S_n(\fT,v)$.

\begin{lemma}($n=2$ \cite{korinman2019classical}, $n=3$ \cite{higgins})\label{exact}
Suppose $\Sigma$ is an essentially bordered pb surface, and $\cE$ is an ideal trangulation of $\Sigma$. Then we have the following exact sequence:
$$\begin{tikzcd}
S_n(\Sigma,v) \arrow[r, "\Theta",tail] & \otimes_{\fT\in tri(\cE)} S_n(\fT,v)
\arrow[r,"T_{\cE}" ] &( \otimes_{\fT\in tri(\cE)} S_n(\fT,v))\otimes (\otimes_{e\in Int(\cE)}S_n(\fB,v))
\end{tikzcd}$$
where $T_{\cE}= \Delta_{Int(\cE)} - fl\circ {}_{Int(\cE)}\Delta$. The arrow with a tail means the corresponding map is injective.
\end{lemma}
\begin{proof}
Theorem 8.6 in \cite{le2021stated}, and Proposition \ref{inj}.
\end{proof}

\subsection{Conventions}

For any topological space $X$, we  use $cl(X)$ to denote the closure of $X$, and use $int(X)$ to denote the interior of $X$, and use $\sharp X$ to denote the number of components of $X$.

An {\bf oriented arc} $\alpha$ in $(M,N)$ is a smooth embedding from $[0,1]$ to $M$ with the orientation given by the positive direction of $[0,1]$ such that $\alpha\cap\partial M
= \{\alpha(0),\alpha(1)\}\cap N$. A {\bf  framed oriented arc} is an oriented arc with transversal framing
such that the framings at $\{\alpha(0),\alpha(1)\}\cap N$, if not empty, are given by the velocity  vectors of $N$.
 An {\bf  oriented circle} in $(M,N)$ is a smooth embedding  $\beta: S^{1}\rightarrow M$ with a chosen orientation such that $\beta\subset int(M)$. A {\bf  framed oriented knot} is an oriented circle with transversal framing.

If both two ends of a framed oriented arc lie in $N$ 
we call it a  framed oriented boundary arc of $(M,N)$, or just {\bf framed oriented boundary arc} when there is no confusion with $(M,N)$. If the two ends of a framed oriented boundary arc are both stated, we call it a {\bf stated framed oriented boundary arc}.

 For a framed oriented boundary arc $\alpha$, we use $\alpha_{i,j}$ to denote $\alpha$ with two ends stated by $s(\alpha(0 )) = j$ and $s(\alpha(1)) = i$. Also for a stated framed oriented boundary arc $\alpha$, suppose 
$s(\alpha(0 )) = j$ and $s(\alpha(1)) = i$, we can also use $\alpha_{i,j}$ to denote $\alpha$ to indicate the information that $s(\alpha(0 )) = j$ and $s(\alpha(1)) = i$.
 
For two framed oriented arcs $\alpha,\beta$, we say $\alpha*\beta$ is well-defined if (1) $\alpha\cap \beta=\{\alpha(0) \} = \{\beta(1)\}$  and they have the same framing and velocity vector at point $\beta(1) = \alpha(0)$, or (2) $\alpha\cap \beta=\{\alpha(0),\alpha(1) \} = \{\beta(0),\beta(1)\}$, where $\alpha(0) = \beta(1)$ and
$\alpha(1) = \beta(0)$, and they have the same framings and velocity vectors at their intersecting points. And we use $\alpha*\beta$ to denote the new framed oriented arc (or framed oriented knot) obtained by connecting $\alpha$ and $\beta$ at their intersecting points. Note that it is possible that $\alpha*\beta$ is not a well-defined framed oriented arc (or framed oriented knot)
 because the intersecting points could be contained in $N$. If this happens, we just isotopicly push  the parts nearby 
intersecting points  to the inside of $M$. Then we obtain a well-defined framed oriented arc (or framed oriented knot), which is still denoted as $\alpha*\beta$.
Note that for case (2), we have both $\alpha*\beta$ and $\beta*\alpha$ are well-defined and they represent the same framed oriented
knot.

Suppose $R_1,R_2$ are two algebras, and $f$ is map from $R_1$ to $R_2$. Let $A=(a_{i,j})$ be a $k_1$ by $k_2$ matrix in $R_1$ where $k_1,k_2$ are two positive integers. We define $f(A)$ to be a $k_1$ by $k_2$ matrix in $R_2$
with $[f(A)]_{i,j} = f(a_{i,j})$. If $f$ is an algebra homomorphism, we have 
$f(I)= I, f(A_1 A_2) = f(A_1)f(A_2), f(A_3+A_4) = f(A_3) + f(A_4), f(A_5^{-1}) = (f(A_5))^{-1}$
where $I$ is the identity matrix in any size and $A_t,1\leq t\leq 5,$ are matrices in $R_1$ such that the above operations for $A_t$ make sense.

In this paper, when we talk about three manifold, we always mean a three manifold with a chosen orientation and a chosen
 Riemannian metric.

\section{Algebra structure for $S_n(M,N,1)$ and the coordinate ring}\label{sec3}

In this section we will give an algebra structure to $S_n(M,N,1)$, and introduce an algebraic set related to $(M,N)$. The main goal of this section is to construct a surjective algebra homomorphism from $S_n(M,N,1)$ to the coordinate ring of this algebraic set. In next section, we will prove the well-definedness and surjectivity of this algebra homomorphism.

Recall that for any positive integer $k$, we define 
$$[k]=\frac{q^k-q^{-k}}{q-q^{-1}}.$$

\begin{lemma}
For any positive integer $k$, we have 
$$\sum_{\sigma\in S_k} (q^2)^{\ell(\sigma)} = [k]! q^{\frac{k(k-1)}{2}}$$
where $[k]! =\prod_{1\leq i\leq k}[i]$.
\end{lemma}
\begin{proof}
We prove this by using mathmatical induction on $k$. Obviously it is true for $k=1$. 

Suppose we have $\sum_{\sigma\in S_k} (q^2)^{\ell(\sigma)} = [k]! q^{\frac{k(k-1)}{2}}.$ Then 
\begin{align*}
&\sum_{\sigma\in S_{k+1}} (q^2)^{\ell(\sigma)} = \sum_{1\leq i\leq k+1}\left (
\sum_{\sigma\in S_{k+1},\sigma(k+1) 
=i} q^{2\ell(\sigma)}\right ) \\
=&(\sum_{1\leq i\leq k+1} q^{2(k+1-i)}) [k]! q^{\frac{k(k-1)}{2}}=
[k+1]! q^{\frac{(k+1)(k)}{2}}.
\end{align*}
\end{proof}

\subsection{Hight exchange relations for boundary arcs with opposite orientations}  In this subsection, we try to derive the relation between $\raisebox{-.20in}{

\begin{tikzpicture}
\tikzset{->-/.style=

{decoration={markings,mark=at position #1 with

{\arrow{latex}}},postaction={decorate}}}
\filldraw[draw=white,fill=gray!20] (-0,-0.2) rectangle (1, 1.2);
\draw [line width =1.5pt,decoration={markings, mark=at position 1 with {\arrow{>}}},postaction={decorate}](1,-0.2)--(1,1.2);
\draw [line width =1pt,decoration={markings, mark=at position 0.5 with {\arrow{>}}},postaction={decorate}](0,0.8)--(1,0.8);
\draw [line width =1pt,decoration={markings, mark=at position 0.5 with {\arrow{<}}},postaction={decorate}](0,0.2)--(1,0.2);
\node [right]  at(1,0.8) {$j$};
\node [right]  at(1,0.2) {$i$};
\end{tikzpicture}
}$ and $\raisebox{-.20in}{

\begin{tikzpicture}
\tikzset{->-/.style=

{decoration={markings,mark=at position #1 with

{\arrow{latex}}},postaction={decorate}}}
\filldraw[draw=white,fill=gray!20] (-0,-0.2) rectangle (1, 1.2);
\draw [line width =1.5pt,decoration={markings, mark=at position 1 with {\arrow{>}}},postaction={decorate}](1,1.2)--(1,-0.2);
\draw [line width =1pt,decoration={markings, mark=at position 0.5 with {\arrow{>}}},postaction={decorate}](0,0.8)--(1,0.8);
\draw [line width =1pt,decoration={markings, mark=at position 0.5 with {\arrow{<}}},postaction={decorate}](0,0.2)--(1,0.2);
\node [right]  at(1,0.8) {$j$};
\node [right]  at(1,0.2) {$i$};
\end{tikzpicture}
}$
(between $\raisebox{-.20in}{

\begin{tikzpicture}
\tikzset{->-/.style=

{decoration={markings,mark=at position #1 with

{\arrow{latex}}},postaction={decorate}}}
\filldraw[draw=white,fill=gray!20] (-0,-0.2) rectangle (1, 1.2);
\draw [line width =1.5pt,decoration={markings, mark=at position 1 with {\arrow{>}}},postaction={decorate}](1,-0.2)--(1,1.2);
\draw [line width =1pt,decoration={markings, mark=at position 0.5 with {\arrow{<}}},postaction={decorate}](0,0.8)--(1,0.8);
\draw [line width =1pt,decoration={markings, mark=at position 0.5 with {\arrow{>}}},postaction={decorate}](0,0.2)--(1,0.2);
\node [right]  at(1,0.8) {$j$};
\node [right]  at(1,0.2) {$i$};
\end{tikzpicture}
}$ and $\raisebox{-.20in}{

\begin{tikzpicture}
\tikzset{->-/.style=

{decoration={markings,mark=at position #1 with

{\arrow{latex}}},postaction={decorate}}}
\filldraw[draw=white,fill=gray!20] (-0,-0.2) rectangle (1, 1.2);
\draw [line width =1.5pt,decoration={markings, mark=at position 1 with {\arrow{>}}},postaction={decorate}](1,1.2)--(1,-0.2);
\draw [line width =1pt,decoration={markings, mark=at position 0.5 with {\arrow{<}}},postaction={decorate}](0,0.8)--(1,0.8);
\draw [line width =1pt,decoration={markings, mark=at position 0.5 with {\arrow{>}}},postaction={decorate}](0,0.2)--(1,0.2);
\node [right]  at(1,0.8) {$j$};
\node [right]  at(1,0.2) {$i$};
\end{tikzpicture}
}$).

\begin{proposition}\label{prop3.1}
Let $(M,N)$ be any marked three manifold with $N\neq \emptyset$. In $S_n(M,N,v)$, we have 

\begin{align*}
\raisebox{-.20in}{
\begin{tikzpicture}
\tikzset{->-/.style=
{decoration={markings,mark=at position #1 with
{\arrow{latex}}},postaction={decorate}}}
\filldraw[draw=white,fill=gray!20] (-0,-0.2) rectangle (1, 1.2);
\draw [line width =1.5pt,decoration={markings, mark=at position 1 with {\arrow{>}}},postaction={decorate}](1,-0.2)--(1,1.2);
\draw [line width =1pt,decoration={markings, mark=at position 0.5 with {\arrow{>}}},postaction={decorate}](0,0.8)--(1,0.8);
\draw [line width =1pt,decoration={markings, mark=at position 0.5 with {\arrow{<}}},postaction={decorate}](0,0.2)--(1,0.2);
\node [right]  at(1,0.8) {$j$};
\node [right]  at(1,0.2) {$i$};
\end{tikzpicture}}
=
\raisebox{-.20in}{
\begin{tikzpicture}
\tikzset{->-/.style=
{decoration={markings,mark=at position #1 with
{\arrow{latex}}},postaction={decorate}}}
\filldraw[draw=white,fill=gray!20] (-0,-0.2) rectangle (1, 1.2);
\draw [line width =1.5pt,decoration={markings, mark=at position 1 with {\arrow{>}}},postaction={decorate}](1,1.2)--(1,-0.2);
\draw [line width =1pt](0.6,0.6)--(1,1);
\draw [line width =1pt](0.6,0.4)--(1,0);
\draw[line width =1pt,decoration={markings, mark=at position 0.5 with {\arrow{<}}},postaction={decorate}] (0,0)--(0.4,0.4);
\draw[line width =1pt,decoration={markings, mark=at position 0.5 with {\arrow{>}}},postaction={decorate}] (0,1)--(0.4,0.6);
\draw[line width =1pt] (0.4,0.6)--(0.6,0.4);
\node [right]  at(1,1) {$i$};
\node [right]  at(1,0) {$j$};
\end{tikzpicture}}
&= \left \{
 \begin{array}{ll}
     q^{\frac{1-n}{n}}
\left(
\raisebox{-.20in}{
\begin{tikzpicture}
\tikzset{->-/.style=
{decoration={markings,mark=at position #1 with
{\arrow{latex}}},postaction={decorate}}}
\filldraw[draw=white,fill=gray!20] (-0,-0.2) rectangle (1, 1.2);
\draw [line width =1.5pt,decoration={markings, mark=at position 1 with {\arrow{>}}},postaction={decorate}](1,1.2)--(1,-0.2);
\draw [line width =1pt,decoration={markings, mark=at position 0.5 with {\arrow{>}}},postaction={decorate}](0,0.8)--(1,0.8);
\draw [line width =1pt,decoration={markings, mark=at position 0.5 with {\arrow{<}}},postaction={decorate}](0,0.2)--(1,0.2);
\node [right]  at(1,0.8) {$j$};
\node [right]  at(1,0.2) {$i$};
\end{tikzpicture}}
+c_i (1-q^2)\sum_{j<k\leq n}
c_{\bar{k}}^{-1}
\raisebox{-.20in}{
\begin{tikzpicture}
\tikzset{->-/.style=
{decoration={markings,mark=at position #1 with
{\arrow{latex}}},postaction={decorate}}}
\filldraw[draw=white,fill=gray!20] (-0,-0.2) rectangle (1, 1.2);
\draw [line width =1.5pt,decoration={markings, mark=at position 1 with {\arrow{>}}},postaction={decorate}](1,1.2)--(1,-0.2);
\draw [line width =1pt,decoration={markings, mark=at position 0.5 with {\arrow{>}}},postaction={decorate}](0,0.8)--(1,0.8);
\draw [line width =1pt,decoration={markings, mark=at position 0.5 with {\arrow{<}}},postaction={decorate}](0,0.2)--(1,0.2);
\node [right]  at(1,0.8) {$k$};
\node [right]  at(1,0.2) {$\bar{k}$};
\end{tikzpicture}}
\right),                    & \text{if }j=\bar{i},\\
     q^{\frac{1}{n}}
\raisebox{-.20in}{
\begin{tikzpicture}
\tikzset{->-/.style=
{decoration={markings,mark=at position #1 with
{\arrow{latex}}},postaction={decorate}}}
\filldraw[draw=white,fill=gray!20] (-0,-0.2) rectangle (1, 1.2);
\draw [line width =1.5pt,decoration={markings, mark=at position 1 with {\arrow{>}}},postaction={decorate}](1,1.2)--(1,-0.2);
\draw [line width =1pt,decoration={markings, mark=at position 0.5 with {\arrow{>}}},postaction={decorate}](0,0.8)--(1,0.8);
\draw [line width =1pt,decoration={markings, mark=at position 0.5 with {\arrow{<}}},postaction={decorate}](0,0.2)--(1,0.2);
\node [right]  at(1,0.8) {$j$};
\node [right]  at(1,0.2) {$i$};
\end{tikzpicture}}
,     & \text{if }j\neq\bar{i},\\
 \end{array}
 \right.\\
\raisebox{-.20in}{
\begin{tikzpicture}
\tikzset{->-/.style=
{decoration={markings,mark=at position #1 with
{\arrow{latex}}},postaction={decorate}}}
\filldraw[draw=white,fill=gray!20] (-0,-0.2) rectangle (1, 1.2);
\draw [line width =1.5pt,decoration={markings, mark=at position 1 with {\arrow{>}}},postaction={decorate}](1,-0.2)--(1,1.2);
\draw [line width =1pt,decoration={markings, mark=at position 0.5 with {\arrow{<}}},postaction={decorate}](0,0.8)--(1,0.8);
\draw [line width =1pt,decoration={markings, mark=at position 0.5 with {\arrow{>}}},postaction={decorate}](0,0.2)--(1,0.2);
\node [right]  at(1,0.8) {$j$};
\node [right]  at(1,0.2) {$i$};
\end{tikzpicture}}
=
\raisebox{-.20in}{
\begin{tikzpicture}
\tikzset{->-/.style=
{decoration={markings,mark=at position #1 with
{\arrow{latex}}},postaction={decorate}}}
\filldraw[draw=white,fill=gray!20] (-0,-0.2) rectangle (1, 1.2);
\draw [line width =1.5pt,decoration={markings, mark=at position 1 with {\arrow{>}}},postaction={decorate}](1,1.2)--(1,-0.2);
\draw [line width =1pt](0.6,0.6)--(1,1);
\draw [line width =1pt](0.6,0.4)--(1,0);
\draw[line width =1pt,decoration={markings, mark=at position 0.5 with {\arrow{>}}},postaction={decorate}] (0,0)--(0.4,0.4);
\draw[line width =1pt,decoration={markings, mark=at position 0.5 with {\arrow{<}}},postaction={decorate}] (0,1)--(0.4,0.6);
\draw[line width =1pt] (0.4,0.6)--(0.6,0.4);
\node [right]  at(1,1) {$i$};
\node [right]  at(1,0) {$j$};
\end{tikzpicture}}
&= \left \{
 \begin{array}{ll}
     q^{\frac{1-n}{n}}
\left(
\raisebox{-.20in}{
\begin{tikzpicture}
\tikzset{->-/.style=
{decoration={markings,mark=at position #1 with
{\arrow{latex}}},postaction={decorate}}}
\filldraw[draw=white,fill=gray!20] (-0,-0.2) rectangle (1, 1.2);
\draw [line width =1.5pt,decoration={markings, mark=at position 1 with {\arrow{>}}},postaction={decorate}](1,1.2)--(1,-0.2);
\draw [line width =1pt,decoration={markings, mark=at position 0.5 with {\arrow{<}}},postaction={decorate}](0,0.8)--(1,0.8);
\draw [line width =1pt,decoration={markings, mark=at position 0.5 with {\arrow{>}}},postaction={decorate}](0,0.2)--(1,0.2);
\node [right]  at(1,0.8) {$j$};
\node [right]  at(1,0.2) {$i$};
\end{tikzpicture}}
+c_i (1-q^2)\sum_{j<k\leq n}
c_{\bar{k}}^{-1}
\raisebox{-.20in}{
\begin{tikzpicture}
\tikzset{->-/.style=
{decoration={markings,mark=at position #1 with
{\arrow{latex}}},postaction={decorate}}}
\filldraw[draw=white,fill=gray!20] (-0,-0.2) rectangle (1, 1.2);
\draw [line width =1.5pt,decoration={markings, mark=at position 1 with {\arrow{>}}},postaction={decorate}](1,1.2)--(1,-0.2);
\draw [line width =1pt,decoration={markings, mark=at position 0.5 with {\arrow{<}}},postaction={decorate}](0,0.8)--(1,0.8);
\draw [line width =1pt,decoration={markings, mark=at position 0.5 with {\arrow{>}}},postaction={decorate}](0,0.2)--(1,0.2);
\node [right]  at(1,0.8) {$k$};
\node [right]  at(1,0.2) {$\bar{k}$};
\end{tikzpicture}}
\right),                    & \text{if }j=\bar{i},\\
     q^{\frac{1}{n}}
\raisebox{-.20in}{
\begin{tikzpicture}
\tikzset{->-/.style=
{decoration={markings,mark=at position #1 with
{\arrow{latex}}},postaction={decorate}}}
\filldraw[draw=white,fill=gray!20] (-0,-0.2) rectangle (1, 1.2);
\draw [line width =1.5pt,decoration={markings, mark=at position 1 with {\arrow{>}}},postaction={decorate}](1,1.2)--(1,-0.2);
\draw [line width =1pt,decoration={markings, mark=at position 0.5 with {\arrow{<}}},postaction={decorate}](0,0.8)--(1,0.8);
\draw [line width =1pt,decoration={markings, mark=at position 0.5 with {\arrow{>}}},postaction={decorate}](0,0.2)--(1,0.2);
\node [right]  at(1,0.8) {$j$};
\node [right]  at(1,0.2) {$i$};
\end{tikzpicture}}
,     & \text{if }j\neq\bar{i},\\
 \end{array}
 \right.\\
\end{align*}

\end{proposition}
\begin{proof}
We only prove the first equation. The second equation can be obtained by reversing all the orientations of $n$-webs in the first equation.

We have
\begin{equation}\label{eq2.1}
\raisebox{-.20in}{
\begin{tikzpicture}
\tikzset{->-/.style=
{decoration={markings,mark=at position #1 with
{\arrow{latex}}},postaction={decorate}}}
\filldraw[draw=white,fill=gray!20] (-0,-0.2) rectangle (1, 1.2);
\draw [line width =1.5pt,decoration={markings, mark=at position 1 with {\arrow{>}}},postaction={decorate}](1,-0.2)--(1,1.2);
\draw [line width =1pt,decoration={markings, mark=at position 0.5 with {\arrow{>}}},postaction={decorate}](0,0.8)--(1,0.8);
\draw [line width =1pt,decoration={markings, mark=at position 0.5 with {\arrow{<}}},postaction={decorate}](0,0.2)--(1,0.2);
\node [right]  at(1,0.8) {$j$};
\node [right]  at(1,0.2) {$i$};
\end{tikzpicture}}
=
\raisebox{-.20in}{
\begin{tikzpicture}
\tikzset{->-/.style=
{decoration={markings,mark=at position #1 with
{\arrow{latex}}},postaction={decorate}}}
\filldraw[draw=white,fill=gray!20] (-0,-0.2) rectangle (1, 1.2);
\draw [line width =1.5pt,decoration={markings, mark=at position 1 with {\arrow{>}}},postaction={decorate}](1,1.2)--(1,-0.2);
\draw [line width =1pt](0.6,0.6)--(1,1);
\draw [line width =1pt](0.6,0.4)--(1,0);
\draw[line width =1pt,decoration={markings, mark=at position 0.5 with {\arrow{<}}},postaction={decorate}] (0,0)--(0.4,0.4);
\draw[line width =1pt,decoration={markings, mark=at position 0.5 with {\arrow{>}}},postaction={decorate}] (0,1)--(0.4,0.6);
\draw[line width =1pt] (0.4,0.6)--(0.6,0.4);
\node [right]  at(1,1) {$i$};
\node [right]  at(1,0) {$j$};
\end{tikzpicture}}
=  \delta_{\bar j,i } c_i q^{\frac{1-n}{n}} \; 
\raisebox{-.20in}{
\begin{tikzpicture}
\tikzset{->-/.style=
{decoration={markings,mark=at position #1 with
{\arrow{latex}}},postaction={decorate}}}
\filldraw[draw=white,fill=gray!20] (-0.7,-0.7) rectangle (0,0.7);
\draw [line width =1.5pt](0,0.7)--(0,-0.7);
\draw [color = black, line width =1pt] (-0.7 ,-0.3) arc (-90:90:0.5 and 0.3);
\draw [line width =1pt,decoration={markings, mark=at position 0.5 with {\arrow{<}}},postaction={decorate}](-0.2,-0.02)--(-0.2,0.02);
\end{tikzpicture}}
-(-1)^{\frac{n(n-1)}{2}} ([n-2]!)^{-1}q^{\frac{1}{n}}
\raisebox{-.60in}{

\begin{tikzpicture}
\tikzset{->-/.style=

{decoration={markings,mark=at position #1 with

{\arrow{latex}}},postaction={decorate}}}

\filldraw[draw=white,fill=gray!20] (0,0) rectangle (2.6, 4);
\draw[line width =2pt,decoration={markings, mark=at position 1.0 with {\arrow{>}}},postaction={decorate}](2.6,4) --(2.6,0);
\draw [line width =0.8pt,decoration={markings, mark=at position 0.5 with {\arrow{>}}},postaction={decorate}](1.3,1)--(0,0.5);
\draw[line width =0.8pt,decoration={markings, mark=at position 0.5 with {\arrow{>}}},postaction={decorate}] (1.3,1)--(2.6,0.5)node[right]{{$j$}};
\draw [line width =0.8pt,decoration={markings, mark=at position 1.0 with {\arrow{>}}},postaction={decorate}](1.3,1)--(1.3,2) node[right]{{$n-2$}};
\draw [line width =0.8pt](1.3,2)--(1.3,3);
\draw [line width =0.8pt,decoration={markings, mark=at position 0.5 with {\arrow{>}}},postaction={decorate}](0,3.5)--(1.3,3);
\draw[line width =0.8pt,decoration={markings, mark=at position 0.5 with {\arrow{<}}},postaction={decorate}] (1.3,3)--(2.6,3.5) node[right]{{$i$}};
\end{tikzpicture}
}\hspace*{-.1in}.
\end{equation}
The second equality is because of the parallel equation of equation (51) in \cite{le2021stated}.
We have
$$
\raisebox{-.60in}{

\begin{tikzpicture}
\tikzset{->-/.style=

{decoration={markings,mark=at position #1 with

{\arrow{latex}}},postaction={decorate}}}

\filldraw[draw=white,fill=gray!20] (0,0) rectangle (2.6, 4);
\draw[line width =2pt,decoration={markings, mark=at position 1.0 with {\arrow{>}}},postaction={decorate}](2.6,4) --(2.6,0);
\draw [line width =0.8pt,decoration={markings, mark=at position 0.5 with {\arrow{>}}},postaction={decorate}](1.3,1)--(0,0.5);
\draw[line width =0.8pt,decoration={markings, mark=at position 0.5 with {\arrow{>}}},postaction={decorate}] (1.3,1)--(2.6,0.5)node[right]{{$j$}};
\draw [line width =0.8pt,decoration={markings, mark=at position 1.0 with {\arrow{>}}},postaction={decorate}](1.3,1)--(1.3,2) node[right]{{$n-2$}};
\draw [line width =0.8pt](1.3,2)--(1.3,3);
\draw [line width =0.8pt,decoration={markings, mark=at position 0.5 with {\arrow{>}}},postaction={decorate}](0,3.5)--(1.3,3);
\draw[line width =0.8pt,decoration={markings, mark=at position 0.5 with {\arrow{<}}},postaction={decorate}] (1.3,3)--(2.6,3.5) node[right]{{$i$}};
\end{tikzpicture}
}\hspace*{-.1in}=a^2
\sum_{\tau,\sigma\in S_n}  (-q)^{\ell(\tau)}(-q)^{\ell(\sigma)}
\raisebox{-.99in}{

\begin{tikzpicture}
\tikzset{->-/.style=

{decoration={markings,mark=at position #1 with

{\arrow{latex}}},postaction={decorate}}}

\filldraw[draw=white,fill=gray!20] (0,0.4) rectangle (2, 7.6);
\draw[line width =2pt,decoration={markings, mark=at position 1.0 with {\arrow{>}}},postaction={decorate}](2,7.6) --(2,0.4);
\draw [line width =0.8pt,decoration={markings, mark=at position 0.8 with {\arrow{<}}},postaction={decorate}](2,4.4) arc (90:270:0.8 and 0.4) node[right]{\small{$\tau(n)$}};
\draw [line width =0.8pt,decoration={markings, mark=at position 0.8 with {\arrow{<}}},postaction={decorate}](2,5.2) arc (90:270:1.2 and 1.2) node[right]{\small{$\tau(4)$}};
\draw [line width =0.8pt,decoration={markings, mark=at position 0.8 with {\arrow{<}}},postaction={decorate}](2,5.6) arc (90:270:1.4 and 1.6) node[right]{\small{$\tau(3)$}};
\draw [line width =0.8pt,decoration={markings, mark=at position 0.5 with {\arrow{<}}},postaction={decorate}](0,2)--(2,2)node[right]{\small{$\tau(2)$}};
\draw [line width =0.8pt,decoration={markings, mark=at position 0.5 with {\arrow{>}}},postaction={decorate}](0,6)--(2,6);
\draw [line width =0.8pt,decoration={markings, mark=at position 0.2 with {\arrow{>}}},postaction={decorate}](2,1.6) arc (90:270:0.8 and 0.4) node[right]{\small{$j$}};
\draw [line width =0.8pt,decoration={markings, mark=at position 0.8 with {\arrow{>}}},postaction={decorate}](2,7.2) arc (90:270:0.8 and 0.4) node[right]{\small{$\sigma(n)$}};
\node at(2.5,4.4) {\small $\sigma(1)$};
\node[right] at(2,5.2) {\small $\sigma(n-3)$};
\node [right] at(2,5.6) {\small $\sigma(n-2)$};
\node [right] at(2,6) {\small $\sigma(n-1)$};
\node [right] at(2,7.2) {\small $i$};
\node at(2.5,1.6) {\small $\tau(1)$};

\node at(1.7,4.8) {\vdots};
\end{tikzpicture}
}\hspace*{-.1in}$$
$$
=a^2 t^{\frac{n}{2}}
\sum_{\substack{\tau(1)=\bar{j},\sigma(n)=\bar{i}\\ \overline{\sigma(1)}=\tau(n),\dots,\overline{\sigma(n-2)}=\tau(3),\sigma,\tau\in S_n   }}(-q)^{\ell(\sigma)+\ell(\tau)} c_i c_{\tau(1)}c_{\sigma(n-1)}^{-1}c_{\sigma(n)}^{-1}
\;
\raisebox{-.20in}{

\begin{tikzpicture}
\tikzset{->-/.style=

{decoration={markings,mark=at position #1 with

{\arrow{latex}}},postaction={decorate}}}
\filldraw[draw=white,fill=gray!20] (-0,-0.2) rectangle (1, 1.2);
\draw [line width =1.5pt,decoration={markings, mark=at position 1 with {\arrow{>}}},postaction={decorate}](1,1.2)--(1,-0.2);
\draw [line width =1pt,decoration={markings, mark=at position 0.5 with {\arrow{>}}},postaction={decorate}](0,0.8)--(1,0.8);
\draw [line width =1pt,decoration={markings, mark=at position 0.5 with {\arrow{<}}},postaction={decorate}](0,0.2)--(1,0.2);
\node [right]  at(1,0.8) {$\sigma(n-1)$};
\node [right]  at(1,0.2) {$\tau(2)$};
\end{tikzpicture}
}
$$ where the first equality comes from relation \eqref{wzh.five}.
Then $j\in \{\sigma(n-1), \sigma(n)\}, i\in\{\tau(1),\tau(2)\},$\\$\{\tau(1),\tau(2)\}=\{\overline{\sigma(n-1)},\overline{\sigma(n)}\}$.

Case 1 when $i\neq \overline{j}$. Then $\tau(1) = \overline{j} \neq i =\overline{\sigma(n)}$, furthermore we have 
$\tau(1) = \overline{\sigma(n-1)}=\overline{j}$ and $\tau(2) = \overline{\sigma(n)} = i$.
Then we have 
\begin{align*}
&a^2 t^{\frac{n}{2}}
\sum_{\substack{\tau(1)=\bar{j},\sigma(n)=\bar{i}\\ \overline{\sigma(1)}=\tau(n),\dots,\overline{\sigma(n-2)}=\tau(3),\sigma,\tau\in S_n   }}(-q)^{\ell(\sigma)+\ell(\tau)} c_i c_{\tau(1)}c_{\sigma(n-1)}^{-1}c_{\sigma(n)}^{-1}
\;
\raisebox{-.20in}{
\begin{tikzpicture}
\tikzset{->-/.style=
{decoration={markings,mark=at position #1 with
{\arrow{latex}}},postaction={decorate}}}
\filldraw[draw=white,fill=gray!20] (-0,-0.2) rectangle (1, 1.2);
\draw [line width =1.5pt,decoration={markings, mark=at position 1 with {\arrow{>}}},postaction={decorate}](1,1.2)--(1,-0.2);
\draw [line width =1pt,decoration={markings, mark=at position 0.5 with {\arrow{>}}},postaction={decorate}](0,0.8)--(1,0.8);
\draw [line width =1pt,decoration={markings, mark=at position 0.5 with {\arrow{<}}},postaction={decorate}](0,0.2)--(1,0.2);
\node [right]  at(1,0.8) {$\sigma(n-1)$};
\node [right]  at(1,0.2) {$\tau(2)$};
\end{tikzpicture}
}
\\
=&a^2 t^{\frac{n}{2}}c_i^{2}c_j^{-2}
\sum_{\substack{\tau(1) = \overline{\sigma(n-1)}=\overline{j},\tau(2) = \overline{\sigma(n)} = i\\ \overline{\sigma(1)}=\tau(n),\dots,\overline{\sigma(n-2)}=\tau(3),\sigma,\tau\in S_n   }}(-q)^{l(\sigma)+l(\tau)} 
\;
\raisebox{-.20in}{
\begin{tikzpicture}
\tikzset{->-/.style=
{decoration={markings,mark=at position #1 with
{\arrow{latex}}},postaction={decorate}}}
\filldraw[draw=white,fill=gray!20] (-0,-0.2) rectangle (1, 1.2);
\draw [line width =1.5pt,decoration={markings, mark=at position 1 with {\arrow{>}}},postaction={decorate}](1,1.2)--(1,-0.2);
\draw [line width =1pt,decoration={markings, mark=at position 0.5 with {\arrow{>}}},postaction={decorate}](0,0.8)--(1,0.8);
\draw [line width =1pt,decoration={markings, mark=at position 0.5 with {\arrow{<}}},postaction={decorate}](0,0.2)--(1,0.2);
\node [right]  at(1,0.8) {$j$};
\node [right]  at(1,0.2) {$i$};
\end{tikzpicture}
}
\\
=&a^2 t^{\frac{n}{2}}c_i^{2}c_j^{-2}
(-q)^{2n + 2i -2j -3}[n-2]!q^{\frac{(n-2)(n-3)}{2}}
\;\raisebox{-.20in}{
\begin{tikzpicture}
\tikzset{->-/.style=
{decoration={markings,mark=at position #1 with
{\arrow{latex}}},postaction={decorate}}}
\filldraw[draw=white,fill=gray!20] (-0,-0.2) rectangle (1, 1.2);
\draw [line width =1.5pt,decoration={markings, mark=at position 1 with {\arrow{>}}},postaction={decorate}](1,1.2)--(1,-0.2);
\draw [line width =1pt,decoration={markings, mark=at position 0.5 with {\arrow{>}}},postaction={decorate}](0,0.8)--(1,0.8);
\draw [line width =1pt,decoration={markings, mark=at position 0.5 with {\arrow{<}}},postaction={decorate}](0,0.2)--(1,0.2);
\node [right]  at(1,0.8) {$j$};
\node [right]  at(1,0.2) {$i$};
\end{tikzpicture}
}\\
=&-(-1)^{\frac{n(n-1)}{2}}[n-2]!
\;
\raisebox{-.20in}{
\begin{tikzpicture}
\tikzset{->-/.style=
{decoration={markings,mark=at position #1 with
{\arrow{latex}}},postaction={decorate}}}
\filldraw[draw=white,fill=gray!20] (-0,-0.2) rectangle (1, 1.2);
\draw [line width =1.5pt,decoration={markings, mark=at position 1 with {\arrow{>}}},postaction={decorate}](1,1.2)--(1,-0.2);
\draw [line width =1pt,decoration={markings, mark=at position 0.5 with {\arrow{>}}},postaction={decorate}](0,0.8)--(1,0.8);
\draw [line width =1pt,decoration={markings, mark=at position 0.5 with {\arrow{<}}},postaction={decorate}](0,0.2)--(1,0.2);
\node [right]  at(1,0.8) {$j$};
\node [right]  at(1,0.2) {$i$};
\end{tikzpicture}
}.
\end{align*}
From equation (\ref{eq2.1}),we have 
$\raisebox{-.20in}{
\begin{tikzpicture}
\tikzset{->-/.style=
{decoration={markings,mark=at position #1 with
{\arrow{latex}}},postaction={decorate}}}
\filldraw[draw=white,fill=gray!20] (-0,-0.2) rectangle (1, 1.2);
\draw [line width =1.5pt,decoration={markings, mark=at position 1 with {\arrow{>}}},postaction={decorate}](1,-0.2)--(1,1.2);
\draw [line width =1pt,decoration={markings, mark=at position 0.5 with {\arrow{>}}},postaction={decorate}](0,0.8)--(1,0.8);
\draw [line width =1pt,decoration={markings, mark=at position 0.5 with {\arrow{<}}},postaction={decorate}](0,0.2)--(1,0.2);
\node [right]  at(1,0.8) {$j$};
\node [right]  at(1,0.2) {$i$};
\end{tikzpicture}}
=
\raisebox{-.20in}{
\begin{tikzpicture}
\tikzset{->-/.style=
{decoration={markings,mark=at position #1 with
{\arrow{latex}}},postaction={decorate}}}
\filldraw[draw=white,fill=gray!20] (-0,-0.2) rectangle (1, 1.2);
\draw [line width =1.5pt,decoration={markings, mark=at position 1 with {\arrow{>}}},postaction={decorate}](1,1.2)--(1,-0.2);
\draw [line width =1pt](0.6,0.6)--(1,1);
\draw [line width =1pt](0.6,0.4)--(1,0);
\draw[line width =1pt,decoration={markings, mark=at position 0.5 with {\arrow{<}}},postaction={decorate}] (0,0)--(0.4,0.4);
\draw[line width =1pt,decoration={markings, mark=at position 0.5 with {\arrow{>}}},postaction={decorate}] (0,1)--(0.4,0.6);
\draw[line width =1pt] (0.4,0.6)--(0.6,0.4);
\node [right]  at(1,1) {$i$};
\node [right]  at(1,0) {$j$};
\end{tikzpicture}}=  q^{\frac{1}{n}}\raisebox{-.20in}{
\begin{tikzpicture}
\tikzset{->-/.style=
{decoration={markings,mark=at position #1 with
{\arrow{latex}}},postaction={decorate}}}
\filldraw[draw=white,fill=gray!20] (-0,-0.2) rectangle (1, 1.2);
\draw [line width =1.5pt,decoration={markings, mark=at position 1 with {\arrow{>}}},postaction={decorate}](1,1.2)--(1,-0.2);
\draw [line width =1pt,decoration={markings, mark=at position 0.5 with {\arrow{>}}},postaction={decorate}](0,0.8)--(1,0.8);
\draw [line width =1pt,decoration={markings, mark=at position 0.5 with {\arrow{<}}},postaction={decorate}](0,0.2)--(1,0.2);
\node [right]  at(1,0.8) {$j$};
\node [right]  at(1,0.2) {$i$};
\end{tikzpicture}
}$.

Case 2 when $i=\overline{j}$. We have $\tau(1) =\overline{j} = i =\overline{\sigma(n)}, 
\tau(2) = \overline{\sigma(n-1)}$, futhermore $\tau(k) = \overline{\sigma(n+1-k)}, 1\leq k\leq n$.
Then we have $\ell(\tau) = \ell(\sigma)$. 
Thus 
\begin{align*}
&a^2 t^{\frac{n}{2}}
\sum_{\substack{\tau(1)=\bar{j},\sigma(n)=\bar{i}\\ \overline{\sigma(1)}=\tau(n),\dots,\overline{\sigma(n-2)}=\tau(3),\sigma,\tau\in S_n   }}(-q)^{\ell(\sigma)+\ell(\tau)} c_i c_{\tau(1)}c_{\sigma(n-1)}^{-1}c_{\sigma(n)}^{-1}
\;
\raisebox{-.20in}{
\begin{tikzpicture}
\tikzset{->-/.style=
{decoration={markings,mark=at position #1 with
{\arrow{latex}}},postaction={decorate}}}
\filldraw[draw=white,fill=gray!20] (-0,-0.2) rectangle (1, 1.2);
\draw [line width =1.5pt,decoration={markings, mark=at position 1 with {\arrow{>}}},postaction={decorate}](1,1.2)--(1,-0.2);
\draw [line width =1pt,decoration={markings, mark=at position 0.5 with {\arrow{>}}},postaction={decorate}](0,0.8)--(1,0.8);
\draw [line width =1pt,decoration={markings, mark=at position 0.5 with {\arrow{<}}},postaction={decorate}](0,0.2)--(1,0.2);
\node [right]  at(1,0.8) {$\sigma(n-1)$};
\node [right]  at(1,0.2) {$\tau(2)$};
\end{tikzpicture}
}
\\
=&a^2 t^{\frac{n}{2}} t^{-1} c_i^{3}
\sum_{\sigma\in S_n,\sigma(n) = \bar{i}}(-q)^{2\ell(\sigma)}c_{\sigma(n-1)}^{-1}
\;
\raisebox{-.20in}{
\begin{tikzpicture}
\tikzset{->-/.style=
{decoration={markings,mark=at position #1 with
{\arrow{latex}}},postaction={decorate}}}
\filldraw[draw=white,fill=gray!20] (-0,-0.2) rectangle (1, 1.2);
\draw [line width =1.5pt,decoration={markings, mark=at position 1 with {\arrow{>}}},postaction={decorate}](1,1.2)--(1,-0.2);
\draw [line width =1pt,decoration={markings, mark=at position 0.5 with {\arrow{>}}},postaction={decorate}](0,0.8)--(1,0.8);
\draw [line width =1pt,decoration={markings, mark=at position 0.5 with {\arrow{<}}},postaction={decorate}](0,0.2)--(1,0.2);
\node [right]  at(1,0.8) {$\sigma(n-1)$};
\node [right]  at(1,0.2) {$\overline{\sigma(n-1)}$};
\end{tikzpicture}
}
\\
=&a^2 t^{\frac{n}{2}} t^{-1} c_i^{3}\sum_{1\leq k\leq n,k\neq \bar{i}}\left (
\sum_{\sigma\in S_n,\sigma(n-1) = k,\sigma(n) = \bar{i}}(-q)^{2\ell(\sigma)}c_{k}^{-1}
\;
\raisebox{-.20in}{
\begin{tikzpicture}
\tikzset{->-/.style=
{decoration={markings,mark=at position #1 with
{\arrow{latex}}},postaction={decorate}}}
\filldraw[draw=white,fill=gray!20] (-0,-0.2) rectangle (1, 1.2);
\draw [line width =1.5pt,decoration={markings, mark=at position 1 with {\arrow{>}}},postaction={decorate}](1,1.2)--(1,-0.2);
\draw [line width =1pt,decoration={markings, mark=at position 0.5 with {\arrow{>}}},postaction={decorate}](0,0.8)--(1,0.8);
\draw [line width =1pt,decoration={markings, mark=at position 0.5 with {\arrow{<}}},postaction={decorate}](0,0.2)--(1,0.2);
\node [right]  at(1,0.8) {$k$};
\node [right]  at(1,0.2) {$\bar{k}$};
\end{tikzpicture}
}
 \right )\\
=&a^2 t^{\frac{n}{2}} t^{-1} c_i^{3}  \sum_{1\leq k\leq n,k\neq \bar{i}} c_{k}^{-1}
\;
\raisebox{-.20in}{
\begin{tikzpicture}
\tikzset{->-/.style=
{decoration={markings,mark=at position #1 with
{\arrow{latex}}},postaction={decorate}}}
\filldraw[draw=white,fill=gray!20] (-0,-0.2) rectangle (1, 1.2);
\draw [line width =1.5pt,decoration={markings, mark=at position 1 with {\arrow{>}}},postaction={decorate}](1,1.2)--(1,-0.2);
\draw [line width =1pt,decoration={markings, mark=at position 0.5 with {\arrow{>}}},postaction={decorate}](0,0.8)--(1,0.8);
\draw [line width =1pt,decoration={markings, mark=at position 0.5 with {\arrow{<}}},postaction={decorate}](0,0.2)--(1,0.2);
\node [right]  at(1,0.8) {$k$};
\node [right]  at(1,0.2) {$\bar{k}$};
\end{tikzpicture}
}
 \left (
\sum_{\sigma\in S_n,\sigma(n-1) = k,\sigma(n) = \bar{i}}(-q)^{2\ell(\sigma)} \right )  \\
=&a^2 t^{\frac{n}{2}} t^{-1} c_i^{3}  \sum_{1\leq k\leq n,k\neq \bar{i}} c_{k}^{-1} [n-2]! q^{\frac{(n-2)(n-3)}{2}} q^{4n -2\bar{i} - 2} q^{-2k + 2\delta_{k>\bar{i}}}
\;
\raisebox{-.20in}{
\begin{tikzpicture}
\tikzset{->-/.style=
{decoration={markings,mark=at position #1 with
{\arrow{latex}}},postaction={decorate}}}
\filldraw[draw=white,fill=gray!20] (-0,-0.2) rectangle (1, 1.2);
\draw [line width =1.5pt,decoration={markings, mark=at position 1 with {\arrow{>}}},postaction={decorate}](1,1.2)--(1,-0.2);
\draw [line width =1pt,decoration={markings, mark=at position 0.5 with {\arrow{>}}},postaction={decorate}](0,0.8)--(1,0.8);
\draw [line width =1pt,decoration={markings, mark=at position 0.5 with {\arrow{<}}},postaction={decorate}](0,0.2)--(1,0.2);
\node [right]  at(1,0.8) {$k$};
\node [right]  at(1,0.2) {$\bar{k}$};
\end{tikzpicture}
}
  \\
=&a^2 t^{\frac{n}{2}} t^{-2} c_i^{3} [n-2]!  q^{\frac{(n-2)(n-3)}{2}} q^{4n -2\bar{i} - 2}\sum_{1\leq k\leq n,k\neq \bar{i}} c_{\bar{k}}  q^{-2k + 2\delta_{k>\bar{i}}}
\;
\raisebox{-.20in}{
\begin{tikzpicture}
\tikzset{->-/.style=
{decoration={markings,mark=at position #1 with
{\arrow{latex}}},postaction={decorate}}}
\filldraw[draw=white,fill=gray!20] (-0,-0.2) rectangle (1, 1.2);
\draw [line width =1.5pt,decoration={markings, mark=at position 1 with {\arrow{>}}},postaction={decorate}](1,1.2)--(1,-0.2);
\draw [line width =1pt,decoration={markings, mark=at position 0.5 with {\arrow{>}}},postaction={decorate}](0,0.8)--(1,0.8);
\draw [line width =1pt,decoration={markings, mark=at position 0.5 with {\arrow{<}}},postaction={decorate}](0,0.2)--(1,0.2);
\node [right]  at(1,0.8) {$k$};
\node [right]  at(1,0.2) {$\bar{k}$};
\end{tikzpicture}
}
   \\
=&a^2 t^{\frac{n}{2}} t^{-2} c_i^{3} [n-2]!  q^{\frac{(n-2)(n-3)}{2}} q^{4n -2\bar{i} - 2}
q^{-1-\frac{1}{n}}
\sum_{1\leq k\leq n,k\neq \bar{i}} c_{\bar{k}}^{-1}  q^{ 2\delta_{k>\bar{i}}}
\;
\raisebox{-.20in}{
\begin{tikzpicture}
\tikzset{->-/.style=
{decoration={markings,mark=at position #1 with
{\arrow{latex}}},postaction={decorate}}}
\filldraw[draw=white,fill=gray!20] (-0,-0.2) rectangle (1, 1.2);
\draw [line width =1.5pt,decoration={markings, mark=at position 1 with {\arrow{>}}},postaction={decorate}](1,1.2)--(1,-0.2);
\draw [line width =1pt,decoration={markings, mark=at position 0.5 with {\arrow{>}}},postaction={decorate}](0,0.8)--(1,0.8);
\draw [line width =1pt,decoration={markings, mark=at position 0.5 with {\arrow{<}}},postaction={decorate}](0,0.2)--(1,0.2);
\node [right]  at(1,0.8) {$k$};
\node [right]  at(1,0.2) {$\bar{k}$};
\end{tikzpicture}
}
  \\
=&(-1)^{\frac{n(n-1)}{2}}c_i [n-2]! q^{-1}
\sum_{1\leq k\leq n,k\neq \bar{i}} c_{\bar{k}}^{-1}  q^{ 2\delta_{k>\bar{i}}}
\;\raisebox{-.20in}{
\begin{tikzpicture}
\tikzset{->-/.style=
{decoration={markings,mark=at position #1 with
{\arrow{latex}}},postaction={decorate}}}
\filldraw[draw=white,fill=gray!20] (-0,-0.2) rectangle (1, 1.2);
\draw [line width =1.5pt,decoration={markings, mark=at position 1 with {\arrow{>}}},postaction={decorate}](1,1.2)--(1,-0.2);
\draw [line width =1pt,decoration={markings, mark=at position 0.5 with {\arrow{>}}},postaction={decorate}](0,0.8)--(1,0.8);
\draw [line width =1pt,decoration={markings, mark=at position 0.5 with {\arrow{<}}},postaction={decorate}](0,0.2)--(1,0.2);
\node [right]  at(1,0.8) {$k$};
\node [right]  at(1,0.2) {$\bar{k}$};
\end{tikzpicture}
}   .\\
\end{align*}
From equation (\ref{eq2.1}) and relation (\ref{wzh.seven}),we have 
$$\raisebox{-.20in}{
\begin{tikzpicture}
\tikzset{->-/.style=
{decoration={markings,mark=at position #1 with
{\arrow{latex}}},postaction={decorate}}}
\filldraw[draw=white,fill=gray!20] (-0,-0.2) rectangle (1, 1.2);
\draw [line width =1.5pt,decoration={markings, mark=at position 1 with {\arrow{>}}},postaction={decorate}](1,-0.2)--(1,1.2);
\draw [line width =1pt,decoration={markings, mark=at position 0.5 with {\arrow{>}}},postaction={decorate}](0,0.8)--(1,0.8);
\draw [line width =1pt,decoration={markings, mark=at position 0.5 with {\arrow{<}}},postaction={decorate}](0,0.2)--(1,0.2);
\node [right]  at(1,0.8) {$j$};
\node [right]  at(1,0.2) {$i$};
\end{tikzpicture}}
=
\raisebox{-.20in}{
\begin{tikzpicture}
\tikzset{->-/.style=
{decoration={markings,mark=at position #1 with
{\arrow{latex}}},postaction={decorate}}}
\filldraw[draw=white,fill=gray!20] (-0,-0.2) rectangle (1, 1.2);
\draw [line width =1.5pt,decoration={markings, mark=at position 1 with {\arrow{>}}},postaction={decorate}](1,1.2)--(1,-0.2);
\draw [line width =1pt](0.6,0.6)--(1,1);
\draw [line width =1pt](0.6,0.4)--(1,0);
\draw[line width =1pt,decoration={markings, mark=at position 0.5 with {\arrow{<}}},postaction={decorate}] (0,0)--(0.4,0.4);
\draw[line width =1pt,decoration={markings, mark=at position 0.5 with {\arrow{>}}},postaction={decorate}] (0,1)--(0.4,0.6);
\draw[line width =1pt] (0.4,0.6)--(0.6,0.4);
\node [right]  at(1,1) {$i$};
\node [right]  at(1,0) {$j$};
\end{tikzpicture}}=  q^{\frac{1-n}{n}}\left(\raisebox{-.20in}{
\begin{tikzpicture}
\tikzset{->-/.style=
{decoration={markings,mark=at position #1 with
{\arrow{latex}}},postaction={decorate}}}
\filldraw[draw=white,fill=gray!20] (-0,-0.2) rectangle (1, 1.2);
\draw [line width =1.5pt,decoration={markings, mark=at position 1 with {\arrow{>}}},postaction={decorate}](1,1.2)--(1,-0.2);
\draw [line width =1pt,decoration={markings, mark=at position 0.5 with {\arrow{>}}},postaction={decorate}](0,0.8)--(1,0.8);
\draw [line width =1pt,decoration={markings, mark=at position 0.5 with {\arrow{<}}},postaction={decorate}](0,0.2)--(1,0.2);
\node [right]  at(1,0.8) {$j$};
\node [right]  at(1,0.2) {$i$};
\end{tikzpicture}
}+c_i (1-q^2)\sum_{j<k\leq n}
c_{\bar{k}}^{-1}
\raisebox{-.20in}{
\begin{tikzpicture}
\tikzset{->-/.style=
{decoration={markings,mark=at position #1 with
{\arrow{latex}}},postaction={decorate}}}
\filldraw[draw=white,fill=gray!20] (-0,-0.2) rectangle (1, 1.2);
\draw [line width =1.5pt,decoration={markings, mark=at position 1 with {\arrow{>}}},postaction={decorate}](1,1.2)--(1,-0.2);
\draw [line width =1pt,decoration={markings, mark=at position 0.5 with {\arrow{>}}},postaction={decorate}](0,0.8)--(1,0.8);
\draw [line width =1pt,decoration={markings, mark=at position 0.5 with {\arrow{<}}},postaction={decorate}](0,0.2)--(1,0.2);
\node [right]  at(1,0.8) {$k$};
\node [right]  at(1,0.2) {$\bar{k}$};
\end{tikzpicture}
}\right).$$

\end{proof}

\subsection{Algebra structure for $S_n(M,N,1)$} In this subsection, we will give an algebra structure for $S_n(M,N,1)$.
To do so, first we define the multiplication of two stated $n$-webs as the disjoint union, that is, we first isotope two webs such that they have no intersection, then take their union as the product.  To prove this multiplication is well defined, we have to show the product is independent of how we take the union of these two webs. 

\begin{corollary}\label{cccc3.2}
For any marked three manifold, we have $S_n(M,N,1)$ is a commutative algebra under the above defined multiplication.
\end{corollary}
\begin{proof}
Because of relations \eqref{w.cross} and \eqref{wzh.eight}, it suffices to show  for any two states $i,j$, we have 
$$\raisebox{-.20in}{
\begin{tikzpicture}
\tikzset{->-/.style=
{decoration={markings,mark=at position #1 with
{\arrow{latex}}},postaction={decorate}}}
\filldraw[draw=white,fill=gray!20] (-0,-0.2) rectangle (1, 1.2);
\draw [line width =1.5pt,decoration={markings, mark=at position 1 with {\arrow{>}}},postaction={decorate}](1,-0.2)--(1,1.2);
\draw [line width =1pt,decoration={markings, mark=at position 0.5 with {\arrow{>}}},postaction={decorate}](0,0.8)--(1,0.8);
\draw [line width =1pt,decoration={markings, mark=at position 0.5 with {\arrow{<}}},postaction={decorate}](0,0.2)--(1,0.2);
\node [right]  at(1,0.8) {$j$};
\node [right]  at(1,0.2) {$i$};
\end{tikzpicture}
}= \raisebox{-.20in}{
\begin{tikzpicture}
\tikzset{->-/.style=
{decoration={markings,mark=at position #1 with
{\arrow{latex}}},postaction={decorate}}}
\filldraw[draw=white,fill=gray!20] (-0,-0.2) rectangle (1, 1.2);
\draw [line width =1.5pt,decoration={markings, mark=at position 1 with {\arrow{>}}},postaction={decorate}](1,1.2)--(1,-0.2);
\draw [line width =1pt,decoration={markings, mark=at position 0.5 with {\arrow{>}}},postaction={decorate}](0,0.8)--(1,0.8);
\draw [line width =1pt,decoration={markings, mark=at position 0.5 with {\arrow{<}}},postaction={decorate}](0,0.2)--(1,0.2);
\node [right]  at(1,0.8) {$j$};
\node [right]  at(1,0.2) {$i$};
\end{tikzpicture}
},\;
\raisebox{-.20in}{
\begin{tikzpicture}
\tikzset{->-/.style=
{decoration={markings,mark=at position #1 with
{\arrow{latex}}},postaction={decorate}}}
\filldraw[draw=white,fill=gray!20] (-0,-0.2) rectangle (1, 1.2);
\draw [line width =1.5pt,decoration={markings, mark=at position 1 with {\arrow{>}}},postaction={decorate}](1,-0.2)--(1,1.2);
\draw [line width =1pt,decoration={markings, mark=at position 0.5 with {\arrow{<}}},postaction={decorate}](0,0.8)--(1,0.8);
\draw [line width =1pt,decoration={markings, mark=at position 0.5 with {\arrow{>}}},postaction={decorate}](0,0.2)--(1,0.2);
\node [right]  at(1,0.8) {$j$};
\node [right]  at(1,0.2) {$i$};
\end{tikzpicture}
}=\raisebox{-.20in}{
\begin{tikzpicture}
\tikzset{->-/.style=
{decoration={markings,mark=at position #1 with
{\arrow{latex}}},postaction={decorate}}}
\filldraw[draw=white,fill=gray!20] (-0,-0.2) rectangle (1, 1.2);
\draw [line width =1.5pt,decoration={markings, mark=at position 1 with {\arrow{>}}},postaction={decorate}](1,1.2)--(1,-0.2);
\draw [line width =1pt,decoration={markings, mark=at position 0.5 with {\arrow{<}}},postaction={decorate}](0,0.8)--(1,0.8);
\draw [line width =1pt,decoration={markings, mark=at position 0.5 with {\arrow{>}}},postaction={decorate}](0,0.2)--(1,0.2);
\node [right]  at(1,0.8) {$j$};
\node [right]  at(1,0.2) {$i$};
\end{tikzpicture}
}$$
 when $v=1$, which can be eaily derived from Proposition \ref{prop3.1}.

\end{proof}

\subsection{Relative spin structure and character variety}\label{sss3.3}
Let $(M,N)$ be any marked three manifold, and $\zeta: UM\rightarrow M$ be the unit tangent bundle. We know the fiber of this bundle is $SO(3)$, whose fundamental group
is $\mathbb{Z}_{2}$. For any point $P\in M$, we use $\vartheta_P$ to denote the nontrivial element in the fundamental group of $\zeta^{-1}(P)$. We have $\vartheta_P$ is homotopic to $\vartheta_Q$ for any two points $P,Q$ in a same component of $M$. For a component $Y$ of $M$, we use $\vartheta_Y$ to denote this homotopy type. When $M$ is connected, we will use $\vartheta$ to denote this unique homotopy type.

For any component $e\in N$, $e$ has  a unique lift $\tilde{e}$ in $UM$. For any point $x\in e$, let $u_x$ be the unit velocity vector at point $x$, let $w_x$ be the unit tangent vector at $x$ such that $w_x$
is orthogonal to $\partial M$ and pointing inside of $M$. Then the orientation of $M$ determines the second unit tangent vector $v_x$ such that $(u_x, v_x, w_x)$ is the orientation of $M$. Obviously 
$\tilde{e} =\{(x,u_x,v_x,w_x)\mid x\in e\}$ is a smooth path in $UM$, which is diffeomorphic to $(0,1)$. Let $\tilde{N} =\cup_{e}\tilde{e}$ where the union takes over all components $e$ of $N$. Note that $\tilde{N}
=\emptyset$ when $N=\emptyset$.

From now on, for a topological space $X$, we will use $Com(X)$ (respectively $PCom(X)$) to denote the set of components of $X$ (the set of path connected components of $X$).

\begin{definition}
A {\bf relative spin structure} of $(M,N)$ is defined to be  a group homomorphism $h:H_1(UM,\tilde{N})\rightarrow \mathbb{Z}_2$ such that $h(\vartheta_Y)= 1$ for $Y\in Com(M)$.
\end{definition}

Note that when $N=\emptyset$, $h$ is just the usual spin structure.

Suppose $(M,N)$ is the disjoint union of $(M_1,N_1)$ and $(M_2,N_2)$. Then we have 
$H_1(UM,\tilde{N}) = H_1(UM_1,\widetilde{N_1})\oplus H_1(UM_2,\widetilde{N_2})$. For each $i=1,2$, let $h_i$
be a relative spin structure for $(M_i,N_i)$. Then $(h_1,h_2): H_1(UM,\tilde{N})\rightarrow \mathbb{Z}_2$, defined by $(h_1,h_2)(x_1,x_2) = h_1(x_1)+h_2(x_2)$, is a relative spin structure for $(M,N)$. Clearly every relative spin structure of $(M,N)$ is of  this form.

\begin{rem}{
Let $X$ be any  path  connected topological space, and $P$ be a set of finite points in $X$. We suppose 
$P = \{p_0,p_1,\dots,p_{m-1}\}$ where $m$ is a positive integer. For each $1\leq i\leq m-1$, let $\alpha_{i}$ be a path connecting $p_{0}$ and $p_{i}$. Then 
$H_1(X,P) = H_1(P)\oplus \mathbb{Z}([\alpha_1])\oplus\dots\oplus \mathbb{Z}([\alpha_{n-1}])$.

Suppose $M$ is connected.
When $N$ has only one component, we have $H_1(UM,\tilde{N}) = H_1(UM)$. Thus in this case the relative spin structure for $(M,N)$ is just the usual spin structure for $M$. When $\sharp N>1$, suppose the set of components of $N$ is
$ \{e_0,e_1,\dots, e_{m-1}\}$. For each $1\leq i\leq m-1$, let $\alpha_{i}$ be a path connecting $\widetilde{e_0}$ and $\widetilde{e_i}$. Then we have 
$H_1(UM,\tilde{N}) = H_1(UM)\oplus \mathbb{Z}([\alpha_1])\oplus\dots\oplus \mathbb{Z}([\alpha_{m-1}])$.
For any spin structure $h$ for $M$, we can extend $h$ to a relative spin structure for $(M,N)$ by defining
$h([\alpha_i]) = r_i\in \mathbb{Z}_{2}, 1\leq i\leq m-1,$ where $r_i,1\leq i\leq m-1,$ are $m-1$ arbitrary elements in $\mathbb{Z}_2$. Reversely, for any relative spin structure, we can restrict $h$ to $H_1(UM)$ to obtain a spin structure for $M$.

Suppose $N^{'}$ is obtained from $N$ ($N\neq \emptyset$) by adding one extra marking $e$ such that $cl(e)\cap cl(N) = \emptyset$. Let $\alpha$ be a path connecting $\tilde{N}$ and $\tilde{e}$ such that
$\alpha(0)\in \tilde{N}$ and $\alpha(1)\in\tilde{e}$. Then $H_1(UM,\widetilde{N^{'}}) = H_1(UM,\tilde{N})\oplus
\mathbb{Z}([\alpha])$.
Any relative spin structure $h$ for $(M,N)$ can be extended to a relative spin structure for $(M,N^{'})$ by defining $h([\alpha]) = r$ where $r$ is an arbitrary element in $\mathbb{Z}_2$.
Reversely any relative spin structure $h$  for $(M,N^{'})$ can be restricted to $H_1(UM,\tilde{N})$ bo obtain a relative spin structure for $(M,N)$.
}\end{rem}

For a path connected topological space $X$, we use
  $\pi_1(X)$ to denote the fundamental group for $X$. For  $[\alpha],[\beta] \in \pi_1(X)$, $[\alpha][\beta] = [\alpha*\beta]$ where $\alpha*\beta$ is obtained by first going through $\beta$, then going through $\alpha$. Note that here $\alpha*\beta$ is different with conventional definition.

\begin{definition}\label{df3.4}
For any path connected toplogical space $X$,
define 
$$\Gamma_n(X) = \mathbb{C}[[\alpha]_{i,j}\mid [\alpha]\in \pi_1(X),1\leq i,j\leq n]/(Q_{[\alpha]}Q_{[\beta]}
= Q_{[\alpha *\beta]}, det(Q_{[\alpha]}) = 1, Q_{[o]}=I)$$
where $[\alpha],[\beta]$ go through all elements in $\pi_1(X)$, $[o]$ is the trivial loop in $\pi_1(X)$
, $Q_{[\eta]} = ([\eta]_{i,j})_{1\leq i,j\leq n}$ for any element $[\eta]\in \pi_1(X)$.

Note that $\pi_1(X)$  has an action on $\Gamma_n(X)$, defined by 
$[\alpha]([\beta]_{i,j}) = [\alpha*\beta*\alpha^{-1}]_{i,j}$ for any $[\alpha],[\beta],1\leq i,j\leq n$.
We use $G_n(X)$ to denote the subalgebra of $\Gamma_n(X)$ fixed by this action.
\end{definition}

Note that Trace($Q_{[\alpha]})\in G_n(X)$ for any $[\alpha]\in \pi_1(X)$, and
 $G_n(X)$ is generated by Trace$(Q_{[\alpha]})$, $[\alpha]\in \pi_1(X)$ as an algebra \cite{S2001SLn}. 

\begin{rem}
We can generalize Definition \ref{df3.4} to general topological space. Let $X$ be a topological space. Suppose $PCom(X) = \{X_1,\dots,X_m\}$, then define 
$$\Gamma_n(X) = \Gamma_n(X_1)\otimes \dots \otimes \Gamma_n(X_m),\;G_n(X) = G_n(X_1)\otimes \dots \otimes G_n(X_m).$$
But $\Gamma_n(X), G_n(X)$ are only well-defined up to isomorphism, since different  order of $X_i$ give different algebras.
\end{rem}

\begin{definition}(\cite{CL2022stated})\label{bf1}
Let $X$ be a  topological space and $\{E_j\}_{j\in J}$ be disjoint contractible subspaces of $X$. The fundamental groupoid $\pi_1(X, \cup_{j\in J}E_j)$  is the groupoid (i.e. a category with invertible morphisms) whose objects are $\{E_j\}_{j\in J}$  and whose morphisms are the homotopy classes of oriented paths in $X$ with end points  in $\cup_{j\in J}E_j$. A morphism of groupoids is a functor of the corresponding categories. We can regard the group as the groupoid consisting of only one object and all group elements being all morphisms.
\end{definition}

Let $\mathcal{A},\mathcal{B}$ be two categories. We try to define a new category   $\mathcal{A}\cup \mathcal{B}$. The objects of $\mathcal{A}\cup \mathcal{B}$ is the union of objects in $\mathcal{A}$ and objects in $\mathcal{B}$.
Let $U,V$ be any two objects in $\mathcal{A}\cup \mathcal{B}$. If $U,V$ both belong to $\mathcal{A}$
(respectively $\mathcal{B}$), then we define Hom${}_{\mathcal{A}\cup \mathcal{B}}(U,V)$ to be
Hom${}_{\mathcal{A}}(U,V)$ (respectively Hom${}_{\mathcal{B}}(U,V)$). Otherwise we define Hom${}_{\mathcal{A}\cup \mathcal{B}}(U,V) = \emptyset$.

Let $X$ be a  topological space. Suppose $PCom(X) = \{X_1,\dots,X_m\}$. For each $1\leq t\leq m$, let $\{E_j\}_{j\in J_{t}}$ be disjoint contractible subspaces of $X_t$.
Obviously we have 
$$\pi_1(X,\cup_{1\leq t\leq m}(\cup_{j\in J}E_j))= \pi_1(X_1, \cup_{j\in J_1}E_j)\cup \dots\cup
\pi_1(X_m, \cup_{j\in J_m}E_j) .$$

\begin{definition}
For any marked three manifold $(M,N)$ with every component of $M$ containing at least one marking,
define $$\chi_n(M,N) = Hom(\pi_1(M,N), SL(n,\mathbb{C})),$$ and 
$$\tilde{\chi}_n(M,N) = \{\tilde{\rho}\in Hom(\pi_1(UM,\tilde{N}), SL(n,\mathbb{C}))\mid\tilde{\rho}
(\vartheta_Y) = d_n I,\;\text{for all }Y\in Com(M)\}$$ where $I$ is the identity matrix.
\end{definition}

Suppose $(M,N)$ is the disjoint union of $(M_1,N_1)$ and $(M_2,N_2)$, then we have 
$$\chi_n(M,N)\simeq \chi_n(M_1,N_1)\times \chi_n(M_2,N_2),\;\tilde{\chi}_n(M,N)
\simeq \tilde{\chi}_n(M_1,N_1)\times \tilde{\chi}_n(M_2,N_2).$$
From Lemma 8.1 in \cite{CL2022stated}, if $M$ is connected we have 
$$\chi_n(M,N)\simeq Hom(\pi_1(M), SL(n,\mathbb{C}))\times SL(n,\mathbb{C})^{\sharp N-1}.$$
 Then $\chi_n(M,N)$  is an affine algebraic set, whose 
coordinate ring is denoted as $R_n(M,N)$.

\begin{definition}\label{rrr}
When $M$ is connected and $N$ is empty, we define $\chi_n(M,N)$ to be the 
$SL(n,\mathbb{C})$-character variety of $M$.
That is $$\chi_n(M,N)=Hom(\pi_1(M), SL(n, \mathbb{C}))/\simeq$$ where $\rho\simeq \rho^{'}\in
Hom(\pi_1(M), SL(n, \mathbb{C}))$ if and only if Trace$\rho([\alpha])$=
Trace$\rho^{'}([\alpha])$ for all $[\alpha] \in Hom(\pi_1(M), SL(n, \mathbb{C}))$.

Similarly we define 
$$\tilde{\chi}_n(M,\emptyset)=\{\tilde{\rho}\in Hom(\pi_1(UM), SL(n, \mathbb{C}))\mid \tilde{\rho}(\vartheta)
= d_n I\}/\simeq$$
where the definition for $\simeq$ is the same as above (that is two elements are considered the same if and only if they have the same Trace).
\end{definition}

\begin{rem}\label{rem01}
From \cite{S2001SLn}, we know $\chi_n(M,\emptyset)$ is an affine algebraic set. We also denote it's  coordinate ring 
as $R_n(M,\emptyset)$. 
 There is a
surjective algebra homomorphism $\cY: G_n(M)\rightarrow R_n(M,\emptyset)$ defined by
$$\cY(\text{Trace}(Q_{[\alpha]}))(\rho) = \text{Trace}(\rho([\alpha])) \;\text{where}
\; [\alpha]\in \pi_1(M), \rho\in \chi_n(M,\emptyset),$$
and Ker$\cY = \sqrt{0}$.
\end{rem}

We can simply generalize  definitions for $\chi_n(M,N),\tilde{\chi}_n(M,N), R_n(M,N)$
to all marked three manifolds by taking product (or tensor product) for disjoint union. 

\begin{proposition} \label{prop3.6}
For any marked three manifold  $(M,N)$, we have $\chi_n(M,N)\simeq \tilde{\chi}_n(M,N)$.
\end{proposition}
\begin{proof}
We can assume $M$ is connected.
Here we only consider the case when $N\neq \emptyset$ since we can prove the case when $N=\emptyset$
by using the same technique.

Let $h$ be any relative spin structure for $(M,N)$, then we use $h$ to establish an isomorphism $f_h:
\chi_n(M,N)\rightarrow \tilde{\chi}_n(M,N)$. For any $\rho \in \chi_n(M,N)$, define 
$$f_h(\rho)(\tilde{\alpha})
= d_n^{h(\tilde{\alpha})}\rho(\alpha)\text{\;where\;}\tilde{\alpha} \in \pi_1(UM,\tilde{N}),\;\alpha = \zeta(\tilde{\alpha}).$$ 
Clearly $f_h(\rho)$ is a homomorphism from $\pi_1(UM,\tilde{N})$ to $SL(n,\mathbb{C})$, and $f_h(\rho)(\vartheta) = d_n^{h(\vartheta)}\rho(\zeta(\vartheta)) = d_n I$, thus $f_h\in \tilde{\chi}_n(M,N)$. 

Then we try to define 
$g_h : \tilde{\chi}_n(M,N) \rightarrow \chi_n(M,N)$, for any $\tilde{\rho}\in \tilde{\chi}_n(M,N)$ and 
$\alpha\in \pi_1(M,N)$
$$g_h(\tilde{\rho})(\alpha) = d_n^{h(\tilde{\alpha})} \tilde{\rho}(\tilde{\alpha})
\;\text{where} \;\tilde{\alpha}\in \pi_1(UM, \tilde{N})\;\text{such that}\; \zeta(\tilde{\alpha}) = \alpha.$$
Suppose $\zeta(\tilde{\alpha}) = \zeta(\tilde{\beta}) = \alpha$. Since 
$\tilde{\rho}(\vartheta) = d_n I$ and $h(\vartheta) = 1$, we have 
$$\tilde{\rho}(\tilde{\alpha}^{-1}\tilde{\beta}) = d_n^{h(\tilde{\alpha}^{-1}) +h(\tilde{\beta})}I
= d_n^{h(\tilde{\alpha})+ h(\tilde{\beta})}I.$$
Then $d_n^{h(\tilde{\beta})} \tilde{\rho}(\tilde{\beta}) = d_n^{h(\tilde{\alpha})} \tilde{\rho}(\tilde{\alpha})$, thus $g_h(\tilde{\rho})$ is well-defined. Obviously we have $g_h(\tilde{\rho}) \in \chi(M,N)$.

We have 
$$f_h(g_h(\tilde{\rho}))(\tilde{\alpha}) = d_n^{h(\tilde{\alpha})} g_h(\tilde{\rho})(\alpha)
= d_n^{h(\tilde{\alpha})}d_n^{h(\tilde{\alpha})} \tilde{\rho}(\tilde{\alpha}) =\tilde{\rho}(\tilde{\alpha}),$$
$$g_h(f_h(\rho))(\alpha) = d_n^{h(\tilde{\alpha})}f_h(\rho)(\tilde{\alpha}) =
d_n^{h(\tilde{\alpha})}d_n^{h(\tilde{\alpha})}\rho(\alpha) = \rho(\alpha),$$
thus $g_h$ and $f_h$ are inverse to each other.
\end{proof}

Thus when there is a fixed relative spin structure, we don't have to distinguish between 
$\chi_n(M,N)$ and $\tilde{\chi}_n(M,N)$, and we also regard $R_n(M,N)$ as the coordinate ring 
for $\tilde{\chi}_n(M,N)$ using Proposition \ref{prop3.6}. 

\begin{rem}
For any two relative spin structures $h_1,h_2$, we have $F_{h_2 - h_1}\circ f_{h_1} = f_{h_2}$
where $F_{h_2 - h_1}$ is an isomorphism from $\tilde{\chi}_n(M,N)$ to $\tilde{\chi}_n(M,N)$ defined as
$$F_{h_2-h_1}(\tilde{\rho})(\tilde{\alpha}) = d_n^{h_2(\tilde{\alpha}) - h_1(\tilde{\alpha})}\tilde{\rho}(\tilde{\alpha}).$$
\end{rem}

\begin{rem}\label{rem3.14}
From Remark \ref{rem01}, we know there is a surjective algebra homomorphism 
$\cY: G_n(M) \rightarrow R_n(M,\emptyset)$, and Ker$\cY =\sqrt{ 0}$. When there is a spin structure $h$, we can regard $R_n(M,\emptyset)$ as the coordinate ring for 
$\tilde{\chi}_n(M,\emptyset)$ using Proposition \ref{prop3.6}. Then $\cY$ is given by
\begin{equation}\label{eee3.14}
\cY(\text{Trace}(Q_{[\alpha]}))(\tilde{\rho}) = d_n^{h([\tilde{\alpha}])} \text{Trace}( \tilde{\rho}([\tilde{\alpha}]))
\end{equation}
where $[\alpha]\in \pi_1(M)$ and $[\tilde{\alpha}]\in \pi_1(UM)$ is any lift of $[\alpha]$.
Since  the definition of $\cY$ in equation (\ref{eee3.14}) is related to $h$, we will use $\cY_h$, instead of $\cY$, to denote the map defined by equation (\ref{eee3.14}).
\end{rem}

\subsection{Surjective algebra homomorphism from $S_n(M,N,1)$ to the coordinate ring}\label{subb3.4}
For any marked three manifold $(M,N)$, we are trying to define a surjective algebra homomorphism $\Phi_h^{(M,N)} :S_n(M,N,1)\rightarrow R_n(M,N)$ (here we regard $R_n(M,N)$ as the coordinate ring 
for $\tilde{\chi}_n(M,N)$).

 Recall that $\tilde{N}$ is lifted by $N$. For any component $e\in N$, we have 
$\tilde{e} = \{(x, u_x, v_x, w_x) \mid x\in e\}$, where $u_x$ is the unit velocity vector at $x$, $w_x$ is the unit tangent vector at $x$ orthogonal to $\partial M$ pointing into $M$ and  $(u_x, v_x, w_x)$ is the orientation of $M$.

For any $n$-web $l$ and a component $e$ of $N$, we  can isotope $l$ such that the velocity vector of $l$ at each it's end point $x$  contained in $e$ is parallel to $v_x$. Then we say $l$ is in {\bf good position with respect to $e$}. If $l$ is in good position with respect to every component of $N$, we say it is in good position with respect to $(M,N)$, or just $l$ is in {\bf good position} when there is no confusion with $(M,N)$.

Let $\alpha$ be any stated framed oriented boundary  arc in $S_n(M,N)$,  then we can lift $\alpha$ to an element $\tilde{\alpha}$ in $\pi_1(UM,\tilde{N})$. We first isotope $\alpha$ such that $\alpha$ is in good position and the framing is normal everywhere. Then $\alpha$ lifts to $\tilde{\alpha}$  where
the first vector is the framing, the second vector is the velocity vector, and the third vector is determined by the orientation of $M$.  
Suppose $s(\alpha(0)) = j$ and $s(\alpha(1)) = i$. 
Then for any
$\tilde{\rho}\in\tilde{\chi}_n(M,N)$, define 
$$tr_{\alpha}(\tilde{\rho}) = [A\tilde{\rho}(\tilde{\alpha})]_{\overline{i},\overline{j}},\;\text{where}\; A_{i,j} = (-1)^{i+1}\delta_{\overline{i},j}, 1\leq i,j\leq n.$$
Note that det$ A =1$ and $A^2 = d_n I$.

Let $\alpha$ be any framed oriented  knot in $S_n(M,N)$, then we lift $\alpha$ to a closed path $\tilde{\alpha}$  in $UM$ as above (first isotope $\alpha$ such that the framing is normal everywhere, then use framing as the first vector and use velocity vector as the second vector). We use a
path to connect $\tilde{N}$ (respectively the base point for $\pi_1(UM)$) and $\tilde{\alpha}$ when
$N\neq \emptyset$ (respectively $N=\emptyset$), this gives an element in $\pi_1(UM,\tilde{N})$ or $\pi_1(UM)$, which is still denoted as $\tilde{\alpha}$. For any $\tilde{\rho}\in\tilde{\chi}_n(M,N)$ define 
$$tr_{\alpha}(\tilde{\rho}) =  \text{Trace}(\tilde{\rho}(\tilde{\alpha})).$$
%
Since Trace is invariant under the same conjugacy class, we have $tr_{\alpha}(\tilde{\rho})$ is well-defined.

\begin{theorem}\label{thm3.11}

Let $(M,N)$ be a marked three manifold with $M$ being connected. 
Then
there exists a surjective algebra homomorphism 
$\Phi^{(M,N)} : S_n(M,N,1)\rightarrow R_n(M,N)$ defined as following: For any stated $n$-web $l$ in $(M,N)$, we use relation (\ref{wzh.five}) to kill all the sinks and sources to obtain $l^{'}$ consisting of acrs and knots if $N\neq\emptyset$ (we use relation (\ref{wzh.four}) to kill all the sinks and sources to obtain $l^{'}$ consisting of knots if $N=\emptyset$). Suppose $l^{'} = \cup_{\alpha}\alpha$ where each $\alpha$ is a stated framed oriented boundary arc or a framed oriented knot, then define
$$\Phi^{(M,N)}(l)(\tilde{\rho}) = \prod_{\alpha} tr_{\alpha}(\tilde{\rho})$$
where $\tilde{\rho}\in \tilde{\chi}_n(M,N)$.

\end{theorem}

Although we assume $M$ is connected in Theorem \ref{thm3.11} for simplicity, we can easily generalize  Theorem \ref{thm3.11} to general marked three manifolds. 

 When there is no confusion with the marked three manifold $(M,N)$, we can  omit the superscript for $\Phi^{(M,N)}$.
We will prove Theorem \ref{thm3.11} in next section. 

\subsection{Compatibility with the construction by Costantino and L{\^e} for essentially bordered pb
 surface when $n=2$.
}

Let $\Sigma$ be an essentailly bordered pb surface, and $(M,N) = \Sigma\times [-1,1]$. 
Recall that for every boundary component $e$, we select a point $x_{e}\in e$, and set $N=\cup_{e} (\{x_e\}\times (-1,1))$ where $e$ is taken over all the boundary components of $\Sigma$.
Then orientation and  Riemannian matric of $M$ are the product orientation and  product Riemannian matric respectively. For simplicity, we can assume $\Sigma$ is connected.

If we regard the state $"2"$ as the state $"+"$, and the state $"1"$ as the state $"-"$. Then 
$S_2(\Sigma,1)$ is the same as the commutative stated skein algebra $\cS_{1}(\Sigma)$, mentioned in Section 8 in \cite{CL2022stated}, as shown in \cite{le2021stated}.
 When $n=2$, Costantino and L{\^e} established an isomorphism $tr : S_2(\Sigma,1)\rightarrow \chi(\Sigma)$ where $\chi(\Sigma)$ is the coordinate ring of the so called flat twisted $SL(2,\mathbb{C})$-bundle.
Here we briefly recall the definition of the  flat twisted $SL(2,\mathbb{C})$-bundle,
 please refer to Section 8 in \cite{CL2022stated} for more details. We follow their notation, and use $\pi_1(U\Sigma,\widetilde{\partial\Sigma})$ to denote their groupoid, where $U\Sigma$ is the  unit tangent bundle over
$\Sigma$ and $\widetilde{\partial\Sigma}$ is the lift of $\partial \Sigma$.
 Then every point in $U\Sigma$ is a pair $(x,v_x)$ where $x\in\Sigma$ and $v_x$ is a unit tangent vector at point $x$.
 For a point $x\in\Sigma$ the fiber $\mathbb{O}$ is a circle, and we orient it according to the  orientation of $\Sigma$. 
Then the flat twisted $SL(2,\mathbb{C})$-bundle is defined to be 
$$\{\rho\in Hom( \pi_1(U\Sigma,\widetilde{\partial\Sigma}), SL(2,\mathbb{C})  )\mid \rho(\mathbb{O}) = -I\}.$$

Let $pr: M\rightarrow \Sigma$ be the projection given by $pr(x,t) = x$ for all $x\in \Sigma,t\in [-1,1]$, and $l:\Sigma \rightarrow M$ be the embedding  given by $l(x) = (x,0)$.
We use $pr_*$ (respectievly $l_*$) to denote the induced map from 
$T(M)$ to $ T(\Sigma)$ (respectively from $ T(\Sigma)$ to $T(M)$). 
For every point $y\in M$, we use $u_y$ to denote the unit vertical tangent vector at $y$ such that $u_y$ points from $-1$ to $1$. Then $l$ induces an embedding 
\begin{align*}
l_{\sharp}: U\Sigma&\rightarrow UM\\
(x,v_x) &\mapsto (l(x), u_{l(x)}, l_*(v_x), w_{l(x)})
\end{align*}
where  $w_{l(x)}$ is determined by the orientation of $M$. Let $VM = \{(y,a_y,b_y,c_y)\in UM\mid a_y = u_y\}$ be a submanifold of $UM$ with one dimension less. Then $pr$ induces a projection
\begin{align*}
 pr_{\sharp}: VM&\rightarrow \Sigma\\
(y, a_y, b_y, c_y)&\mapsto (pr(y), pr_*(b_y)).
\end{align*}
Clearly Im$l_{\sharp}\subset VM$, and $pr_{\sharp}\circ l_{\sharp} = Id_{U\Sigma}$.

 For each boundary component $e$ of $\Sigma$, we use $\widetilde{x_e}$ to denote a point in $\tilde{e}$ whose projection on $e$ is the point $x_e$.

Define
\begin{align*}
f_{l} :\pi_1(U\Sigma,\widetilde{\partial \Sigma}) &\rightarrow \pi_1(UM,\tilde{N})\\
[\alpha]&\mapsto [l_{\sharp}\circ \alpha]
\end{align*}
where $\alpha$ is a representative of $[\alpha]$ such that the two endpoints of $\alpha$ belong to
$\cup_{e}\{\widetilde{x_e}\}$.
And define 
\begin{align*}
f_{pr}:\pi_1(UM,\tilde{N})&\rightarrow \pi_1(\Sigma,\widetilde{\partial\Sigma})\\
[\beta]&\mapsto [pr_{\sharp}\circ \beta]
\end{align*}
where $\beta$ is a representative of $[\beta]$ such that Im$\beta\subset VM$. It is easy to show $f_{l}$ and 
$f_{pr}$ are inverse to each other, and 
 $f_{l}(\mathbb{O})= \vartheta$. 
Then $f_l$ induces isomorphism from $\tilde{\chi}_2(M,N)$ to $\{\rho\in Hom( \pi_1(U\Sigma,\widetilde{\partial\Sigma}), SL(2,\mathbb{C})  )\mid \rho(\mathbb{O}) = -I\}$, which further induces an isomorphism $f_* :\chi(\Sigma)\rightarrow R_2(M,N)$. Then it is a trivial check that 
$$f_* \circ tr = \Phi .$$

\subsection{Compatibility with the splitting map}\label{sub3.4}
In this subsection we discuss the splitting maps for both $R_n(M,N)$  and $S_n(M,N,1)$ and the commutativity between the splitting map and $\Phi$. 

Recall that when $D$ is a properly embedded disk in $M$ and $\beta$ is an oriented open interval contained in $D$, there exists a splitting map $\Theta_{(D,\beta)}:S_n(M,N,1)\rightarrow S_n(\text{Cut}_{(D,\beta)}(M,N),1).$

\begin{lemma}
The above linear map $\Theta_{(D,\beta)}$ is an algebra homomorphism.
\end{lemma}
\begin{proof}
It followes easily from the definition of $\Theta_{(D,\beta)}$.
\end{proof}

We will use $(M^{'}, N^{'})$ to denote $\text{Cut}_{(D,\beta)}(M,N)$. Then there is a projection
pr $:(M^{'}, N^{'})\rightarrow (M,N).$
If we  orient $\partial D$, the orientations of $\partial D$ and $M$ give a way to distinguish between $D_1$ and $D_2$ such that the orientation pointing from $D_2$ to $D_1$ and the orientation of $\partial D$ coincide with the orientaion of $M$, see Figure \ref{fig:1}.
\begin{figure}[!h]
\centering
\includegraphics[scale=0.6]{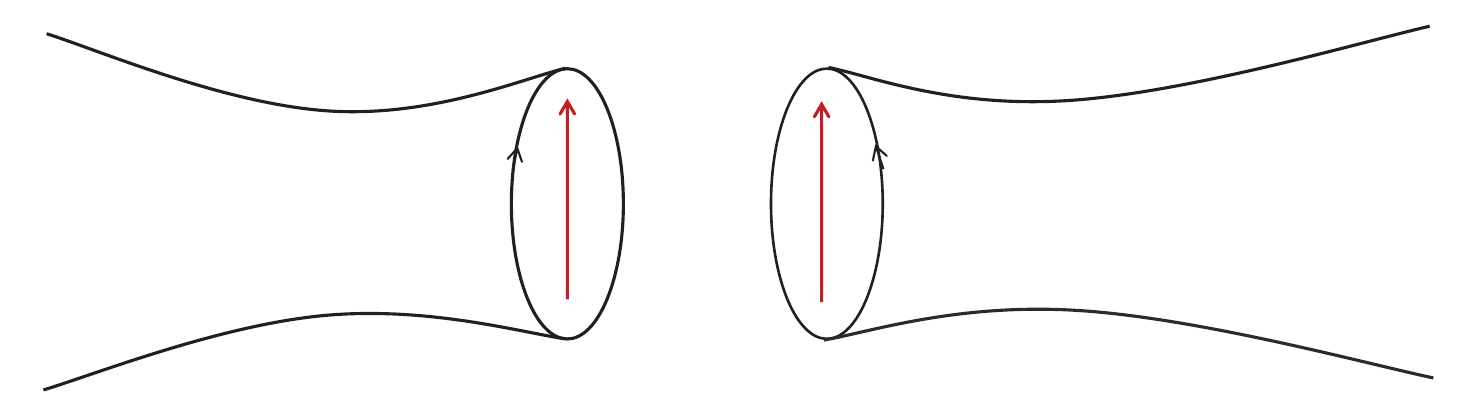}
\caption{The orientation of $D$ is indicated by the arrow, the orientation of $M$ is right handed.
The left (respectively right) disk copy is $D_1$ (respectively $D_2$). The left (respectively right) red arrow is $\beta_1$ (respectively $\beta_2$).}
\label{fig:1}
\end{figure}


In the following discussion, we fix an orientation for $\partial D$. Note that $\beta\in D$ lifts to an element
in $UM$. For every point $P$ in $\beta$, the velocity vector gives the first unit tangent vector, the orientation of $\partial D$ and the orientation of $M$ give the third unit tangent vector (since the orientation of $\partial D$ and the orientation of $M$ give an orientation of $D$, which determines a unit tangent vector at $P$), then the orientation of $M$ gives the second unit tangent vector. We use $\tilde{\beta}$ to denote this lift. Recall that $\zeta : UM\rightarrow M$ is the projection, we define $\tilde{D} = \zeta^{-1}(D)$. The projection pr $:(M^{'}, N^{'})\rightarrow (M,N)$ induces a projection $\tilde{\text{pr}}: UM^{'}\rightarrow UM$. Then 
$\tilde{\text{pr}}^{-1}(\tilde{\beta}) = \tilde{\beta_1}^{'}\cup \tilde{\beta_2}^{'}$
where $\tilde{\beta_1}^{'} \in \widetilde{D_1}, \tilde{\beta_2}^{'} \in \widetilde{D_2}$ ($ \widetilde{D_1}$
and $ \widetilde{D_2}$ are defined in the same way with $\tilde{D}$).  Note that 
$\widetilde{\beta_1} =  \tilde{\beta_1}^{'}$ and $\widetilde{\beta_2} \neq  \tilde{\beta_2}^{'}$. The orientation of $\beta$ determines a path $a_{\beta}$ from $\widetilde{\beta_2}$ to $\tilde{\beta_2}^{'}$ and a path $b_{\beta}$
from $\tilde{\beta_2}^{'}$ to $\widetilde{\beta_2}$ such that both $a_{\beta} *  b_{\beta}$ and 
$b_{\beta}* a_{\beta}$ are in the same homotopy type with $\vartheta$.

\begin{rem}
According to Lemma 8.1 in \cite{CL2022stated},  any $\rho^{'}\in \tilde{\chi}_n(M^{'},N^{'})$ can be extended to a homomorpism $\rho^{''}:\pi_1(UM^{'},\widetilde{N^{'}}\cup \{\tilde{\beta_2}^{'}\})\rightarrow SL(n,\mathbb{C})$ by setting 
$\rho^{''}(a_{\beta}) = d_n A$.

For any $\alpha\in \pi_1(UM,\tilde{N})$, we can isotope $\alpha$ such that $\alpha \cap \tilde{D}
= \alpha\cap \tilde{\beta}$ and $\alpha\cap \tilde{\beta}$ consists of finite points. Then 
$\alpha = \alpha_k* \alpha_{k-1}*\dots*\alpha_1$ where each
$\alpha_i\in \pi_1(UM, \tilde{N}\cup\{\tilde{\beta}\})$  intersects $\tilde{D}$ at most in it's endpoints and exactly along $\tilde{\beta}$. For any $\rho^{'}\in \tilde{\chi}_n(M^{'},N^{'})$, define
 $$\nu^{*}(\rho^{'})(\alpha) = \rho^{''}(\alpha_k^{'})\rho^{''}(\alpha_{k-1}^{'})\dots
\rho^{''}(\alpha_1^{'})$$
where $\alpha_i^{'} = \tilde{\text{pr}}^{-1}(\alpha_i)\in \pi_1(UM^{'}, \widetilde{N^{'}}\cup \{\tilde{\beta_2}^{'}\})$.
\end{rem}


\begin{proposition}
Let $(M,N)$ be any marked three manifold, $(D,\beta)$ be any properly embedded disk with an open oriented interval $\beta\in D$. Then there is a surjective homomorphism $\nu^{*}:\tilde{\chi}_n(M^{'},N^{'})\rightarrow 
\tilde{\chi}_n(M,N)$ where $(M^{'}, N^{'}) =$ Cut$_{(D,\beta)}(M, N )$. Especially $\nu^{*}$ induces
an injective algebra homomorphism $\nu: R_n(M,N)\rightarrow R_n(M^{'},N^{'})$.
\end{proposition}
\begin{proof}
First we show $\nu^{*}(\rho^{'})\in \tilde{\chi}_n(M,N)$. Clearly we have $ \rho^{''}( \vartheta_Y) = d_n I$ for  any component 
$Y$ of $M^{'}$.
Since $\rho^{''}$ preserves height exchange and crossing exchange, to show $\nu^{*}(\rho^{'})$ is well-defined,  it suffices to show 
$\nu^{*}(\rho^{'})$ preserves the following two moves:
\begin{equation}\label{e111}
\raisebox{-.30in}{
\begin{tikzpicture}
\tikzset{->-/.style=
{decoration={markings,mark=at position #1 with
{\arrow{latex}}},postaction={decorate}}}
\draw [line width =1pt,color=red,decoration={markings, mark=at position 1 with {\arrow{>}}},postaction={decorate}](0.8,-0.5)--(0.8,1.5);
\draw[line width =1pt] (0,0) arc (-90:90:0.5);
\draw [line width =1pt](-1,0)--(0,0);
\draw [line width =1pt](-1,1)--(0,1);
\filldraw[draw=black,fill=white] (-0.5,0) circle (0.07);
\node [right] at(0.9,1.3) {$\beta$};
\end{tikzpicture}
}
\longleftrightarrow
\raisebox{-.30in}{
\begin{tikzpicture}
\tikzset{->-/.style=
{decoration={markings,mark=at position #1 with
{\arrow{latex}}},postaction={decorate}}}
\draw [line width =1pt,color=red,decoration={markings, mark=at position 1 with {\arrow{>}}},postaction={decorate}](0,-0.5)--(0,1.5);
\draw[line width =1pt] (0,0) arc (-90:90:0.5);
\draw [line width =1pt](-1,0)--(0,0);
\draw [line width =1pt](-1,1)--(0,1);
\filldraw[draw=black,fill=white] (-0.5,0) circle (0.07);
\node [right] at(0.1,1.3) {$\beta$};
\end{tikzpicture}
}
\end{equation}
\begin{equation}\label{e222}
\raisebox{-.30in}{
\begin{tikzpicture}
\tikzset{->-/.style=
{decoration={markings,mark=at position #1 with
{\arrow{latex}}},postaction={decorate}}}
\draw [line width =1pt,color=red,decoration={markings, mark=at position 1 with {\arrow{>}}},postaction={decorate}](-0.8,-0.5)--(-0.8,1.5);
\draw[line width =1pt] (0,1) arc (90:270:0.5);
\draw [line width =1pt](1,0)--(0,0);
\draw [line width =1pt](1,1)--(0,1);
\filldraw[draw=black,fill=white] (0.5,0) circle (0.07);
\node [left] at(-0.9,1.3) {$\beta$};
\end{tikzpicture}
}
\longleftrightarrow
\raisebox{-.30in}{
\begin{tikzpicture}
\tikzset{->-/.style=
{decoration={markings,mark=at position #1 with
{\arrow{latex}}},postaction={decorate}}}
\draw [line width =1pt,color=red,decoration={markings, mark=at position 1 with {\arrow{>}}},postaction={decorate}](-0,-0.5)--(-0,1.5);
\draw[line width =1pt] (0,1) arc (90:270:0.5);
\draw [line width =1pt](1,0)--(0,0);
\draw [line width =1pt](1,1)--(0,1);
\filldraw[draw=black,fill=white] (0.5,0) circle (0.07);
\node [left] at(-0.1,1.3) {$\beta$};
\end{tikzpicture}
}
\end{equation}
 The red arrow in equations (\ref{e111})  and  (\ref{e222}) is the projection of $\beta$, to get the original  $\beta$, we just rotate the
red arrow in equations (\ref{e111})  and  (\ref{e222}) 90 degrees such that it points towards readers. The black line represents part of path in $UM$, and the white dot represents the direction of the path. The first unit tangent vector of the path is the one pointing towards readers  and the second one is given by the velocity vector of the black line.

Here we only prove $\nu^{*}(\rho^{'})$ preserves
$$
\raisebox{-.30in}{
\begin{tikzpicture}
\tikzset{->-/.style=
{decoration={markings,mark=at position #1 with
{\arrow{latex}}},postaction={decorate}}}
\draw [line width =1pt,color=red,decoration={markings, mark=at position 1 with {\arrow{>}}},postaction={decorate}](0.8,-0.5)--(0.8,1.5);
\draw[line width =1pt] (0,0) arc (-90:90:0.5);
\draw [line width =1pt,decoration={markings, mark=at position 0.5 with {\arrow{>}}},postaction={decorate}](-1,0)--(0,0);
\draw [line width =1pt](-1,1)--(0,1);
\node [right] at(0.9,1.3) {$\beta$};
\end{tikzpicture}
}
\longleftrightarrow
\raisebox{-.30in}{
\begin{tikzpicture}
\tikzset{->-/.style=
{decoration={markings,mark=at position #1 with
{\arrow{latex}}},postaction={decorate}}}
\draw [line width =1pt,color=red,decoration={markings, mark=at position 1 with {\arrow{>}}},postaction={decorate}](0,-0.5)--(0,1.5);
\draw[line width =1pt] (0,0) arc (-90:90:0.5);
\draw [line width =1pt,decoration={markings, mark=at position 0.5 with {\arrow{>}}},postaction={decorate}](-1,0)--(0,0);
\draw [line width =1pt](-1,1)--(0,1);
\node [right] at(0.1,1.3) {$\beta$};
\end{tikzpicture}
}.
$$
The same proving technique applies for other three cases.

We isotope 
$\raisebox{-.30in}{
\begin{tikzpicture}
\tikzset{->-/.style=
{decoration={markings,mark=at position #1 with
{\arrow{latex}}},postaction={decorate}}}
\draw [line width =1pt,color=red,decoration={markings, mark=at position 1 with {\arrow{>}}},postaction={decorate}](0,-0.5)--(0,1.5);
\draw[line width =1pt] (0,0) arc (-90:90:0.5);
\draw [line width =1pt,decoration={markings, mark=at position 0.5 with {\arrow{>}}},postaction={decorate}](-1,0)--(0,0);
\draw [line width =1pt](-1,1)--(0,1);
\node [right] at(0.1,1.3) {$\beta$};
\end{tikzpicture}
}$ to 
$\raisebox{-.30in}{
\begin{tikzpicture}
\tikzset{->-/.style=
{decoration={markings,mark=at position #1 with
{\arrow{latex}}},postaction={decorate}}}
%
\draw [line width =1pt,color=red,decoration={markings, mark=at position 1 with {\arrow{>}}},postaction={decorate}](0.2,-0.5)--(0.2,1.5);
\draw [line width =1pt,decoration={markings, mark=at position 0.5 with {\arrow{>}}},postaction={decorate}](-1,0)--(0,0);
\draw [line width =1pt](-1,1)--(0,1);
\draw[line width =1pt] (0,0) arc (-90:0:0.2);
\draw[line width =1pt] (0.2,0.8) arc (0:90:0.2);
\draw[line width =1pt] (0.6,0.2) arc (0:180:0.2);
\draw[line width =1pt] (0.2,0.8) arc (-180:0:0.2);
\draw[line width =1pt] (0.6,0.2) arc (-180:0:0.3);
\draw[line width =1pt] (1.2,0.8) arc (0:180:0.3);
\draw[line width =1pt] (1.2,0.2)--(1.2,0.8);
\node [right] at(0.3,1.3) {$\beta$};
\end{tikzpicture}
}.$ 
From the definition of $\nu^{*}(\rho^{'})$, we know
\begin{align*}
\nu^{*}(\rho^{'})(\raisebox{-.30in}{
\begin{tikzpicture}
\tikzset{->-/.style=
{decoration={markings,mark=at position #1 with
{\arrow{latex}}},postaction={decorate}}}
\draw [line width =1pt,color=red,decoration={markings, mark=at position 1 with {\arrow{>}}},postaction={decorate}](0,-0.5)--(0,1.5);
\draw[line width =1pt] (0,0) arc (-90:90:0.5);
\draw [line width =1pt,decoration={markings, mark=at position 0.5 with {\arrow{>}}},postaction={decorate}](-1,0)--(0,0);
\draw [line width =1pt](-1,1)--(0,1);
\node [right] at(0.1,1.3) {$\beta$};
\end{tikzpicture}
}) =& \rho^{''}(\raisebox{-.30in}{
\begin{tikzpicture}
\tikzset{->-/.style=
{decoration={markings,mark=at position #1 with
{\arrow{latex}}},postaction={decorate}}}
%
\draw [line width =1pt,color=red,decoration={markings, mark=at position 1 with {\arrow{>}}},postaction={decorate}](0.2,-0.5)--(0.2,1.5);
\draw [line width =1pt,decoration={markings, mark=at position 0.5 with {\arrow{<}}},postaction={decorate}](-1,1)--(0,1);
\draw[line width =1pt] (0.2,0.8) arc (0:90:0.2);
\node [right] at(0.3,1.3) {$\beta$};
\end{tikzpicture}
})
\rho^{''}(\raisebox{-.30in}{
\begin{tikzpicture}
\tikzset{->-/.style=
{decoration={markings,mark=at position #1 with
{\arrow{latex}}},postaction={decorate}}}
%
\draw [line width =1pt,color=red,decoration={markings, mark=at position 1 with {\arrow{>}}},postaction={decorate}](0.2,-0.5)--(0.2,1.5);
\draw[line width =1pt] (0.6,0.2) arc (0:180:0.2);
\draw[line width =1pt] (0.2,0.8) arc (-180:0:0.2);
\draw[line width =1pt] (0.6,0.2) arc (-180:0:0.3);
\draw[line width =1pt] (1.2,0.8) arc (0:180:0.3);
\draw[line width =1pt,decoration={markings, mark=at position 0.5 with {\arrow{>}}},postaction={decorate}] (1.2,0.2)--(1.2,0.8);
\node [right] at(0.3,1.3) {$\beta$};
\end{tikzpicture}
})
\rho^{''}(\raisebox{-.30in}{
\begin{tikzpicture}
\tikzset{->-/.style=
{decoration={markings,mark=at position #1 with
{\arrow{latex}}},postaction={decorate}}}
%
\draw [line width =1pt,color=red,decoration={markings, mark=at position 1 with {\arrow{>}}},postaction={decorate}](0.2,-0.5)--(0.2,1.5);
\draw [line width =1pt,decoration={markings, mark=at position 0.5 with {\arrow{>}}},postaction={decorate}](-1,0)--(0,0);
\draw[line width =1pt] (0,0) arc (-90:0:0.2);
\node [right] at(0.3,1.3) {$\beta$};
\end{tikzpicture}
})
\\
 =& \rho^{''}(\raisebox{-.30in}{
\begin{tikzpicture}
\tikzset{->-/.style=
{decoration={markings,mark=at position #1 with
{\arrow{latex}}},postaction={decorate}}}
%
\draw [line width =1pt,color=red,decoration={markings, mark=at position 1 with {\arrow{>}}},postaction={decorate}](0.2,-0.5)--(0.2,1.5);
\draw [line width =1pt,decoration={markings, mark=at position 0.5 with {\arrow{<}}},postaction={decorate}](-1,1)--(0,1);
\draw[line width =1pt] (0.2,0.8) arc (0:90:0.2);
\node [right] at(0.3,1.3) {$\beta$};
\end{tikzpicture}
})
\rho^{''}(\raisebox{-.30in}{
\begin{tikzpicture}
\tikzset{->-/.style=
{decoration={markings,mark=at position #1 with
{\arrow{latex}}},postaction={decorate}}}
%
\draw [line width =1pt,color=red,decoration={markings, mark=at position 1 with {\arrow{>}}},postaction={decorate}](0.2,-0.5)--(0.2,1.5);
\draw [line width =1pt,decoration={markings, mark=at position 0.5 with {\arrow{>}}},postaction={decorate}](-1,0)--(0,0);
\draw[line width =1pt] (0,0) arc (-90:0:0.2);
\node [right] at(0.3,1.3) {$\beta$};
\end{tikzpicture}
}) =
\rho^{''}(\raisebox{-.30in}{
\begin{tikzpicture}
\tikzset{->-/.style=
{decoration={markings,mark=at position #1 with
{\arrow{latex}}},postaction={decorate}}}
\draw [line width =1pt,color=red,decoration={markings, mark=at position 1 with {\arrow{>}}},postaction={decorate}](0.8,-0.5)--(0.8,1.5);
\draw[line width =1pt] (0,0) arc (-90:90:0.5);
\draw [line width =1pt,decoration={markings, mark=at position 0.5 with {\arrow{>}}},postaction={decorate}](-1,0)--(0,0);
\draw [line width =1pt](-1,1)--(0,1);
\node [right] at(0.9,1.3) {$\beta$};
\end{tikzpicture}
})
\\
=& \nu^{*}(\rho^{'})(\raisebox{-.30in}{
\begin{tikzpicture}
\tikzset{->-/.style=
{decoration={markings,mark=at position #1 with
{\arrow{latex}}},postaction={decorate}}}
\draw [line width =1pt,color=red,decoration={markings, mark=at position 1 with {\arrow{>}}},postaction={decorate}](0.8,-0.5)--(0.8,1.5);
\draw[line width =1pt] (0,0) arc (-90:90:0.5);
\draw [line width =1pt,decoration={markings, mark=at position 0.5 with {\arrow{>}}},postaction={decorate}](-1,0)--(0,0);
\draw [line width =1pt](-1,1)--(0,1);
\node [right] at(0.9,1.3) {$\beta$};
\end{tikzpicture}
}).
\end{align*}
Then, trivially we have $\nu^{*}(\rho^{'})\in \tilde{\chi}_n(M,N)$.

Then we want to show $\nu^{*}$ is surjective. We use $-\tilde{\beta}$ to denote $\tilde{\text{pr}}(\widetilde{\beta_2})$, and use
$\overline{a_{\beta}}$ to denote $\tilde{\text{pr}}(a_{\beta})$. Then $\overline{a_{\beta}}$ is a path from $-\tilde{\beta}$ to
$\tilde{\beta}$. For any $\rho\in \tilde{\chi}_n(M,N)$, we use Lemma 8.1 in \cite{CL2022stated} to extend $\rho$ to
$\rho^{''}:\pi_1(UM,\tilde{N}\cup\{\tilde{\beta},-\tilde{\beta}\})$ setting in particular $\rho^{''}(\overline{a_{\beta}})
= d_n A$. The projection $\tilde{\text{pr}}:UM^{'}\rightarrow UM$ induces a homomorphism $\text{pr}_* :
\pi_1(UM^{'},\widetilde{N^{'}}\cup \{\tilde{\beta_2}^{'}\}) \rightarrow \pi_1(UM,\tilde{N}\cup\{\tilde{\beta},-\tilde{\beta}\})$. Then 
$\rho^{''}\circ \text{pr}_*$ is a homomorphism from $\pi_1(UM^{'},\widetilde{N^{'}}\cup \{\tilde{\beta_2}^{'}\})$ to $SL(n,\mathbb{C})$. Set $\rho^{'} $ to be the restriction of $\rho^{''}\circ \text{pr}_*$ on $\pi_1(UM^{'},\widetilde{N^{'}})$. Then it is east to show we have
$\rho^{'}\in \tilde{\chi}_n(M^{'},N^{'})$ and $\nu^{*}(\rho^{'}) = \rho$.

\end{proof}

\begin{theorem}
Let $(M,N)$ be any marked three manifold, $(D,\beta)$ be any properly embedded disk with an open oriented interval $\beta\in D$. Then we have 
$$\Phi^{(M^{'},N^{'})}\circ \Theta_{(D,\beta)} = \nu\circ \Phi^{(M,N)}.$$
\end{theorem}
\begin{proof}
We can assume $M$ is connected. Note that $M^{'}$ may not be connected.

Since both $\Phi^{(M^{'},N^{'})}\circ \Theta_{(D,\beta)}$ and $\nu\circ \Phi^{(M,N)}$ are algebra homomorphisms, it suffices to show $\Phi^{(M^{'},N^{'})}(\Theta_{(D,\beta)} (\alpha))= \nu( \Phi^{(M,N)}(\alpha))$ for any framed oriented knot or stated framed oriented boundary arc $\alpha$. If there is no intersection between $\alpha$ and $D$, it is obvious. Then we look at the case when $\alpha$ intersects $D$. We isotope $\alpha$ such that $\alpha$ is transverse to $D$, $\alpha\cap D\subset \beta$, and the framing at each point of $\alpha\cap \beta$ is the velocity vector of $\beta$.

If $\alpha$ is a stated framed oriented boundary arc with $s(\alpha(0)) = i$ and $s(\alpha(1)) = j$. Then $\alpha = \alpha_m * \alpha_{m-1}* \dots * \alpha_1$ where all $\alpha_i$ are framed oriented arcs and are parts of $\alpha$ such that each $\alpha_t$ has two ends on $\beta$ for $2\leq t\leq m-1$ and
$\alpha_1(1), \alpha_m(0)\in\beta$ and the interior of each $\alpha_t$ has no intersection with $D$. Let $\alpha_t^{'} = \text{pr}^{-1}(\alpha_t), 1\leq t\leq m$, then
$$\Theta_{(D,\beta)}(\alpha) = \sum_{1\leq k_1,\dots,k_{m-1}\leq n} (\alpha_m^{'})_{j,k_{m-1}}
(\alpha_{m-1}^{'})_{k_{m-1},k_{m-2}}\dots (\alpha_1^{'})_{k_1,i}.$$

For any $\rho^{'}\in \tilde{\chi}_n(M^{'},N^{'})$, we have 
\begin{equation*}
\begin{split}
\Phi^{(M^{'},N^{'})}(\Theta_{(D,\beta)} (\alpha))(\rho^{'})
&= \sum_{1\leq k_1,\dots,k_{m-1}\leq n} (A\rho^{'}(\widetilde{\alpha^{'}_{m}}))_{\overline{j},\overline{k_{m-1}}}
 (A\rho^{'}(\widetilde{\alpha^{'}_{m-1}}))_{\overline{k_{m-1}},\overline{k_{m-2}}}\dots(A\rho^{'}(\widetilde{\alpha^{'}_{1}}))_{\overline{k_1},\overline{i}}\\
&=(A\rho^{'}(\widetilde{\alpha^{'}_{m}})A\rho^{'}(\widetilde{\alpha^{'}_{m-1}})\dots A\rho^{'}(\widetilde{\alpha^{'}_{1}}))_{\overline{j},\overline{i}}
\end{split}
\end{equation*}
and
\begin{equation*}
\nu( \Phi^{(M,N)}(\alpha))(\rho^{'}) =  \Phi^{(M,N)}(\alpha)(\nu^{*}(\rho^{'})) =(A\;\nu^{*}(\rho^{'})(\tilde{\alpha}))_{\overline{j},\overline{i}}.
\end{equation*}
According to the definition of $\nu^{*}$, we know 
$$\nu^{*}(\rho^{'})(\tilde{\alpha}) = \rho^{''}(\tilde{\alpha}_m^{'})\rho^{''}(\tilde{\alpha}_{k-1}^{'})\dots
\rho^{''}(\tilde{\alpha}_1^{'}).$$
It is easy to see 
$$\rho^{''}(\tilde{\alpha}_m^{'})\rho^{''}(\tilde{\alpha}_{k-1}^{'})\dots
\rho^{''}(\tilde{\alpha}_1^{'})
=\rho^{''}(\widetilde{\alpha^{'}_{m}}) \rho^{''}(\gamma_{m-1}) \rho^{''}(\widetilde{\alpha^{'}_{m-1}})
\rho^{''}(\gamma_{m-2})\dots \rho^{''}(\gamma_1) \rho^{''}(\widetilde{\alpha^{'}_{1}})$$
where $\gamma_t = a_{\beta}^{-1}$ or $b_{\beta}^{-1}$ for $1\leq t\leq m-1$. Since 
$\rho^{''}( a_{\beta}^{-1}) =  \rho^{''}( b_{\beta}^{-1}) = A$, we have 
\begin{align*}
\nu^{*}(\rho^{'})(\tilde{\alpha}) &= \rho^{''}( \widetilde{\alpha^{'}_{m}}) \rho^{''}(\gamma_{m-1}) \rho^{''}(\widetilde{\alpha^{'}_{m-1}})
\rho^{''}(\gamma_{m-2})\dots \rho^{''}(\gamma_1) \rho^{''}(\widetilde{\alpha^{'}_{1}})\\
&=  \rho^{'}( \widetilde{\alpha^{'}_{m}}) A \rho^{'}(\widetilde{\alpha^{'}_{m-1}})
A\dots A\rho^{'}(\widetilde{\alpha^{'}_{1}}).
\end{align*}
Thus we get
$$\nu( \Phi^{(M,N)}(\alpha))(\rho^{'}) =(A\rho^{'}(\widetilde{\alpha^{'}_{m}})A\rho^{'}(\widetilde{\alpha^{'}_{m-1}})\dots A\rho^{'}(\widetilde{\alpha^{'}_{1}}))_{\overline{j},\overline{i}} = \Phi^{(M^{'},N^{'})}(\Theta_{(D,\beta)} (\alpha))(\rho^{'}).$$

If $\alpha$ is a framed knot. Let $\eta$ be a path in $UM$ connecting the base point of $\pi_1(UM)$
to $\tilde{\alpha}$ when $N=\emptyset$ or a path connecting $\tilde{N}$ to $\tilde{\alpha}$ when
$N\neq \emptyset$ such that $\eta\cap \tilde{D} = \emptyset$. Similarly suppose  $\alpha = \alpha_m * \alpha_{m-1}* \dots * \alpha_1$ where all $\alpha_i$ are framed oriented arcs and parts of $\alpha$ such that each $\alpha_t$ has two ends on $\beta$ and does not intersect with $D$ on it's interior. 
Let $\alpha_t^{'} = \text{pr}^{-1}(\alpha_t), 1\leq t\leq m$, then
$$\Theta_{(D,\beta)}(\alpha) = \sum_{1\leq i,k_1,\dots,k_{m-1}\leq n} (\alpha_m^{'})_{i,k_{m-1}}
(\alpha_{m-1}^{'})_{k_{m-1},k_{m-2}}\dots (\alpha_1^{'})_{k_1,i}.$$

For any $\rho^{'}\in \tilde{\chi}_n(M^{'},N^{'})$, we have 
\begin{equation*}
\begin{split}
\Phi^{(M^{'},N^{'})}(\Theta_{(D,\beta)} (\alpha))(\rho^{'})
&= \sum_{1\leq i,k_1,\dots,k_{m-1}\leq n} (A\rho^{'}(\widetilde{\alpha^{'}_{m}}))_{\overline{i},\overline{k_{m-1}}}
 (A\rho^{'}(\widetilde{\alpha^{'}_{m-1}}))_{\overline{k_{m-1}},\overline{k_{m-2}}}\dots(A\rho^{'}(\widetilde{\alpha^{'}_{1}}))_{\overline{k_1},\overline{i}}\\
&=\text{Trace}(A\rho^{'}(\widetilde{\alpha^{'}_{m}})A\rho^{'}(\widetilde{\alpha^{'}_{m-1}})\dots A\rho^{'}(\widetilde{\alpha^{'}_{1}}))
\end{split}
\end{equation*}
and
\begin{equation*}
\nu( \Phi^{(M,N)}(\alpha))(\rho^{'}) =  \Phi^{(M,N)}(\alpha)(\nu^{*}(\rho^{'})) =\text{Trace}(\nu^{*}(\rho^{'})(\eta^{-1}*\tilde{\alpha} *\eta)).
\end{equation*}
We assume $\eta(1)\in \widetilde{\alpha^{'}_{1}}$, otherwise we can relabel $\alpha_i$ to make this happpen, and $\eta(1)$ divides $\widetilde{\alpha^{'}_{1}}$
into two parts $\widetilde{\alpha^{'}_{1}}^{'}, \widetilde{\alpha^{'}_{1}}^{''}$
such that $\widetilde{\alpha^{'}_{1}} = \widetilde{\alpha^{'}_{1}}^{'}* \widetilde{\alpha^{'}_{1}}^{''}$.
Using the same technique as $\alpha$ being an arc, we get
\begin{align*}
\nu^{*}(\rho^{'})(\eta^{-1}*\tilde{\alpha}*\eta) = \rho^{''}(\eta^{-1} *\widetilde{\alpha^{'}_{1}}^{''})
 A\rho^{''}( \widetilde{\alpha^{'}_{m}}) A \rho^{''}(\widetilde{\alpha^{'}_{m-1}})
A\dots A\rho^{''}(\widetilde{\alpha^{'}_{1}}^{'} * \eta).
\end{align*}
Then we have 
\begin{align*}
\nu( \Phi^{(M,N)}(\alpha))(\rho^{'})
&=\text{Trace}( \rho^{''}(\eta^{-1} *\widetilde{\alpha^{'}_{1}}^{''})
 A\rho^{''}( \widetilde{\alpha^{'}_{m}}) A \rho^{''}(\widetilde{\alpha^{'}_{m-1}})
A\dots A\rho^{''}(\widetilde{\alpha^{'}_{1}}^{'} * \eta))\\
&= \text{Trace}( 
 A\rho^{''}( \widetilde{\alpha^{'}_{m}}) A \rho^{''}(\widetilde{\alpha^{'}_{m-1}})
A\dots A\rho^{''}(\widetilde{\alpha^{'}_{1}}^{'} * \eta) \rho^{''}(\eta^{-1} *\widetilde{\alpha^{'}_{1}}^{''}))\\
&= \text{Trace}( 
 A\rho^{''}( \widetilde{\alpha^{'}_{m}}) A \rho^{''}(\widetilde{\alpha^{'}_{m-1}})
A\dots A\rho^{''}(\widetilde{\alpha^{'}_{1}}^{'} *\widetilde{\alpha^{'}_{1}}^{''}))\\
&= \text{Trace}( 
 A\rho^{''}( \widetilde{\alpha^{'}_{m}}) A \rho^{''}(\widetilde{\alpha^{'}_{m-1}})
A\dots A\rho^{''}(\widetilde{\alpha^{'}_{1}}))\\
&= \text{Trace}( 
 A\rho^{'}( \widetilde{\alpha^{'}_{m}}) A \rho^{'}(\widetilde{\alpha^{'}_{m-1}})
A\dots A\rho^{'}(\widetilde{\alpha^{'}_{1}})) \\
&= \Phi^{(M^{'},N^{'})}(\Theta_{(D,\beta)} (\alpha))(\rho^{'}).
\end{align*}

\end{proof}

\section{Proof for Theorem \ref{thm3.11}}
For any $m$ by $k$ matrix $A$, $C_t(A),1\leq t\leq k,$ denotes the $t$-th column of $A$, 
$R_t(A),1\leq t\leq m$, denotes the $t$-th row of $A$.

\subsection{The case when $N=\emptyset$}\label{sub4.1}

Sikora proved $S_n(M;\mathbb{C},1)\simeq G_n(M)$ \cite{sikora2005skein}.
 L{\^e} and Sikora proved $S_n(M,\emptyset,1)\simeq S_n(M;\mathbb{C},1)$, which is related to the spin structure $h$ \cite{le2021stated}.
Then it is easy to show the combination of 
$S_n(M,\emptyset,1)\simeq S_n(M;\mathbb{C},1)\simeq G_n(M)\rightarrow R_n(M,\emptyset)$ is $\Phi$
,where the third map $G_n(M)\rightarrow R_n(M,\emptyset)$ is $\cY_h$ in Reamrk \ref{rem3.14}. Thus $\Phi$ is a well-defined surjective algebra homomorphism. Especially Ker$\Phi = \sqrt{0}$ since Ker$\cY_h = \sqrt{0}$.

\subsection{Independence of how to kill sinks and souces ($N\neq \emptyset$)}\label{sub4.2}
When we try to use  relation (\ref{wzh.five}) to kill all the sinks and sources, we first drag all the sinks and sources close enough to some component of $N$, then use relation (\ref{wzh.five}). In this subsection, we want to show $\Phi$ is independent of how to kill sinks and sources, that is, to show $\Phi$ is independent of how we drag sinks and sources close to $N$.

Let $l$ be a stated $n$-web. Suppose $l^{'}$ and $l^{''}$ are obtained from $l$ by killing all the sinks and sources. Note that for every sink or source of $l$, we may use different ways to kill this sink or source to get $l^{'}$ and $l^{''}$. First assume we kill all the sinks and sources in the same way to obtain $l^{'}$ and $l^{''}$ except one source or sink, which is denoted as $\mathfrak{S}$. Let $l_1$ be obtained from $l$ by first  killing all the sinks and sources except $\mathfrak{S}$ using the same way as $l^{'}$ and $l^{''}$ then eliminating the component containing $\mathfrak{S}$.

Suppose $\mathfrak{S}$ is a sink.
Then
\begin{align*}
l^{'} = [\sum_{\sigma\in S_n}(-1)^{\ell(\sigma)}(\eta_n*\alpha_n)_{\sigma(n),u_n}(\eta_{n-1}*\alpha_{n-1})_{\sigma(n-1),u_{n-1}}\dots (\eta_{1}*\alpha_{1})_{\sigma(1),u_1}]\;l_1\\
l^{''} = [\sum_{\sigma\in S_n}(-1)^{\ell(\sigma)}(\gamma_n*\alpha_n)_{\sigma(n),u_n}(\gamma_{n-1}*\alpha_{n-1})_{\sigma(n-1),u_{n-1}}\dots (\gamma_{1}*\alpha_{1})_{\sigma(1),u_1}]\;l_1
\end{align*}
where $\alpha_t,\eta_t,\gamma_t,1\leq t\leq n, $ are framed oriented  arcs such that $\eta_t*\alpha_t,
\gamma_t*\alpha_t,1\leq t\leq n$, are well-defined framed oriented  boundary arcs in $(M,N)$. Note that
$\eta_t(1),1\leq t\leq n$, belong to  a same component of $N$ (the same with $\gamma_t(1)$), and $\eta_t,1\leq t\leq n$ are isotopic to each other (the same with $\gamma_t$).

For any element $\rho\in\tilde{\chi}_n(M,N)$, we have 
\begin{align*}
&(\sum_{\sigma\in S_n}(-1)^{\ell(\sigma)}tr_{(\eta_n*\alpha_n)_{\sigma(n),u_n}}
tr_{(\eta_{n-1}*\alpha_{n-1})_{\sigma(n-1),u_{n-1}}}\dots tr_{(\eta_1*\alpha_1)_{\sigma(1),u_1}})(\rho)\\
= &\sum_{\sigma\in S_n}(-1)^{\ell(\sigma)}[A\rho(\widetilde{\eta_n*\alpha_n})]_{\overline{\sigma(n)},\overline{u_n}}
[A\rho(\widetilde{\eta_{n-1}*\alpha_{n-1}})]_{\overline{\sigma(n-1)},\overline{u_{n-1}}}\dots
[A\rho(\widetilde{\eta_1*\alpha_1})]_{\overline{\sigma(1)},\overline{u_1}}\\
= &(-1)^{\frac{n(n-1)}{2}}\sum_{\sigma\in S_n}(-1)^{\ell(\sigma)}[A\rho(\widetilde{\eta_n}*\widetilde{\alpha_n})]_{\sigma(n),\overline{u_n}}
[A\rho(\widetilde{\eta_{n-1}}*\widetilde{\alpha_{n-1}})]_{\sigma(n-1),\overline{u_{n-1}}}\dots
[A\rho(\widetilde{\eta_1}*\widetilde{\alpha_1})]_{\sigma(1),\overline{u_1}}\\
= &(-1)^{\frac{n(n-1)}{2}}\text{det}
\begin{pmatrix}
C_{\overline{u_1}}(A\rho(\widetilde{\eta_1}*\widetilde{\alpha_1}))&
\dots& C_{\overline{u_{n-1}}}(A\rho(\widetilde{\eta_{n-1}}*\widetilde{\alpha_{n-1}}))&
C_{\overline{u_n}}(A\rho(\widetilde{\eta_n}*\widetilde{\alpha_n}))
\end{pmatrix}\\
= &(-1)^{\frac{n(n-1)}{2}}\text{det}(A) \text{det}
\begin{pmatrix}
C_{\overline{u_1}}(\rho(\widetilde{\eta_1}*\widetilde{\alpha_1}))&\dots&
C_{\overline{u_{n-1}}}(\rho(\widetilde{\eta_{n-1}}*\widetilde{\alpha_{n-1}}))&
C_{\overline{u_n}}(\rho(\widetilde{\eta_n}*\widetilde{\alpha_n}))
\end{pmatrix}\\
= &(-1)^{\frac{n(n-1)}{2}} \text{det}
\begin{pmatrix}
C_{\overline{u_1}}(\rho(\widetilde{\eta_1}*\widetilde{\alpha_1}))&\dots &
C_{\overline{u_{n-1}}}(\rho(\widetilde{\eta_{n-1}}*\widetilde{\alpha_{n-1}})) &
C_{\overline{u_n}}(\rho(\widetilde{\eta_n}*\widetilde{\alpha_n}))
\end{pmatrix}.
\end{align*}
Similarly we have 
\begin{align*}
&(\sum_{\sigma\in S_n}(-1)^{\ell(\sigma)}tr_{(\gamma_n*\alpha_n)_{\sigma(n),u_n}}
tr_{(\gamma_{n-1}*\alpha_{n-1})_{\sigma(n-1),u_{n-1}}}\dots tr_{(\gamma_1*\alpha_1)_{\sigma(1),u_1}})(\rho)\\
= &(-1)^{\frac{n(n-1)}{2}} \text{det}
\begin{pmatrix}
C_{\overline{u_1}}(\rho(\widetilde{\gamma_1}*\widetilde{\alpha_1}))&\dots&
C_{\overline{u_{n-1}}}(\rho(\widetilde{\gamma_{n-1}}*\widetilde{\alpha_{n-1}}))&
C_{\overline{u_n}}(\rho(\widetilde{\gamma_n}*\widetilde{\alpha_n}))
\end{pmatrix}\\
= &(-1)^{\frac{n(n-1)}{2}} \text{det}
\begin{pmatrix}
C_{\overline{u_1}}(\rho(\widetilde{\gamma_1}*\widetilde{\eta_1}^{-1})\rho(\widetilde{\eta_1}*\widetilde{\alpha_1}))&\dots&
C_{\overline{u_n}}(\rho(\widetilde{\gamma_n}*\widetilde{\eta_n}^{-1})\rho(\widetilde{\eta_n}*\widetilde{\alpha_n}))
\end{pmatrix}\\
= &(-1)^{\frac{n(n-1)}{2}} \text{det}
\begin{pmatrix}
C_{\overline{u_1}}(\rho(\widetilde{\eta_1}*\widetilde{\alpha_1}))&\dots&
C_{\overline{u_{n-1}}}(\rho(\widetilde{\eta_{n-1}}*\widetilde{\alpha_{n-1}}))&
C_{\overline{u_n}}(\rho(\widetilde{\eta_n}*\widetilde{\alpha_n}))
\end{pmatrix}.
\end{align*}
The last equality is because $\rho(\widetilde{\gamma_1}*\widetilde{\eta_1}^{-1})
=\dots = \rho(\widetilde{\gamma_n}*\widetilde{\eta_n}^{-1})\in SL(n,\mathbb{C})$.

Suppose $\mathfrak{S}$ is a source. Similarly we have 
\begin{align*}
l^{'} = [\sum_{\sigma\in S_n}(-1)^{\ell(\sigma)}(\beta_n*\epsilon_n)_{v_n, \sigma(n)}(\beta_{n-1}*\epsilon_{n-1})_{v_{n-1}, \sigma(n-1)}\dots (\beta_1*\epsilon_1)_{v_1, \sigma(1)}]\;l_1\\
l^{''} = [\sum_{\sigma\in S_n}(-1)^{\ell(\sigma)}(\beta_n*\zeta_n)_{v_n, \sigma(n)}(\beta_{n-1}*\zeta_{n-1})_{v_{n-1}, \sigma(n-1)}\dots (\beta_1*\zeta_1)_{v_1, \sigma(1)}]\;l_1\\
\end{align*}
where $\beta_t,\epsilon_t,\zeta_t,1\leq t\leq n$, are framed oriented  arcs such that $\beta_t*\epsilon_t,
\beta_t*\zeta_t,1\leq t\leq n$, are well-defined framed oriented  boundary arcs in $(M,N)$. 

For any element $\rho\in\tilde{\chi}_n(M,N)$, similarly we can get 
\begin{align*}
&(\sum_{\sigma\in S_n}(-1)^{\ell(\sigma)}tr_{(\beta_n*\epsilon_n)_{v_n, \sigma(n)}}
tr_{((\beta_{n-1}*\epsilon_{n-1})_{v_{n-1}, \sigma(n-1)}}\dots tr_{(\beta_1*\epsilon_1)_{v_1, \sigma(1)}})(\rho)\\
= &(-1)^{\frac{n(n-1)}{2}} \text{det}
\begin{pmatrix}
R_{\overline{v_1}}(A\rho(\widetilde{\beta_1}*\widetilde{\epsilon_1}))\\
\vdots\\
R_{\overline{v_{n-1}}}(\rho(\widetilde{\beta_{n-1}}*\widetilde{\epsilon_{n-1}}))\\
R_{\overline{v_n}}(\rho(\widetilde{\beta_n}*\widetilde{\epsilon_n}))
\end{pmatrix}.
\end{align*}
Then we have 
\begin{align*}
&(\sum_{\sigma\in S_n}(-1)^{\ell(\sigma)}tr_{(\beta_n*\zeta_n)_{v_n, \sigma(n)}}
tr_{((\beta_{n-1}*\zeta_{n-1})_{v_{n-1}, \sigma(n-1)}}\dots tr_{(\beta_1*\zeta_1)_{v_1, \sigma(1)}})(\rho)\\
= &(-1)^{\frac{n(n-1)}{2}} \text{det}
\begin{pmatrix}
R_{\overline{v_1}}(A\rho(\widetilde{\beta_1}*\widetilde{\zeta_1}))\\
\vdots\\
R_{\overline{v_{n-1}}}(\rho(\widetilde{\beta_{n-1}}*\widetilde{\zeta_{n-1}}))\\
R_{\overline{v_n}}(\rho(\widetilde{\beta_n}*\widetilde{\zeta_n}))
\end{pmatrix}\\&=
(-1)^{\frac{n(n-1)}{2}} \text{det}
\begin{pmatrix}
R_{\overline{v_1}}(A\rho(\widetilde{\beta_1}*\widetilde{\epsilon_1})\rho(\widetilde{\epsilon_1}^{-1} * \widetilde{\zeta_1}) )\\
\vdots\\
R_{\overline{v_{n-1}}}(\rho(\widetilde{\beta_{n-1}}*\widetilde{\epsilon_{n-1}})\rho(\widetilde{\epsilon_{n-1}}^{-1} * \widetilde{\zeta_{n-1}}))\\
R_{\overline{v_n}}(\rho(\widetilde{\beta_n}*\widetilde{\epsilon_n})\rho(\widetilde{\epsilon_n}^{-1} * \widetilde{\zeta_n}) )\\
\end{pmatrix}\\
&=(-1)^{\frac{n(n-1)}{2}} \text{det}
\begin{pmatrix}
R_{\overline{v_1}}(A\rho(\widetilde{\beta_1}*\widetilde{\epsilon_1}))\\
\vdots\\
R_{\overline{v_{n-1}}}(\rho(\widetilde{\beta_{n-1}}*\widetilde{\epsilon_{n-1}}))\\
R_{\overline{v_n}}(\rho(\widetilde{\beta_n}*\widetilde{\epsilon_n}))
\end{pmatrix}.
\end{align*}

Thus we have $\Phi(l^{'}) = \Phi(l^{''})$. For general case, we have a sequence 
$l^{'} = l^{(1)},l^{(2)},\dots, l^{(k)}= l^{''}$ such that $l^{(t)},1\leq t\leq k,$ are obtained from $l$
by kill all sources and sinks using relation (\ref{wzh.five}) and $l^{(t)}, l^{(t+1)},1\leq t\leq k-1,$ are obtained from $l$ in only one different sink or source. Then $\Phi(l^{'}) = \Phi(l^{(1)})=\Phi(l^{(2)})=\dots=
\Phi( l^{(k)})= \Phi(l^{''})$.
Then $\Phi$ is well-defined on the set of framed $n$-webs.

Suppose the stated $n$-webs $l_1$ and $l_2$ are isotopic. From the definition of $\Phi$, we first use relation relation (\ref{wzh.five}) to kill all the sinks and sources to obtain $(l_1)^{'}$. According to the isotopy between $l_1$ and $l_2$ and how we skill sinks and sources in $l_1$,
 we can pick a way to kill all the sinks and sources in $l_2$ to obtain $(l_2)^{'}$ such that  $(l_1)^{'}$ and 
 $(l_2)^{'}$ are homopotic relative to $N$. Thus $\Phi(l_1) = \Phi(l_2)$, that is, $\Phi$ is well-defined on the set of isotopy classes of stated $n$-webs.

\subsection{Checking for relations ($N\neq \emptyset$)}\label{sub4.3}
We use $l^{'}$ (respectively $l^{''}$) to denote the stated $n$-webs on the left (repectively right) hand side of "$=$" in relations (\ref{w.cross})-(\ref{wzh.eight}).

It is obvious that $\Phi$ preserves relation (\ref{w.cross}).

Since, for any $\rho \in \tilde{\chi}_n(M,N)$, $\rho(\vartheta) = d_n I$, then $\Phi$ respects relations
(\ref{w.twist})
and
(\ref{w.unknot}).

Then we want to show $\Phi$ preserves relation  (\ref{wzh.four}). From the definition of $\Phi$ and Subsection \ref{sub4.2}, we can suppose the parts outside of the box are $2n$ framed oriented arcs connected to the box. We label the framed oriented arcs connected to the box on the left edge as $\alpha_n, \dots, \alpha_2,\alpha_1$ from top to bottom, 
label the framed oriented arcs connected to the box on the right edge as $\beta_n, \dots, \beta_2,\beta_1$ from top to bottom. Suppose $s(\alpha_t(0)) = u_t, s(\beta_t(1))=v_t,1\leq t\leq n$. To kill the sink and source in $l^{'}$ using relation (\ref{wzh.five}), we use the same path to drag them close to $N$. 

For $\rho\in\tilde{\chi}_n(M,N)$, we have
\begin{align*}
\Phi(l^{'})(\rho) =& \text{det}
\begin{pmatrix}
R_{\overline{v_n}}(A\rho(\widetilde{\beta_n}*\widetilde{\epsilon_n}))\\
\vdots\\
R_{\overline{v_1}}(A\rho(\widetilde{\beta_1}*\widetilde{\epsilon_1}))
\end{pmatrix}\text{det}
\begin{pmatrix}
C_{\overline{u_1}}(\rho(\widetilde{\eta_1}*\widetilde{\alpha_1}))&\dots &
C_{\overline{u_n}}(\rho(\widetilde{\eta_n}*\widetilde{\alpha_n}))
\end{pmatrix}\\
=&\text{det}
\begin{pmatrix}
R_{\overline{v_n}}(A\rho(\widetilde{\beta_n}*\widetilde{\epsilon_n}))C_{\overline{u_1}}(\rho(\widetilde{\eta_1}*\widetilde{\alpha_1}))&\dots&
R_{\overline{v_n}}(A\rho(\widetilde{\beta_n}*\widetilde{\epsilon_n}))C_{\overline{u_n}}(\rho(\widetilde{\eta_n}*\widetilde{\alpha_n}))\\
\vdots& &\vdots\\
R_{\overline{v_1}}(A\rho(\widetilde{\beta_1}*\widetilde{\epsilon_1}))C_{\overline{u_1}}(\rho(\widetilde{\eta_1}*\widetilde{\alpha_1}))&\dots&
R_{\overline{v_1}}(A\rho(\widetilde{\beta_1}*\widetilde{\epsilon_1}))C_{\overline{u_n}}(\rho(\widetilde{\eta_n}*\widetilde{\alpha_n}))\\
\end{pmatrix}\\
=&\text{det}
\begin{pmatrix}
[A\rho(\widetilde{\beta_n}*\widetilde{\epsilon_n})\rho(\widetilde{\eta_1}*\widetilde{\alpha_1})]_{\overline{v_n},\overline{u_1}}&\dots&
[A\rho(\widetilde{\beta_n}*\widetilde{\epsilon_n})\rho(\widetilde{\eta_n}*\widetilde{\alpha_n})]_{\overline{v_n},\overline{u_n}}\\
\vdots& &\vdots\\
[A\rho(\widetilde{\beta_1}*\widetilde{\epsilon_1})\rho(\widetilde{\eta_1}*\widetilde{\alpha_1})]_{\overline{v_1},\overline{u_1}}&\dots&
[A\rho(\widetilde{\beta_1}*\widetilde{\epsilon_1})\rho(\widetilde{\eta_n}*\widetilde{\alpha_n})]_{\overline{v_1},\overline{u_n}}\\
\end{pmatrix}\\
=&(-1)^{\frac{n(n-1)}{2}}\text{det}
\begin{pmatrix}
[A\rho(\widetilde{\beta_1}*\widetilde{\epsilon_1}*\widetilde{\eta_1}*\widetilde{\alpha_1})]_{\overline{v_1},\overline{u_1}}&\dots&
[A\rho(\widetilde{\beta_1}*\widetilde{\epsilon_1}*\widetilde{\eta_n}*\widetilde{\alpha_n})]_{\overline{v_1},\overline{u_n}}\\
\vdots& &\vdots\\
[A\rho(\widetilde{\beta_n}*\widetilde{\epsilon_n}*\widetilde{\eta_1}*\widetilde{\alpha_1})]_{\overline{v_n},\overline{u_1}}&\dots&
[A\rho(\widetilde{\beta_n}*\widetilde{\epsilon_n}*\widetilde{\eta_n}*\widetilde{\alpha_n})]_{\overline{v_n},\overline{u_n}}\\
\end{pmatrix}\\
\end{align*}

For each pair $1\leq i,j\leq n$, we use $a_{i,j}$ to denote the oriented straight line in the shaded box (it's framing is the one pointing to readers)
connecting  $\alpha_j(1)$  and $\beta_i(0)$  such that $\beta_i * a_{i,j} * \alpha_j$ is a well-defined stated framed oriented boundary arc.
Because we use the same path to drag the source and the sink,
then $$\rho(\widetilde{\beta_i}*\widetilde{\epsilon_i}*\widetilde{\eta_j}*\widetilde{\alpha_j})
= d_n \rho(\widetilde{\beta_i}*\widetilde{a_{i,j}}*\widetilde{\alpha_j})$$ for all $1\leq i,j\leq n$,
or 
$$\rho(\widetilde{\beta_i}*\widetilde{\epsilon_i}*\widetilde{\eta_j}*\widetilde{\alpha_j})
=  \rho(\widetilde{\beta_i}*\widetilde{a_{i,j}}*\widetilde{\alpha_j})$$ for all $1\leq i,j\leq n$.

For any $\rho\in\tilde{\chi}_n(M,N)$, since $(d_n)^{n} = 1$, we have 
\begin{align*}
&\Phi(l^{''})(\rho) \\= &(-1)^{\frac{n(n-1)}{2}}
\sum_{\sigma\in S_n}(-1)^{l(\sigma)}[A\rho(\widetilde{\beta_{\sigma(1)}}*\widetilde{\epsilon_{\sigma(1)}}*\widetilde{\eta_1}*\widetilde{\alpha_1})]_{\overline{v_{\sigma(1)}},\overline{u_1}}\dots
[A\rho(\widetilde{\beta_{\sigma(n)}}*\widetilde{\epsilon_{\sigma(n)}}*\widetilde{\eta_n}*\widetilde{\alpha_n})]_{\overline{v_{\sigma(n)}},\overline{u_n}}\\
=&(-1)^{\frac{n(n-1)}{2}}\text{det}
\begin{pmatrix}
[A\rho(\widetilde{\beta_1}*\widetilde{\epsilon_1}*\widetilde{\eta_1}*\widetilde{\alpha_1})]_{\overline{v_1},\overline{u_1}}&\dots&
[A\rho(\widetilde{\beta_1}*\widetilde{\epsilon_1}*\widetilde{\eta_n}*\widetilde{\alpha_n})]_{\overline{v_1},\overline{u_n}}\\
\vdots& &\vdots\\
[A\rho(\widetilde{\beta_n}*\widetilde{\epsilon_n}*\widetilde{\eta_1}*\widetilde{\alpha_1})]_{\overline{v_n},\overline{u_1}}&\dots&
[A\rho(\widetilde{\beta_n}*\widetilde{\epsilon_n}*\widetilde{\eta_n}*\widetilde{\alpha_n})]_{\overline{v_n},\overline{u_n}}\\
\end{pmatrix}\\
\end{align*}

Thus $\Phi(l^{'}) = \Phi(l^{''})$.

From the definition of $\Phi$, we know $\Phi$ respects relation
(\ref{wzh.five}).

Then we want to show $\Phi$ preserves relation (\ref{wzh.six}), that is, to show
$$\Phi(\raisebox{-.20in}{
\begin{tikzpicture}
\tikzset{->-/.style=
{decoration={markings,mark=at position #1 with
{\arrow{latex}}},postaction={decorate}}}
\filldraw[draw=white,fill=gray!20] (-0.7,-0.7) rectangle (0,0.7);
\draw [line width =1.5pt,decoration={markings, mark=at position 1 with {\arrow{>}}},postaction={decorate}](0,0.7)--(0,-0.7);
\draw [color = black, line width =1pt] (0 ,0.3) arc (90:270:0.5 and 0.3);
\node [right]  at(0,0.3) {$i$};
\node [right] at(0,-0.3){$j$};
\draw [line width =1pt,decoration={markings, mark=at position 0.5 with {\arrow{>}}},postaction={decorate}](-0.5,0.02)--(-0.5,-0.02);
\end{tikzpicture}} ) = \Phi( \raisebox{-.20in}{
\begin{tikzpicture}
\tikzset{->-/.style=
{decoration={markings,mark=at position #1 with
{\arrow{latex}}},postaction={decorate}}}
\filldraw[draw=white,fill=gray!20] (-0.7,-0.7) rectangle (0,0.7);
\draw [line width =1.5pt,decoration={markings, mark=at position 1 with {\arrow{>}}},postaction={decorate}](0,0.7)--(0,-0.7);
\draw [color = black, line width =1pt] (0 ,0.3) arc (90:270:0.5 and 0.3);
\node [right]  at(0,0.3) {$i$};
\node [right] at(0,-0.3){$j$};
\draw [line width =1pt,decoration={markings, mark=at position 0.5 with {\arrow{<}}},postaction={decorate}](-0.5,0.02)--(-0.5,-0.02);
\end{tikzpicture}} ) = \delta_{\bar j,i }\,  (-1)^{n-i}.$$
We use $\alpha_1$ to denote $\raisebox{-.20in}{
\begin{tikzpicture}
\tikzset{->-/.style=
{decoration={markings,mark=at position #1 with
{\arrow{latex}}},postaction={decorate}}}
\filldraw[draw=white,fill=gray!20] (-0.7,-0.7) rectangle (0,0.7);
\draw [line width =1.5pt,decoration={markings, mark=at position 1 with {\arrow{>}}},postaction={decorate}](0,0.7)--(0,-0.7);
\draw [color = black, line width =1pt] (0 ,0.3) arc (90:270:0.5 and 0.3);
\node [right]  at(0,0.3) {$i$};
\node [right] at(0,-0.3){$j$};
\draw [line width =1pt,decoration={markings, mark=at position 0.5 with {\arrow{>}}},postaction={decorate}](-0.5,0.02)--(-0.5,-0.02);
\end{tikzpicture}} $ and $\alpha_2$ to denote 
$\raisebox{-.20in}{
\begin{tikzpicture}
\tikzset{->-/.style=
{decoration={markings,mark=at position #1 with
{\arrow{latex}}},postaction={decorate}}}
\filldraw[draw=white,fill=gray!20] (-0.7,-0.7) rectangle (0,0.7);
\draw [line width =1.5pt,decoration={markings, mark=at position 1 with {\arrow{>}}},postaction={decorate}](0,0.7)--(0,-0.7);
\draw [color = black, line width =1pt] (0 ,0.3) arc (90:270:0.5 and 0.3);
\node [right]  at(0,0.3) {$i$};
\node [right] at(0,-0.3){$j$};
\draw [line width =1pt,decoration={markings, mark=at position 0.5 with {\arrow{<}}},postaction={decorate}](-0.5,0.02)--(-0.5,-0.02);
\end{tikzpicture}} $, we have 
$$\Phi(\alpha_1)(\rho) = [A\rho(\widetilde{\alpha_1})]_{\bar{j},\bar{i}}
= d_n A_{\bar{j},\bar{i}} =  \delta_{\bar j,i }\,  (-1)^{n-i},\Phi(\alpha_2)(\rho) = [A\rho(\widetilde{\alpha_2})]_{\bar{i},\bar{j}}
= A_{\bar{i},\bar{j}} =  \delta_{\bar j,i }\,  (-1)^{n-i}.$$

For relation (\ref{wzh.seven}), we only prove the case when the white dot represents an arrow going from right to left.
From the definition of $\Phi$ and Subsection \ref{sub4.2}, we only have two cases to consider: (1) the left hand side of "$=$" is a knot, (2)
the left hand side of "$=$" is an arc. 

When the left hand side of "$=$" is a knot, the right hand side of "$=$" is a framed oriented boundary arc, which is denoted as $\alpha$. 
Then for $\rho\in\tilde{\chi}_n(M,N)$, we have 
\begin{align*}
\Phi(l^{''})(\rho) = \sum_{1\leq i\leq n}  (-1)^{i+1}[A\rho(\tilde{\alpha})]_{i,\bar{i}}=
\sum_{1\leq i\leq n}  [\rho(\tilde{\alpha})]_{\bar{i},\bar{i}} = \text{Trace}(\rho(\tilde{\alpha}))
= \Phi(l^{'})(\rho).
\end{align*}

When the left hand side of "$=$" is an arc, the right hand side of "$=$" consists two framed oriented boundary arcs, which are denoted as $\gamma_2$ and $\gamma_1$ such that $\gamma_2$ is above $\gamma_1$ is the box. 
Suppose $s(\gamma_1(0)) = v, s(\gamma_2(1)) = u$.
Then for $\rho\in\tilde{\chi}_n(M,N)$, we have 
\begin{align*}
\Phi(l^{''})(\rho) &= \sum_{1\leq i\leq n}  (-1)^{i+1}[A\rho(\widetilde{\gamma_2})]_{\bar{u},\bar{i}}
[A\rho(\widetilde{\gamma_1})]_{{i},\bar{v}}=
 \sum_{1\leq i\leq n}  [A\rho(\widetilde{\gamma_2})]_{\bar{u},\bar{i}}
[\rho(\widetilde{\gamma_1})]_{\bar{i},\bar{v}}\\
&=[A\rho(\widetilde{\gamma_2})\rho(\widetilde{\gamma_1})]_{\bar{u},\bar{v}}
=[A\rho(\widetilde{\gamma_2}*\widetilde{\gamma_1})]_{\bar{u},\bar{v}}
= \Phi(l^{'})(\rho).
\end{align*}

From the definition of $\Phi$, we know $\Phi$ respects relation
 (\ref{wzh.eight}).

\subsection{Algebra homomorphism and surjectivity}
The definition of $\Phi$ indicates it is an algebra homormophism. 

When $N$ is empty, we already proved $\Phi$ is surjective.
If $N\neq \emptyset$.
 For any element $[\alpha]\in \pi_1(M,N)$, we choose a representative $\alpha$ for $[\alpha]$ such that $\alpha$ has no self-intersection and only intersects $\partial M$ at it's endpoints. Then we can give a framing to $\alpha$ to make $\alpha$ a framed oriented boundary arc for $(M,N)$. Then $R_n(M,N)$ is generated by $\Phi(\alpha_{j,i}),1\leq j,i\leq n, [\alpha]
\in\pi_1(M,N)$, as an algebra. Thus $\Phi$ is surjective.

In Section \ref{subb5}, we will give a unique way to lift  $[\alpha]$ to a framed oriented boundary arc.

\section{Classical limit  and Ker$\Phi$}\label{subb5}

In this section we try to understand the classical limit of stated $SL(n)$-skein module of three manifold. Then we will use the classical limit to show the Kernal of $\Phi$ is $\sqrt{0}$. We can use the  Lemma \ref{6688} to reduce the general three manifold to connected three manifold.

\begin{lemma}(\cite{blyth2018module,przytycki1998fundamentals})\label{8866}
Suppose $0\rightarrow A_1\rightarrow B_1 \rightarrow C_1\rightarrow 0,$ and 
$0\rightarrow A_2\rightarrow B_2 \rightarrow C_2\rightarrow 0$ are two short exact sequences, then 
$$0\rightarrow A_1\otimes B_2 + B_1\otimes A_2\rightarrow B_1\otimes B_2\rightarrow C_1\otimes C_2\rightarrow 0$$ is an exact sequence. All $A_i,B_i,C_i$ are vector spaces over $\mathbb{C}$ and all the involved maps are linear maps.

\end{lemma}

Actually
Lemma \ref{8866}  is true with general commutative ring $R$ with unit and general $R$-modules.

\begin{lemma}\label{6688}
Suppose $(M,N)$ is the disjoint union of $(M_1,N_1)$ and $(M_2,N_2)$. If Ker$\,\Phi^{(M_i,N_i)}
=\sqrt{0}_{S_n(M_i,N_i,1)}$ for $i=1,2$, then we have Ker$\,\Phi^{(M,N)}
=\sqrt{0}_{S_n(M,N,1)}$.
\end{lemma}

\begin{proof}
Since $R_n(M,N)$ contains no nonzero nilpotents, we have  $\sqrt{0}_{S_n(M,N,1)}\subset$\,Ker$\,\Phi^{(M,N)}.$
We have $S_n(M,N,1) = S_n(M_1,N_1,1)\otimes S_n(M_2,N_2,1),\;R_n(M,N) = R_n(M_1,N_1) \otimes R_n(M_2,N_2)$. From the assumption, we have the following two exact sequences:
\begin{align*}
0\rightarrow \sqrt{0}_{S_n(M_i,N_i,1)}\rightarrow S_n(M_i,N_i,1) \rightarrow R_n(M_i,N_i)\rightarrow 0
\end{align*}
for $i=1,2$.
Then from Lemma \ref{8866}, we get the following exact sequence:
\begin{align*}0\rightarrow \sqrt{0}_{S_n(M_1,N_1,1)}\otimes S_n(M_2,N_2,1) + S_n(M_1,N_1,1)\otimes\sqrt{0}_{S_n(M_2,N_2,1)}\rightarrow\\ S_n(M_1,N_1,1)\otimes S_n(M_2,N_2,1)\rightarrow R_n(M_1,N_1)\otimes R_n(M_2,N_2)\rightarrow 0.
\end{align*}
Thus we have 
$$\text{Ker}\Phi^{(M,N)} = \sqrt{0}_{S_n(M_1,N_1,1)}\otimes S_n(M_2,N_2,1) + S_n(M_1,N_1,1)\otimes\sqrt{0}_{S_n(M_2,N_2,1)}\subset \sqrt{0}_{S_n(M,N,1)}.$$
Thus we have 
$$\text{Ker}\Phi^{(M,N)} = \sqrt{0}_{S_n(M_1,N_1,1)}\otimes S_n(M_2,N_2,1) + S_n(M_1,N_1,1)\otimes\sqrt{0}_{S_n(M_2,N_2,1)}= \sqrt{0}_{S_n(M,N,1)}.$$

\end{proof}

Then in the remaining of this section, we will assume all the marked three manifolds involved are connected.
We also fix a relative spin structure $h$ for $(M,N)$. For any framed oriented  boundary arc $\alpha$ in $(M,N)$, we consider $[\alpha]$ as an element in $\pi_1(M,N)$ by forgetting the framing of $\alpha$.

\begin{rem}\label{rre5.1}
An element $[\alpha] \in \pi_1(M,N)$ and $1\leq i,j\leq n$ uniquely determine an element in $S_n(M,N,1)$ in the following way: We choose a good representative $\alpha$ such that $\alpha(1)$ is higher than $\alpha(0)$ if $\alpha(0)$ and $\alpha(1)$ belong to the same component of $N$,  $\alpha$ only intersects $\partial M$ at it's endpoints, and $\alpha$ does not intersect itself. Then we give a framing to $\alpha$ respecting $N$, that is, the framing at endpoints are given by the velocity vectors of $N$. We denote this framed oriented boundary arc as $\hat{\alpha}$. We choose the framing such that $h(\widetilde{\hat{\alpha}}) = 0$, then we obtain an element $\hat{\alpha}_{i,j}
\in S_n(M,N,1)$. Suppose we choose a different good representative $\alpha^{'}$. We have $[\alpha] = [\alpha^{'}]\in \pi_1(M,N)$ and $h(\widetilde{\hat{\alpha^{'}}}) = h(\widetilde{\hat{\alpha}}) = 0$. Then $\hat{\alpha^{'}}_{i,j} = \hat{\alpha}_{i,j}$ because of relations (\ref{w.cross}), (\ref{w.twist}) , (\ref{wzh.eight}) and Corollary \ref{cccc3.2}.

 We use 
$S^{[\alpha]}_{i,j}$ to denote $\hat{\alpha}_{i,j}$, and $S^{[\alpha]}$ to denote an $n$ by $n$  matrix in $S_n(M,N,1)$ such that $(S^{[\alpha]})_{i,j} = S^{[\alpha]}_{i,j},1\leq i,j\leq n$.

For any two stated oriented framed boundary arcs $\alpha_1,\alpha_2$, suppose $s(\alpha_1(0))
= s(\alpha_2(0))$ and $s(\alpha_1(1))
= s(\alpha_2(1))$. If $h(\widetilde{\alpha_1}) = h(\widetilde{\alpha_2}) $ and $[\alpha_1] = [\alpha_2]\in
\pi_1(M,N)$, then $\alpha_1 =\alpha_2 \in S_n(M,N,1)$ because of relations (\ref{w.cross}), (\ref{w.twist}), (\ref{wzh.eight}) and Corollary \ref{cccc3.2}.

The for any stated oriented framed boundary arcs $\alpha_{i,j}$, we have $$\alpha_{i,j} =
d_n^{h(\tilde{\alpha})} S^{[\alpha]}_{i,j}\in S_n(M,N,1).$$
\end{rem}

\begin{proposition}\label{prop5.2}
(a) For any two elements $[\alpha],[\beta]\in \pi_1(M,N)$, if $[\beta][\alpha]$ makes sense, then 
$A S^{[\beta*\alpha]} = A S^{[\beta]} A S^{[\alpha]}$.

(b) For any $[\eta]\in \pi_1(M,N)$, we have 
det$(S^{[\eta]}) = 1$. Especially det$(A S^{[\eta]}) = 1$.

(c) Suppose $[o]\in \pi_1(M,N)$ is the identity morphism for an object, then $S^{[o]} = d_n A$.
Especially $A S^{[o]} = I$.
\end{proposition}
\begin{proof}
(a) We have 
$$(S^{[\beta]} A S^{[\alpha]})_{i,j} = \sum_{1\leq k\leq n} (-1)^{k+1}S^{[\beta]}_{i,k} S^{[\alpha]}_{\bar{k}, j}
= S^{[\beta*\alpha]}_{i,j} = (S^{[\beta*\alpha]})_{i,j}$$
where the second equality is because of relation (\ref{wzh.seven}).
Thus $A S^{[\beta*\alpha]} = A S^{[\beta]} A S^{[\alpha]}$.

%

(b)We have 
$$\text{det}(S^{[\eta]}) = \sum_{\sigma\in S_n} (-1)^{l(\sigma)} S^{[\eta]}_{1,\sigma(1)}
 S^{[\eta]}_{2,\sigma(2)} \dots  S^{[\eta]}_{n,\sigma(n)}
= \raisebox{-.30in}{
\begin{tikzpicture}
\tikzset{->-/.style=
{decoration={markings,mark=at position #1 with
{\arrow{latex}}},postaction={decorate}}}
\filldraw[draw=white,fill=gray!20] (0,-0.7) rectangle (1.2,1.3);
\draw [line width =1.5pt,decoration={markings, mark=at position 1 with {\arrow{>}}},postaction={decorate}](1.2,-0.7)--(1.2,1.3);
\draw [line width =1pt,decoration={markings, mark=at position 0.5 with {\arrow{<}}},postaction={decorate}](1.2,1)  --(0.2,0);
\draw [line width =1pt,decoration={markings, mark=at position 0.5 with {\arrow{<}}},postaction={decorate}](1.2,0)  --(0.2,0);
\draw [line width =1pt,decoration={markings, mark=at position 0.5 with {\arrow{<}}},postaction={decorate}](1.2,-0.4)--(0.2,0);
\node  at(1,0.5) {$\vdots$};
\node [right] at(1.2,1) {$1$};
\node [right] at(1.2,0) {$n-1$};
\node [right] at(1.2,-0.4) {$n$};
\end{tikzpicture}}
 = 1$$
where the second equality is from relation (\ref{wzh.five}) and 
the last equality is because of equation (54) in  \cite{le2021stated}.

(c) For $1\leq i,j\leq n$, we have 
$$S^{[o]}_{i,j} = \raisebox{-.20in}{
\begin{tikzpicture}
\tikzset{->-/.style=
{decoration={markings,mark=at position #1 with
{\arrow{latex}}},postaction={decorate}}}
\filldraw[draw=white,fill=gray!20] (-0.7,-0.7) rectangle (0,0.7);
\draw [line width =1.5pt,decoration={markings, mark=at position 1 with {\arrow{>}}},postaction={decorate}](0,-0.7)--(0,0.7);
\draw [color = black, line width =1pt] (0 ,0.3) arc (90:270:0.5 and 0.3);
\node [right]  at(0,0.3) {$i$};
\node [right] at(0,-0.3){$j$};
\draw [line width =1pt,decoration={markings, mark=at position 0.5 with {\arrow{<}}},postaction={decorate}](-0.5,0.02)--(-0.5,-0.02);
\end{tikzpicture}} = d_n A_{i,j}.$$
Thus $ S^{[o]} = d_n A$.

\end{proof}

\subsection{Isomomorphism between $S_n(M,N,1)$ and $\Gamma_n(M)$ when $N$ has one component}

In this subsection, $N$ always containes one component unless specified.
If $N$ has only one component, then $\pi_1(M,N) = \pi_1(M)$ (we choose the base point for $\pi_1(M)$ to be a point in $N$).

\begin{lemma}\label{lmm5.3}

Let $(M,N)$ be a marked three manifold with $N$ consisting of one component. There exists an algebra homomorphism
$F:\Gamma_n(M)\rightarrow S_n(M,N,1)$ defined by 
$$F([\alpha]_{i,j}) = F((Q_{[\alpha]})_{i,j}) = (A S^{[\alpha]})_{i,j}$$
where $[\alpha]\in \pi_1(M,N),1\leq i,j\leq n.$

\end{lemma}

\begin{proof}

Lemma \ref{prop5.2} shows $F$ respects all the relations defined for $\Gamma_n(M)$, thus
$F$ is a well-defined algebra homomorphism.
\end{proof}

Let $\alpha$ be a framed oriented arc in $S_n(M,N,1)$. Then $[\alpha]$ is an element in $\pi_1(M,N)$ by forgetting  the framing of $\alpha$. We define $G(\alpha_{i,j}) = d_n^{h(\tilde{\alpha})+1}(-1)^{i+1}[\alpha]_{\bar{i},j}\in \Gamma_n(M)$.
For a  framed oriented  knot $\alpha$, first we forget the framing of $\alpha$ and then we use a path $\beta$ to connect $\alpha$ and $N$. Then we obtain an elemnt in $\pi_1(M,N)$, which is denoted as $[\alpha_{\beta}]$. 
We define $G(\alpha) = d_n^{h(\tilde{\alpha})}\text{Trace}(Q_{[\alpha_{\beta}]})\in \Gamma_n(M)$. It  is easy to show $G(\alpha)$ is independent of the choice of $\beta$.

For any stated $n$-web $l$, we use relation (\ref{wzh.five}) to kill all the sinks and sources  to obtain a new stated $n$-web $l^{'}$. Suppose $l^{'} = \cup_{\alpha}\alpha$ where each $\alpha$ is a stated framed oriented boundary arc or a framed oriented knot, define $G(l) = \Pi_{\alpha} G(\alpha)$.

\begin{lemma}\label{lmm5.4}
The above map $G:S_n(M,N,1)\rightarrow \Gamma_n(M)$ is a well-defined algebra homomorphism.
\end{lemma}
\begin{proof}
We prove $G$ is well-defined in two steps. First we prove the definition of $G$ is independent of the choice of how we kill sinks and sources, then we prove $G$ respects all the relations defined for $S_n(M,N,1)$. Note that these two steps appeared when we tried to prove Theorem \ref{thm3.11}. Actually the proving techniques here are the same with the techniques used in Subsections \ref{sub4.2} and \ref{sub4.3}. So here we  omit all the details. 
\end{proof}

\begin{theorem}\label{thm5.5}
Let $(M,N)$ be a marked three manifold with $N$ consisting of one component. There exist  algebra homomorphisms
$F:\Gamma_n(M)\rightarrow S_n(M,N,1), G:S_n(M,N,1)\rightarrow \Gamma_n(M)$ such that 
$F\circ G = Id_{S_n(M,N,1)}$ and $G\circ F = Id_{\Gamma_n(M)}$. Especially $\Gamma_n(M)\simeq S_n(M,N,1)$.
\end{theorem}
\begin{proof}
Lemmas \ref{lmm5.3} and \ref{lmm5.4} show the existence of $F$  and $G$. It remains to show they are inverse to each other. 

For any $[\alpha]\in \pi_1(M,N), 1\leq i,j\leq n$, we have 
$$G(F([\alpha]_{i,j})) = G ((-1)^{i+1} S^{[\alpha]}_{\bar{i}, j})
= (-1)^{i+1}d_n (-1)^{\bar{i}+1}[\alpha]_{i,j} = [\alpha]_{i,j}.$$
Thus $G\circ F = Id_{\Gamma_n(M)}$ since $[\alpha]_{i,j}$,  $[\alpha]\in \pi_1(M,N), 1\leq i,j\leq n$,  generate $\Gamma_n(M)$ as an algebra.

Obviously all the stated framed oriented  boundary arcs generate $S_n(M,N,1)$ as an algebra. For any stated oriented framed boundary arc $\alpha_{i,j}\in S_n(M,N,1)$, we have
$$F(G(\alpha_{i,j})) = d_n^{h(\tilde{\alpha})+1}(-1)^{i+1}F([\alpha]_{\bar{i},j})
= d_n^{h(\tilde{\alpha})+1}(-1)^{i+1} (-1)^{\bar{i}+1}S^{[\alpha]}_{i,j} = d_n^{h(\tilde{\alpha})}S^{[\alpha]}_{i,j} = \alpha_{i,j}. $$
Thus $F\circ G = Id_{S_n(M,N,1)}.$

\end{proof}

\begin{rem}
Korinman and Murakami proved the isomorphism between $S_2(M,N,1)$ and $\Gamma_2(M)$ by using a  different technique
\cite{korinman2022relating}.
\end{rem}

\subsection{Adding one extra marking to marked three manifold}

In this subsection, we will investigate the effects on $S_n(M,N,1)$ when we put one extra marking on $\partial M$. 

\begin{definition}
Let $(M,N)$ be a marked three manifold. We say that $N^{'}$  is obtained from $N$ by adding one extra marking if
$N^{'}= N\cup e$ where $e$ is an oriented open interval in $\partial M$ such that $cl(e)\cap cl(N) =\emptyset$. We call the linear map 
 $S_n(M,N,v)\rightarrow S_n(M,N^{'},v)$, induced by embedding $(M,N)\rightarrow (M,N^{'})$, adding marking map.
Obviously this map is an algebra homomorphism when $v=1$. We will use $l_{ad}^{e}$ to denote the adding marking map. We can omit the superscript when there is no confusion with marking $e$.
\end{definition}

\begin{rem}\label{rem5.8}
Suppose $(M,N)$ is a marked three manifold with $N\neq\emptyset$, and $N^{'}$ is obtained from $N$ by adding one extra marking $e$. Let $\alpha$ be a path connecting $N$ and $e$. We require $\alpha$ does not intersect itself, $\alpha$ only intersects $\partial M$ at it's endpoints, and $\alpha(0)$ belongs to a component $e_1\subset N$, and $\alpha(1)\in e$. Then we give a framing to $\alpha$ to obtain a  framed oriented boundary arc in $(M,N^{'})$, which is still denoted as $\alpha$, such that $h(\tilde{\alpha}) = 0$. Similarly we give a framing to $\alpha^{-1}$ such that
$h(\widetilde{\alpha^{-1}}) = 0$. Then we have $\alpha_{i,j} = S^{[\alpha]}_{i,j}, \alpha^{-1}_{i,j}
= S^{[\alpha^{-1}]}_{i,j},1\leq i,j\leq n.$ From Proposition \ref{prop5.2}, we know $AS^{[\alpha]}AS^{[\alpha^{-1}]} = AS^{[\alpha^{-1}]}AS^{[\alpha]} = AS^{[o]} = I.$
\end{rem}

We can regard $S_n(M,N^{'},1)$ as an $S_n(M,N,1)$-algebra because of the adding marking map. 

\begin{lemma}\label{lmm5.8}
Suppose $(M,N)$ is a marked three manifold with $N\neq\emptyset$, and $N^{'}$ is obtained from $N$ by adding one extra marking $e$. Then as an $S_n(M,N,1)$-algebra, $S_n(M,N^{'},1)$ is generated by $\alpha_{i,j},1\leq i,j\leq n$.
\end{lemma}
\begin{proof}
Let $T$ be the $S_n(M,N,1)$-subalgebra of $S_n(M,N^{'},1)$ generated by $\alpha_{i,j},1\leq i,j\leq n$. 
Since det$(S^{[\alpha]}) = 1\in S_n(M,N^{'},1)$, we have $(S^{[\alpha]})^{-1}$ is well-defined and 
$[(S^{[\alpha]})^{-1}]_{i,j}, 1\leq i,j\leq n,$  are polynomials in $\alpha_{i,j},1\leq i,j\leq n$. Especially we have $[(S^{[\alpha]})^{-1}]_{i,j}\in T , 1\leq i,j\leq n$.
We also have $S^{[\alpha^{-1}]}= A^{-1} (S^{[\alpha]})^{-1} A^{-1} $, thus 
$\alpha^{-1}_{i,j}
= S^{[\alpha^{-1}]}_{i,j}\in T,1\leq i,j\leq n.$

From relation (\ref{wzh.five}), we know, as an $S_n(M,N,1)$-algebra, $S_n(M,N^{'},1)$ is generated by stated framed oriented boundary arcs with at least one end point in $e$. Suppose $\beta_{i,j}$ is  such  an arc in $S_n(M,N^{'},1)$. Recall that $\beta_{i,j} = d_n^{h(\tilde{\beta})}S^{[\beta]}_{i,j}$.

For the case when  $\beta(0),\beta(1)\in e$, we have 
$$A S^{[\beta]} = AS^{[\alpha]} AS^{[\alpha^{-1}*\beta*\alpha]} AS^{[\alpha^{-1}]}$$
where $[\alpha^{-1}*\beta*\alpha]$ is a path with two end points in $N$. Especially we get
$$S^{[\beta]} = S^{[\alpha]} AS^{[\alpha^{-1}*\beta*\alpha]} AS^{[\alpha^{-1}]}.$$ Then 
$S^{[\beta]}_{i,j}\in T,1\leq i,j\leq n$, because $S^{[\alpha]}_{i,j}, S^{[\alpha^{-1}*\beta*\alpha]}_{i,j},
S^{[\alpha^{-1}]}_{i,j}\in T ,1\leq i,j\leq n.$ Thus $\beta_{i,j} = d_n^{h(\tilde{\beta})}S^{[\beta]}_{i,j}\in T$.

For the other cases, we can use the same way to show $\beta_{i,j} = d_n^{h(\tilde{\beta})}S^{[\beta]}_{i,j}\in T$.
Thus $T = S_n(M,N^{'},1).$

\end{proof}

We use $O(SLn)$ to denote the algebra by setting $q=1$ for $O_q(SLn)$. To distinguish the generators for
$O(SLn)$ and $O_q(SLn)$, we use $x_{i,j}$, instead of $u_{i,j}$, to denote the generators of $O(SLn)$. Then 
$$O(SLn)=\mathbb{C}[x_{i,j}\mid 1\leq i,j\leq n]/(\text{det}(X) = 1)$$
where $X$ is an $n$ by $n$ matrix such that $X_{i,j} = x_{i,j},1\leq i,j\leq n.$ 
We have $X^{-1}$ makes sense and is an $n$ be $n$ matrix in $O(SLn)$ because det$(X) = 1$. For $1\leq i,j\leq n$, We use $x^{-1}_{i,j}$ to denote $(X^{-1})_{i,j}$.
Obviously $S_n(M,N,1)\otimes O(SLn)$ is an $S_n(M,N,1)$-algebra, and as an $S_n(M,N,1)$-algebra we have
$$S_n(M,N,1)\otimes O(SLn) = S_n(M,N,1)[x_{i,j}\mid 1\leq i,j\leq n]/(\text{det}(X) = 1)$$ 
by regarding $1\otimes x_{i,j}$ as $x_{i,j}$.

\begin{lemma}\label{lmm5.9}
Suppose $(M,N)$ is a marked three manifold with $N\neq\emptyset$, and $N^{'}$ is obtained from $N$ by adding one extra marking $e$. Then there exists an $S_n(M,N,1)$-algebra homomorphism
\begin{align*}
 \imath:S_n(M,N,1)\otimes O(SLn)&\rightarrow S_n(M,N^{'},1)\\
 1\otimes x_{i,j}&\mapsto (A S^{[\alpha]})_{i,j}.
\end{align*}
\end{lemma}
\begin{proof}
Since  $S_n(M,N^{'},1)$ is a commutative $S_n(M,N,1)$-algebra and det$(AS^{[\alpha]}) = 1\in S_n(M,N^{'},1)$, then
$\imath$ is a well-defined $S_n(M,N,1)$-algebra homomorphism. 
\end{proof}

Next we try to define an $S_n(M,N,1)$-algebra homormorphism $$\jmath:S_n(M,N^{'},1)\rightarrow 
S_n(M,N,1)\otimes O(SLn).$$

Let $l$ be an stated $n$-web in $(M,N^{'})$, and $s_{l}$ be the state of $l$. If $l\cap e =\emptyset$, define 
$\jmath(l) = l\otimes 1\in S_n(M,N,1)\otimes O(SLn)$.

 If $l\cap e\neq \emptyset$, we 
suppose $|l\cap e| = m$, then we label the ends of $l$ on $e$ from $1$ to $m$. We use $E_k$ to denote the end of $l$ at $e$ labeled by number $k, 1\leq k\leq m.$
Define
\begin{align*}
\begin{split}
 f_k&= \left \{
 \begin{array}{ll}
     -1,                    & \text{if }E_k\text{ points towards }e,\\
     1,     & \text{if }E_k\text{ points out of }e,\\
 \end{array}
 \right.\\
 g_k&= \left \{
 \begin{array}{ll}
     Id\in S_n,                    & \text{if }E_k\text{ points towards }e,\\
     \delta\in S_n,     & \text{if }E_k\text{ points out of }e,\\
 \end{array}
 \right.\\
 h_k (i,j)&= \left \{
 \begin{array}{ll}
     i,j,                    & \text{if }E_k\text{ points towards }e,\\
     j,i,     & \text{if }E_k\text{ points out of }e,\\
 \end{array}
 \right.
 \end{split}
 \end{align*}
where $\delta(\lambda) = \bar{\lambda},1\leq \lambda\leq n, 1\leq i,j\leq n, 1\leq k\leq m.$

We can connect $E_k$ with $\alpha^{-1}$ or $\alpha$ by the following way:
Suppose $E_k$ points towards $e$. First we isotope $\alpha^{-1}$ by moving $\alpha^{-1}(0)$ along $e$ to meet the end $E_k$. Then we isotope $l,  \alpha^{-1}$ nearby their endpoints at $e$ such that they are both in good position with respect to $e$.  And then we  connect $E_k$ with $\alpha^{-1}$. When $E_k$ points out of $e$, we can use the way to connect $E_k$ with $\alpha$.


Then we try to define an element $l(\alpha_{j_1}^{f_1},\alpha_{j_2}^{f_2},\dots, \alpha_{j_m}^{f_m})\in S_n(M,N,1)$ by the following way: For each $1\leq k\leq m$ we connect $E_k$ with $\alpha^{f_k}$, and give state
$j_k$ to the other end of $\alpha^{f_k}$ that is not used to connect $E_k$. During the process of connecting each $E_k$ and $\alpha^{f_k}$, we can isotope $\alpha^{f_k}$ such that 
$l(\alpha_{j_1}^{f_1},\alpha_{j_2}^{f_2},\dots, \alpha_{j_m}^{f_m})$ does not intersect itself. After connecting each   $E_k$ and $\alpha^{f_k}$, we can isotope the parts nearby the connecting points  such that $l(\alpha_{j_1}^{f_1},\alpha_{j_2}^{f_2},\dots, \alpha_{j_m}^{f_m})$ only intersects $\partial M$ at it's endpoints. Then
  $l(\alpha_{j_1}^{f_1},\alpha_{j_2}^{f_2},\dots, \alpha_{j_m}^{f_m})\in S_n(M,N,1)$. Obviously 
$l(\alpha_{j_1}^{f_1},\alpha_{j_2}^{f_2},\dots, \alpha_{j_m}^{f_m})$ is a well-defined element in $S_n(M,N,1)$. 
We define 
$$\jmath(l) = \sum_{1\leq j_1,\dots,j_m\leq n}c_{g_1(j_1)}\dots
c_{g_m(j_m)}\; l(\alpha_{j_1}^{f_1},\dots, \alpha_{j_m}^{f_m})\otimes
\mu(\alpha^{-f_1}_{h_1(i_1,\overline{j_1})})\dots
\mu(\alpha^{-f_m}_{h_m(i_m,\overline{j_m})})$$
where $\mu(\alpha_{i,j}) = d_n (-1)^{i+1} x_{\bar{i},j}, \mu(\alpha^{-1}_{i,j}) = d_n (-1)^{i+1} x^{-1}_{\bar{i},j}$, $i_k = s_l(E_k)$, $1\leq k\leq m,$ $c_{t} = (-1)^{n-t}, 1\leq t\leq n$.

Note that if  $l_1$ and $l_2$ are isotopic to each other, we have $$l_1(\alpha_{j_1}^{f_1},\alpha_{j_2}^{f_2},\dots, \alpha_{j_m}^{f_m})= l_2(\alpha_{j_1}^{f_1},\alpha_{j_2}^{f_2},\dots, \alpha_{j_m}^{f_m})\in S_n(M,N,1)$$ where the labelings of endpoints of 
$l_i$, $i=1,2$, on $e$ are preserved by isotopy. Then $\jmath$ respects isotopy classes, that is, $\jmath$
is defined on the  isotopy classes of  stated $n$-webs. For any two stated $n$-webs $l_1,l_2$, we isotope $l_t,t=1,2,$ such that $l_1\cap l_2 =\emptyset$, then we have $\jmath(l_1\cup l_2) = \jmath(l_1)\jmath(l_2)$.

\begin{lemma}\label{lmm5.10}
Suppose $(M,N)$ is a marked three manifold with $N\neq\emptyset$, and $N^{'}$ is obtained from $N$ by adding one extra marking $e$. Then $\jmath:S_n(M,N^{'},1)\rightarrow 
S_n(M,N,1)\otimes O(SLn)$ is a well-defined 
$S_n(M,N,1)$-algebra homormorphism.
\end{lemma}
\begin{proof}
From the above discussion, it suffices to show $\jmath$ preserves  relations 
(\ref{w.cross})-(\ref{wzh.eight}) for well-definedness. We use $l$ (respectively $l^{'}$) to denote the left handside (respectively right handside) of "=" in these relations.

It obvious that $\jmath$ preserves relations (\ref{w.cross})-(\ref{wzh.four}).

It is obvious that $\jmath$ preserves relations (\ref{wzh.five})-(\ref{wzh.eight}) if the boundary component in the picture is not $e$. Then we suppose the boundary component in all these pictures is $e$.

Then we look at relation (\ref{wzh.five}).  We only prove the case where the white dot represents the arrow going from left to right, that is, all the arrows point towards $e$.
We choose a labeling for endpoints of $l$ on $e$.
From bottom to top, we label the endpoints in  the right picture from $m+1$ to $m+n$. The other endpoints of $l^{'}$ not in the picture are labeled in the same way as $l$.

Then we have
\begin{align*}
&\jmath(l^{'}) = \sum_{\sigma\in S_n}(-1)^{\ell(\sigma)}
\sum_{\substack{1\leq j_1,\dots,j_m\leq n\\1\leq k_1,\dots,k_n\leq n}}c_{g_1(j_1)}\dots
c_{g_m(j_m)} c_{k_1}\dots c_{k_n}\\ 
&l^{'}(\alpha_{j_1}^{f_1},\dots, \alpha_{j_m}^{f_m},\alpha^{-1}_{k_1},\dots,\alpha^{-1}_{k_n}))\otimes
\mu(\alpha^{-f_1}_{h_1(i_1,\overline{j_1})}, \dots
\mu(\alpha^{-f_m}_{h_m(i_m,\overline{j_m})}) x_{\sigma(1),\overline{k_1}}\dots x_{\sigma(n),\overline{k_n}}\\
&= 
\sum_{\substack{1\leq j_1,\dots,j_m\leq n\\1\leq k_1,\dots,k_n\leq n}}c_{g_1(j_1)}\dots
c_{g_m(j_m)} c_{k_1}\dots c_{k_n}\\ 
&l^{'}(\alpha_{j_1}^{f_1},\dots, \alpha_{j_m}^{f_m},\alpha^{-1}_{k_1},\dots,\alpha^{-1}_{k_n}))\otimes
\mu(\alpha^{-f_1}_{h_1(i_1,\overline{j_1})}, \dots
\mu(\alpha^{-f_m}_{h_m(i_m,\overline{j_m})})
\sum_{\sigma\in S_n}(-1)^{\ell(\sigma)} x_{\sigma(1),\overline{k_1}}\dots x_{\sigma(n),\overline{k_n}}\\
&= 
\sum_{\substack{1\leq j_1,\dots,j_m\leq n\\1\leq k_1,\dots,k_n\leq n}}c_{g_1(j_1)}\dots
c_{g_m(j_m)} c_{k_1}\dots c_{k_n}\\ 
&l^{'}(\alpha_{j_1}^{f_1},\dots, \alpha_{j_m}^{f_m},\alpha^{-1}_{k_1},\dots,\alpha^{-1}_{k_n}))\otimes
\mu(\alpha^{-f_1}_{h_1(i_1,\overline{j_1})}, \dots
\mu(\alpha^{-f_m}_{h_m(i_m,\overline{j_m})})\text{det}
\begin{pmatrix}
x_{1,\overline{k_1}}&\dots&x_{1,\overline{k_n}}\\
\vdots& &\vdots\\
x_{n,\overline{k_1}}&\dots&x_{n,\overline{k_n}}\\
\end{pmatrix}\\
&= (-1)^{\frac{n(n-1)}{2}}
\sum_{\substack{1\leq j_1,\dots,j_m\leq n\\\tau\in S_n}}c_{g_1(j_1)}\dots
c_{g_m(j_m)} \\ 
&l^{'}(\alpha_{j_1}^{f_1},\dots, \alpha_{j_m}^{f_m},\alpha^{-1}_{\tau(1)},\dots,\alpha^{-1}_{\tau(n)}))\otimes
\mu(\alpha^{-f_1}_{h_1(i_1,\overline{j_1})}, \dots
\mu(\alpha^{-f_m}_{h_m(i_m,\overline{j_m})})\text{det}
\begin{pmatrix}
x_{1,\overline{\tau(1)}}&\dots&x_{1,\overline{\tau(n)}}\\
\vdots& &\vdots\\
x_{n,\overline{\tau(1)}}&\dots&x_{n,\overline{\tau(n)}}\\
\end{pmatrix}\\
&= 
\sum_{\substack{1\leq j_1,\dots,j_m\leq n\\\tau\in S_n}} (-1)^{\ell(\tau)} c_{g_1(j_1)}\dots
c_{g_m(j_m)} \\
&l^{'}(\alpha_{j_1}^{f_1},\dots, \alpha_{j_m}^{f_m},\alpha^{-1}_{\tau(1)},\dots,\alpha^{-1}_{\tau(n)}))\otimes
\mu(\alpha^{-f_1}_{h_1(i_1,\overline{j_1})}, \dots
\mu(\alpha^{-f_m}_{h_m(i_m,\overline{j_m})})\\
&= 
\sum_{\substack{1\leq j_1,\dots,j_m\leq n}} c_{g_1(j_1)}\dots
c_{g_m(j_m)} 
l(\alpha_{j_1}^{f_1},\dots, \alpha_{j_m}^{f_m})\otimes
\mu(\alpha^{-f_1}_{h_1(i_1,\overline{j_1})}, \dots
\mu(\alpha^{-f_m}_{h_m(i_m,\overline{j_m})}) = \jmath(l).\\
\end{align*}

To show $\jmath$ preserves relation (\ref{wzh.six}). Here we only prove $\jmath$
preserves 
$\raisebox{-.20in}{
\begin{tikzpicture}
\tikzset{->-/.style=
{decoration={markings,mark=at position #1 with
{\arrow{latex}}},postaction={decorate}}}
\filldraw[draw=white,fill=gray!20] (-0.7,-0.7) rectangle (0,0.7);
\draw [line width =1.5pt,decoration={markings, mark=at position 1 with {\arrow{>}}},postaction={decorate}](0,0.7)--(0,-0.7);
\draw [color = black, line width =1pt] (0 ,0.3) arc (90:270:0.5 and 0.3);
\node [right]  at(0,0.3) {$i$};
\node [right] at(0,-0.3){$j$};
\draw [line width =1pt,decoration={markings, mark=at position 0.5 with {\arrow{>}}},postaction={decorate}](-0.5,0.02)--(-0.5,-0.02);
\end{tikzpicture}}   = \delta_{\bar j,i }\,  (-1)^{n-i}$. We label the top endpoint by $1$ and the other one by 2.  
Then we  have
\begin{align*}
\jmath( \raisebox{-.20in}{
\begin{tikzpicture}
\tikzset{->-/.style=
{decoration={markings,mark=at position #1 with
{\arrow{latex}}},postaction={decorate}}}
\filldraw[draw=white,fill=gray!20] (-0.7,-0.7) rectangle (0,0.7);
\draw [line width =1.5pt,decoration={markings, mark=at position 1 with {\arrow{>}}},postaction={decorate}](0,0.7)--(0,-0.7);
\draw [color = black, line width =1pt] (0 ,0.3) arc (90:270:0.5 and 0.3);
\node [right]  at(0,0.3) {$i$};
\node [right] at(0,-0.3){$j$};
\draw [line width =1pt,decoration={markings, mark=at position 0.5 with {\arrow{>}}},postaction={decorate}](-0.5,0.02)--(-0.5,-0.02);
\end{tikzpicture}})&=\sum_{1\leq j_1,j_2\leq n}(-1)^{j+\overline{j_1}}c_{\overline{j_1}}c_{j_2} l(\alpha_{j_1},\alpha_{j_2}^{-1})\otimes x^{-1}_{j_1,i} x_{\overline{j},\overline{j_2}}\\
&=\sum_{1\leq j_1,j_2\leq n}(-1)^{j+\overline{j_1}}c_{\overline{j_1}}c_{j_2}
 \raisebox{-.20in}{
\begin{tikzpicture}
\tikzset{->-/.style=
{decoration={markings,mark=at position #1 with
{\arrow{latex}}},postaction={decorate}}}
\filldraw[draw=white,fill=gray!20] (-0.7,-0.7) rectangle (0,0.7);
\draw [line width =1.5pt,decoration={markings, mark=at position 1 with {\arrow{>}}},postaction={decorate}](0,0.7)--(0,-0.7);
\draw [color = black, line width =1pt] (0 ,0.3) arc (90:270:0.5 and 0.3);
\node [right]  at(0,0.3) {$j_1$};
\node [right] at(0,-0.3){$j_2$};
\draw [line width =1pt,decoration={markings, mark=at position 0.5 with {\arrow{>}}},postaction={decorate}](-0.5,0.02)--(-0.5,-0.02);
\end{tikzpicture}}
\otimes x^{-1}_{j_1,i} x_{\overline{j},\overline{j_2}}\\
&=(-1)^{j+1}\sum_{1\leq j_1\leq n}1\otimes x^{-1}_{j_1,i} x_{\overline{j},j_1}
=(-1)^{j+1}\delta_{\bar{j},i}1\otimes 1\\&= \delta_{\bar j,i }\,  (-1)^{n-i} 1\otimes 1.\\
\end{align*}

Next we look at relation (\ref{wzh.seven}). Here we only prove the case where the white dot represents an arrow going from right to left. We choose a labeling for endpoints of $l$ on $e$.
We label the top (respectively bottom) endpoint in the right picture by $m+1$ (respectively $m+2$). The other endpoints of $l^{'}$ not in the picture are labeled in the same way as $l$. Then we have 
\begin{align*}
\jmath(l^{'}) &= \sum_{1\leq i\leq n}(-1)^{i+1}
\sum_{\substack{1\leq j_1,\dots,j_m\leq n\\1\leq k_1,k_2\leq n}}c_{g_1(j_1)}\dots
c_{g_m(j_m)} c_{\overline{k_1}} c_{k_2} (-1)^{\bar{i}+\overline{k_1}}\\ 
&l^{'}(\alpha_{j_1}^{f_1},\dots, \alpha_{j_m}^{f_m},\alpha_{k_1},\alpha^{-1}_{k_2}))\otimes
\mu(\alpha^{-f_1}_{h_1(i_1,\overline{j_1})}) \dots
\mu(\alpha^{-f_m}_{h_m(i_m,\overline{j_m})}) x_{i,\overline{k_2}} x^{-1}_{k_1,i}\\
& = 
\sum_{\substack{1\leq j_1,\dots,j_m\leq n\\1\leq k_1,k_2\leq n}}c_{g_1(j_1)}\dots
c_{g_m(j_m)}  c_{k_2}\\ 
&l^{'}(\alpha_{j_1}^{f_1},\dots, \alpha_{j_m}^{f_m},\alpha_{k_1},\alpha^{-1}_{k_2}))\otimes
\mu(\alpha^{-f_1}_{h_1(i_1,\overline{j_1})}) \dots
\mu(\alpha^{-f_m}_{h_m(i_m,\overline{j_m})}) \sum_{1\leq i\leq n} x_{i,\overline{k_2}} x^{-1}_{k_1,i}\\
& = 
\sum_{\substack{1\leq j_1,\dots,j_m\leq n\\1\leq k_1\leq n}}c_{g_1(j_1)}\dots
c_{g_m(j_m)}  (-1)^{k_1+1}\\
& l^{'}(\alpha_{j_1}^{f_1},\dots, \alpha_{j_m}^{f_m},\alpha_{k_1},\alpha^{-1}_{\overline{k_1}}))\otimes
\mu(\alpha^{-f_1}_{h_1(i_1,\overline{j_1})}) \dots
\mu(\alpha^{-f_m}_{h_m(i_m,\overline{j_m})}) \\
& = 
\sum_{\substack{1\leq j_1,\dots,j_m\leq n}}c_{g_1(j_1)}\dots
c_{g_m(j_m)} 
 l^{'}(\alpha_{j_1}^{f_1},\dots, \alpha_{j_m}^{f_m})\otimes
\mu(\alpha^{-f_1}_{h_1(i_1,\overline{j_1})}) \dots
\mu(\alpha^{-f_m}_{h_m(i_m,\overline{j_m})}) 
\\
& =\jmath(l).
\end{align*}

It is obvious that $\jmath$ preserves relation (\ref{wzh.eight}).

Then $\jmath$ is well-defined. Trivially it is an algebra homomorphism. For any $\alpha\in S_n(M,N,1)$, we have $\jmath(l_{ad}(\alpha)) = \alpha\otimes 1$, thus $\jmath$ is an $S_n(M,N,1)$-algebra homomorphism.


\end{proof}

\begin{theorem}\label{thm5.11}
Suppose $(M,N)$ is a marked three manifold with $N\neq\emptyset$, and $N^{'}$ is obtained from $N$ by adding one extra marking. Then there exist $S_n(M,N,1)$-algebra homomorphisms $\imath: S_n(M,N,1)\otimes
O(SLn)\rightarrow S_n(M,N^{'},1)$ and $\jmath:S_n(M,N^{'},1)\rightarrow 
S_n(M,N,1)\otimes O(SLn)$ such that $\jmath\circ\imath =Id_{S_n(M,N,1)\otimes
O(SLn)} , \imath\circ\jmath = Id_{S_n(M,N^{'},1)}$. Especially $S_n(M,N^{'},1)\simeq
S_n(M,N,1)\otimes O(SLn)$.
\end{theorem}
\begin{proof}
The existence of $\imath$ and $\jmath$ are given by Lemmas \ref{lmm5.9} and \ref{lmm5.10}.

Let $i,j\in\mathbb{J}$. Then we have 
\begin{align*}
\jmath(\imath(1\otimes x_{i,j})) &= (-1)^{i+1}\jmath(\alpha_{\bar{i},j})=(-1)^{i+1}\sum_{1\leq k\leq n}C_kd_n (-1)^{\bar{i}+1} \alpha(\alpha^{-1}_{k})\otimes x_{i,\bar{k}}\\
&= \sum_{1\leq k\leq n}(-1)^{n-k}
\raisebox{-.20in}{
\begin{tikzpicture}
\tikzset{->-/.style=
{decoration={markings,mark=at position #1 with
{\arrow{latex}}},postaction={decorate}}}
\filldraw[draw=white,fill=gray!20] (-0.7,-0.7) rectangle (0,0.7);
\draw [line width =1.5pt,decoration={markings, mark=at position 1 with {\arrow{>}}},postaction={decorate}](0,-0.7)--(0,0.7);
\draw [color = black, line width =1pt] (0 ,0.3) arc (90:270:0.5 and 0.3);
\node [right]  at(0,0.3) {$k$};
\node [right] at(0,-0.3){$j$};
\draw [line width =1pt,decoration={markings, mark=at position 0.5 with {\arrow{<}}},postaction={decorate}](-0.5,0.02)--(-0.5,-0.02);
\end{tikzpicture}}
\otimes x_{i,\bar{k}} = 1\otimes x_{i,j}.
\end{align*}
Since $1\otimes x_{i,j}$ are $S_n(M,N,1)$-algebra generators, we have $\jmath\circ\imath =Id_{S_n(M,N,1)\otimes
O(SLn)}$.

We also have 
$$\imath(\jmath(\alpha_{i,j})) = (-1)^{\bar{i}+1}\imath(1\otimes x_{\bar{i},j}) = \alpha_{i,j}.$$
From Lemma \ref{lmm5.8}, we get $\imath\circ\jmath = Id_{S_n(M,N^{'},1)}$.
\end{proof}

\begin{theorem}\label{thm5.13}
Let $(M,N)$ be a marked three manifold. If $N=\emptyset$, we have $S_n(M,N,1)\simeq G_n(M)$.
If $N\neq \emptyset$, we have $S_n(M,N,1)\simeq \Gamma_n(M)\otimes O(SLn)^{\otimes(\sharp N-1)}$.
\end{theorem}
\begin{proof}
Subsection \ref{sub4.1}, Theorems \ref{thm5.5} and \ref{thm5.11}.
\end{proof}

\begin{corollary}\label{cli}
Let $(M,N)$ be a marked three manifold with $N\neq \emptyset$. Suppose $\pi_1(M)$ is a free group generated by $m$ elements. Then we have 
 $S_n(M,N,1)\simeq  O(SLn)^{\otimes(m + \sharp N-1)}$.
\end{corollary}

The second conclusion in Theorem 7.13 in \cite{le2021stated} indicates the classical limit for essentially bordered pb surfaces, which coincides with Corollary \ref{cli}.


\begin{corollary}\label{Cor5.13}
Suppose $(M,N)$ is a marked three manifold, and $N^{'}$ is obtained from $N$ by adding one extra marking.
Then the adding mark map $l_{ad}:S_n(M,N,1)\rightarrow S_n(M,N^{'},1)$ is injective.
\end{corollary}
\begin{proof}
If $N$ is empty. We look at the following diagram:
$$\begin{tikzcd}
S_n(M,\emptyset,1)  \arrow[r, "l_{ad}"]
\arrow[d, "\simeq"]  
&  S_n(M,N^{'},1) \arrow[d, "G"] \\
G_n(M) \arrow[r, "\lambda "] 
&  \Gamma_n(M)\\
\end{tikzcd}$$
where the isomorphism from $S_n(M,\emptyset,1)$ to $G_n(M)$ is the one introduced in Subsection \ref{sub4.1} (the spin structure used for this isomorphism is the restriction of the relative spin structure for $(M,N^{'})$), and $\lambda$ is the embedding. It is easy to check the above diagram is commutative. Then $l_{ad}$ is injective because $G$ is an isomorphism.

If $N$ is not empty. For any $\alpha\in S_n(M,N,1)$, we have $\jmath(l_{ad}(\alpha)) = \alpha\otimes 1$. Then $l_{ad}$ is injective because $\jmath$ is an isomorphism and the map from 
$S_n(M,N,1)$ to $S_n(M,N,1)\otimes O(SLn)$ given by $\alpha\mapsto \alpha\otimes 1$ is injective.
\end{proof}

From Corollary \ref{Cor5.13} and the adding marking map, we can regard $S_n(M,N,1)$
as a subalgebra of $S_n(M,N^{'},1)$.

\subsection{Ker$\Phi$ = $\sqrt{0}$}

Suppose $N$ has only one component.
Then we define an alegbra isomorphism $H:\Gamma_n(M)\rightarrow \Gamma_n(M)$ and 
a surjective algebra homormorphism $\tau:\Gamma_n(M)\rightarrow R_n(M,N)$.
Let $[\alpha]$ be an element in  $\pi_1(M,N)$ and 
$i,j$ be two integers between $1$ and $n$. Define $H([\alpha]_{i,j})
= [\alpha]_{\bar{i},\bar{j}}$, it is easy to show $H$ is a well-defined algebra isomorphism. For any $\rho\in\tilde{\chi}_n(M,N)$, define  $\tau([\alpha]_{i,j})(\rho) = [\rho(\tilde{\alpha})]_{i,j}$
where $\tilde{\alpha}\in \pi_1(UM, \tilde{N})$ is a lift for $\alpha$ such 
that $h(\tilde{\alpha}) = 0$. From the proof of Proposition \ref{prop3.6}, we know the definition of $\tau$ is independent of the choice of the lift for $\alpha$. It is also obvious to show $\tau$ is a well-defined surjective algebra homomorphism. Especially from Proposition \ref{prop3.6} and definitions for $\Gamma_n(M), R_n(M,N)$, we have Ker$\tau = \sqrt{0}_{\Gamma_n(M)}.$

\begin{lemma}\label{lmm5.15}
Let $(M,N)$ be marked three manifold with $N$ consisting of one open oriented interval. Then we have the following commutative diagram:
$$\begin{tikzcd}
\Gamma_n(M)  \arrow[r, "F"]
\arrow[d, "H"]  
&  S_n(M,N,1) \arrow[d, "\Phi"] \\
\Gamma_n(M) \arrow[r, "\tau "] 
&  R_n(M,N)\\
\end{tikzcd}.$$ Especially Ker\,$\Phi=\sqrt{0}_{S_n(M,N,1)}.$
\end{lemma}
\begin{proof}
For any $[\alpha]\in\pi_1(M,N),1\leq i,j\leq n$, we know $F([\alpha]_{i,j}) = (-1)^{i+1} \hat{\alpha}_{\bar{i},j}$
where $\hat{\alpha}$ is a framed oriented boundary arc such that $h(\widetilde{\hat{\alpha}})
= 0$ and $[\hat{\alpha}] = [\alpha]\in\pi_1(M,N)$. Then for any $\rho\in\tilde{\chi}_n(M,N)$, we have
$$\Phi(F([\alpha]_{i,j}))(\rho) =(-1)^{i+1} \Phi(\hat{\alpha}_{\bar{i},j})(\rho) =
 (-1)^{i+1} [A\rho(\widetilde{\hat{\alpha}})]_{i,\bar{j}} = [\rho(\widetilde{\hat{\alpha}})]_{\bar{i},\bar{j}}.$$
Since $\widetilde{\hat{\alpha}}\in \pi_1(UM,\tilde{N})$ is a lift for $\alpha$ and $h(\widetilde{\hat{\alpha}})=0$,
we have 
$$\tau(H([\alpha]_{i,j}))(\rho) = \tau([\alpha]_{\bar{i},\bar{j}})(\rho)
= [\rho(\widetilde{\hat{\alpha}})]_{\bar{i},\bar{j}}.$$
Thus the diagram commutes.

Since both $F$ and $H$ are isomorphisms and Ker$\tau=\sqrt{0}_{\Gamma_n(M)}$, we get Ker$\Phi=\sqrt{0}_{S_n(M,N,1)}$.

\end{proof}

\begin{lemma}\label{lmm5.16}
Suppose $(M,N)$ is a marked three manifold with $N\neq\emptyset$, and $N^{'}$ is obtained from $N$ by adding one extra marking. 
Then Ker\,$\Phi^{(M,N^{'})}$ is the ideal generated by Ker\,$\Phi^{(M,N)}$ (here we regard $S_n(M,N,1)$ as a subalgebra of $S_n(M,N^{'},1)$). 
\end{lemma}
\begin{proof}
Here we  use the notations in Remark \ref{rem5.8}.

From Proposition \ref{prop3.6} and Lemma 8.1 in \cite{CL2022stated}, we know there is an algebra isomorphism 
$h: R_n(M,N)\otimes O(SLn)\rightarrow R_n(M,N^{'})$ defined by
$$h(r\otimes x_{i,j})(\rho) = r(\rho|_{\pi_1(UM,\tilde{N})})[\rho(\tilde{\alpha})]_{i,j}$$
where $r\in R_n(M,N), 1\leq i,j\leq n,$ $\rho\in \tilde{\chi}_n(M,N^{'})$ and 
$\rho|_{\pi_1(UM,\tilde{N})} $ is the restriction of $\rho$ on $\pi_1(UM,\tilde{N})$.
We  have another algebra isomorphism $f:O(SLn)\rightarrow O(SLn)$ given by $x_{i,j}\rightarrow x_{\bar{i},\bar{j}}$.

Then we want to show the following diagram is commutative:
$$\begin{tikzcd}
S_n(M,N,1)\otimes O(SLn)  \arrow[r, "\imath"]
\arrow[d, "\Phi^{(M,N)}\otimes f"]  
&  S_n(M,N^{'},1) \arrow[d, "\Phi^{(M,N^{'})}"] \\
R_n(M,N)\otimes O(SLn) \arrow[r, "h"] 
&  R_n(M,N^{'})\\
\end{tikzcd}.$$
Let $\rho$ be element in $\tilde{\chi}_n(M,N^{'})$, $i,j$ be two integers between $1$ and $n$,
$\beta_{k,t}$ be a stated framed oriented boundary arc in $(M,N)$. We have 
$$
\Phi^{(M,N^{'})}(\imath(\beta_{k,t}\otimes 1))(\rho) = \Phi^{(M,N^{'})}(\beta_{k,t})(\rho)
= [A\rho(\tilde{\beta})]_{\bar{k},\bar{t}}\;,$$
and
\begin{align*}
(h\circ (\Phi^{(M,N)}\otimes f)) (\beta_{k,t}\otimes 1)(\rho)
&= h(\Phi^{(M,N)}(\beta_{k,t})\otimes 1)(\rho)  = \Phi^{(M,N)}(\beta_{k,t}) (\rho|_{\pi_1(UM,\tilde{N})})\\
&= [A \rho|_{\pi_1(UM,\tilde{N})} (\tilde{\beta})]_{\bar{k},\bar{t}} = [A\rho(\tilde{\beta})]_{\bar{k},\bar{t}}\, .
\end{align*}
We also have 
$$
\Phi^{(M,N^{'})}(\imath(1\otimes x_{i,j}))(\rho) = (-1)^{i+1} \Phi^{(M,N^{'})}(\alpha_{\bar{i},j})(\rho)
= (-1)^{i+1} [A \rho(\tilde{\alpha})]_{i,\bar{j}} =  [ \rho(\tilde{\alpha})]_{\bar{i},\bar{j}}\;,
$$
and 
$$(h\circ (\Phi^{(M,N)}\otimes f)) (1\otimes x_{i,j})(\rho) =
h(1\otimes x_{\bar{i},\bar{j}})(\rho) = [\rho(\tilde{\alpha})]_{\bar{i},\bar{j}}\, .$$
Thus the above diagram is commutative because $\beta_{k,t}\otimes 1, 1\otimes x_{i,j}$ generate
$S_n(M,N,1)\otimes O(SLn)$ as an algebra and all the maps in the diagram are algebra homomorphisms.

We have Ker$(\Phi^{(M,N)}\otimes f) =$Ker($\Phi^{(M,N)}\otimes Id_{O(SLn)}$)
$=(\text{Ker}\,\Phi^{(M,N)})\otimes O(SLn)$, where $(\text{Ker}\,\Phi^{(M,N)})\otimes O(SLn)$ is an ideal generated by $(\text{Ker}\,\Phi^{(M,N)})\otimes1$. Then Ker\,$\Phi^{(M,N^{'})}$is the ideal generated by
Ker\,$\Phi^{(M,N)}$ since $\imath((\text{Ker}\,\Phi^{(M,N)})\otimes1) =$Ker\,$\Phi^{(M,N)}$.

\end{proof}

\begin{lemma}\label{lmm5.17}
Let $(M,N)$ be a marked three manifold with $N\neq \emptyset$. We have (a) Ker\,$\Phi^{(M,N)}
=\sqrt{0}_{S_n(M,N,1)}$, and (b) $\sqrt{0}_{S_n(M,N,1)}$ is the ideal generated by
$\sqrt{0}_{S_n(M,\{e\},1)}$ where $e$ is a component of $N$ (here we regard $S_n(M,\{e\},1)$ as a subalgebra of $S_n(M,N,1)$).
\end{lemma}
\begin{proof}
 
For any marked three manifold $(M_1,N_1)$ with $N_1\neq \emptyset$, suppose $N_1^{'}$ is obtained from $N_1$ by adding one extra marking, then Ker\,$\Phi^{(M_1,N_1^{'})} = (\text{Ker\,}\Phi^{(M_1,N_1)})$ from Lemma \ref{lmm5.16}. Thus if Ker\,$\Phi^{(M_1,N_1)} =\sqrt{0}_{S_n(M_1,N_1,1)}$, then 
$\text{Ker\,}\Phi^{(M_1,N_1^{'})} = (\text{Ker\,}\Phi^{(M_1,N_1)})\subset \sqrt{0}_{S_n(M_1,N_1^{'},1)}$. We also have $\sqrt{0}_{S_n(M_1,N_1^{'},1)}\subset \text{Ker\,}\Phi^{(M_1,N_1^{'})}$ since the coordinate ring has no nonzero nilponents. Then $\text{Ker\,}\Phi^{(M_1,N_1^{'})}=\sqrt{0}_{S_n(M_1,N_1^{'},1)}$. Thus we have 
Ker\,$\Phi^{(M_1,N_1)} =\sqrt{0}_{S_n(M_1,N_1,1)}$ implies $\text{Ker\,}\Phi^{(M_1,N_1^{'})}=\sqrt{0}_{S_n(M_1,N_1^{'},1)}$. Combine with the fact that (a) is true if $N$ consists of only one oriented open interval (Lemma \ref{lmm5.15}), then we get (a) is true for general marked three manifold $(M,N)$ with $N\neq \emptyset$.

If $N$ consists of one component, clearly (b) holds. If $\sharp N>1$, suppose $Com(N)=\{e_1,e_2,\dots,e_m\}$.
For any $1\leq i\leq m$, define $N_{(i)}=e_1\cup\dots\cup e_i$. Then we have 
$$S_n(M,N_{(1)},1)\subset S_n(M,N_{(2)},1)\subset,\dots,\subset S_n(M,N_{(m)},1).$$
Since Ker\,$\Phi^{(M,N_{(i+1)})}$ is an ideal of $S_n(M,N_{(i+1)},1)$ generated by
Ker\,$\Phi^{(M,N_{(i)})}$, then we have 
$$\text{Ker}\,\Phi^{(M,N)} = \text{Ker\,}\Phi^{(M,N_{(m)})}$$
is an ideal of $S_n(M,N,1)$
generated by Ker\,$\Phi_{(M,N_{(1)})}$, which actually is $\text{Ker\,}\Phi_{(M,\{e_1\})}$. From (a), we know 
Ker\,$\Phi^{(M,N)}= \sqrt{0}_{S_n(M,N,1)}$ and Ker\,$\Phi^{(M,\{e_1\})}= \sqrt{0}_{S_n(M,\{e_1\},1)}$. Since we can label any component of $N$ as $e_1$, then (b) is true.
\end{proof}

\begin{theorem}
For any marked three manifold $(M,N)$,
 we have Ker\,$\Phi^{(M,N)}
=\sqrt{0}_{S_n(M,N,1)}$. 
\end{theorem}
\begin{proof}
Subsection \ref{sub4.1} and (a) in Lemma \ref{lmm5.17}.
\end{proof}

\section{Splitting map and adding marking map}
In this section we will mainly discuss the splitting map and the add marking map for general $v$. We will show the Kernal of the splitting map is equal to the Kernal of the adding marking map. This conclusion helps us to prove the injectivity for splitting map for a  large family of marked three manifolds. Also we will show the splitting map is  injective for any marked three manifold when $v=1$.

\subsection{Gluing the thickening of the ideal triangle to marked three manifolds}\label{sbbbbbbb6.1}
Let $\Delta$ denote the marked three manifold in the following picture (normally we use $\Delta$ to denote comultiplication, only in this section we use $\Delta$ to denote the marked three manifold in the following picture):
$$
\raisebox{-.35in}{
\begin{tikzpicture}
\tikzset{->-/.style=
{decoration={markings,mark=at position #1 with
{\arrow{latex}}},postaction={decorate}}}
\draw[line width = 1pt] (0,0) rectangle (4.8, 2);
\draw [dashed] (0,0)--(2.4,2.4);
\draw [dashed] (2.4,2.4)--(2.4,4.4);
\draw [dashed] (2.4,2.4)--(4.8,0);
\draw [line width = 1pt] (0,2)--(2.4,4.4);
\draw [line width = 1pt] (2.4,4.4)--(4.8,2);
\draw [color=red, line width = 1pt] (1.2,1.2)--(1.2,1.9);
\draw [color=red, line width = 1pt] (1.2,2.1)--(1.2,3.2);
\draw [color=red,->, line width = 1pt] (1.2,2.1)--(1.2,2.8);
\draw [color=red, line width = 1pt] (3.6,1.2)--(3.6,1.9);
\draw [color=red, line width = 1pt] (3.6,2.1)--(3.6,3.2);
\draw [color=red,->, line width = 1pt] (3.6,2.1)--(3.6,2.8);
\draw [color=red, line width = 1pt] (2.6,0)--(2.6,2);
\draw [color=red,->, line width = 1pt] (2.6,0)--(2.6,1.6);
\node[right] at(1.2,2.5){$e_1$};
\node[left] at(3.6,2.5){$e_2$};
\node[right] at(2.6,1.3){$e_3$};
\end{tikzpicture}}
$$

Recall that $\fT$ denotes
the standard ideal triangle. Then $\Delta$, after removing the three vertical edges, is isomorphic to $\fT\times [-1,1]$. 

For each $i=1,2,3,$ let $D_{i}$ be an embedded disk in $\partial \Delta$ such that $cl(e_i)\subset int(D_i)$ and there is no intersection among these three disks. From now on, when we draw $\Delta$, we may omit all the black lines, that is, we only draw three red arrows. And we only draw involved markings and stated $n$-webs when we try to draw stated $n$-webs in marked three manifolds.

 Let $(M,N)$ be any marked three manifold with $\sharp N\geq 2$. Suppose $e_1^{'},e_2^{'}$ are two components of $N$. For each $i=1,2$, let $D_{i}^{'}$ be an embedded disk on the boundary of $M$ such that the intersection between the closure of $N$ and $D_i$ is the closure of $e_i$ and the closure of $e_i$ is contained in the interior of $D_i$ and $D_1\cap D_2 =\emptyset$. For each $i=1,2$, let $\phi_{i} : D_i^{'} \rightarrow D_i$ be a diffeomorphism such that $\phi_i(e_{i}^{'}) = e_i$ and $\phi_i$ preserves the orientations of $e_i$ and $e_i^{'}$. We set 
$$M_{e_1^{'}\Delta e_2^{'}} = (M\cup \Delta)/(\phi_i(x) = x, x\in D_i^{'},i=1,2),\;
 N_{e_1^{'}\Delta e_2^{'}} = (N-(e_1^{'}\cup e_2^{'}))\cup e_3.$$
Then $(M_{e_1^{'}\Delta e_2^{'}},N_{e_1^{'}\Delta e_2^{'}})$ is a marked three manifold. We use 
$(M,N)_{e_1^{'}\Delta e_2^{'}}$ to denote  this marked three manifold.

 Then there is a linear map
$QF_{e_1^{'},e_2^{'}} : S_n(M,N,v)\rightarrow S_n((M,N)_{e_1^{'}\Delta e_2^{'}},v)$.
L{\^e} and Sikora  introduced this map when $(M,N)$ is the thickening of a punctured bordered surface
\cite{le2021stated}. The definition here for $QF_{e_1^{'},e_2^{'}}$ is similar with the one defined in \cite{le2021stated} for punctured bordered surface. We use $l$ to denote the obvious embedding from $M$ to 
$M_{e_1^{'}\Delta e_2^{'}}$.  For any stated $n$-web $\alpha \in (M,N)$, we extend the ends of $l(\alpha)$ on each $e_i,i=1,2$, to $e_3$ such that the framing of extended parts contained in $\Delta$
is given by the positive direction of $[-1,1]$ and all the ends on $e_3$ extended from $e_1$  are higher than all the ends extended from $e_2$. To be precise, see the following picture:
$$
\raisebox{-.60in}{
\begin{tikzpicture}
\tikzset{->-/.style=
{decoration={markings,mark=at position #1 with
{\arrow{latex}}},postaction={decorate}}}
\draw [color=red, line width = 1pt] (0,0)--(0,2);
\draw [color=red,->, line width = 1pt] (0,0)--(0,1);
\draw [color=red, line width = 1pt] (1.2,-2)--(1.2,0);
\draw [color=red,->, line width = 1pt] (1.2,-2)--(1.2,-1);
\draw [color=red, line width = 1pt] (2.4,0)--(2.4,2);
\draw [color=red,->, line width = 1pt] (2.4,0)--(2.4,1);
\draw [color=blue, line width = 1pt] (-1,1.2)--(0,1.2);
\draw [color=blue, line width = 1pt] (-1,1.8)--(0,1.8);
\draw [color=blue, line width = 1pt] (2.4,0.2)--(3.4,0.2);
\draw [color=blue, line width = 1pt] (2.4,0.8)--(3.4,0.8);
\node[left] at(0,1.6){$\vdots$};
\node[right] at(2.4,0.6){$\vdots$};
\node[right] at(0,1.2){\small $i_1$};
\node[right] at(0,1.8){\small $i_{k_1}$};
\node[left] at(2.4,0.2){\small $j_1$};
\node[left] at(2.4,0.8){\small $j_{k_2}$};
\end{tikzpicture}}
\longrightarrow
\raisebox{-.60in}{
\begin{tikzpicture}
\tikzset{->-/.style=
{decoration={markings,mark=at position #1 with
{\arrow{latex}}},postaction={decorate}}}
\draw [color=red, line width = 1pt] (0,0)--(0,2);
\draw [color=red,->, line width = 1pt] (0,0)--(0,1);
\draw [color=red, line width = 1pt] (1.2,-2)--(1.2,0);
\draw [color=red,->, line width = 1pt] (1.2,-2)--(1.2,-1);
\draw [color=red, line width = 1pt] (2.4,0)--(2.4,2);
\draw [color=red,->, line width = 1pt] (2.4,0)--(2.4,1);
\draw [color=blue, line width = 1pt] (-1,1.2)--(0,1.2);
\draw [color=blue, line width = 1pt] (-1,1.8)--(0,1.8);
\draw [color=blue, line width = 1pt] (2.4,0.2)--(3.4,0.2);
\draw [color=blue, line width = 1pt] (2.4,0.8)--(3.4,0.8);
\draw [color=blue, line width = 1pt] (1.2,-0.8)--(0,1.2);
\draw [color=blue, line width = 1pt] (1.2,-0.2)--(0,1.8);
\draw [color=blue, line width = 1pt] (1.2, -1.8)--(2.4,0.2);
\draw [color=blue, line width = 1pt] (1.2,-1.2)--(2.4,0.8);
\node[left] at(0,1.6){$\vdots$};
\node[right] at(2.4,0.6){$\vdots$};
\node[right] at(1.2,-0.8){\small $i_1$};
\node[right] at(1.2,-0.2){\small $i_{k_1}$};
\node[left] at(1.2, -1.8){\small $j_1$};
\node[left] at(1.2,-1.2){\small $j_{k_2}$};
\end{tikzpicture}}
$$
where the blue lines are parts of stated $n$-webs with arbitray orientations and the framing in the picture is given by the red arrows.
It is easy to show $QF_{e_1^{'},e_2^{'}}$ is a well-defined linear map. To simplify notation, we normally omit the subscript for $QF_{e_1^{'},e_2^{'}}$ when there is no confusion.

Let $\beta$ be a path in $\partial \Delta$ connecting $e_1$ and $e_2$, see the following picture (the green dashed line is $\beta$):
$$
\raisebox{-.35in}{
\begin{tikzpicture}
\tikzset{->-/.style=
{decoration={markings,mark=at position #1 with
{\arrow{latex}}},postaction={decorate}}}
\draw[line width = 1pt] (0,0) rectangle (4.8, 2);
\draw [dashed] (0,0)--(2.4,2.4);
\draw [dashed] (2.4,2.4)--(2.4,4.4);
\draw [dashed] (2.4,2.4)--(4.8,0);
\draw [dashed,color=green] (3.2,1.6)--(1.2,1.2);
\draw [dashed,->,color=green] (3.2,1.6)--(1.6,1.28);
\draw [dashed,color=green] (3.6,3.2)--(3.2,1.6);
\draw [line width = 1pt] (0,2)--(2.4,4.4);
\draw [line width = 1pt] (2.4,4.4)--(4.8,2);
\draw [color=red, line width = 1pt] (1.2,1.2)--(1.2,1.9);
\draw [color=red, line width = 1pt] (1.2,2.1)--(1.2,3.2);
\draw [color=red,->, line width = 1pt] (1.2,2.1)--(1.2,2.8);
\draw [color=red, line width = 1pt] (3.6,1.2)--(3.6,1.9);
\draw [color=red, line width = 1pt] (3.6,2.1)--(3.6,3.2);
\draw [color=red,->, line width = 1pt] (3.6,2.1)--(3.6,2.8);
\draw [color=red, line width = 1pt] (2.6,0)--(2.6,2);
\draw [color=red,->, line width = 1pt] (2.6,0)--(2.6,1.6);
\node[right] at(1.2,2.5){$e_1$};
\node[left] at(3.6,2.5){$e_2$};
\node[right] at(2.6,1.3){$e_3$};
\end{tikzpicture}}
$$
 Let $\Delta^{'}$ be obtained from $\Delta$ by connecting $e_1,e_2$ using $\beta$.
Then $\Delta^{'}$ (after removing some unimportant parts on the boundary) is isomorphic to the thickening of the bigon. Then there is a counit $\epsilon$ for 
$S_n(\Delta^{'},v)$ \cite{le2021stated}. We use $l_{\Delta}$ to denote the map from 
$S_n(\Delta,v)$ to $S_n(\Delta^{'},v)$ induced by the embedding $\Delta\rightarrow \Delta^{'}$. 
We use $\epsilon_{\Delta}$ to denote $\epsilon \circ l_{\Delta}$.
Define 
$$Cut =(\epsilon_{\Delta}\otimes Id_{S_n(M,N,q)})\circ \Theta_{(D_1,e_1)}\circ \Theta_{(D_2,e_2)}
:S_n((M,N)_{e_1^{'}\Delta e_2^{'}},v)\rightarrow S_n(M,N,v) .$$

\begin{proposition}
With the above notations, we have $Cut$ and $QF$ are inverse to each other.
\end{proposition}
\begin{proof}
Let 
$
\raisebox{-.20in}{
\begin{tikzpicture}
\tikzset{->-/.style=
{decoration={markings,mark=at position #1 with
{\arrow{latex}}},postaction={decorate}}}
\draw[color = red, line width=1pt] (0,0)--(0,1);
\draw[color = red, ->,line width=1pt] (0,0)--(0,0.5);
\draw[color = red, line width=1pt] (2,0)--(2,1);
\draw[color = red, ->,line width=1pt] (2,0)--(2,0.5);
\draw[color = blue, line width=2pt] (-0.7,0.8)--(0,0.8);
\draw[color = blue,line width=2pt] (2,0.2)--(2.7,0.2);
\node[right] at(0,0.8){\small $\vec{u}$};
\node[left] at(2,0.2){\small $\vec{v}$};
\end{tikzpicture}}
$ be any stated $n$-web in $(M,N)$, where the thick blue line represents some arbitrarily oriented parallel framed lines with framing given by red arrows
and $\vec{u},\vec{v}$ are sequences of ordered states, which indicate the states of the parallel framed lines.
Then we have
\begin{align*}
&Cut\circ QF (
\raisebox{-.20in}{
\begin{tikzpicture}
\tikzset{->-/.style=
{decoration={markings,mark=at position #1 with
{\arrow{latex}}},postaction={decorate}}}
\draw[color = red, line width=1pt] (0,0)--(0,1);
\draw[color = red, ->,line width=1pt] (0,0)--(0,0.5);
\draw[color = red, line width=1pt] (2,0)--(2,1);
\draw[color = red, ->,line width=1pt] (2,0)--(2,0.5);
\draw[color = blue, line width=2pt] (-0.7,0.8)--(0,0.8);
\draw[color = blue,line width=2pt] (2,0.2)--(2.7,0.2);
\node[right] at(0,0.8){\small $\vec{u}$};
\node[left] at(2,0.2){\small $\vec{v}$};
\end{tikzpicture}})
=
Cut(
\raisebox{-.35in}{
\begin{tikzpicture}
\tikzset{->-/.style=
{decoration={markings,mark=at position #1 with
{\arrow{latex}}},postaction={decorate}}}
\draw[color = red, line width=1pt] (0,0)--(0,1);
\draw[color = red, ->,line width=1pt] (0,0)--(0,0.5);
\draw[color = red, line width=1pt] (2,0)--(2,1);
\draw[color = red, ->,line width=1pt] (2,0)--(2,0.5);
\draw[color = red, line width=1pt] (1,-1)--(1,0);
\draw[color = red, ->,line width=1pt] (1,-1)--(1,-0.5);
\draw[color = blue, line width=2pt] (-0.7,0.8)--(0,0.8);
\draw[color = blue,line width=2pt] (2,0.2)--(2.7,0.2);
\draw[color = blue,line width=2pt] (0,0.8)--(1,-0.2);
\draw[color = blue,line width=2pt] (2,0.2)--(1, -0.8);
\node[right] at(1,-0.2){\small $\vec{u}$};
\node[left] at(1, -0.8){\small $\vec{v}$};
\end{tikzpicture}})\\
=&(\epsilon_{\Delta}\otimes Id)(\sum_{\vec{a},\vec{b}}
\raisebox{-.35in}{
\begin{tikzpicture}
\tikzset{->-/.style=
{decoration={markings,mark=at position #1 with
{\arrow{latex}}},postaction={decorate}}}
\draw[color = red, line width=1pt] (0,0)--(0,1);
\draw[color = red, ->,line width=1pt] (0,0)--(0,0.5);
\draw[color = red, line width=1pt] (2,0)--(2,1);
\draw[color = red, ->,line width=1pt] (2,0)--(2,0.5);
\draw[color = red, line width=1pt] (1,-1)--(1,0);
\draw[color = red, ->,line width=1pt] (1,-1)--(1,-0.5);
\draw[color = blue,line width=2pt] (0,0.8)--(1,-0.2);
\draw[color = blue,line width=2pt] (2,0.2)--(1, -0.8);
\node[right] at(1,-0.2){\small $\vec{u}$};
\node[left] at(1, -0.8){\small $\vec{v}$};
\node[right] at(2,0.2){\small $\vec{b}$};
\node[left] at(0,0.8){\small $\vec{a}$};
\end{tikzpicture}}
\otimes
\raisebox{-.20in}{
\begin{tikzpicture}
\tikzset{->-/.style=
{decoration={markings,mark=at position #1 with
{\arrow{latex}}},postaction={decorate}}}
\draw[color = red, line width=1pt] (0,0)--(0,1);
\draw[color = red, ->,line width=1pt] (0,0)--(0,0.5);
\draw[color = red, line width=1pt] (2,0)--(2,1);
\draw[color = red, ->,line width=1pt] (2,0)--(2,0.5);
\draw[color = blue, line width=2pt] (-0.7,0.8)--(0,0.8);
\draw[color = blue,line width=2pt] (2,0.2)--(2.7,0.2);
\node[right] at(0,0.8){\small $\vec{a}$};
\node[left] at(2,0.2){\small $\vec{b}$};
\end{tikzpicture}}
)\\
=&(\epsilon\otimes Id)(\sum_{\vec{a},\vec{b}}
\raisebox{-.20in}{
\begin{tikzpicture}
\tikzset{->-/.style=
{decoration={markings,mark=at position #1 with
{\arrow{latex}}},postaction={decorate}}}
\draw[color = red, line width=1pt] (0,0)--(0,1);
\draw[color = red, ->,line width=1pt] (0,0)--(0,0.5);
\draw[color = red, line width=1pt] (2,0)--(2,1);
\draw[color = red, ->,line width=1pt] (2,0)--(2,0.5);
\draw[color = blue, line width=2pt] (0,0.8)--(2,0.8);
\draw[color = blue,line width=2pt] (0,0.2)--(2,0.2);
\node[right] at(2,0.2){\small $\vec{v}$};
\node[left] at(0,0.2){\small $\vec{b}$};
\node[right] at(2,0.8){\small $\vec{u}$};
\node[left] at(0,0.8){\small $\vec{a}$};
\end{tikzpicture}}
\otimes
\raisebox{-.20in}{
\begin{tikzpicture}
\tikzset{->-/.style=
{decoration={markings,mark=at position #1 with
{\arrow{latex}}},postaction={decorate}}}
\draw[color = red, line width=1pt] (0,0)--(0,1);
\draw[color = red, ->,line width=1pt] (0,0)--(0,0.5);
\draw[color = red, line width=1pt] (2,0)--(2,1);
\draw[color = red, ->,line width=1pt] (2,0)--(2,0.5);
\draw[color = blue, line width=2pt] (-0.7,0.8)--(0,0.8);
\draw[color = blue,line width=2pt] (2,0.2)--(2.7,0.2);
\node[right] at(0,0.8){\small $\vec{a}$};
\node[left] at(2,0.2){\small $\vec{b}$};
\end{tikzpicture}}
)\\
=&\raisebox{-.20in}{
\begin{tikzpicture}
\tikzset{->-/.style=
{decoration={markings,mark=at position #1 with
{\arrow{latex}}},postaction={decorate}}}
\draw[color = red, line width=1pt] (0,0)--(0,1);
\draw[color = red, ->,line width=1pt] (0,0)--(0,0.5);
\draw[color = red, line width=1pt] (2,0)--(2,1);
\draw[color = red, ->,line width=1pt] (2,0)--(2,0.5);
\draw[color = blue, line width=2pt] (-0.7,0.8)--(0,0.8);
\draw[color = blue,line width=2pt] (2,0.2)--(2.7,0.2);
\node[right] at(0,0.8){\small $\vec{u}$};
\node[left] at(2,0.2){\small $\vec{v}$};
\end{tikzpicture}}.
\end{align*}

For any stated $n$-web $\alpha$ in $(M,N)_{e_1^{'}\Delta e_2^{'}}$, we can isotope $\alpha$ such that $\alpha\cap \Delta$ looks like the following picture:
$$
\raisebox{-.35in}{
\begin{tikzpicture}
\tikzset{->-/.style=
{decoration={markings,mark=at position #1 with
{\arrow{latex}}},postaction={decorate}}}
\draw[color = red, line width=1pt] (0,0)--(0,1);
\draw[color = red, ->,line width=1pt] (0,0)--(0,0.5);
\draw[color = red, line width=1pt] (2,0)--(2,1);
\draw[color = red, ->,line width=1pt] (2,0)--(2,0.5);
\draw[color = red, line width=1pt] (1,-1)--(1,0);
\draw[color = red, ->,line width=1pt] (1,-1)--(1,-0.5);
\draw[color = blue,line width=2pt] (0,0.8)--(1,-0.2);
\draw[color = blue,line width=2pt] (2,0.2)--(1, -0.8);
\node[right] at(1,-0.2){\small $\vec{u}$};
\node[left] at(1, -0.8){\small $\vec{v}$};
\draw[color = blue,line width=2pt] (0.65 ,1.3) arc (225:315:0.5);
\end{tikzpicture}}.
$$
Then we can use relation (\ref{wzh.seven})  to kill all the framed lines in the most above thick blue arc one by one. Then $\alpha = \sum_{t}k_t \alpha_{t}$ where, for each $t$, $\alpha_t\cap \Delta$ looks like:
$$
\raisebox{-.35in}{
\begin{tikzpicture}
\tikzset{->-/.style=
{decoration={markings,mark=at position #1 with
{\arrow{latex}}},postaction={decorate}}}
\draw[color = red, line width=1pt] (0,0)--(0,1);
\draw[color = red, ->,line width=1pt] (0,0)--(0,0.5);
\draw[color = red, line width=1pt] (2,0)--(2,1);
\draw[color = red, ->,line width=1pt] (2,0)--(2,0.5);
\draw[color = red, line width=1pt] (1,-1)--(1,0);
\draw[color = red, ->,line width=1pt] (1,-1)--(1,-0.5);
\draw[color = blue,line width=2pt] (0,0.8)--(1,-0.2);
\draw[color = blue,line width=2pt] (2,0.2)--(1, -0.8);
\node[right] at(1,-0.2){\small $\vec{a}$};
\node[left] at(1, -0.8){\small $\vec{b}$};
\end{tikzpicture}}.
$$
We have 
\begin{align*}
&QF\circ Cut(\raisebox{-.35in}{
\begin{tikzpicture}
\tikzset{->-/.style=
{decoration={markings,mark=at position #1 with
{\arrow{latex}}},postaction={decorate}}}
\draw[color = red, line width=1pt] (0,0)--(0,1);
\draw[color = red, ->,line width=1pt] (0,0)--(0,0.5);
\draw[color = red, line width=1pt] (2,0)--(2,1);
\draw[color = red, ->,line width=1pt] (2,0)--(2,0.5);
\draw[color = red, line width=1pt] (1,-1)--(1,0);
\draw[color = red, ->,line width=1pt] (1,-1)--(1,-0.5);
\draw[color = blue, line width=2pt] (-0.7,0.8)--(0,0.8);
\draw[color = blue,line width=2pt] (2,0.2)--(2.7,0.2);
\draw[color = blue,line width=2pt] (0,0.8)--(1,-0.2);
\draw[color = blue,line width=2pt] (2,0.2)--(1, -0.8);
\node[right] at(1,-0.2){\small $\vec{a}$};
\node[left] at(1, -0.8){\small $\vec{b}$};
\end{tikzpicture}})\\
=&
QF\circ (\epsilon_{\Delta}\otimes Id)(\sum_{\vec{c},\vec{d}}
\raisebox{-.35in}{
\begin{tikzpicture}
\tikzset{->-/.style=
{decoration={markings,mark=at position #1 with
{\arrow{latex}}},postaction={decorate}}}
\draw[color = red, line width=1pt] (0,0)--(0,1);
\draw[color = red, ->,line width=1pt] (0,0)--(0,0.5);
\draw[color = red, line width=1pt] (2,0)--(2,1);
\draw[color = red, ->,line width=1pt] (2,0)--(2,0.5);
\draw[color = red, line width=1pt] (1,-1)--(1,0);
\draw[color = red, ->,line width=1pt] (1,-1)--(1,-0.5);
\draw[color = blue,line width=2pt] (0,0.8)--(1,-0.2);
\draw[color = blue,line width=2pt] (2,0.2)--(1, -0.8);
\node[left] at(0,0.8){\small $\vec{c}$};
\node[right] at(2, 0.2){\small $\vec{d}$};
\node[right] at(1,-0.2){\small $\vec{a}$};
\node[left] at(1, -0.8){\small $\vec{b}$};
\end{tikzpicture}}\otimes
\raisebox{-.20in}{
\begin{tikzpicture}
\tikzset{->-/.style=
{decoration={markings,mark=at position #1 with
{\arrow{latex}}},postaction={decorate}}}
\draw[color = red, line width=1pt] (0,0)--(0,1);
\draw[color = red, ->,line width=1pt] (0,0)--(0,0.5);
\draw[color = red, line width=1pt] (2,0)--(2,1);
\draw[color = red, ->,line width=1pt] (2,0)--(2,0.5);
\draw[color = blue, line width=2pt] (-0.7,0.8)--(0,0.8);
\draw[color = blue,line width=2pt] (2,0.2)--(2.7,0.2);
\node[left] at(2,0.2){\small $\vec{d}$};
\node[right] at(0, 0.8){\small $\vec{c}$};
\end{tikzpicture}})\\
=&
QF\circ (\epsilon\otimes Id)( \sum_{\vec{c},\vec{d}}
\raisebox{-.20in}{
\begin{tikzpicture}
\tikzset{->-/.style=
{decoration={markings,mark=at position #1 with
{\arrow{latex}}},postaction={decorate}}}
\draw[color = red, line width=1pt] (0,0)--(0,1);
\draw[color = red, ->,line width=1pt] (0,0)--(0,0.5);
\draw[color = red, line width=1pt] (2,0)--(2,1);
\draw[color = red, ->,line width=1pt] (2,0)--(2,0.5);
\draw[color = blue, line width=2pt] (0,0.2)--(2,0.2);
\draw[color = blue,line width=2pt] (0,0.8)--(2,0.8);
\node[right] at(2,0.2){\small $\vec{b}$};
\node[left] at(0, 0.8){\small $\vec{c}$};
\node[right] at(2,0.8){\small $\vec{a}$};
\node[left] at(0,0.2){\small $\vec{d}$};
\end{tikzpicture}}\otimes
\raisebox{-.20in}{
\begin{tikzpicture}
\tikzset{->-/.style=
{decoration={markings,mark=at position #1 with
{\arrow{latex}}},postaction={decorate}}}
\draw[color = red, line width=1pt] (0,0)--(0,1);
\draw[color = red, ->,line width=1pt] (0,0)--(0,0.5);
\draw[color = red, line width=1pt] (2,0)--(2,1);
\draw[color = red, ->,line width=1pt] (2,0)--(2,0.5);
\draw[color = blue, line width=2pt] (-0.7,0.8)--(0,0.8);
\draw[color = blue,line width=2pt] (2,0.2)--(2.7,0.2);
\node[left] at(2,0.2){\small $\vec{d}$};
\node[right] at(0, 0.8){\small $\vec{c}$};
\end{tikzpicture}})\\
=&
\sum_{\vec{c},\vec{d}} \epsilon(
\raisebox{-.20in}{
\begin{tikzpicture}
\tikzset{->-/.style=
{decoration={markings,mark=at position #1 with
{\arrow{latex}}},postaction={decorate}}}
\draw[color = red, line width=1pt] (0,0)--(0,1);
\draw[color = red, ->,line width=1pt] (0,0)--(0,0.5);
\draw[color = red, line width=1pt] (2,0)--(2,1);
\draw[color = red, ->,line width=1pt] (2,0)--(2,0.5);
\draw[color = blue, line width=2pt] (0,0.2)--(2,0.2);
\draw[color = blue,line width=2pt] (0,0.8)--(2,0.8);
\node[right] at(2,0.2){\small $\vec{b}$};
\node[left] at(0, 0.8){\small $\vec{c}$};
\node[right] at(2,0.8){\small $\vec{a}$};
\node[left] at(0,0.2){\small $\vec{d}$};
\end{tikzpicture}})\,
QF(
\raisebox{-.20in}{
\begin{tikzpicture}
\tikzset{->-/.style=
{decoration={markings,mark=at position #1 with
{\arrow{latex}}},postaction={decorate}}}
\draw[color = red, line width=1pt] (0,0)--(0,1);
\draw[color = red, ->,line width=1pt] (0,0)--(0,0.5);
\draw[color = red, line width=1pt] (2,0)--(2,1);
\draw[color = red, ->,line width=1pt] (2,0)--(2,0.5);
\draw[color = blue, line width=2pt] (-0.7,0.8)--(0,0.8);
\draw[color = blue,line width=2pt] (2,0.2)--(2.7,0.2);
\node[left] at(2,0.2){\small $\vec{d}$};
\node[right] at(0, 0.8){\small $\vec{c}$};
\end{tikzpicture}})\\
=&
\sum_{\vec{c},\vec{d}} \epsilon(
\raisebox{-.20in}{
\begin{tikzpicture}
\tikzset{->-/.style=
{decoration={markings,mark=at position #1 with
{\arrow{latex}}},postaction={decorate}}}
\draw[color = red, line width=1pt] (0,0)--(0,1);
\draw[color = red, ->,line width=1pt] (0,0)--(0,0.5);
\draw[color = red, line width=1pt] (2,0)--(2,1);
\draw[color = red, ->,line width=1pt] (2,0)--(2,0.5);
\draw[color = blue, line width=2pt] (0,0.2)--(2,0.2);
\draw[color = blue,line width=2pt] (0,0.8)--(2,0.8);
\node[right] at(2,0.2){\small $\vec{b}$};
\node[left] at(0, 0.8){\small $\vec{c}$};
\node[right] at(2,0.8){\small $\vec{a}$};
\node[left] at(0,0.2){\small $\vec{d}$};
\end{tikzpicture}})
\raisebox{-.35in}{
\begin{tikzpicture}
\tikzset{->-/.style=
{decoration={markings,mark=at position #1 with
{\arrow{latex}}},postaction={decorate}}}
\draw[color = red, line width=1pt] (0,0)--(0,1);
\draw[color = red, ->,line width=1pt] (0,0)--(0,0.5);
\draw[color = red, line width=1pt] (2,0)--(2,1);
\draw[color = red, ->,line width=1pt] (2,0)--(2,0.5);
\draw[color = red, line width=1pt] (1,-1)--(1,0);
\draw[color = red, ->,line width=1pt] (1,-1)--(1,-0.5);
\draw[color = blue, line width=2pt] (-0.7,0.8)--(0,0.8);
\draw[color = blue,line width=2pt] (2,0.2)--(2.7,0.2);
\draw[color = blue,line width=2pt] (0,0.8)--(1,-0.2);
\draw[color = blue,line width=2pt] (2,0.2)--(1, -0.8);
\node[right] at(1,-0.2){\small $\vec{c}$};
\node[left] at(1, -0.8){\small $\vec{d}$};
\end{tikzpicture}}\\
=&
\raisebox{-.35in}{
\begin{tikzpicture}
\tikzset{->-/.style=
{decoration={markings,mark=at position #1 with
{\arrow{latex}}},postaction={decorate}}}
\draw[color = red, line width=1pt] (0,0)--(0,1);
\draw[color = red, ->,line width=1pt] (0,0)--(0,0.5);
\draw[color = red, line width=1pt] (2,0)--(2,1);
\draw[color = red, ->,line width=1pt] (2,0)--(2,0.5);
\draw[color = red, line width=1pt] (1,-1)--(1,0);
\draw[color = red, ->,line width=1pt] (1,-1)--(1,-0.5);
\draw[color = blue, line width=2pt] (-0.7,0.8)--(0,0.8);
\draw[color = blue,line width=2pt] (2,0.2)--(2.7,0.2);
\draw[color = blue,line width=2pt] (0,0.8)--(1,-0.2);
\draw[color = blue,line width=2pt] (2,0.2)--(1, -0.8);
\node[right] at(1,-0.2){\small $\vec{a}$};
\node[left] at(1, -0.8){\small $\vec{b}$};
\end{tikzpicture}} .
\end{align*}
Thus we have 
$$(QF\circ Cut)(\alpha) = \sum_{t}k_t (QF\circ Cut)(\alpha_{t}) = \sum_{t}k_t \alpha_{t} =\alpha.$$

\end{proof}

\subsection{The connection between the splitting map and the adding marking map}\label{sbbbbb6.2}
In this subsection, we will show the splitting and the adding marking map are related by some isomorphisms, which implies they have the same Kernal.

\begin{lemma}(\cite{le2021stated})
Let $(M,N)$ be a marked three manifold, and $e$ be a component of $N$. Then there is a linear isomorphism
$h_{e}: S_n(M,N,v)\rightarrow S_n(M,N,v)$ given by

$$
h_{e}\left(
\raisebox{-.40in}{

\begin{tikzpicture}
\tikzset{->-/.style=

{decoration={markings,mark=at position #1 with

{\arrow{latex}}},postaction={decorate}}}

\filldraw[draw=white,fill=gray!20] (-1.5,0) rectangle (0, 2.5);
\draw [line width =1.5pt,decoration={markings, mark=at position 1 with {\arrow{>}}},postaction={decorate}](0,2.5)--(0,0);
\draw[line width =1pt] (-1.5,0.5)--(0,0.5);
\draw[line width =1pt] (-1.5,1)--(0,1);
\draw[line width =1pt] (-1.5,2)--(0,2);
\node [left] at(0,1.5) {$\vdots$};
\node [right] at(0,0.5) {$i_1$};
\node [right] at(0,1) {$i_2$};
\node [right] at(0,2) {$i_k$};
\end{tikzpicture}
}\right)= \left(\frac 1{ \prod_{j=1}^k c_{{i_j}}}\right) \cdot
\raisebox{-.40in}{

\begin{tikzpicture}
\tikzset{->-/.style=

{decoration={markings,mark=at position #1 with

{\arrow{latex}}},postaction={decorate}}}

\filldraw[draw=white,fill=gray!20] (-1.5,0) rectangle (0, 2.5);
\draw [line width =1pt] (-1,0.2) rectangle (-0.5, 2.3);
\draw [line width =1.5pt,decoration={markings, mark=at position 1 with {\arrow{>}}},postaction={decorate}](0,2.5)--(0,0);
\draw[line width =1pt] (-1.5,0.5)--(-1,0.5);
\draw[line width =1pt] (-0.5,0.5)--(0,0.5);
\draw[line width =1pt] (-1.5,1)--(-1,1);
\draw[line width =1pt] (-0.5,1.5)--(0,1.5);
\draw[line width =1pt] (-1.5,2)--(-1,2);
\draw[line width =1pt] (-0.5,2)--(0,2);
\node [left] at(0,1) {$\vdots$};
\node [left] at(-1,1.5) {$\vdots$};
\node [right] at(0,0.5) {$\overline{i_k}$};
\node [right] at(0,1.5) {$\overline{i_2}$};
\node [right] at(0,2) {$\overline{i_1}$};
\node  at(-0.75,1.25) {$\bar{H}$};
\end{tikzpicture}
}
$$
where $\bar H$ denotes the negative half-twist  and the  thick line with an arrow is part of the marking $e$. Note that
the orientations of the horizontal lines are arbitrary.
\end{lemma}  
\begin{proof}
The isomorphism $h_e$ here is $htw_{\beta}^{-1}$ in Proposition 4.11 in \cite{le2021stated}.
\end{proof}

Let $(M,N)$ be a marked three maniold. Suppose $D$ is a properly embedded disk and $\beta$ is an oriented open interval contained in $D$. We know there is a linear map
$\Theta_{(D,\beta)}:S_n(M,N,v)\rightarrow S_n(\text{Cut}_{(D,\beta)}(M,N),v)$.  $\text{Cut}_{(D,\beta)}(M,N)$ has two markings labeled by $\beta_1,\beta_2$ as in Figure \ref{fig:1}. Then 
$(M,N^{'})\simeq(\text{Cut}_{(D,\beta)}(M,N))_{\beta_1\Delta \beta_2}$ where $N^{'}$ is obtained from $N$ by adding one extra marking. 
We use $\varphi$ to denote this obvious isomorphism, and use $\varphi_{*}$ to denote the isomorphism 
from $S_n((M,N^{'}),v)$ to $S_n((\text{Cut}_{(D,\beta)})_{\beta_1\Delta \beta_2},v)$ induced by $\varphi$.
Recall that there is linear map $l_{ad}: S_n(M,N,v)\rightarrow S_n(M,N^{'},v)$ induced by the embedding $(M,N)\rightarrow (M,N^{'})$.

\begin{proposition}\label{prop6.3}
With the above notations, we have 
$$QF\circ h_{\beta_2}\circ \Theta_{(D,\beta)} = \varphi_{*}\circ l_{ad}.$$ Especially Ker$\,\Theta_{(D,\beta)}
=$Ker$\,l_{ad}$.
\end{proposition}
\begin{proof} We have
\begin{align*}
&QF\circ h_{\beta_2}\circ \Theta_{(D,\beta)}(
\raisebox{-.20in}{
\begin{tikzpicture}
\tikzset{->-/.style=
{decoration={markings,mark=at position #1 with
{\arrow{latex}}},postaction={decorate}}}
\draw[color = red, line width=1pt] (0,0)--(0,1);
\draw[color = red, ->,line width=1pt] (0,0)--(0,0.5);
\draw[color = blue, line width=1pt] (-1,0.9)--(1,0.9);
\draw[color = blue, line width=1pt] (-1,0.7)--(1,0.7);
\draw[color = blue, line width=1pt] (-1,0.1)--(1,0.1);
\node[left] at(0, 0.5){\small $\vdots$};
\end{tikzpicture}})
=\sum_{1\leq i_1,\dots,i_k\leq n}
QF\circ h_{\beta_2}(
\raisebox{-.20in}{
\begin{tikzpicture}
\tikzset{->-/.style=
{decoration={markings,mark=at position #1 with
{\arrow{latex}}},postaction={decorate}}}
\draw[color = red, line width=1pt] (0,0)--(0,1);
\draw[color = red, ->,line width=1pt] (0,0)--(0,0.5);
\draw[color = blue, line width=1pt] (-1,0.9)--(0,0.9);
\draw[color = blue, line width=1pt] (-1,0.7)--(0,0.7);
\draw[color = blue, line width=1pt] (-1,0.1)--(0,0.1);
\draw[color = red, line width=1pt] (2,0)--(2,1);
\draw[color = red, ->,line width=1pt] (2,0)--(2,0.5);
\draw[color = blue, line width=1pt] (2,0.9)--(3,0.9);
\draw[color = blue, line width=1pt] (2,0.7)--(3,0.7);
\draw[color = blue, line width=1pt] (2,0.1)--(3,0.1);
\node[left] at(0, 0.5){\small $\vdots$};
\node[right] at(2, 0.5){\small $\vdots$};
\node[right] at(0, 1){\small $i_k$};
\node[right] at(0, 0.6){\small $i_{k-1}$};
\node[right] at(0, 0.1){\small $i_1$};
\node[left] at(2, 1){\small $i_k$};
\node[left] at(2, 0.6){\small $i_{k-1}$};
\node[left] at(2, 0.1){\small $i_1$};
\end{tikzpicture}})\\
=& \sum_{1\leq i_1,\dots,i_k\leq n} (\Pi_{j=1}^{k}c_{i_j}^{-1})
QF(
\raisebox{-.35in}{
\begin{tikzpicture}
\tikzset{->-/.style=
{decoration={markings,mark=at position #1 with
{\arrow{latex}}},postaction={decorate}}}
\draw[color = red, line width=1pt] (0,0)--(0,2);
\draw[color = red, ->,line width=1pt] (0,0)--(0,1);
\draw[color = blue, line width=1pt] (-1,1.9)--(0,1.9);
\draw[color = blue, line width=1pt] (-1,1.7)--(0,1.7);
\draw[color = blue, line width=1pt] (-1,1.1)--(0,1.1);
\draw[color = red, line width=1pt] (2,0)--(2,2);
\draw[color = red, ->,line width=1pt] (2,0)--(2,1);
\draw[color = blue, line width=1pt] (2.8,0.1)--(3,0.1);
\draw[color = blue, line width=1pt] (2.8,0.7)--(3,0.7);
\draw[color = blue, line width=1pt] (2.8,0.9)--(3,0.9);
\draw[color = blue, line width=1pt] (2,0.9)--(2.4,0.9);
\draw[color = blue, line width=1pt] (2,0.3)--(2.4,0.3);
\draw[color = blue, line width=1pt] (2,0.1)--(2.4,0.1);
\draw[color = blue, line width=1pt]  (2.4,0) rectangle (2.8,1);
\node[left] at(0, 1.5){\small $\vdots$};
\node[right] at(2, 0.7){\small $\vdots$};
\node[right] at(0, 2){\small $i_k$};
\node[right] at(0, 1.6){\small $i_{k-1}$};
\node[right] at(0, 1.1){\small $i_1$};
\node[left] at(2, 0.9){\small $\overline{i_1}$};
\node[left] at(2, 0.4){\small $\overline{i_{k-1}}$};
\node[left] at(2, 0){\small $\overline{i_k}$};
\node at(2.6, 0.5){\small $\bar{H}$};
\end{tikzpicture}})\\
=&
\sum_{1\leq i_1,\dots,i_k\leq n} (\Pi_{j=1}^{k}c_{i_j}^{-1})
\raisebox{-.60in}{
\begin{tikzpicture}
\tikzset{->-/.style=
{decoration={markings,mark=at position #1 with
{\arrow{latex}}},postaction={decorate}}}
\draw[color = red, line width=1pt] (0,0)--(0,2);
\draw[color = red, ->,line width=1pt] (0,0)--(0,1);
\draw[color = red, line width=1pt] (1,-1)--(1,1);
\draw[color = red, ->,line width=1pt] (1,-1)--(1,0);
\draw[color = blue, line width=1pt] (-1,1.9)--(0,1.9);
\draw[color = blue, line width=1pt] (-1,1.7)--(0,1.7);
\draw[color = blue, line width=1pt] (-1,1.1)--(0,1.1);
\draw[color = blue, line width=1pt] (0,1.9)--(1,0.9);
\draw[color = blue, line width=1pt] (1,0.7)--(0,1.7);
\draw[color = blue, line width=1pt] (1,0.1)--(0,1.1);
\draw[color = red, line width=1pt] (2,0)--(2,2);
\draw[color = red, ->,line width=1pt] (2,0)--(2,1);
\draw[color = blue, line width=1pt] (2.8,0.1)--(3,0.1);
\draw[color = blue, line width=1pt] (2.8,0.7)--(3,0.7);
\draw[color = blue, line width=1pt] (2.8,0.9)--(3,0.9);
\draw[color = blue, line width=1pt] (2,0.9)--(2.4,0.9);
\draw[color = blue, line width=1pt] (2,0.3)--(2.4,0.3);
\draw[color = blue, line width=1pt] (2,0.1)--(2.4,0.1);
\draw[color = blue, line width=1pt] (2,0.9)--(1,-0.1);
\draw[color = blue, line width=1pt] (2,0.3)--(1,-0.7);
\draw[color = blue, line width=1pt] (2,0.1)--(1,-0.9);
\draw[color = blue, line width=1pt]  (2.4,0) rectangle (2.8,1);
\node[left] at(0, 1.5){\small $\vdots$};
\node[right] at(2, 0.5){\small $\vdots$};
\node[right] at(1,1){\small $i_k$};
\node[right] at(1, 0.6){\small $i_{k-1}$};
\node[right] at(1, 0.1){\small $i_1$};
\node[left] at(1,-0.1){\small $\overline{i_1}$};
\node[left] at(1,-0.6){\small $\overline{i_{k-1}}$};
\node[left] at(1,-1){\small $\overline{i_k}$};
\node at(2.6, 0.5){\small $\bar{H}$};
\end{tikzpicture}}\\
=&
\raisebox{-.20in}{
\begin{tikzpicture}
\tikzset{->-/.style=
{decoration={markings,mark=at position #1 with
{\arrow{latex}}},postaction={decorate}}}
\draw[color = red, line width=1pt] (0,0)--(0,1);
\draw[color = red, ->,line width=1pt] (0,0)--(0,0.5);
\draw[color = blue, line width=1pt] (-1,0.9)--(-0.085,0.9);
\draw[color = blue, line width=1pt] (-1,0.7)--(-0.085,0.7);
\draw[color = blue, line width=1pt] (-1,0.1)--(-0.085,0.1);
\draw[color = blue, line width=1pt] (0.085,0.9)--(1,0.9);
\draw[color = blue, line width=1pt] (0.085,0.7)--(1,0.7);
\draw[color = blue, line width=1pt] (0.085,0.1)--(1,0.1);
\node[left] at(0, 0.5){\small $\vdots$};
\end{tikzpicture}}
\end{align*}
where the last equality is from relation (\ref{wzh.seven}).
\end{proof}

\subsection{Injectivity for splitting map and adding marking map}
The main goal of this subsection is to show injectivity for splitting map and adding marking map for a large family of marked three manifolds. Thanks to Proposition \ref{prop6.3}, we know the splitting map is injective if and only if the corresponding adding marking map is injective.

\begin{lemma}\label{lmm6.4}
Let $(M,N)$ be a marked three manifold. For each $i=1,2$, suppose $N_i$ is obtained from $N$ by adding one extra marking $e_i$. If $e_1,e_2$ belong to the same component of $\partial M$, then there exists an isomorphism $f:(M,N_1)\rightarrow (M,N_2)$ such that $f_{*}\circ l_{ad}^{e_1} = l_{ad}^{e_2}$. Especially Ker\,$l_{ad}^{e_1} = $
Ker\,$l_{ad}^{e_2}$.
\end{lemma}
\begin{proof}
Since $e_1,e_2$ belong to the same component of $\partial M$, we can use a path $\beta$ in $\partial M$ to connect $e_1$ and $e_2$ such that $\beta(0)=(cl(e_1))(1),\beta(1) = (cl(e_2))(0)$ and $\beta\cap cl(N\cup e_1\cup e_2) =\{\beta(0), \beta(1)\}$. Then $e_2*\beta*e_1$ is a path in $\partial M$. Let $U$ be a regular open neighborhood of $e_2*\beta*e_1$ such that $U$ restracts to $e_2*\beta*e_1$ and $U\cap\text{cl}(N)=\emptyset$.
Then there exists an isomorphism $f:(M,N_1)\rightarrow (M,N_2)$ such that $f$ is identity on $M-U$ and $f(e_1) = e_2$ (actually $f$ can be the isomorphism that drags $e_1$ to $e_2$ along  $\beta$). Clearly we have $f_{*}\circ l_{ad}^{e_1} = l_{ad}^{e_2}$.

\end{proof}

\begin{definition}\label{df6.5}
Let $(M,N)$ be a marked three manifold, and $V$ be a component of $\partial M$. Let $N^{'}$ be obtained from $N$ by adding one extra marking $e$ with $e\subset V$. Define 
$$\text{Ker\,}((M,N),V,v) =\text{Ker\,}l_{ad}^{e}.$$
Lemma \ref{lmm6.4} shows $\text{Ker\,}((M,N),V,v)$ is independent of the choice of $e$.
\end{definition}

\begin{corollary}\label{cor6.6}
Let $(M,N)$ be a marked three manifold, and $V$ be a component of $\partial M$.  Then 
$\text{Ker\,}((M,N),V,1) =0.$
\end{corollary}
\begin{proof}
Corollary \ref{Cor5.13}, Definition \ref{df6.5}.
\end{proof}

\begin{theorem}\label{thm6.7}
Let $(M,N)$ be a marked three manifold. Let $D$ be a properly embedded disk with an oriented open interval $\beta\subset int(D)$. Suppose $\partial D$ is contained in the component $V$ of $\partial M$. Then we have  $\text{Ker\,}((M,N),V,v) = Ker\,\Theta_{(D,\beta)}$. Especially $\text{Ker\,}((M,N),V,v) = 0$ if and only if $\Theta_{(D,\beta)}$ is injective.
\end{theorem}
\begin{proof}
Proposition \ref{prop6.3}. 
\end{proof}

\begin{corollary}
Let $(M,N)$ be  a marked three manifold. Let $D$ be a properly embedded disk with an oriented open interval $\beta\subset int( D)$. When $v=1$, we have $\Theta_{(D,\beta)}$ is always injective.
\end{corollary}
\begin{proof}
Corollary \ref{cor6.6}, Theorem \ref{thm6.7}.
\end{proof}

\begin{corollary}
Conjecture 7.12 in \cite{le2021stated} holds when $v=1$.
\end{corollary}

\begin{theorem}\label{thm6.9}
Let  $(M,N)$ be a marked three manifold, and $V$ be a component of $\partial M$. Assume $V\cap N\neq \emptyset$.

 (a) Ker\,$((M,N),V,v)=0$.

(b) Let $N^{'}$ be obtained from $N$ by adding one extra marking in $V$, then $$S_n(M,N^{'},v)\simeq S_n(M,N,v)\otimes O_q(SLn).$$
\end{theorem}
\begin{proof}
(a)
Let $B$ be the three dimensional solid ball with two markings on it's boundary. We label one marking of $B$ as $b$. It is well-known that
$S_n(B,v)= O_q(SLn)$. Since $V\cap N\neq \emptyset$, we can choose one component $a$ of $N$ such that $a\subset V$. Then $((M,N)\cup B)_{a\Delta b}\simeq (M,N^{''})$ where $(M,N)\cup B$ is taking the disjoint union and $N^{''}$ is obtained from $N$ by adding one extra marking in $V$. We use $J$ to denote the isomorphism  from $(M,N^{''})$ to $((M,N)\cup B)_{a\Delta b}$, use 
$L$ to denote the embedding from $(M,N)$ to $(M,N)\cup B$. It is easy to show we have the following commutative diagram:
$$\begin{tikzcd}
S_n(M,N,v)  \arrow[r, "L_{*}"]
\arrow[d, "l_{ad}"]  
&  S_n((M,N)\cup B,v)  \arrow[d, "QF"] \\
 S_n(M,N^{''},v)  \arrow[r, "J_{*}"] 
&  S_n(((M,N)\cup B)_{a\Delta b},v)\\
\end{tikzcd}.$$
Then $l_{ad}$ is injective since both $QF$ and $J_{*}$ are isomorphisms and $L_{*}$ is injective.
Thus $$\text{Ker\,}((M,N),V,v) = \text{Ker\,}l_{ad} =0.$$

(b) We have $((M,N)\cup B)_{a\Delta b}\simeq (M,N^{''})\simeq (M,N^{'})$, thus
$$S_n(M,N^{'},v)\simeq S_n(((M,N)\cup B)_{a\Delta b},v)\simeq S_n(M,N,v)\otimes O_q(SLn).$$
\end{proof}


\begin{corollary}\label{cor6.10}
Let $(M,N)$ be a marked three manifold. Let $D$ be a properly embedded disk with an oriented open interval $\beta\subset int( D)$. Suppose $\partial D$ is contained in the compotent $V$ of $\partial M$, and 
$N\cap V\neq \emptyset$. Then $\Theta_{(D,\beta)}$ is injective.
\end{corollary}
\begin{proof}
Theorems \ref{thm6.7} and \ref{thm6.9}.
\end{proof}

\begin{corollary}\label{cor6.11}
Let $(M,N)$ be a marked three manifold. Suppose every component of $\partial M$ contains at least one marking. Then the splitting map
 $\Theta$ is always injective.
\end{corollary}
\begin{proof}
Corollary \ref{cor6.10}.
\end{proof}

\begin{rem}
Corollary \ref{cor6.11}  proves Proposition \ref{inj}.
\end{rem}


\def \End {End(S_n(M,N,v))}

\def \cF {\mathcal{F}}

\def \tF {\tilde{\cF}}

\section{Frobenius homomorphism for $SL(n)$}\label{subb7}

The main goal of this subsection is to construct the Frobenius homomorphism  $$\cF :S_n(M,N,1)\rightarrow S_n(M,N,v)$$ when $v$ is a primitive $m$-th root of unit with $m$ being coprime with $2n$ and every component of $M$ contains at least one marking. 

We already know the  generators for algebra $S_n(M,N,1)$ and relations for these generators. It seems like we can define the Frobenius homomorphism on these generators and then check all the relations. But $S_n(M,N,v)$ does not have algebra structure unless $(M,N)$ is the thickening of a pb surface.

We use $\End$ to denote the set of  linear maps from $S_n(M,N,v)$ to $S_n(M,N,v)$, then $\End$ has a natural algebra structure given by combination of maps. Then we  define an algebra homomorphism $\tF: S_n(M,N,1)\rightarrow \End$ by defining $\tF$ on the generators of $S_n(M,N,1)$. We  show the definition of $\tF$ is indenpendent of the choice of generators of $S_n(M,N,1)$. Then we define $\cF$ to be the combination between $\tF: S_n(M,N,1)\rightarrow \End$ and an 
 obvious linear map from $\End$ to $S_n(M,N,v)$, defined by sending $f\in\End$ to $f(\emptyset)$ ($\emptyset$ is the empty stated  $n$-web).  Furthermore $\tF$ gives an $S_n(M,N,1)$-module structure to $S_n(M,N,v)$.


We will show $\cF$ is commutative with the splitting map $\Theta$. Furthermore if $(M,N)$ is the thickening of an essentially bordered pb surface, we have $\cF$ is an injective algebra homomorphism, and Im$\cF$ lives in the center.

Let $\alpha$ be a (stated) framed oriented boundary arc or a framed oriented knot in $(M,N)$,
we use $\alpha^{(m)}$ denote the disjoint union of $m$ parallel copies of $\alpha$ (taken
in the direction of the framing).  We require $\alpha^{(m)}$ lives in a small enough open tubular  neighborhood of $\alpha$. From now on, when we say $m$ parallel copies of an arc or a knot, we always mean taking the disjoin union of $m$ parallel copies in the framing direction (also the $m$ parallel copies live in a small enough open tubular  neighborhood).
We use 
$
\begin{tikzpicture}
\tikzset{->-/.style=
{decoration={markings,mark=at position #1 with
{\arrow{latex}}},postaction={decorate}}}
\draw [color = black, line width =1pt](0.6,0) --(0,0);
\draw[line width =1pt,decoration={markings, mark=at position 0.8 with {\arrow{>}}},postaction={decorate}](1,0) --(1.5,0);
\node at(0.8,0) {\small $m$};
\draw[color=black] (0.8,0) circle(0.2);
\end{tikzpicture}
$
to denote the 
$m$ parallel copies of
$
\begin{tikzpicture}
\tikzset{->-/.style=
{decoration={markings,mark=at position #1 with
{\arrow{latex}}},postaction={decorate}}}
\draw [color = black, line width =1pt](0.6,0) --(1,0);
\draw[line width =1pt,decoration={markings, mark=at position 0.8 with {\arrow{>}}},postaction={decorate}](1,0) --(1.5,0);
\end{tikzpicture}
$.

\subsection{On $m$ parallel copies of the stated framed oriented boundary arc}

In this subsection, we will focus on finding relations for $m$ parallel copies of the stated framed oriented boundary arc.

\newcommand{\twoonetwowall}[7][]{
\raisebox{-0.4\height}{\begin{tikzpicture}
        \ifthenelse{\equal{#2}{<-}}{\tikzmath{\y0=0.2;\y1=.1;}}{
        \ifthenelse{\equal{#2}{->}}{\tikzmath{\y0=0; \y1=.2;}}{\tikzmath{\y0=0.1; \y1=0.1;}}}
        \tikzmath{\xw=1.5;\yb=\y0+0.1; \ya=\yb+0.4; \ycent=\ya/2+\yb/2; \yh=\ya+ .1+\y1;}
        \fill[gray!20!white] (0,0)rectangle(\xw,\yh);
        \ifthenelse{\equal{#2}{n}}{}{\draw[wall, #2] (\xw,0) --(\xw,\yh);} 
        \ifthenelse{\equal{#1}{w}}{\draw[wall] (0,0) --(0,\yh);}{}
        \draw[edge, -o-={0.8}{#3}] (1.05,\ycent) to[out=30, in=180] (\xw,\ya) node[right] {\small #6};
        \draw[edge, -o-={0.8}{#4}] (1.05,\ycent) to[out=-30, in=180] (\xw,\yb) node[right] {\small #7};
        \draw[edge] (0.45,\ycent) -- (1.05,\ycent);
        \coordinate[label={above:{\small #5}}] (B) at (0.85,\yb+0.2);
        \draw[edge, -o-={0.4}{#3}] (0,\ya) to[out=0, in=150] (0.45,\ycent);
        \draw[edge, -o-={0.4}{#4}] (0,\yb) to[out=0, in=-150]  (0.45,\ycent);
\end{tikzpicture}
}\hspace*{-.1in}}

\begin{lemma}\label{lmmm7.1}
In $S_n(M,N,v)$, we have
$\raisebox{-.20in}{

} = 0$.
\end{lemma}

\begin{proof}
This can be easily proved by using relations (\ref{wzh.five}) and (\ref{wzh.six}).
\end{proof}

\begin{lemma}\label{lmmm7.2}
In $S_n(M,N,v)$, we have
\begin{align}\label{eqq14}
 q^{\frac{m}{n}}
\raisebox{-.10in}{
%
}\;
\end{align}
where on the left hand side of the equality
$%
$ is part of $m$ parallel copies of a stated framed oriented boundary arc and the single vertical oriented line is not a part of these $m$ parallel copies.
\end{lemma}
\begin{proof}
Using relation (\ref{w.cross}), we get
\begin{align*}
&\text{Eq}_1:\;\;
q^{\frac{1}{n}}
\raisebox{-.10in}{
%
}\;
\end{align*}
where the second last equality is because of equation (51) in \cite{le2021stated} and Lemma \ref{lmmm7.1}.
Then we get the equation (\ref{eqq14}).
\end{proof}

\begin{lemma}\label{llmm7.3}
In $S_n(M,N,v)$, we have
\begin{align*}
 q^{\frac{m}{n}}
\raisebox{-.10in}{
%
}\;
\end{align*}
where on the left hand side of the equality
$%
$ is part of $m$ parallel copies of a stated framed oriented boundary arc and the single vertical oriented line is not a part of these $m$ parallel copies.
\end{lemma}
\begin{proof}
The proof is the same with the proof for Lemma \ref{lmmm7.2}.
\end{proof}

\begin{corollary}\label{corr7.3}
 If $q^{\frac{m}{n}} = 1\text{ or }-1$, we have
\begin{align}
\raisebox{-.10in}{
%
$ is part of $m$ parallel copies of a stated framed oriented boundary arc and the single vertical  line is not a part of these $m$ parallel copies. Note that there are four possibilities to give orientations for the $n$-webs shown in the above local picture.
\end{corollary}
\begin{proof}
Lemmas \ref{lmmm7.2}, \ref{llmm7.3}.
\end{proof}

\begin{lemma}
Let $\alpha$ be a framed oriented boundary arc in $(M,N)$. Suppose 
$\raisebox{-.20in}{
%
}$ is a local picture for $\alpha^{(2)}$, and the other two ends that are not shown in the local picture are assigned with the same state. Then we have 
$$\raisebox{-.20in}{
%
}$$
when $j<i$.
\end{lemma}
\begin{proof}
A small tubular open neighborhood of $\alpha$ in $M$ is isomorphic
to the thickening of the bigon, with $\alpha$ being identified with the core of
the bigon. Then utilizing functoriality we may assume, without loss of generality, that $(M, N )$
is the thickening of the  bigon, $\alpha$ is the core of the bigon. From Theorem \ref{t.Hopf} and the definition of $O_q(SLn)$, the Lemma is obviously true for the thickening of the bigon.

Another proof can be using relation (\ref{wzh.eight}) and Lemma \ref{lmmm7.1}.
\end{proof}

\begin{lemma}\label{lmmm7.5}
In $S_n(M,N,v)$, we have 
$$
\raisebox{-.30in}{
}= \left \{
 \begin{array}{ll}
     (q^{-\frac{1}{n}+\delta_{ij}})^{m}
\raisebox{-.30in}{
\begin{tikzpicture}
\tikzset{->-/.style=
{decoration={markings,mark=at position #1 with
{\arrow{latex}}},postaction={decorate}}}
\draw [line width =1.5pt,decoration={markings, mark=at position 0.91 with {\arrow{>}}},postaction={decorate}] (0,1)--(0,-1);
\draw [color = black, line width =1pt](-1,-0.5) --(0,-0.5);
\draw [color = black, line width =1pt](-1,0.5) --(-0.7,0.5);
\draw [color = black, line width =1pt](-0.3,0.5) --(0,0.5);
\node [right]at(0,0.5) {\small $j$};
\node [right]at(0,-0.5) {\small $i$};
\node at(-0.5,0.5) {\small $m$};
\draw[color=black] (-0.5,0.5) circle (0.2);
\end{tikzpicture}},                    & \text{if }j\geq i,\\
\end{array}
 \right.
$$
where all the orientations of the $n$-webs are the same, that is, they are all pointing towards the marking or pointing out of the marking, on the left hand side of the equality 
$%
$  is part of  $m$ parallel copies of a stated framed oriented boundary arc and the other single  line (the one stated by $i$) is not part of these $m$ parallel copies.
\end{lemma}
\begin{proof}
When $j\geq i$, it is trivial using relation (\ref{wzh.eight}).

Suppose $j<i$.
We prove this case by using mathematical induction on $m$. When $m=1$, it is true because of relation (\ref{wzh.eight}). Assume it is true for $m-1$. Then we have 
\begin{align*}
&\raisebox{-.30in}{
%
 \right.
\end{equation}
where on the left hand side of the equality
$%
$ is part of  $m$ parallel copies of a stated framed oriented boundary arc and the other single oriented line (the one stated by $i$) is not part of these $m$ parallel copies.
\end{lemma}
\begin{proof}
The proof uses the same technique as Lemma \ref{lmmm7.5}. From Proposition \ref{prop3.1}, it is trivial if $j\neq\bar{i}$. Suppose  $j=\bar{i}$, then we can prove this case by using mathmatical induction on $m$. Proposition \ref{prop3.1} guarantees the initial step. Then we can use same technique as Lemma \ref{lmmm7.5} to prove the inductive step.
\end{proof}

\begin{rem}\label{remm7.8}
 We can get a parallel equation as equation \eqref{eqqq16} by reversing all the orientations of $n$-webs in it.
\end{rem}

\begin{corollary}\label{corr7.6}
If $v$ is a primitive
 $m$-th root of unity with $m$ and $2n$ being coprime with each other, we have 
$$
\raisebox{-.30in}{
\begin{tikzpicture}
\tikzset{->-/.style=
{decoration={markings,mark=at position #1 with
{\arrow{latex}}},postaction={decorate}}}
\draw [line width =1.5pt,decoration={markings, mark=at position 0.91 with {\arrow{>}}},postaction={decorate}] (0,1)--(0,-1);
\draw [color = black, line width =1pt](-1,0.5) --(-0,-0.5);
\draw [color = black, line width =1pt](-1,-0.5) --(-0.6,-0.1);
\draw [color = black, line width =1pt](-0.4,0.1) --(0,0.5);
\draw [color = black, line width =1pt](-1.5,-0.5) --(-1,-0.5);
\draw [color = black, line width =1pt](-1.5,0.5) --(-1.4,0.5);
\node [right]at(0,0.5) {\small $i$};
\node [right]at(0,-0.5) {\small $j$};
\node at(-1.2,0.5) {\small $m$};
\draw[color=black] (-1.2,0.5) circle (0.2);
\end{tikzpicture}}= 
\raisebox{-.30in}{
\begin{tikzpicture}
\tikzset{->-/.style=
{decoration={markings,mark=at position #1 with
{\arrow{latex}}},postaction={decorate}}}
\draw [line width =1.5pt,decoration={markings, mark=at position 0.91 with {\arrow{>}}},postaction={decorate}] (0,1)--(0,-1);
\draw [color = black, line width =1pt](-1,-0.5) --(0,-0.5);
\draw [color = black, line width =1pt](-1,0.5) --(-0.7,0.5);
\draw [color = black, line width =1pt](-0.3,0.5) --(0,0.5);
\node [right]at(0,0.5) {\small $j$};
\node [right]at(0,-0.5) {\small $i$};
\node at(-0.5,0.5) {\small $m$};
\draw[color=black] (-0.5,0.5) circle (0.2);
\end{tikzpicture}}
$$
where  
$\begin{tikzpicture}
\tikzset{->-/.style=
{decoration={markings,mark=at position #1 with
{\arrow{latex}}},postaction={decorate}}}
\draw [color = black, line width =1pt](0.8,0) --(1.2,0);
\draw [color = black, line width =1pt](0,0) --(0.4,0);
%
\node at(0.6,0) {\small $m$};
\draw[color=black] (0.6,0) circle(0.2);
\end{tikzpicture}$ is part of  $m$ parallel copies of a stated framed oriented boundary arc and the other single  line (the one stated by $i$) is not part of these $m$ parallel copies. Note that there are four possibilities for us to give orientations for the stated $n$-webs shown in the local picture.
\end{corollary}
\begin{proof}
From the assumption, we have $q^m = (q^{\frac{1}{n}})^m = 1$ and $q^2$ is a primitive $m$-th root of unity. Then the Corollary \ref{corr7.6} comes from Lemmas \ref{lmmm7.5}, \ref{lmmm7.7} and Remark \ref{remm7.8}.
\end{proof}

{\bf Conventions:} In the following of this section, we always assume, unless especially specified,
 $v$ is a primitive $m$-th root of unity with $m$ and $2n$ being coprime with each other, all the marked three manifolds involved have at least one marking at every component, $h$ is a  relative spin structure for $(M,N)$.

Recall that the stated skein algebra of bigon $\fB$ has  a Hopf algebra structure, and $S_n(\fB,v)$ is isomorphic to $O_q(SLn)$ as a Hopf algebra. We use $\Delta,\epsilon$ to denote it's coproduct and counit respectively. And for any $i,j\in\mathbb{J}$, $b_{i,j} =\raisebox{-.20in}{

\begin{tikzpicture}
\tikzset{->-/.style=

{decoration={markings,mark=at position #1 with

{\arrow{latex}}},postaction={decorate}}}

\filldraw[draw=white,fill=gray!20] (0,0) rectangle (1, 1);
\draw [line width =1pt,decoration={markings, mark=at position 0.5 with {\arrow{<}}},postaction={decorate}](0,0.5)--(1,0.5);
\draw[line width =1pt] (0,0)--(0,1);
\draw[line width =1pt] (1,0)--(1,1);
\node [left] at(0,0.5) {$i$};
\node [right] at(1,0.5) {$j$};
\end{tikzpicture}
}$.

\begin{lemma}\label{llll7.14}
In $S_n(\fB,v)$, we have 
\begin{align*}
&\sum_{\sigma\in S_n} (-1)^{\ell(\sigma)} (b_{1,\sigma(1)})^{(m)}(b_{2,\sigma(2)})^{(m)}\dots (b_{n,\sigma(n)})^{(m)}\\=&
\sum_{\sigma\in S_n} (-1)^{\ell(\sigma)} (b_{\sigma(1),1}^{1})^{(m)}(b_{\sigma(2),2})^{(m)}\dots (b_{\sigma(n),n})^{(m)} = 1.
\end{align*}
\end{lemma}
\begin{proof}
Theorem \ref{t.Hopf}, Lemma \ref{mmm2.4}.
\end{proof}

\begin{lemma}\label{llmm7.14}
In $S_n(\fB,v)$, we have 
$$\Delta((b_{i,j})^{(m)} )= \sum_{1\leq k\leq n} (b_{i,k})^{(m)}\otimes (b_{k,j})^{(m)}.$$
\end{lemma}
\begin{proof}
Theorem \ref{t.Hopf},  Lemma \ref{mmm2.4}.
\end{proof}

\begin{definition}
Let  $l$ be a stated $n$-web consisting of stated framed oriented boundary arcs. Suppose $l = \cup_{\alpha}\alpha$ where each $\alpha$ is a stated framed oriented booundary arc. Define $l^{(m)} = \cup_{\alpha}\alpha^{(m)}$.
\end{definition}

\begin{corollary}\label{corr7.15}
Suppose $D$ is a  properly embedded disk in a marked 3-manifold $(M, N )$ with an  oriented open interval $\beta\subset D$. Let $\alpha$ be a stated $n$-web consisting of stated
framed oriented boundary arcs. Suppose $\alpha$ is $(D, \beta)$-transverse and  intersects $\beta$ in exactly one point.
 For any state $k$, let $\alpha_k$, which is a stated $n$-web in $\text{Cut}_{(D, \beta)}(M,N)$,  be the lift of $\alpha$ with both newly created boundary points having state $k$. Then $\Theta(\alpha^{(m)}) = \sum_{1\leq k\leq n}(\alpha_k)^{(m)}$.
\end{corollary}
\begin{proof}
We use the same trick used in Lemma 4.2 in \cite{bloomquist2020chebyshev}.
We can assume $\alpha$ has one component.
 An open small tubular neighborhood of $\alpha\cup D$ in $M$ is isomorphic  to the thickening of a bigon such that $\alpha$ is the core of  the bigon and $D$ is the thickening of an ideal arc connecting the two ideal points of the bigon. This completes the proof because of functoriality and Lemma \ref{llmm7.14}.
\end{proof}

\begin{corollary}\label{corr7.16}
Suppose $D$ is a  properly embedded disk in a marked 3-manifold $(M, N )$ with an  oriented open interval $\beta\subset D$. Let $\alpha$   be a stated $n$-web consisting of stated
framed oriented boundary arcs.
 Suppose $\alpha$ is $(D, \beta)$-transverse, and $\alpha\cap \beta\neq\emptyset$.  For any map $s:\beta\cap \alpha\rightarrow\{1,2,\dots,n\}$, let $\alpha_s$, which is a stated $n$-web in $\text{Cut}_{(D, \beta)}(M,N)$, be the lift of $\alpha$ such that for every $P\in \beta\cap\alpha$ the two newly created boundary points corresponding to $P$ both  have state $s(P)$. Then 
\begin{equation}\label{eee}
\Theta(\alpha^{(m)}) = \sum_{s:\beta\cap \alpha\rightarrow\{1,2,\dots,n\}}(\alpha_s)^{(m)}.
\end{equation}
 Note that 
$
\Theta(\alpha) = \sum_{s:\beta\cap \alpha\rightarrow\{1,2,\dots,n\}}\alpha_s.
$
\end{corollary}
\begin{proof}
Here we use the technique used in page 24 in \cite{bloomquist2020chebyshev}.

Let $U$ be  an open small  tubular neighborhood of $\alpha\cup D$. Then we have $U$ is the thickening of a pb surface. Because of functoriality, we only need to prove equation (\ref{eee}) when $M=U$. Thus we can suppose that $\alpha$ is a simple diagram on  a pb surface  $\Sigma$ consisting of
 stated  arcs, with $c$ an ideal arc, and
$\alpha$ is transversal to $c$. Then equation (\ref{eee})  becomes
\begin{equation}
\Theta_{c}(\alpha^{(m)}) =  \sum_{s:c\cap \alpha\rightarrow\{1,2,\dots,n\}}(\alpha(h,s))^{(m)}
\label{eq.m3}
\end{equation}
where  $h$ be a linear order on the set $\alpha \cap c$ and $\alpha(h,s)$ is defined in Subsection \ref{sub27}.

 If $|\alpha\cap c|=1$, then equation (\ref{eq.m3}) follows exactly from Corollary \ref{corr7.15}.

\def\al{\alpha}
Now suppose $|\alpha\cap c|>1$.
Let $\cV$ be a finite subset of $c$ and  $c\setminus \cV = \cup_{i=1}^k c_i$ such that each $c_i$ intersects $\alpha$ at exactly one point.
 Let $\hat{\Sigma} = \Sigma \setminus \cV$.  Then we are ready to use Lemma \ref{llmm2.2}. Let $\Sigma'$ be the result of splitting $\Sigma$ along $c$, and $\hat{\Sigma}^{'}$ be the result of splitting $\hat{\Sigma}$ along all $c_i$.
The linear order $h$ on $\alpha\cap c$ induces a linear order on $\{c_i\mid 1\leq i\leq k\}$, which is also denoted as $h$.
Let $\hat{\alpha}\in S_n(\hat{\Sigma},v)$ be the element defined by the same $\alpha$, but considered as an element of $S_n(\hat{\Sigma},v)$. Then clearly 
$
\iota_*(\hat{\alpha}) = \al
$,
where $\iota_*: S_n(\hat{\Sigma},v) \rightarrow  S_n(\Sigma,v)$ is the induced homomorphism.
Since $\hat{\alpha}$ intersects each $c_i$ in exactly one point, from Corollary \ref{corr7.15} we have 
$$\Theta_{c}(\hat{\alpha}^{(m)}) = \sum_{s:c\cap \alpha\rightarrow\{1,2,\dots,n\}}(\hat{\alpha}_s)^{(m)},$$
where $\hat{\alpha}_s$ is the lift of $\hat{\alpha}$ such that for every $P\in c\cap\alpha$ the two newly created boundary points corresponding to $P$ both  have state $s(P)$.

 Then we have 
\begin{align*}
&\Theta_c(\alpha) = \Theta_c(\iota_{*}(\hat{\alpha})) =(\iota_h)_{*}( \Theta_c(\hat{\alpha}))\\
= &(\iota_h)_{*}(\sum_{s:c\cap \alpha\rightarrow\{1,2,\dots,n\}}(\hat{\alpha}_s)^{(m)})\\
=& \sum_{s:u\cap \alpha\rightarrow\{1,2,\dots,n\}}(\alpha(h,s))^{(m)}
\end{align*}
where $(\iota_h)_{*}$ is the induced homomorphism, see Lemma \ref{llmm2.2}.

\end{proof}

\def\mp{\raisebox{-30pt}{\incl{2.5 cm}{monogon1}}}
\def\mg{\raisebox{-30pt}{\incl{2.5 cm}{monogon2}}}
\def\bp{\raisebox{-30pt}{\incl{2.5 cm}{bigon1}}}
\def\bg{\raisebox{-30pt}{\incl{2.5 cm}{bigon2}}}

\begin{lemma}\label{lllll7.18}
We use $P_{1,2}$ to denote the once punctured bigon.
Let $\alpha$ (respectivly $\alpha^{'}$) be the framed oriented boundary arc in  the top left (respectively top right) picture in Figure \ref{fig:2}. Let $1\leq i,j\leq n$.
Then
in $S_n(P_{1,2},v)$, we have 
$$ (\alpha_{i,j})^{(m)} = (\alpha^{'}_{i,j})^{(m)}.$$
\end{lemma}

\begin{proof}
As in Figure \ref{fig:2}, we use the red ideal arc to cut $P_{1,2}$. For any state $v$, we use $\alpha_v$ to denote the lift of $\alpha_{i,j}$ with the newly created two endpoints stated by $v$, similarly we can define $(\alpha^{'})_{v}$. From Corollary \ref{corr7.16}, we know 
$$\Theta((\alpha_{i,j})^{(m)}) = \sum_{1\leq v\leq n} (\alpha_v)^{(m)},\; \Theta((\alpha^{'}_{i,j})^{(m)}) = \sum_{1\leq v\leq n} ((\alpha^{'})_v)^{(m)}.$$
From Corollary  \ref{corr7.3}, we know $(\alpha_v)^{(m)}=((\alpha^{'})_v)^{(m)}$. Thus we have $\Theta(\alpha^{(m)}) = \Theta((\alpha^{'})^{(m)})$. This completes the proof by the injectivity of $\Theta$, see Proposition \ref{inj}.
\end{proof}

\begin{lemma}\label{llmmm7.19}
We use $P_{1,1}$ to denote the once punctured monogon.
Let $\beta$ (respectivly $\beta^{'}$) be the framed oriented boundary arc in  the bottom left (respectively bottom right) picture in Figure \ref{fig:2}. Let $1\leq i,j\leq n$.
Then
in $S_n(P_{1,1},v)$, we have 
$$ (\beta_{i,j})^{(m)} = (\beta^{'}_{i,j})^{(m)}.$$
\end{lemma}
\begin{proof}
The proof is the same with Lemma \ref{lllll7.18}. The only difference is that we will use Corollary \ref{corr7.6}, instead of  Corollary  \ref{corr7.3}.
\end{proof}

\begin{figure}[H]
\centering
\includegraphics[scale=0.4]{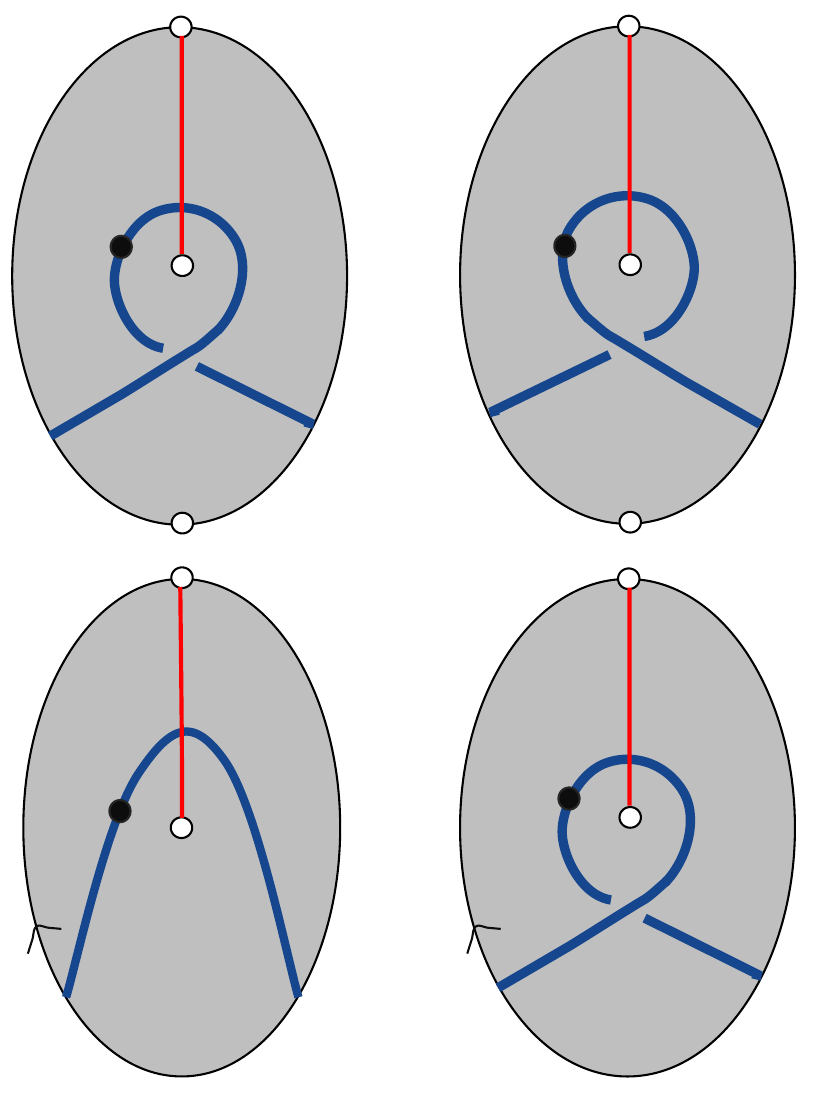}
\caption{Blue lines represent framed oriented boundary arcs, where the framing is the one pointing towards readers and the orientation is indicated by the black dot. Red lines are the cutting ideal arcs.}
\label{fig:2}
\end{figure}

\begin{corollary}\label{ccc7.20}
 In $S_n(M,N,v)$, we have 
\begin{align}
\raisebox{-.10in}{
\begin{tikzpicture}
\tikzset{->-/.style=
{decoration={markings,mark=at position #1 with
{\arrow{latex}}},postaction={decorate}}}
\draw [color = black, line width =1pt](0,0) --(-0.7,0);
\draw [color = black, line width =1pt](0,0) --(0.4,0);
\draw[color = black, line width =1pt](0.8,0) --(1.3,0);
\draw[color = black, line width =1pt](0,0.1) --(0,0.3);
\draw[color = black, line width =1pt](0,0.7) --(0,1.2);
\draw [color = black, line width =1pt](0,-0.5) --(0,-0.1);
\node at(0.6,0) {\small $m$};
\draw[color=black] (0.6,0) circle(0.2);
\node at(0,0.5) {\small $m$};
\draw[color=black] (0,0.5) circle(0.2);
\end{tikzpicture}}
=
\raisebox{-.10in}{
\begin{tikzpicture}
\tikzset{->-/.style=
{decoration={markings,mark=at position #1 with
{\arrow{latex}}},postaction={decorate}}}
\draw [color = black, line width =1pt](-0.1,0) --(-0.7,0);
\draw [color = black, line width =1pt](0.1,0) --(0.4,0);
\draw[color = black, line width =1pt](0.8,0) --(1.3,0);
\draw[color = black, line width =1pt](0,0) --(0,0.3);
\draw [color = black, line width =1pt](0,-0.5) --(0,0);
\draw[color = black, line width =1pt](0,0.7) --(0,1.2);
\node at(0,0.5) {\small $m$};
\draw[color=black] (0,0.5) circle(0.2);
\node at(0.6,0) {\small $m$};
\draw[color=black] (0.6,0) circle(0.2);
\end{tikzpicture}}
\end{align}
where on the left (right) hand side of the equality the two
$\begin{tikzpicture}
\tikzset{->-/.style=
{decoration={markings,mark=at position #1 with
{\arrow{latex}}},postaction={decorate}}}
\draw [color = black, line width =1pt](0,0) --(-0.7,0);
\draw [color = black, line width =1pt](0,0) --(0.4,0);
\draw[color = black, line width =1pt](0.8,0) --(1.3,0);
\node at(0.6,0) {\small $m$};
\draw[color=black] (0.6,0) circle(0.2);
\end{tikzpicture}$ are parts of the same $m$ parallel copies of a stated framed oriented boundary arc.
\end{corollary}
\begin{proof}
The Corollary can be easily proved by using functoriality and Lemma
\ref{lllll7.18}.
\end{proof}

\begin{corollary}\label{ccc7.21}
In $S_n(M,N,v)$, we have 
$$
\raisebox{-.30in}{
\begin{tikzpicture}
\tikzset{->-/.style=
{decoration={markings,mark=at position #1 with
{\arrow{latex}}},postaction={decorate}}}
\draw [line width =1.5pt,decoration={markings, mark=at position 0.91 with {\arrow{>}}},postaction={decorate}] (0,1)--(0,-1);
\draw [color = black, line width =1pt](-1,0.5) --(-0,-0.5);
\draw [color = black, line width =1pt](-1,-0.5) --(-0.6,-0.1);
\draw [color = black, line width =1pt](-0.4,0.1) --(0,0.5);
\draw [color = black, line width =1pt](-1.5,-0.5) --(-1.4,-0.5);
\draw [color = black, line width =1pt](-1.5,0.5) --(-1.4,0.5);
\node [right]at(0,0.5) {\small $i$};
\node [right]at(0,-0.5) {\small $j$};
\node at(-1.2,0.5) {\small $m$};
\draw[color=black] (-1.2,0.5) circle (0.2);
\node at(-1.2,-0.5) {\small $m$};
\draw[color=black] (-1.2,-0.5) circle (0.2);
\end{tikzpicture}}= 
\raisebox{-.30in}{
\begin{tikzpicture}
\tikzset{->-/.style=
{decoration={markings,mark=at position #1 with
{\arrow{latex}}},postaction={decorate}}}
\draw [line width =1.5pt,decoration={markings, mark=at position 0.91 with {\arrow{>}}},postaction={decorate}] (0,1)--(0,-1);
\draw [color = black, line width =1pt](-1,-0.5) --(-0.7,-0.5);
\draw [color = black, line width =1pt](-0.3,-0.5) --(0,-0.5);
\draw [color = black, line width =1pt](-1,0.5) --(-0.7,0.5);
\draw [color = black, line width =1pt](-0.3,0.5) --(0,0.5);
\node [right]at(0,0.5) {\small $j$};
\node [right]at(0,-0.5) {\small $i$};
\node at(-0.5,0.5) {\small $m$};
\draw[color=black] (-0.5,0.5) circle (0.2);
\node at(-0.5,-0.5) {\small $m$};
\draw[color=black] (-0.5,-0.5) circle (0.2);
\end{tikzpicture}}
$$
where on the left (right) hand side of the equality the two
$\begin{tikzpicture}
\tikzset{->-/.style=
{decoration={markings,mark=at position #1 with
{\arrow{latex}}},postaction={decorate}}}
\draw [color = black, line width =1pt](0,0) --(-0.7,0);
\draw [color = black, line width =1pt](0,0) --(0.4,0);
\draw[color = black, line width =1pt](0.8,0) --(1.3,0);
\node at(0.6,0) {\small $m$};
\draw[color=black] (0.6,0) circle(0.2);
\end{tikzpicture}$ are parts of the same $m$ parallel copies of a stated framed oriented boundary arc.
\end{corollary}
\begin{proof}
The Corollary can be easily proved by using functoriality and Lemma
\ref{llmmm7.19}.
\end{proof}

\begin{lemma}\label{lllm7.17}

In $S_n(M,N,v)$, we have 
\begin{equation*}
\raisebox{-.20in}{
\begin{tikzpicture}
\tikzset{->-/.style=
{decoration={markings,mark=at position #1 with
{\arrow{latex}}},postaction={decorate}}}
\draw [line width =1.5pt,decoration={markings, mark=at position 0.91 with {\arrow{>}}},postaction={decorate}] (0,1)--(0,0);
\draw [color = black, line width =1pt] (-1,0.5)--(-0.7,0.5);
\draw [color = black, line width =1pt] (-1.4,0.5)--(-1.7,0.5);
\draw [color = black, line width =1pt](-0,0.5)--(-0.3,0.5);
\node at(-0.5,0.5) {\small $m$};
\node [right] at(-0,0.5) {\small $k$};
\draw[color=black] (-0.5,0.5) circle (0.2);
\draw (-1.4,0.2) rectangle (-1,0.8);
\node  at(-1.2,0.5) {\small $\bar{H}$};
\end{tikzpicture}}
= 
\raisebox{-.20in}{
\begin{tikzpicture}
\tikzset{->-/.style=
{decoration={markings,mark=at position #1 with
{\arrow{latex}}},postaction={decorate}}}
\draw [line width =1.5pt,decoration={markings, mark=at position 0.91 with {\arrow{>}}},postaction={decorate}] (0,1)--(0,0);
\draw [color = black, line width =1pt] (-1,0.5)--(-0.7,0.5);
\draw [color = black, line width =1pt](-0,0.5)--(-0.3,0.5);
\node at(-0.5,0.5) {\small $m$};
\node [right] at(-0,0.5) {\small $k$};
\draw[color=black] (-0.5,0.5) circle (0.2);
\end{tikzpicture}}=
\raisebox{-.20in}{
\begin{tikzpicture}
\tikzset{->-/.style=
{decoration={markings,mark=at position #1 with
{\arrow{latex}}},postaction={decorate}}}
\draw [line width =1.5pt,decoration={markings, mark=at position 0.91 with {\arrow{>}}},postaction={decorate}] (0,1)--(0,0);
\draw [color = black, line width =1pt] (-1,0.5)--(-0.7,0.5);
\draw [color = black, line width =1pt] (-1.4,0.5)--(-1.7,0.5);
\draw [color = black, line width =1pt](-0,0.5)--(-0.3,0.5);
\node at(-0.5,0.5) {\small $m$};
\node [right] at(-0,0.5) {\small $k$};
\draw[color=black] (-0.5,0.5) circle (0.2);
\draw (-1.4,0.2) rectangle (-1,0.8);
\node  at(-1.2,0.5) {\small $H$};
\end{tikzpicture}}
\end{equation*}
where $H,\bar{H}$ are postive half twist and negative half twist respectively, and
 the parts not shown in the local picture can be arbitrary, that is, 
$\begin{tikzpicture}
\tikzset{->-/.style=
{decoration={markings,mark=at position #1 with
{\arrow{latex}}},postaction={decorate}}}
\draw [color = black, line width =1pt](0.8,0) --(1.2,0);
\draw [color = black, line width =1pt](0,0) --(0.4,0);
%
\node at(0.6,0) {\small $m$};
\draw[color=black] (0.6,0) circle(0.2);
\end{tikzpicture}$ may not be part of  $m$ parallel copies of some stated framed oriented boundary arc. 
\end{lemma}

\begin{proof}
This Lemma can be easily proved by using
equation (51) in \cite{le2021stated} and Lemma \ref{lmmm7.1}.
\end{proof}

\begin{lemma}\label{mmm7.23}

In $S_n(M,N,v)$, we have 
\begin{equation}\label{86}
\sum_{1\leq k\leq n}(-1)^{k+1}
\raisebox{-.30in}{
\begin{tikzpicture}
\tikzset{->-/.style=
{decoration={markings,mark=at position #1 with
{\arrow{latex}}},postaction={decorate}}}
\draw [line width =1.5pt,decoration={markings, mark=at position 0.91 with {\arrow{>}}},postaction={decorate}] (0,-1)--(0,1);
\draw [line width =1pt,decoration={markings, mark=at position 0.5 with {\arrow{>}}},postaction={decorate}] (-1.3,-0.5)--(-0.7,-0.5);
\draw [color = black, line width =1pt](-0,-0.5)--(-0.3,-0.5);
\node at(-0.5,-0.5) {\small $m$};
\node [right] at(-0,-0.5) {\small $\bar{k}$};
\draw[color=black] (-0.5,-0.5) circle (0.2);
\draw [line width =1pt,decoration={markings, mark=at position 0.5 with {\arrow{<}}},postaction={decorate}] (-1.3,0.5)--(-0.7,0.5);
\draw [color = black, line width =1pt](-0,0.5)--(-0.3,0.5);
\node at(-0.5,0.5) {\small $m$};
\node [right] at(-0,0.5) {\small $k$};
\draw[color=black] (-0.5,0.5) circle (0.2);
\end{tikzpicture}}
= 
\raisebox{-.30in}{
\begin{tikzpicture}
\tikzset{->-/.style=
{decoration={markings,mark=at position #1 with
{\arrow{latex}}},postaction={decorate}}}
\draw [line width =1.5pt,decoration={markings, mark=at position 0.91 with {\arrow{>}}},postaction={decorate}] (0.7,-1)--(0.7,1);
\draw [line width =1pt,decoration={markings, mark=at position 0.5 with {\arrow{>}}},postaction={decorate}] (-1.3,-0.5)--(-0.7,-0.5);
\draw [color = black, line width =1pt](-0,-0.5)--(-0.3,-0.5);
\node at(-0.5,-0.5) {\small $m$};
\draw[color=black] (-0.5,-0.5) circle (0.2);
\draw [color = black, line width =1pt] (-1.3,0.5)--(-0.7,0.5);
\draw [color = black, line width =1pt](-0,0.5)--(-0.7,0.5);
\draw [color = black, line width =1pt] (0 ,-0.5) arc (-90:90:0.5);
\end{tikzpicture}}
\end{equation}
where  all three
$\begin{tikzpicture}
\tikzset{->-/.style=
{decoration={markings,mark=at position #1 with
{\arrow{latex}}},postaction={decorate}}}
\draw [color = black, line width =1pt](0.8,0) --(1.2,0);
\draw [color = black, line width =1pt](0,0) --(0.4,0);
%
\node at(0.6,0) {\small $m$};
\draw[color=black] (0.6,0) circle(0.2);
\end{tikzpicture}$ are parts of  $m$ parallel copies of some stated framed oriented boundary arc. 
\end{lemma}

\begin{proof}
Because of functoriality, we can assume besides the components that are shown in equation \eqref{86}, the stated $n$-webs have  no other components.
We use $\beta$ to denote the marking shown in equation (\ref{86}). Then there is a right coaction 
$\Delta_{\beta}:S_n(M,N,v)\rightarrow S_n(M,N,v)\otimes O_q(SLn)$, see Subsection 7.1 in \cite{le2021stated}. The coaction $\Delta_{\beta}$ is actually defined by cutting out a bigon. From Corollary \ref{corr7.16}, we have 
$$
\Delta_{\beta}(\raisebox{-.30in}{
\begin{tikzpicture}
\tikzset{->-/.style=
{decoration={markings,mark=at position #1 with
{\arrow{latex}}},postaction={decorate}}}
\draw [line width =1.5pt,decoration={markings, mark=at position 0.91 with {\arrow{>}}},postaction={decorate}] (0.7,-1)--(0.7,1);
\draw [line width =1pt,decoration={markings, mark=at position 0.5 with {\arrow{>}}},postaction={decorate}] (-1.3,-0.5)--(-0.7,-0.5);
\draw [color = black, line width =1pt](-0,-0.5)--(-0.3,-0.5);
\node at(-0.5,-0.5) {\small $m$};
\draw[color=black] (-0.5,-0.5) circle (0.2);
\draw [color = black, line width =1pt] (-1.3,0.5)--(-0.7,0.5);
\draw [color = black, line width =1pt](-0,0.5)--(-0.7,0.5);
\draw [color = black, line width =1pt] (0 ,-0.5) arc (-90:90:0.5);
\end{tikzpicture}}) =
\sum_{1\leq u,v\leq n}
\raisebox{-.30in}{
\begin{tikzpicture}
\tikzset{->-/.style=
{decoration={markings,mark=at position #1 with
{\arrow{latex}}},postaction={decorate}}}
\draw [line width =1.5pt,decoration={markings, mark=at position 0.91 with {\arrow{>}}},postaction={decorate}] (0,-1)--(0,1);
\draw [line width =1pt,decoration={markings, mark=at position 0.5 with {\arrow{>}}},postaction={decorate}] (-1.3,-0.5)--(-0.7,-0.5);
\draw [color = black, line width =1pt](-0,-0.5)--(-0.3,-0.5);
\node at(-0.5,-0.5) {\small $m$};
\node [right] at(-0,-0.5) {\small $v$};
\draw[color=black] (-0.5,-0.5) circle (0.2);
\draw [line width =1pt,decoration={markings, mark=at position 0.5 with {\arrow{<}}},postaction={decorate}] (-1.3,0.5)--(-0.7,0.5);
\draw [color = black, line width =1pt](-0,0.5)--(-0.3,0.5);
\node at(-0.5,0.5) {\small $m$};
\node [right] at(-0,0.5) {\small $u$};
\draw[color=black] (-0.5,0.5) circle (0.2);
\end{tikzpicture}}\otimes
\raisebox{-.30in}{
\begin{tikzpicture}
\tikzset{->-/.style=
{decoration={markings,mark=at position #1 with
{\arrow{latex}}},postaction={decorate}}}
\draw [line width =1.5pt,decoration={markings, mark=at position 0.91 with {\arrow{>}}},postaction={decorate}] (-1.3,-1)--(-1.3,1);
\draw [line width =1.5pt,decoration={markings, mark=at position 0.91 with {\arrow{>}}},postaction={decorate}] (0.7,-1)--(0.7,1);
\draw [line width =1pt,decoration={markings, mark=at position 0.5 with {\arrow{>}}},postaction={decorate}] (-1.3,-0.5)--(-0.7,-0.5);
\draw [color = black, line width =1pt](-0,-0.5)--(-0.3,-0.5);
\node at(-0.5,-0.5) {\small $m$};
\draw[color=black] (-0.5,-0.5) circle (0.2);
\draw [color = black, line width =1pt] (-1.3,0.5)--(-0.7,0.5);
\draw [color = black, line width =1pt](-0,0.5)--(-0.7,0.5);
\node [left] at(-1.3,0.5) {\small $u$};
\node [left] at(-1.3,-0.5) {\small $v$};
\draw [color = black, line width =1pt] (0 ,-0.5) arc (-90:90:0.5);
\end{tikzpicture}}.
$$
Thus we have 
\begin{align*}&
\raisebox{-.30in}{
\begin{tikzpicture}
\tikzset{->-/.style=
{decoration={markings,mark=at position #1 with
{\arrow{latex}}},postaction={decorate}}}
\draw [line width =1.5pt,decoration={markings, mark=at position 0.91 with {\arrow{>}}},postaction={decorate}] (0.7,-1)--(0.7,1);
\draw [line width =1pt,decoration={markings, mark=at position 0.5 with {\arrow{>}}},postaction={decorate}] (-1.3,-0.5)--(-0.7,-0.5);
\draw [color = black, line width =1pt](-0,-0.5)--(-0.3,-0.5);
\node at(-0.5,-0.5) {\small $m$};
\draw[color=black] (-0.5,-0.5) circle (0.2);
\draw [color = black, line width =1pt] (-1.3,0.5)--(-0.7,0.5);
\draw [color = black, line width =1pt](-0,0.5)--(-0.7,0.5);
\draw [color = black, line width =1pt] (0 ,-0.5) arc (-90:90:0.5);
\end{tikzpicture}}
=
\sum_{1\leq u,v\leq n}
\varepsilon(\raisebox{-.30in}{
\begin{tikzpicture}
\tikzset{->-/.style=
{decoration={markings,mark=at position #1 with
{\arrow{latex}}},postaction={decorate}}}
\draw [line width =1.5pt,decoration={markings, mark=at position 0.91 with {\arrow{>}}},postaction={decorate}] (-1.3,-1)--(-1.3,1);
\draw [line width =1.5pt,decoration={markings, mark=at position 0.91 with {\arrow{>}}},postaction={decorate}] (0.7,-1)--(0.7,1);
\draw [line width =1pt,decoration={markings, mark=at position 0.5 with {\arrow{>}}},postaction={decorate}] (-1.3,-0.5)--(-0.7,-0.5);
\draw [color = black, line width =1pt](-0,-0.5)--(-0.3,-0.5);
\node at(-0.5,-0.5) {\small $m$};
\draw[color=black] (-0.5,-0.5) circle (0.2);
\draw [color = black, line width =1pt] (-1.3,0.5)--(-0.7,0.5);
\draw [color = black, line width =1pt](-0,0.5)--(-0.7,0.5);
\node [left] at(-1.3,0.5) {\small $u$};
\node [left] at(-1.3,-0.5) {\small $v$};
\draw [color = black, line width =1pt] (0 ,-0.5) arc (-90:90:0.5);
\end{tikzpicture}})\,
\raisebox{-.30in}{
\begin{tikzpicture}
\tikzset{->-/.style=
{decoration={markings,mark=at position #1 with
{\arrow{latex}}},postaction={decorate}}}
\draw [line width =1.5pt,decoration={markings, mark=at position 0.91 with {\arrow{>}}},postaction={decorate}] (0,-1)--(0,1);
\draw [line width =1pt,decoration={markings, mark=at position 0.5 with {\arrow{>}}},postaction={decorate}] (-1.3,-0.5)--(-0.7,-0.5);
\draw [color = black, line width =1pt](-0,-0.5)--(-0.3,-0.5);
\node at(-0.5,-0.5) {\small $m$};
\node [right] at(-0,-0.5) {\small $v$};
\draw[color=black] (-0.5,-0.5) circle (0.2);
\draw [line width =1pt,decoration={markings, mark=at position 0.5 with {\arrow{<}}},postaction={decorate}] (-1.3,0.5)--(-0.7,0.5);
\draw [color = black, line width =1pt](-0,0.5)--(-0.3,0.5);
\node at(-0.5,0.5) {\small $m$};
\node [right] at(-0,0.5) {\small $u$};
\draw[color=black] (-0.5,0.5) circle (0.2);
\end{tikzpicture}}\\
=&
\sum_{1\leq u,v\leq n} \delta_{u,\bar{v}} (c_v)^{m}
\,
\raisebox{-.30in}{
\begin{tikzpicture}
\tikzset{->-/.style=
{decoration={markings,mark=at position #1 with
{\arrow{latex}}},postaction={decorate}}}
\draw [line width =1.5pt,decoration={markings, mark=at position 0.91 with {\arrow{>}}},postaction={decorate}] (0,-1)--(0,1);
\draw [line width =1pt,decoration={markings, mark=at position 0.5 with {\arrow{>}}},postaction={decorate}] (-1.3,-0.5)--(-0.7,-0.5);
\draw [color = black, line width =1pt](-0,-0.5)--(-0.3,-0.5);
\node at(-0.5,-0.5) {\small $m$};
\node [right] at(-0,-0.5) {\small $v$};
\draw[color=black] (-0.5,-0.5) circle (0.2);
\draw [line width =1pt,decoration={markings, mark=at position 0.5 with {\arrow{<}}},postaction={decorate}] (-1.3,0.5)--(-0.7,0.5);
\draw [color = black, line width =1pt](-0,0.5)--(-0.3,0.5);
\node at(-0.5,0.5) {\small $m$};
\node [right] at(-0,0.5) {\small $u$};
\draw[color=black] (-0.5,0.5) circle (0.2);
\end{tikzpicture}}=
\sum_{1\leq v\leq n} (-1)^{\bar{v}+1}
\,
\raisebox{-.30in}{
\begin{tikzpicture}
\tikzset{->-/.style=
{decoration={markings,mark=at position #1 with
{\arrow{latex}}},postaction={decorate}}}
\draw [line width =1.5pt,decoration={markings, mark=at position 0.91 with {\arrow{>}}},postaction={decorate}] (0,-1)--(0,1);
\draw [line width =1pt,decoration={markings, mark=at position 0.5 with {\arrow{>}}},postaction={decorate}] (-1.3,-0.5)--(-0.7,-0.5);
\draw [color = black, line width =1pt](-0,-0.5)--(-0.3,-0.5);
\node at(-0.5,-0.5) {\small $m$};
\node [right] at(-0,-0.5) {\small $v$};
\draw[color=black] (-0.5,-0.5) circle (0.2);
\draw [line width =1pt,decoration={markings, mark=at position 0.5 with {\arrow{<}}},postaction={decorate}] (-1.3,0.5)--(-0.7,0.5);
\draw [color = black, line width =1pt](-0,0.5)--(-0.3,0.5);
\node at(-0.5,0.5) {\small $m$};
\node [right] at(-0,0.5) {\small $\bar{v}$};
\draw[color=black] (-0.5,0.5) circle (0.2);
\end{tikzpicture}}
\end{align*}
where the second equality is because of   Lemma \ref{lllm7.17} and relation (\ref{wzh.six}).
\end{proof}

\begin{lemma}\label{lll7.20}

In $S_n(M,N,v)$, we have 
\begin{equation*}
\raisebox{-.10in}{
\begin{tikzpicture}
\tikzset{->-/.style=
{decoration={markings,mark=at position #1 with
{\arrow{latex}}},postaction={decorate}}}
\draw [color = black, line width =1pt](-1,0)--(-0.25,0);
\draw [color = black, line width =1pt](0,0)--(0.6,0);
\draw [color = black, line width =1pt](1,0)--(1.5,0);
\node at(0.8,0) {\small $m$};
\draw[color=black] (0.8,0) circle (0.2);
\draw [color = black, line width =1pt] (0.166 ,0.08) arc (-37:270:0.2);
\end{tikzpicture}}=
\raisebox{-.10in}{
\begin{tikzpicture}
\tikzset{->-/.style=
{decoration={markings,mark=at position #1 with
{\arrow{latex}}},postaction={decorate}}}
\draw [color = black, line width =1pt](-1,0)--(0,0);
\draw [color = black, line width =1pt](0.23,0)--(0.6,0);
\draw [color = black, line width =1pt](1,0)--(1.5,0);
\node at(0.8,0) {\small $m$};
\draw[color=black] (0.8,0) circle (0.2);
\draw [color = black, line width =1pt] (0 ,0) arc (-90:215:0.2);
\end{tikzpicture}}
= d_n
\raisebox{-.10in}{
\begin{tikzpicture}
\tikzset{->-/.style=
{decoration={markings,mark=at position #1 with
{\arrow{latex}}},postaction={decorate}}}
\draw [color = black, line width =1pt](-0.3,0)--(0.5,0);
\draw [color = black, line width =1pt](0.23,0)--(0.6,0);
\draw [color = black, line width =1pt](1,0)--(1.5,0);
\node at(0.8,0) {\small $m$};
\draw[color=black] (0.8,0) circle (0.2);
\end{tikzpicture}}
\end{equation*}
where  
$\begin{tikzpicture}
\tikzset{->-/.style=
{decoration={markings,mark=at position #1 with
{\arrow{latex}}},postaction={decorate}}}
\draw [color = black, line width =1pt](0.8,0) --(1.2,0);
\draw [color = black, line width =1pt](0,0) --(0.4,0);
%
\node at(0.6,0) {\small $m$};
\draw[color=black] (0.6,0) circle(0.2);
\end{tikzpicture}$ is part of  $m$ parallel copies of some stated framed oriented boundary arc. 
\end{lemma}
\begin{proof}
Since a positive kink for the $m$ parallel copies of some framed line is isotopic to combining a full positive twist and giving a positive kink for each parallel framed line, then we have 
$$
\raisebox{-.10in}{
\begin{tikzpicture}
\tikzset{->-/.style=
{decoration={markings,mark=at position #1 with
{\arrow{latex}}},postaction={decorate}}}
\draw [color = black, line width =1pt](-1,0)--(-0.25,0);
\draw [color = black, line width =1pt](0,0)--(0.6,0);
\draw [color = black, line width =1pt](1,0)--(1.5,0);
\node at(0.8,0) {\small $m$};
\draw[color=black] (0.8,0) circle (0.2);
\draw [color = black, line width =1pt] (0.166 ,0.08) arc (-37:270:0.2);
\end{tikzpicture}} =
t^{m}
\raisebox{-.10in}{
\begin{tikzpicture}
\tikzset{->-/.style=
{decoration={markings,mark=at position #1 with
{\arrow{latex}}},postaction={decorate}}}
\draw [color = black, line width =1pt](-1,0)--(-0.2,0);
\draw [color = black, line width =1pt](0.2,0)--(0.6,0);
\draw [color = black, line width =1pt](1,0)--(1.5,0);
\node at(0.8,0) {\small $m$};
\node at(0,0) {\small $F$};
\draw[color=black] (0.8,0) circle (0.2);
\draw (-0.2,-0.3) rectangle (0.2,0.3);
\end{tikzpicture}}=
d_n
\raisebox{-.10in}{
\begin{tikzpicture}
\tikzset{->-/.style=
{decoration={markings,mark=at position #1 with
{\arrow{latex}}},postaction={decorate}}}
\draw [color = black, line width =1pt](-0.3,0)--(0.5,0);
\draw [color = black, line width =1pt](0.23,0)--(0.6,0);
\draw [color = black, line width =1pt](1,0)--(1.5,0);
\node at(0.8,0) {\small $m$};
\draw[color=black] (0.8,0) circle (0.2);
\end{tikzpicture}}
$$
where the first equality is from relation (\ref{w.twist}) and the second equality is from Lemma \ref{lllm7.17}.

Similarly we can show $\raisebox{-.10in}{
\begin{tikzpicture}
\tikzset{->-/.style=
{decoration={markings,mark=at position #1 with
{\arrow{latex}}},postaction={decorate}}}
\draw [color = black, line width =1pt](-1,0)--(0,0);
\draw [color = black, line width =1pt](0.23,0)--(0.6,0);
\draw [color = black, line width =1pt](1,0)--(1.5,0);
\node at(0.8,0) {\small $m$};
\draw[color=black] (0.8,0) circle (0.2);
\draw [color = black, line width =1pt] (0 ,0) arc (-90:215:0.2);
\end{tikzpicture}}
= d_n
\raisebox{-.10in}{
\begin{tikzpicture}
\tikzset{->-/.style=
{decoration={markings,mark=at position #1 with
{\arrow{latex}}},postaction={decorate}}}
\draw [color = black, line width =1pt](-0.3,0)--(0.5,0);
\draw [color = black, line width =1pt](0.23,0)--(0.6,0);
\draw [color = black, line width =1pt](1,0)--(1.5,0);
\node at(0.8,0) {\small $m$};
\draw[color=black] (0.8,0) circle (0.2);
\end{tikzpicture}}$\,.
\end{proof}

\def \SM {S_n(M,N,v)}

\def \End {\text{End}(\SM)}
\def \Endm {\text{End}(\SM)^{(m)}}

\subsection{An action of $S_n(M,N,1)$  on $S_n(M,N,v)$}\label{subb7.2}

Let $(M,N)$ be a marked three manifold. Recall that
for $l = \cup_{\alpha}\alpha$, where each $\alpha$ is a stated framed oriented boundary arc in $(M,N)$, we define $l^{(m)}$ to be $ \cup_{\alpha}\alpha^{(m)}$.

\begin{lemma}\label{lmmm7.10}
Let  $l$ be a stated $n$-web consisting of stated framed oriented boundary arcs. Let $T_1,T_2$ be any two isopotic stated $n$-webs such that $T_1\cap l = T_2\cap l =\emptyset$. Then we have 
$$l^{(m)}\cup T_1 = l^{(m)}\cup T_2\in S_n(M,N,v).$$
\end{lemma}
\begin{proof}
Corollaries \ref{corr7.3}, \ref{corr7.6}.
\end{proof}

For any stated $n$-web $l$ consisting of stated framed oriented booundary arcs,
 we try define a linear map $F_{l}:\SM\rightarrow \SM$. For any stated $n$-web $\alpha$, first we isotope $\alpha$ such that $\alpha\cap l =\emptyset$, then define $F_l(\alpha) = l^{(m)}\cup\alpha\in\SM.$ From Lemma \ref{lmmm7.10}, we know $F_l(\alpha)$ is independent of how we isotope $\alpha$. Thus $F_l$ is well-defined on the set of isotopy classes of stated $n$-webs. Since all the skein relations, used to define stated $SL(n)$-skein module are local, $F_l$ preserves all these relations. Thus $F_l$ is a linear map from $\SM$ to $\SM$.
We regard empty $n$-web as a stated $n$-web consisting of stated framed oriented boundary arcs, then 
$F_{\emptyset} = Id_{\SM}$.

 Recall that
 we use 
$\End$ to denote the set of linear maps from $\SM$ to itself, and $\End$ has an obvious algebra structure. Then $F_l\in\End$ for stated $n$-web $l$ consisting of stated framed oriented boundary arcs. 

Let $\alpha,\beta$ be any two stated $n$-webs both consisting of stated  framed oriented boundary arcs, then we 
have $F_{\alpha} F_{\beta} = F_{\beta} F_{\alpha} = F_{\alpha\cup\beta}$ from Corollaries \ref{corr7.3}, \ref{corr7.6}, where $\alpha\cup\beta$ is the disjoint union. Let $\Endm$ be a subvector space of $\End$ linearly spanned by $F_l$ for all stated $n$-webs $l$ consisting of  stated framed oriented booundary arcs. From the above discussion, we know $\Endm$ is a commutative subalgebra of $\End$.

\begin{lemma}\label{mmmm7.25}
Let $\alpha$ be any framed oriented boundary arc in $(M,N)$. In $\Endm$ we have
\begin{align*}
&\sum_{\sigma\in S_n} (-1)^{\ell(\sigma)} F_{\alpha_{1,\sigma(1)}}F_{\alpha_{2,\sigma(2)}}\dots F_{ \alpha_{n,\sigma(n)}}\\=&
\sum_{\sigma\in S_n} (-1)^{\ell(\sigma)} F_{\alpha_{\sigma(1),1}}F_{\alpha_{\sigma(2),2}}\dots F_{\alpha_{\sigma(n),n})}= 1.
\end{align*}
\end{lemma}
\begin{proof}
Let $\beta$ be a stated $n$-web that has no intersection with $\alpha$.
A small tubular open neighborhood $U$ of $\alpha$ in $M$ is isomorphic
to the thickening of the bigon, with $\alpha$ being identified with the core of
the bigon and $U\cap \beta=\emptyset$. We isotope $\alpha_{i,j}$ inside $U$ such that there is no intersection among them. Then
\begin{align*}
&\sum_{\sigma\in S_n} (-1)^{\ell(\sigma)} F_{\alpha_{1,\sigma(1)}}F_{\alpha_{2,\sigma(2)}}\dots F_{ \alpha_{n,\sigma(n)}}(\beta)\\=&
\sum_{\sigma\in S_n} (-1)^{\ell(\sigma)} (\alpha_{1,\sigma(1)})^{(m)}\cup(\alpha_{2,\sigma(2)})^{(m)}\cup\dots\cup (\alpha_{n,\sigma(n)})^{(m)}\cup\beta= \beta.
\end{align*}
where the last equality is because of functoriality and
 Lemma \ref{llll7.14}. 
Thus we have $$\sum_{\sigma\in S_n} (-1)^{\ell(\sigma)} F_{\alpha_{1,\sigma(1)}}F_{\alpha_{2,\sigma(2)}}\dots F_{ \alpha_{n,\sigma(n)}}= 1.$$ Similarly we can prove 
$
\sum_{\sigma\in S_n} (-1)^{\ell(\sigma)} F_{\alpha_{\sigma(1),1}}F_{\alpha_{\sigma(2),2}}\dots F_{\alpha_{\sigma(n),n})}= 1$.
\end{proof}

\begin{lemma}\label{llll7.25}
For any two stated  framed oriented  boundary arcs $\alpha_1,\alpha_2$, suppose $s(\alpha_1(0))
= s(\alpha_2(0))$, $s(\alpha_1(1))
= s(\alpha_2(1))$, and $[\alpha_1] = [\alpha_2]\in
\pi_1(M,N)$. 

(a) If $h(\widetilde{\alpha_1}) = h(\widetilde{\alpha_2}) $, then $F_{\alpha_1} =F_{\alpha_2} \in \Endm.$

(b) If $h(\widetilde{\alpha_1}) \neq h(\widetilde{\alpha_2}) $, then $F_{\alpha_1} = d_n F_{\alpha_2} \in \Endm.$
\end{lemma}

\begin{proof}
Here we  use a  standard fact that two embeddings of a compact graph in $M$ are homotopic if and only if  one can be obtained from the other by crossing changes and   isotopy \cite{skeingroup}.
Let $\beta$ be a stated $n$-web that has no intersection with $\alpha_1\cup\alpha_2$.
 From the above standard fact and Corollaries \ref{corr7.3}, \ref{corr7.6}, \ref{ccc7.20}, \ref{ccc7.21}, we know  $(\alpha_1)^{(m)}\cup\beta =(\alpha_3)^{(m)}\cup\beta$ where $\alpha_3$ is exactly $\alpha_2$ but with different framing and $h(\widetilde{\alpha_1}) = h(\widetilde{\alpha_3})$. Thus $\alpha_3$ can be obtained from $\alpha_2$ by adding a number of kinks, we suppose this number is $\lambda$. Then we have $(\alpha_1)^{(m)}\cup\beta=(\alpha_3)^{(m)}\cup\beta = d_n^{\lambda} (\alpha_2)^{(m)}\cup\beta$ from Lemma \ref{lll7.20}.

If $h(\widetilde{\alpha_1}) = h(\widetilde{\alpha_2}) $, then  $h(\widetilde{\alpha_2}) = h(\widetilde{\alpha_3}) $. Thus we have $\lambda$ is even, then 
$$F_{\alpha_1}(\beta) =  (\alpha_1)^{(m)}\cup\beta= d_n^{\lambda} (\alpha_2)^{(m)} \cup\beta=(\alpha_2)^{(m)} \cup\beta = F_{\alpha_2}(\beta).$$

If $h(\widetilde{\alpha_1}) \neq h(\widetilde{\alpha_2}) $, then  $h(\widetilde{\alpha_2}) \neq h(\widetilde{\alpha_3}) $. Thus we have $\lambda$ is odd, then 
$$F_{\alpha_1}(\beta) =  (\alpha_1)^{(m)}\cup\beta= d_n^{\lambda} (\alpha_2)^{(m)} \cup\beta=d_n(\alpha_2)^{(m)} \cup\beta =d_n F_{\alpha_2}(\beta).$$
\end{proof}

\begin{rem}
An element $[\alpha] \in \pi_1(M,N)$ and $1\leq i,j\leq n$ uniquely determine an element in $\Endm$ in the following way: We choose a good representative $\alpha$ such that $\alpha(1)$ is higher than $\alpha(0)$ if $\alpha(0)$ and $\alpha(1)$ belong to the same component of $N$,  $\alpha$ only intersects $\partial M$ at it's endpoints, and $\alpha$ does not intersect itself. Then we give a framing to $\alpha$ respecting $N$, that is, the framing at endpoints are given by the velocity vectors of $N$. We denote this framed oriented boundary arc as $\hat{\alpha}$. We choose the framing such that $h(\widetilde{\hat{\alpha}}) = 0$, then we obtain an element 
$F_{\hat{\alpha}_{i,j}}\in\Endm$.
Suppose we choose a different good representative $\alpha^{'}$. We have $[\hat{\alpha}] = [\hat{\alpha^{'}}]\in \pi_1(M,N)$ and $h(\widetilde{\hat{\alpha^{'}}}) = h(\widetilde{\hat{\alpha}}) = 0$. Then $F_{\hat{\alpha}_{i,j}}
=F_{\hat{\alpha{'}}_{i,j}}
\in\Endm$ because of   Lemma \ref{llll7.25}.

We use $S^{[\alpha],m}_{i,j}$ to denote $F_{\hat{\alpha}_{i,j}}$, and use $S^{[\alpha],m}$ to denote an $n$ by $n$  matrix in $\Endm$ such that $(S^{[\alpha],m})_{i,j} = S^{[\alpha],m}_{i,j},1\leq i,j\leq n$.

Then for any stated framed oriented  boundary arc $\alpha_{i,j}$, we have $$F_{\alpha_{i,j}} =
d_n^{h(\tilde{\alpha})} S^{[\alpha],m}_{i,j}\in \Endm.$$
\end{rem}

Here we recall the definition for matrix $A$, $$A_{i,j} = (-1)^{i+1}\delta_{\overline{i},j}, 1\leq i,j\leq n.$$
And we have det$ A =1$ and $A^2 = d_n I$.

For any complex number $k\in\mathbb{C}$, we will use $k$ to denote $kId_{\SM}\in\End$ for simplicity.

\begin{proposition}\label{prop7.28}
(a) For any two elements $[\alpha],[\beta]\in \pi_1(M,N)$, if $[\beta][\alpha]$ makes sense, then 
$A S^{[\beta*\alpha],m} = A S^{[\beta],m} A S^{[\alpha],m}$.

(b) For any $[\eta]\in \pi_1(M,N)$, we have 
det$(S^{[\eta],m}) = 1$. Especially det$(A S^{[\eta],m}) = 1$.

(c) Suppose $[o]\in \pi_1(M,N)$ is the identity morphism for an object, then $S^{[o],m} = d_n A$.
Especially $A S^{[o],m} = I$.
\end{proposition}
\begin{proof}

(a) We have 
$$(S^{[\beta],m} A S^{[\alpha],m})_{i,j} = \sum_{1\leq k\leq n}(-1)^{k+1} S^{[\beta],m}_{i,k} S^{[\alpha],m}_{\bar{k},j} = S_{i,j}^{[\beta*\alpha],m} = (S^{[\beta*\alpha],m})_{i,j}$$
where the second equality is from Lemma \ref{mmm7.23}.
Thus we have $A S^{[\beta*\alpha],m} = A S^{[\beta],m} A S^{[\alpha],m}$.

(b)It is from Lemma \ref{mmmm7.25}.

(c)For $1\leq i,j\leq n$, we have 
$$S^{[o],m}_{i,j} = (\raisebox{-.20in}{
\begin{tikzpicture}
\tikzset{->-/.style=
{decoration={markings,mark=at position #1 with
{\arrow{latex}}},postaction={decorate}}}
\filldraw[draw=white,fill=gray!20] (-0.7,-0.7) rectangle (0,0.7);
\draw [line width =1.5pt,decoration={markings, mark=at position 1 with {\arrow{>}}},postaction={decorate}](0,-0.7)--(0,0.7);
\draw [color = black, line width =1pt] (0 ,0.3) arc (90:270:0.5 and 0.3);
\node [right]  at(0,0.3) {$i$};
\node [right] at(0,-0.3){$j$};
\draw [line width =1pt,decoration={markings, mark=at position 0.5 with {\arrow{<}}},postaction={decorate}](-0.5,0.02)--(-0.5,-0.02);
\end{tikzpicture}})^{(m)} = \delta_{\bar{i},j}(c_i^{-1})^{m}= d_n A_{i,j}.$$
Thus $ S^{[o],m} = d_n A$. 

\end{proof}

Suppose $M$ is connected, and the components of $N$ consist of $e_0,e_1,\dots,e_{k-1}$ where $k$ is a positive integer.
If $k\geq 2$,
for each $1\leq t\leq k-1$, let $\alpha_t$ be a path connecting $e_0$ and $e_t$ with $\alpha_t(0)\in e_0$ and $\alpha_t(1)\in e_t$. We use $[o]$ to denote the identity element in $\pi_1(M,e_0)$. 
From Section \ref{subb5}, we know $S_n(M,N,1)$ is a commutative algebra generated by $$S^{[\alpha]}_{i,j},\;[\alpha]\in \pi_1(M,e_0)
\cup \{[\alpha_1],\dots,[\alpha_{k-1}]\}, 1\leq i,j\leq n,$$ subject to the relations
\begin{equation}\label{eqqq}
\begin{split}
det(S^{[\alpha]}) = 1&\text{ for all } [\alpha] \in \pi_1(M,e_0)
\cup \{[\alpha_1],\dots,[\alpha_{k-1}]\}, \;A S^{[o]} = I, \\
&AS^{[\beta]}A S^{[\eta]} =A S^{[\beta*\eta]}\text{ for all } [\beta],[\eta]\in\pi_1(M,e_0).
\end{split}
\end{equation}
Note that if $k=1$, the set $\{[\alpha_1],\dots,[\alpha_{k-1}]\}$ is empty.

Then we define an algebra homomorphism $\tilde{ \cF}:S_n(M,N,1)\rightarrow \Endm$ by defining $\tilde{\cF}$
 on the above generators, and then check all the relations in (\ref{eqqq}).

\begin{theorem}\label{thhh7.29}
There exists an algebra homomorphism 
\begin{align*}
\tilde{ \cF}:S_n(M,N,1)&\rightarrow \Endm\\
S^{[\alpha]}_{i,j}&\mapsto S^{[\alpha],m}_{i,j}
\end{align*}
where $[\alpha]\in \pi_1(M,e_0)
\cup \{[\alpha_1],\dots,[\alpha_{k-1}]\}, 1\leq i,j\leq n$.
\end{theorem}
\begin{proof}
We know $\Endm$ is a commutative algebra. And Proposition
\ref{prop7.28} shows $\tilde{\cF}$ preserves all the relations in (\ref{eqqq}). Thus $\tilde{\cF}$ is a well-defined algebra homomorphism.
\end{proof}
Note that we have $\tilde{\cF}(S^{[\alpha]}) = S^{[\alpha],m}$ for all $[\alpha]\in \pi_1(M,e_0)
\cup \{[\alpha_1],\dots,[\alpha_{k-1}]\}$. We will show this is true for any $[\alpha]\in \pi_1(M,N)$.

The construction for $\tilde{\cF}$ in Theorem \ref{thhh7.29} depends on the choice of the generators for algebra $S_n(M,N,1)$. We will show $\tilde{\cF}$ is independent of the choice of these generators. Actually, we will show a stronger result in the following Theorem. 

\begin{theorem}\label{t888}
For any stated $n$-web $l$ consisting of stated framed oriented boundary arcs, we have $\tilde{\cF}(l) = F_{l}$. Especially $\tilde{\cF}$ is independent of the choice of the generators for algebra $S_n(M,N,1)$, and $\tilde{\cF}$ is surjective.
\end{theorem}
\begin{proof}
First we show for any $[\alpha]\in \pi_1(M,N)$ and $1\leq i,j\leq n$, we have $\tilde{\cF}(S^{[\alpha]}_{i,j}) = S^{[\alpha],m}_{i,j}$, that is, to show $\tilde{\cF}(S^{[\alpha]}) = S^{[\alpha],m}$. 
 This is clearly true if $k=1$, since $\pi_1(M,N) = \pi_1(M,e_0)$ when $k=1$.

Suppose $k\geq 2$. 
For any $1\leq t\leq k-1$, we have 
\begin{align*}
I &= \tilde{\cF}(AS^{[\alpha_t]} AS^{[\alpha_t^{-1}]}) = \tilde{\cF}(AS^{[\alpha_t^{-1}]} AS^{[\alpha_t]})
=A\tilde{\cF}(S^{[\alpha_t]}) A \tilde{\cF}(S^{[\alpha_t^{-1}]}) =A \tilde{\cF}(S^{[\alpha_t^{-1}]}) A\tilde{\cF}( S^{[\alpha_t]})\\
I &= AS^{[\alpha_t],m} AS^{[\alpha_t^{-1}],m} = AS^{[\alpha_t^{-1}],m} AS^{[\alpha_t],m}
=A\tilde{\cF}(S^{[\alpha_t]}) A S^{[\alpha_t^{-1}],m}=A S^{[\alpha_t^{-1}],m} A\tilde{\cF}( S^{[\alpha_t]}).
 \end{align*}
Thus we get $AS^{[\alpha_t^{-1}],m} = A \tilde{\cF}(S^{[\alpha_t^{-1}]})$, which implies 
$S^{[\alpha_t^{-1}],m} =  \tilde{\cF}(S^{[\alpha_t^{-1}]})$.

For any $[\alpha]\in \pi_1(M,N)$,  suppose $\alpha(0)\in e_{u}, \alpha(1)\in e_{v}$. 
If $u=v=0$,  it is obvious that $\tilde{\cF}(S^{[\alpha]}) = S^{[\alpha],m}$.
If $u\neq 0, v\neq 0,$, then we have $[\alpha_v^{-1}*\alpha *\alpha_u]\in \pi_1(M,e_0)$. Then we have 
\begin{align*}
&A\tilde{\cF}(S^{[\alpha]}) = \tilde{\cF}(AS^{[\alpha]})= \tilde{\cF}(AS^{[\alpha_v]} AS^{[\alpha_v^{-1}*\alpha *\alpha_u]} AS^{[\alpha_u^{-1}]})\\
&=A \tilde{\cF}(S^{[\alpha_v]}) A \tilde{\cF}(S^{[\alpha_v^{-1}*\alpha *\alpha_u]}) A \tilde{\cF}(S^{[\alpha_u^{-1}]})\\
&=AS^{[\alpha_v],m} AS^{[\alpha_v^{-1}*\alpha *\alpha_u],m} AS^{[\alpha_u^{-1}],m} =  AS^{[\alpha],m}.
\end{align*}
Thus we have $\hat{\cF}(S^{[\alpha]}) = S^{[\alpha],m}$. Similar we can prove $\tilde{\cF}(S^{[\alpha]}) = S^{[\alpha],m}$ when $u=0,v\neq 0$ or $u\neq 0,v= 0$.

For any stated framed oriented boundary arc $\alpha_{i,j}$, we have $\alpha_{i,j} =d_n^{h(\tilde{\alpha})} S^{[\alpha]}_{i,j}$ in $S_n(M,N,1)$. Then 
$$\tilde{\cF}(\alpha_{i,j}) =d_n^{h(\tilde{\alpha})} \tilde{\cF}(S^{[\alpha]}_{i,j})
= d_n^{h(\tilde{\alpha})} S^{[\alpha],m}_{i,j} = F_{\alpha_{i,j}}.$$

Let $l=\cup_{1\leq i\leq k }   \alpha_i$ where each $\alpha_i$ is a stated framed oriented boundary arc. Then 
$$\tilde{\cF}(l) = \tilde{\cF}(\alpha_1\cdots\alpha_k)
= \tilde{\cF}(\alpha_1)\cdots\tilde{\cF}(\alpha_k)
= F_{\alpha_1}\cdots F_{\alpha_k} = F_{\cup_{1\leq i\leq k}    \alpha_i}= F_l.$$

The surjectivity of $\tilde{\cF}$ is obvious.
\end{proof}

\def \SMO {S_n(M,N,1)}
Since $\tilde{\cF}:S_n(M,N,1)\rightarrow \Endm$ is an algebra homomorphism, this gives an action of 
$S_n(M,N,1)$ on $\SM$, defined by for any $\alpha\in\SMO,\beta\in\SM$, $\alpha\cdot\beta = \tilde{\cF}(\alpha)(\beta)$.

\subsection{Construction for Frobenius map}\label{newF}
In this section, we will define the Frobenius map $\cF:\SMO\rightarrow \SM.$
%

\begin{theorem}\label{thmm7.31}
For any marked three manifold $(M,N)$,
there exists a unique linear map $\cF:S_n(M,N,1)\rightarrow S_n(M,N,v)$ such that $\cF(l) = l^{(m)}$
for any stated $n$-web $l$ consisting of stated framed oriented boundary arcs.
\end{theorem}
\begin{proof}
Since we assume every component of $M$ contains at least one marking, $\SM$ is a linear span by stated $n$-webs  consisting of stated framed oriented boundary arcs, thus the uniqueness is obvious.

To construct $\cF$, we 
 can suppose $M$ is connected. We use $T$ to denote the linear map from $\Endm$ to $\SM$, defined by 
$T(f) = f(\emptyset)$ where $f\in\Endm$ and $\emptyset$
represents the empty stated $n$-web. 
Define $\cF = T\circ\tilde{\cF}$, then $\cF$ is a linear map from $\SMO$ to $\SM$ such that 
$\cF(l) = l^{(m)}$
for any stated $n$-web $l$ consisting of stated framed oriented boundary arcs.

\end{proof}

\begin{rem}
Let $(M,N)$ be a marked three manifold, and $x$ be an element in $ \SM$. Then the proof for Theorem 
\ref{thmm7.31} shows there exists a unique linear map $\cF_x :\SMO\rightarrow \SM$ such that 
$\cF_x(l) = \tF(l)(x)$
for any $l\in\SMO$. Then $\cF = \cF_{\emptyset}$
where $\emptyset$ is the emptyset stated $n$-web.
\end{rem}

\def \tF {\tilde{\cF}}
\def \hF {\hat{\cF}}

For any two elements $\alpha,\beta\in \SMO$, we know $\tF(\alpha)\in \Endm$ and 
$\cF(\beta)\in \SM$, then $\tF(\alpha)(\cF(\beta))\in\SM.$ It easy to check we have 
$\tF(\alpha)(\cF(\beta))=\tF(\beta)(\cF(\alpha))\in\SM.$  Actually if we suppose 
$\alpha = \sum_{1\leq i\leq u}k_i\alpha_i,\beta =\sum_{1\leq j\leq v}t_j \beta_j$
where $\alpha_i,\beta_j$ consist of stated framed oriented boundary arcs and $\alpha_i\cap\beta_j=\emptyset$, then we have 
\begin{equation}\label{mul}
\tF(\alpha)(\cF(\beta))=\tF(\beta)(\cF(\alpha))=\sum_{1\leq i\leq u,1\leq j\leq v}
k_i t_j (\alpha_i)^{(m)}\cup(\beta_j)^{(m)} \in\SM.
\end{equation}

\def \SMm {S_n(M,N,v)^{(m)}}

We use $\SMm$ to denote $\text{Im}\cF$, then there is a surjective map $$\hF:\SMO\rightarrow\SMm$$
induced by $\cF$.
\begin{theorem}\label{structure}
For any marked three manifold $(M,N)$, there is a commutative algebra structure for $\SMm$, which makes 
$\hF:\SMO\rightarrow \SMm$ a surjective algebra homomorphism.
\end{theorem}
\begin{proof}
For any two elements $x,y\in \SMm$, suppose $x=\cF(\alpha), y=\cF(\beta)$ where $\alpha,\beta\in \SMO$, then define $xy = \tF(\alpha)(\cF(\beta))\in\SM$. We need to check this is a well-defined multiplication, that is, to check $xy$ is independent of the choice of $\alpha,\beta$. Suppose we also have 
$x=\cF(\alpha_1), y=\cF(\beta_1)$ where $\alpha_1,\beta_1\in \SMO$.  Then 
$$\tF(\alpha)(\cF(\beta)) = \tF(\alpha)(\cF(\beta_1)) = \tF(\beta_1)(\cF(\alpha)) =
\tF(\beta_1)(\cF(\alpha_1)) = \tF(\alpha_1)(\cF(\beta_1)),$$
which shows the well-defineness of this multiplication. We also have 
$$xy = \tF(\alpha)(\cF(\beta))  = \tF(\beta)(\cF(\alpha))= yx,$$ which shows this multiplication is commutative.
From equation \eqref{mul}, we can easily show $\hF$ is a surjective algebra homomorphism.
\end{proof}

\subsection{Compatibility between Frobenius homomorphism and splitting map}

\begin{theorem}\label{com}
Suppose $D$ is a disk properly embedded into a marked 3-manifold $(M, N )$ and $D$  contains an oriented open interval $\beta$. Let $(M^{'},N^{'})$ be the result of splitting $(M, N )$ along $(D, \beta)$.
Then we have the following commutative diagram:
$$\begin{tikzcd}
S_n(M,N,1)  \arrow[r, "\Theta"]
\arrow[d, "\cF"]  
&  S_n(M^{'},N^{'},1)  \arrow[d, "\cF"] \\
 S_n(M,N,v)  \arrow[r, "\Theta"] 
&  S_n(M^{'},N^{'},v)\\
\end{tikzcd}.$$
\end{theorem}
\begin{proof}
We just need to check $\cF( \Theta(\alpha)) = \Theta(\cF(\alpha))$ for any stated $n$-web $\alpha$ consisting of stated
 framed oriented boundary arcs, since $S_n(M,N,1)$ is linearly spanned by all these $n$-webs. And Corollary \ref{corr7.16} can prove this. 

\end{proof}

Theorem \ref{com} also shows the splitting map restricts to a map  
$$\Theta|_{S_n(M,N,v)^{(m)}}: S_n(M,N,v)^{(m)} \rightarrow S_n(M^{'},N^{'},v)^{(m)},$$
and $\Theta|_{S_n(M,N,v)^{(m)}}\circ\hat{\cF} = \hat{\cF}\circ \Theta$.
It is also easy to show $\Theta|_{S_n(M,N,v)^{(m)}}$ is an algebra homomorphism. 

\subsection{Center of $S_n(\Sigma,v)$ and
injectivity of Frobenius homomorphism}

When $(M,N)$ is the thickening of an essentially bordered pb surface $\Sigma$, obviously we have $S_n(\Sigma,v)^{(m)}$ is a subalgebra of $S_n(\Sigma,v)$.
The algebra structure for $S_n(\Sigma,v)^{(m)}$ inherited from $S_n(\Sigma,v)$ is the same with the algebra structure for $S_n(\Sigma,v)^{(m)}$ defined in Theorem \ref{structure}.
 Then $\cF :S_n(\Sigma,1)\rightarrow S_n(\Sigma,v)$ becomes an algebra homomorphism.

\begin{theorem}\label{Fpb}
Let $\Sigma$ be an essentially bordered pb surface, there exists a unique algebra homomorphism $\cF :S_n(\Sigma,1)\rightarrow S_n(\Sigma,v)$ such that $Im\cF$ lives in the center of $S_n(\Sigma,v)$ and 
$\cF(\alpha) = \alpha^{(m)}$ for any stated framed oriented boundary arc $\alpha$.
\end{theorem}
\begin{proof}
Lemma \ref{lmmm7.10},
Theorem \ref{thmm7.31}.
\end{proof}

Then we are trying to show $\cF$ in Theorem \ref{Fpb} is an embedding. 

\begin{lemma}\label{w}\cite{PW}
There exists a Hopf algebra homomorphism:
\begin{align*}
F_n : O(SLn) &\rightarrow O_q(SLn)\\
x_{i,j}&\mapsto (u_{i,j})^{m}.
\end{align*}
\end{lemma}

Define the monoid
\begin{equation*}\label{eq.Gamma}
\Gamma = \Mat_n(\BN)/ (\Id).
\end{equation*}
Here $\Mat_n(\BN)=\BN^{n\times n}$ is an additive monoid, and $(\Id)$ is the submonoid generated by the identity matrix. Two matrices $m,m'\in \Mat_n(\BN)$ determine the same element in $\Gamma$ if and only if $m-m'= k \Id$ for $k\in \BZ$. Each $m\in \Gamma$ has a unique lift $\hat m\in \Mat_n(\BN)$ 
 such that $\min_{i} \hat m_{ii}=0$. 

\begin{proposition}(\cite{gavarini2007pbw,leY}) \label{r.basisOq}
For any linear order $d_{ord}$ 
on  $\mathbb{J}^2$, the set
\begin{equation}\label{eq.Bord}
B^{d_{ord}}:=\{b(m):= \prod_{(i,j)\in \mathbb{J}^2} (u_{i,j})^{\hat m _{ij}} \mid m \in \Gamma = \Mat_n(\BN)/(\Id)\},
\end{equation}
where the product is taken with respect to the order $d_{ord}$, is a basis of $O_q(SLn)$. 
\end{proposition}

\def\BR{\mathbb R}

\begin{lemma}
$F_n$ in Lemma \ref{w} is an embedding.
\end{lemma}

\begin{proof}
Using Proposition \ref{r.basisOq}, it is easy to check $F_n$ maps  the basis of $O(SLn)$ injectively to the basis of $O_q(SLn)$.
\end{proof}

\begin{lemma}\label{lll7.37}
$\cF:S_n(\fB,1)\rightarrow S_n(\fB,v)$ is injective.
\end{lemma}
\begin{proof}

For any $i,j\in\mathbb{J}$, it is easy to check 
$$\cF(f_{big}(x_{i,j})) = \cF(b_{i,j}) = (b_{i,j})^{(m)},\;
  f_{big}(F_n(x_{i,j})) = f_{big}((u_{i,j})^{m}) = (b_{i,j})^{m} =  (b_{i,j})^{(m)} .$$
Then we have $\cF\circ f_{big} = f_{big}\circ F_n$ since all the maps involved are algebra homomorphims.
 Thus $\cF$ is injective because $f_{big}$ is an ismorphism and $F_n$ is injective.
\end{proof}

The proof in the Lemma \ref{lll7.37} actually shows $\cF:S_n(\fB,1)\rightarrow S_n(\fB,v)$ is an injective Hopf algebra homomorphism.

Recall that $\fT$ denotes the standard ideal triangle.

We know there is a linear isomorphism $QF:S_n(\fB,v)\otimes S_n(\fB,v)\rightarrow S_n(\fT,v)$, see Example 7.9 in \cite{le2021stated}. Note that $QF$ is not an algebra homomorphism unless $v=1$.

\begin{equation}\label{QFF}
\begin{tikzcd}
\raisebox{-.30in}{
\begin{tikzpicture}
\tikzset{->-/.style=
{decoration={markings,mark=at position #1 with
{\arrow{latex}}},postaction={decorate}}}
%
\draw [line width =1pt,decoration={markings, mark=at position 0.8 with {\arrow{>}}},postaction={decorate},shift ={(0.2 ,0)}] (1.5,0)--(0,2);
\draw [line width =1pt,decoration={markings, mark=at position 0.8 with {\arrow{>}}},postaction={decorate}] (0,2)--(-1.5,0);
\draw [line width =1pt,decoration={markings, mark=at position 0.8 with {\arrow{>}}},postaction={decorate},shift ={(1 ,0.6)}] (1.5,0)--(0,2);
\draw [line width =1pt,decoration={markings, mark=at position 0.8 with {\arrow{>}}},postaction={decorate},shift ={(-0.8 ,0.6)}] (0,2)--(-1.5,0);
\draw [color = blue, line width =2pt,shift ={(0.2 ,0)}] (0.75,1)--(1.55,1.6);
\draw [color = blue, line width =2pt] (-0.75,1)--(-1.55,1.6);
\end{tikzpicture}}  \arrow[r, "QF"]
& \raisebox{-.30in}{
\begin{tikzpicture}
\tikzset{->-/.style=
{decoration={markings,mark=at position #1 with
{\arrow{latex}}},postaction={decorate}}}
%
\draw [line width =1pt,decoration={markings, mark=at position 0.8 with {\arrow{>}}},postaction={decorate}] (1.5,0)--(0,2);
\draw [line width =1pt,decoration={markings, mark=at position 0.8 with {\arrow{>}}},postaction={decorate}] (0,2)--(-1.5,0);
\draw [line width =1pt,decoration={markings, mark=at position 0.8 with {\arrow{>}}},postaction={decorate}] (1.5,0)--(-1.5,0);
\filldraw[draw=black!80,fill=white!20] (0 ,2) circle (0.1);
\filldraw[draw=black!80,fill=white!20] (-1.5 ,0) circle (0.1);
\filldraw[draw=black!80,fill=white!20] (1.5 ,0) circle (0.1);
\draw [line width =1pt,decoration={markings, mark=at position 0.8 with {\arrow{>}}},postaction={decorate},shift ={(0.8 ,0.6)}] (1.5,0)--(0,2);
\draw [line width =1pt,decoration={markings, mark=at position 0.8 with {\arrow{>}}},postaction={decorate},shift ={(-0.8 ,0.6)}] (0,2)--(-1.5,0);
\draw [color = blue, line width =2pt] (0.75,1)--(1.55,1.6);
\draw [color = blue, line width =2pt] (-0.75,1)--(-1.55,1.6);
\draw [color = blue, line width =2pt] (0.75 ,1) arc (137:180:1 and 1.5);
\draw [color = blue, line width =2pt] (-0.49 ,0) arc (0:43:1 and 1.5);
\end{tikzpicture}} \\
\end{tikzcd}.
\end{equation}

\begin{lemma}\label{lll7.38}
$\cF:S_n(\fT,1)\rightarrow S_n(\fT,v)$ is injective.
\end{lemma}
\begin{proof}
Look at the following diagram:
\begin{equation}\label{comm}
\begin{tikzcd}
S_n(\fB,1)\otimes S_n(\fB,1)  \arrow[r, "\cF\otimes \cF"]
\arrow[d, "QF"]  
& S_n(\fB,v)\otimes S_n(\fB,v)  \arrow[d, "QF"] \\
 S_n(\fT,1)  \arrow[r, "\cF"] 
&  S_n(\fT,v)\\
\end{tikzcd}.
\end{equation}
As a vector space, $S_n(\fB,1)$ is generated by parallel  stated arcs  that are all  parallel to the core of $\fB$
(these  arcs may have different orientations and states). $QF$ acts on these arcs by extending them to the bottom edge, see equation (\ref{QFF}). $\cF$ acts on these arcs by taking $m$ parallel copies of each arc. Clearly these two actions commute with each other. Thus the diagram in (\ref{comm}) is commutative.
Then $\cF:S_n(\fT,1)\rightarrow S_n(\fT,v)$ is injective since $QF$ is isomorphism and $\cF:S_n(\fB,1)\rightarrow S_n(\fB,v)$ is injective, see Lemma
\ref{lll7.37}.
\end{proof}

\begin{equation}\label{888}
\begin{tikzcd}
S_n(\Sigma,1) \arrow[r, "\Theta",tail]\arrow[d,dashrightarrow,"F_{\cE}"]
 & \otimes_{\fT\in tri(\cE)} S_n(\fT,1)
\arrow[r,"T_{\cE}" ]\arrow[d, "\otimes_{\fT\in tri(\cE)}\cF"]   &( \otimes_{\fT\in tri(\cE)} S_n(\fT,v))\otimes (\otimes_{e\in Int(\cE)}S_n(\fB,1))
 \arrow[d, "(\otimes_{\fT\in tri(\cE)}\cF)\otimes (\otimes_{e\in Int(\cE)} \cF)"] \\
S_n(\Sigma,v) \arrow[r, "\Theta",tail] & \otimes_{\fT\in tri(\cE)} S_n(\fT,v)
\arrow[r,"T_{\cE}" ] &( \otimes_{\fT\in tri(\cE)} S_n(\fT,v))\otimes (\otimes_{e\in Int(\cE)}S_n(\fB,v))
\end{tikzcd}
\end{equation}
Please refer to Lemma \ref{exact} for the definition of $T_{\cE}$.

\begin{theorem}($n=2$ \cite{korinman2019classical}, $n=3$ \cite{higgins})\label{Injec}
Let $\Sigma$ be an essentially bordered pb surface, and $\cE$ be an ideal triangulation of $\Sigma$. Then we have 

(a) the right square in (\ref{888}) is commutative,

(b) there exists a unique algebra homomorphism $F_{\cE}:S_n(\Sigma,1)\rightarrow S_n(\Sigma,v)$ such that the left square is also commutative,

(c) $F_{\cE} = \cF$, especially we have $\cF$ is injective and $F_{\cE}$ is independent of the triangulation $\cE$.

\end{theorem}

\begin{proof}
(a) Theorem \ref{com}.

(b) From Lemma \ref{exact}, we know the two rows are exact. Then there exists a unique $F_{\cE}$ such that the left square is also commutative because the right square is commuative and the two rows are exact.

(c) From Theorem \ref{com}, we know $\cF:S_n(\Sigma,1)\rightarrow S_n(\Sigma,v)$ also makes the left square commutative. Then from the uniqueness in (b), we know $F_{\cE}=\cF$, which implies $F_{\cE}$ is  independent of $\cE$. Since both $\Theta$ and 
$\otimes_{\fT\in tri(\cE)}\cF$ are injective (Proposition \ref{inj}, Lemma \ref{lll7.38}), then $F_{\cE}=\cF$ is also injective. 
\end{proof}

\begin{corollary}
Let $\Sigma$ be an essentially bordered pb surface. We have
$$\hat{\cF}:S_n(\Sigma,1)\rightarrow S_n(\Sigma,v)^{(m)}$$
is an isomorphism.
\end{corollary}
\begin{proof}
From Theorem \ref{Injec}, we know $\hat{\cF}$ is injective. We also have $\hat{\cF}$ is  surjective. 
\end{proof}

Let $\Sigma$ be a pb surface.
For each interior puncture $p$, we have two peripheral loops with vertical framing going around $p$, which are denoted as $a_p,b_p$. Clearly both $a_p$ and $b_p$ live in the center of $S_n(\Sigma,v)$.

\begin{definition}
Suppose $\Sigma$ is an essentially bordered pb surface. Let
$Z_n(\Sigma)$ be the subalgebra of $S_n(\Sigma,v)$ generated by $S_n(\Sigma,v)^{(m)}$ and $a_p,b_p$ for all inner punctures $p$. 
\end{definition}

\begin{corollary}
Suppose $\Sigma$ is an essentially bordered pb surface. Then
$Z_n(\Sigma)$ lives in the center of $S_n(\Sigma,v)$. And $Z_n(\Sigma)$ is a finitely generated commutative algebra.
\end{corollary}
\begin{proof}
Note that $S_n(\Sigma,v)^{(m)}= Im\hat{\cF} = Im \cF$.  Then
from Theorem \ref{Fpb}, we know $Z_n(\Sigma)$ lives in the center of $S_n(\Sigma,v)$. 
Since $S_n(\Sigma,1)$ is finitely generated as an algebra and $\Sigma$ has finitely many interior punctures, then  
$Z_n(\Sigma)$ is a finitely generated commutative algebra.
\end{proof}

\subsection{Classical shadow}
\begin{definition}(\cite{JS})
A commutative algebra is called affine, if it does not contain nonzero nilpotents and it is finitely generated as an algebra.

\end{definition}

Recall that,
for any commutative algebra $A$, we use MaxSpec$(A)$ to denote the set of maximal ideals of $A$. Then MaxSpec$(A)$ is an affine algebraic variety and $A$ is the coordinate ring of MaxSpec$(A)$, if $A$ is affine \cite{JS}.

\begin{proposition}\label{fp}
Suppose $\Sigma$ is an essentially bordered pb surface. Then both $S_n(\Sigma,v)^{(m)}$ and
$Z_n(\Sigma)$ are  affine algebras.
\end{proposition}

\begin{proof}
We already know both $S_n(\Sigma,v)^{(m)}$ and
$Z_n(\Sigma)$ are finitely generated commutative algebra. Since $S_n(\Sigma,v)$ is a domain \cite{leY}, then both $S_n(\Sigma,v)^{(m)}$ and
$Z(\Sigma)$  have no nonzero nilpotents.

\end{proof}

\begin{rem}\label{rem7.48}
Let $A$ be an algebra, and $Z$ be an affine subalgebra of the center of $A$. We use $\text{Irrep}_{A}$ to denote the set of finite dimensional irreducible  representations considered up to isomorphism (that is, two irreducible representations are considered the same if they are isomorphic).
We can omit the subscript for $\text{Irrep}_{A}$, when there is no confusion with $A$.
Then there is a map $\mathcal{X} :\text{Irrep}\rightarrow \text{MaxSpec}(Z)$ defined as following:
Let $\rho:A\rightarrow End(V)$ be a finite dimensional irreducible representation of $A$.
Since $Z$ is contained in the center of $A$, for every $x\in Z$ there exists a complex number $r_{\rho}(x)$ such that $\rho(x)= r_{\rho}(x)Id_V$. We get an algebra homomorphism 
$r_{\rho}:Z\rightarrow \mathbb{C}$. Then the irreducible representation $\rho$ uniquely determines a
point Ker$ (r_{\rho})$ in MaxSpec($Z$). We define $\mathcal{X}(\rho) = $Ker$ (r_{\rho})$.
\end{rem}

From  Proposition \ref{fp}, we know $Z_n(\Sigma)$ is an affine subalgebra of the center of $S_n(\Sigma,v)$, and
 MaxSpec($Z_n(\Sigma)$) is an affine algebraic variety. Then there is a map 
$$\mathcal{X} :\text{Irrep}\rightarrow \text{MaxSpec}(Z_n(\Sigma)).$$

In next section, we will show the center of $S_n(\Sigma,v)$  is actually affine. And the corresponding map $\mathcal{X}$ is surjective and $\mathcal{X}$ is injective on the preimage of a Zariski open  dense subset. 



\subsection{Frobenius homomorphism for reduced stated $SL(n)$-skein algebra}
The reduced stated $SL(n)$-skein algebra $\overline{S}_n(\Sigma,v)$ is defined in subsection 7.1 in \cite{leY} for any pb surface $\Sigma$, $$\overline{S}_n(\Sigma,v) = S_n(\Sigma,v)/I^{\text{bad}},$$ where 
$I^{\text{bad}}$ is the two sided ideal of $S_n(\Sigma,v)$ generated by all bad arcs.
 We use $\overline{S}_n(\Sigma,v)^{(m)}$ to denote the image of $S_n(\Sigma,v)^{(m)}$ under the projection from $S_n(\Sigma,v)$ to $\overline{S}_n(\Sigma,v)$.
 
For any essentially bordered pb surface $\Sigma$,
clearly $\cF:S_n(\Sigma,1)\rightarrow S_n(\Sigma,v)$ sends the bad arc to $m$ parallel copies of the bad arc. Then 
$\cF$ induces an algebra homomorphism $\overline{\cF}:
\overline{S}_n(\Sigma,1) \rightarrow \overline{S}_n(\Sigma,v) $, and Im$\overline{\cF}
= \overline{S}_n(\Sigma,v)^{(m)}$ lives in the center of 
$\overline{S}_n(\Sigma,v)$.

\def\cB{\mathcal B}

\section{Unicity Theorem for stated $SL(n)$-skein algebras}\label{sss888}

For any algebra $A$, we will use $C(A)$ to denote it's center.

An algebra $A$ is called {\bf prime} if all $a,b\in A$ satisfy the following condition: if $arb = 0$ for  all $r\in A$ then $a = 0$ or $b = 0$. Clearly if $A$ is a domain, then $A$ is prime.
An algebra $A$ is  {\bf almost Azumaya} if there is a nonzero element $c\in C(A)$ such that $A_{c}$ is an Azumaya algebra, where
$A_c$ is the localization of $A$ by $c$ \cite{unicity,korinman2021unicity}.
An algebra $A$ is called {\bf affine almost Azumaya}
if $A$ satisfies the following conditions \cite{korinman2021unicity}: 

(1) $A$ is finitely generated as an algebra

 (2) A is prime,

 (3) A is finitely generated as a module over its center. 

\begin{theorem}(\cite{unicity})
If $A$ is a prime algebra that is finitely generated as a module over its  center then $A$ is almost Azumaya. Especially, affine almost Azumaya algebra is  almost Azumaya.
\end{theorem}

\begin{lemma}\label{lmlm8.2}
If $A$ is an affine almost Azumaya algebra, then $C(A)$ is an affine algebra.
\end{lemma}
\begin{proof}
From Lemma 2.1 in \cite{unicity}, we know (1) and (3) in the definition of   affine almost Azumaya algebra indicate 
$C(A)$ is a finitely generated algebra. And (2) indicates $C(A)$ is a domain.
\end{proof}

Thus MaxSpec$(C(A))$ is an irreducible affine algebraic variety if $A$ is an affine almost Azumaya algebra.  

Suppose $A$ is an algebra, and $M$ is a finitely generated $A$-module. We define the dimension of $M$ over $A$ to be the minimum number of generators of $M$ over $A$.

When $A$ is an affine almost Azumaya algebra,
we use $\widetilde{C(A)}$ to denote  the field of fractions of $C(A)$, and use $\widetilde{A}$ to denote the vector space $A\otimes_{C(A)} \widetilde{C(A)}$ over $\widetilde{C(A)}$. Then we define the rank of $A$ over $C(A)$ to be the dimension of $\widetilde{A}$ over $\widetilde{C(A)}$. Obviously we have the rank of $A$ over $C(A)$ is less than or equal to the dimension of $A$ over $C(A)$.

\begin{theorem}\label{tmmm8.3}(\cite{unicity,korinman2021unicity})
If $A$ is an affine almost Azumaya  algebra of rank $K$ and of dimension $r$ over its center $C(A)$, then:

(a) any irreducible representation of $A$ has dimension at most the square root of $K$; 

 (b) The map $ \mathcal{X} :\text{Irrep}\rightarrow  \text{MaxSpec}(C(A))$, defined in Remark \ref{rem7.48}, is surjective, and the cardinality of $\mathcal{X}^{-1}(v)$ is not more than $r$ for any $v\in \text{MaxSpec}(C(A))$;

(c) there exists a Zariski open dense subset $U \subset \text{MaxSpec}(C(A))$ such that for any two irreducible representations $V_1$, $V_2$ of $A$ with $\mathcal{X}(V_1) = \mathcal{X}(V_2)\in U$, then $V_1$ and $V_2$ are isomorphic and have  dimension the square root of $K$. Moreover any representation sending $C(A)$ to scalar operators and  whose induced character lies in $U$ is semi-simple.
\end{theorem}


Frohman, Kania-Bartoszynska, and L{\^e} proved $S_2(\Sigma,v)$ is an affine almost Azumaya algebra when $v$ is a root of unity and $\partial\Sigma =\emptyset$ \cite{unicity}. Korinman proved $S_2(\Sigma,v)$ and 
$\overline{S}_2(\Sigma,v)$ are both affine almost Azumaya  when $v$ is a root of unity of odd order  \cite{korinman2021unicity}.

\subsection{More on $S_n(M,N,1)$-module structure for $S_n(M,N,v)$ and 
Unicity Theorem for $SL(n)$}\label{suuu8.1}
In this subsection, we always assume every component of the marked three manifold contains at least one marking, 
 the surface is  essentially bordered pb surface, and $v$ is a root of unity of  order $m$ with $m$ and $2n$ being coprime with each other.
The main goal of this subsection is to show $S_n(\Sigma,v)$ is affine almost Azumaya.

From Subsection \ref{subb7.2}, we know $\SMO$ acts on $\SM$.
This module structure is defined by 
$l\cdot \alpha = l^{(m)}\cup \alpha$
where $l$ is a stated $n$-web consisting of stated framed oriented boundary arcs, $\alpha$ is any stated $n$-web in $S_n(M,N,v)$, and $l\cap\alpha =\emptyset$. 

Let $f:(M_1,N_1)\rightarrow (M_2,N_2)$ be an embedding for marked three manifolds. We know $f$ induces a linear map $f_{*}:S_n(M_1,N_1,v)\rightarrow S_n(M_2,N_2,v)$ and an algebra homomorphism $$f_*: S_n(M_1,N_1,1)\rightarrow S_n(M_2,N_2,1).$$ Then we  have the following Lemma.

\begin{lemma}\label{mmm8.4}
The above linear map $f_{*}:S_n(M_1,N_1,v)\rightarrow S_n(M_2,N_2,v)$ respects the module structures in a sense that 
$$f_*(x\cdot y) = f_*(x)\cdot f_*(y)$$
where $x\in S_n(M_1.N_1,1),y\in S_n(M_1,N_1,v)$.
\end{lemma}
\begin{proof}
Let 
$\alpha$ be a stated $n$-web consisting of stated framed oriented boundary arcs, $\beta$ be any stated $n$-web in $S_n(M,N,v)$, and $\alpha\cap\beta =\emptyset$. Then we have 
$$f_*(\alpha\cdot\beta) = f_*(\alpha^{(m)}\cup\beta
) = (f_*(\alpha))^{(m)}\cup f_*(\beta) = f_*(\alpha)\cdot f_*(\beta).$$
\end{proof}

\begin{lemma}\label{lmlm8.5}
Let $f:(M_1,N_1)\rightarrow (M_2,N_2)$ be an embedding for marked three manifolds. Suppose the linear map $f_*: S_n(M_1,N_1,v)\rightarrow S_n(M_2,N_2,v)$ is surjective, and
 $S_n(M_1,N_1,v)$ is finitely generated as an $S_n(M_1,N_1,1)$-module. Then 
 $S_n(M_2,N_2,v)$ is finitely generated as an $S_n(M_2,N_2,1)$-module, and the dimension of 
$S_n(M_2,N_2,v)$  as an $S_n(M_2,N_2,1)$-module is not more than the dimension of 
$S_n(M_1,N_1,v)$  as an $S_n(M_1,N_1,1)$-module.
\end{lemma}
\begin{proof}
Suppose $S_n(M_1,N_1,v)$ is  generated by $x_1,x_2,\dots, x_{t}$ as an $S_n(M_1,N_1,1)$-module.
From Lemma \ref{mmm8.4}, we get $S_n(M_2,N_2,v)$ is  generated by $$f_*(x_1),f_*(x_2),\dots,f_*( x_{t})$$ as an $S_n(M_2,N_2,1)$-module because  $f_*: S_n(M_1,N_1,v)\rightarrow S_n(M_2,N_2,v)$ is surjective.
\end{proof}

\begin{rem}\label{new}
For an essentially bordered pb surface $\Sigma$, we know $\hF:S_n(\Sigma,1)\rightarrow S_n(\Sigma,v)^{(m)}$ is an isomorphism, and $S_n(\Sigma,v)^{(m)}$ lives in the center of $S_n(\Sigma,v)$. Thus $S_n(\Sigma,v)$ has an
$S_n(\Sigma,v)^{(m)}$-module structure given by multiplication. $S_n(\Sigma,v)$ also has an $S_n(\Sigma,1)$-module structure.
 Then the identity map $Id: S_n(\Sigma,v)\rightarrow S_n(\Sigma,v)$ preserves the above two module structures, in a sense that, for any $x\in S_n(\Sigma,1),y\in S_n(\Sigma,v)$ we have $x\cdot y = \hF(x)y$.
\end{rem}

\begin{lemma}\label{lmlm8.6}
$S_n(\cB,v)$ is finitely generated as an $S_n(\cB,1)$-module. And the dimension of $S_n(\cB,v)$  as an $S_n(\cB,1)$-module is at most $m^{n^2} - (m-1)^{n} m^{n^2-n}$.
\end{lemma}

\begin{proof}
Let $d_{ord}$ be a linear order on set $\mathbb{J}^{2}$.
From Proposition \ref{r.basisOq} and Theorem \ref{t.Hopf},
we know the set
$$\{\prod_{(i,j)\in \mathbb{J}^2} (b_{i,j})^{\hat m _{ij}} \mid m \in \Gamma = \Mat_n(\BN)/(\Id)\}$$
is a basis for $S_n(\cB,v)$, where the product is taken with respect to the order $d_{ord}$, and 
$$\{\prod_{(i,j)\in \mathbb{J}^2} (b_{i,j})^{m \hat m _{ij}} \mid m \in \Gamma = \Mat_n(\BN)/(\Id)\}$$
is a basis for $S_n(\cB,v)^{(m)}$. Clearly $S_n(\cB,v)$ is generated by the set
\begin{equation}\label{q28}
\{\prod_{(i,j)\in \mathbb{J}^2} (b_{i,j})^{ m _{ij}} \mid m_{ij}\in\mathbb{Z},0\leq m_{i,j}\leq m-1,\exists\, i\in\mathbb{J}\text{ s.t. } m_{ii} = 0\}
\end{equation} 
as an $S_n(\cB,v)^{(m)}$-module, where the product is taken with respect to the order $d_{ord}$. And the cardinality of the set in (\ref{q28}) is $m^{n^2} - (m-1)^{n} m^{n^2-n}$. 

Then Remark \ref{new} indicates Lemma \ref{lmlm8.6}.

\end{proof}

Note that when $n>1$, the set in (\ref{q28}) is not independent. 
 We will use $<n,m>$ to denote the integer $m^{n^2} - (m-1)^{n} m^{n^2-n}$.

\begin{lemma}\label{lmlm8.7}
Suppose $\Sigma$ is the disjoint union of  $k$ bigons, $k\geq 1$. Then $S_n(\Sigma,v)$ is finitely generated as an $S_n(\Sigma,1)$-module. And the dimension of $S_n(\Sigma,v)$  as an $S_n(\Sigma,1)$-module is at most $(<n,m>)^{k}$.
\end{lemma}

\begin{proof}
Lemma \ref{lmlm8.6}.
\end{proof}

Suppose $\Sigma$ is an essentially bordered pb  surface. Let  $B = \{b_1,\dots,b_r \}$ be the collection of properly embedded
disjoint compact oriented arcs in $\Sigma.$ We say $B$ is {\bf saturated} if we have (1) after   cutting $\Sigma$ along $B$, every component of the cutting surface contains exactly one ideal point (2) $B$ is maximal under condition (1).

Let $U(b_1),\dots,U(b_r)$ be a collection of disjoint open tubular neighborhoods of $b_1,\dots,b_r$,  respectively. Each $U(b_i)$ is diffeomorphic with $b_i \times(0,1)$ (the diffeomorphism is orientation preserving)  and we require that $(\partial b_i) \times (0,1) \subset\partial\Sigma$. Set $U(B) = \cup_{1\leq i\leq r}U(b_i)$. Note that each $U(b_i)$ is naturally a bigon, then $U(B)$ is the disjoint union of bigons. From \ref{key} to \ref{keyth}, we will always assume $\Sigma$ is an essentially bordered pb surface.

\begin{theorem}(\cite{le2021stated})\label{key}
Assume  $B = \{b_1,\dots,b_r \}$ is a  saturated system of arcs of $\Sigma$.

 (1) We have $r = r(\Sigma) \triangleq \sharp (\partial \Sigma) - E(\Sigma)$, where $\sharp( \partial \Sigma)$ is the number of boundary components  of $\Sigma$ and $E(\Sigma)$ denotes the Euler characteristics of $\Sigma$.

 (2) The embedding $U(B)\rightarrow \Sigma$ with negative b-orderings for all boundary edges $b$ of $\Sigma$,  induces a linear isomorphism.
\end{theorem}

\begin{proposition}\label{finite}
We have 

(a) $S_n(\Sigma,v)$ is finitely generated as an
$S_n(\Sigma,v)^{(m)}$-module, and the dimension of $S_n(\Sigma,v)$ over $S_n(\Sigma,v)^{(m)}$ is at most $(<n,m>)^{r(\Sigma)}$.

(b) $\overline{S}_n(\Sigma,v)$ is finitely generated as an
$\overline{S}_n(\Sigma,v)^{(m)}$-module, and the dimension of $\overline{S}_n(\Sigma,v)$ over $\overline{S}_n(\Sigma,v)^{(m)}$ is at most $(<n,m>)^{r(\Sigma)}$.
\end{proposition}
\begin{proof}
From 
Lemmas \ref{lmlm8.5}, \ref{lmlm8.7}, and Theorem \ref{key}, and Remark \ref{new}, we can easily prove (a). Trivially (a) indicates (b).
\end{proof}

\begin{lemma}\label{lmlm8.10}
We have  $S_n(\Sigma,v)$
(respectively $\overline{S}_n(\Sigma,v)$) is finitely generated as a module over $C(S_n(\Sigma,v))$
(respectively $C(\overline{S}_n(\Sigma,v))$). And the dimension of $S_n(\Sigma,v)$
(respectively $\overline{S}_n(\Sigma,v)$) over $C(S_n(\Sigma,v))$
(respectively $C(\overline{S}_n(\Sigma,v))$)
 is not more than $(<n,m>)^{r(\Sigma)}$.
\end{lemma}
\begin{proof}
Since $S_n(\Sigma,v)^{(m)}\subset C(S_n(\Sigma,v))$ and 
$\overline{S}_n(\Sigma,v)^{(m)}\subset C(\overline{S}_n(\Sigma,v))$, then Proposition \ref{finite} indicates the Lemma.
\end{proof}


\begin{theorem}\label{tmmm8.11}
 We have  $S_n(\Sigma,v)$ is affine almost Azumaya. Thus it is also almost Azumaya.
\end{theorem}
\begin{proof}
Since $\Sigma$ be an essentially bordered pb surface, then $S_n(\Sigma,v)$ is a domain and is a finitely generated algebra, Theorem 6.1 in \cite{leY}. Combine with Lemma \ref{lmlm8.10}, we get $S_n(\Sigma,v)$ is affine almost Azumaya.
\end{proof}

\begin{corollary}\label{ooo}
 We have  $C(S_n(\Sigma,v))$ is affine.
\end{corollary}
\begin{proof}
Lemma \ref{lmlm8.2}, and Theorem \ref{tmmm8.11}.
\end{proof}

Then we are ready to state the Unicity Theorem for stated $SL(n)$-skein algebra when the surface is an essentailly bordered pb surface.

\begin{theorem}\label{keyth}
 Suppose the rank of $S_n(\Sigma,v)$ over $C(S_n(\Sigma,v))$ is $K$ and the dimension of $S_n(\Sigma,v)$ over $C(S_n(\Sigma,v))$ is $r$, then we have 

(a) $K\leq r\leq (<n,m>)^{r(\Sigma)}$;

(b) any irreducible representation of $S_n(\Sigma,v)$ has dimension at most the square root of $K$;

 (c) the map $ \mathcal{X} :\text{Irrep}\rightarrow  \text{MaxSpec}(C(S_n(\Sigma,v)))$, defined in Remark \ref{rem7.48}, is surjective, and the cardinality of $\mathcal{X}^{-1}(v)$ is not more than $r$ for any $v\in \text{MaxSpec}(C(S_n(\Sigma,v)))$;

(d) there exists a Zariski open dense subset $U \subset \text{MaxSpec}(C(S_n(\Sigma,v)))$ such that for any two irreducible representations $V_1$, $V_2$ of $S_n(\Sigma,v)$ with $\mathcal{X}(V_1) = \mathcal{X}(V_2)\in U$, then $V_1$ and $V_2$ are isomorphic and have  dimension the square root of $K$. Moreover any representation sending $C(S_n(\Sigma,v))$ to scalar operators and  whose induced character lies in $U$ is semi-simple.

\end{theorem}

\begin{proof}
Lemma \ref{lmlm8.10}, Theorems \ref{tmmm8.3} and \ref{tmmm8.11}.
\end{proof}

The rank $K$ in Theorem \ref{keyth} is very important to understand the representation theory for stated $SL(n)$-skein algebra. Unfortunately, it is very hard to calculate it. Frohman, Kania-Bartoszynska, L{\^e} precisely calculated this rank $K$
 in \cite{frohman2021dimension} when $n=2$ and $\partial \Sigma =\emptyset$.

For $\overline{S}_n(\Sigma,v)$, now we are not clear whether $\overline{S}_n(\Sigma,v)$ is a domain (or prime) for general essentially  bordered pb surface. In Theorem 8.1 in \cite{leY}, L{\^e} and Yu proved $\overline{S}_n(\Sigma,v)$ is a domain when $\Sigma$ is a polygon. Thus we have the corresponding statements as Theorem \ref{tmmm8.11}, Corollary \ref{ooo}, and Theorem \ref{keyth} for 
reduced stated $SL(n)$-skein algebra when $\Sigma$ is a polygon. For general essentially bordered pb surface, we can get $C(\overline{S}_n(\Sigma,v))$ is a finitely generated algebra because $\overline{S}_n(\Sigma,v)$ is a finitely  generated algebra and finitely generated as a module over it's center.

\subsection{Further steps}

In this subsection, we still assume $v$ is a root of unity of order $m$ with $m$ and $2n$ being coprime with each other. But there is no restriction for marked three manifolds and surfaces.

The key step to prove Unicity theorem is to prove $S_n(\Sigma,v)$ is finitely generated as an $S_n(\Sigma,1)^{(m)}$-module.
It is reasonable to conjecture $\SM$ has a finitely generated  $\SMO$-module structure for any marked three manifold $(M,N)$.

A {\bf marked surface} is a pair $(\Sigma, P)$ where $\Sigma$ is a compact oriented surface and $P$ is a finite subset of $\partial \Sigma$. If every component of $\Sigma$ intersects $P$, we call $(\Sigma,P)$ an {\bf essentially marked surface}. 

For a marked surface $(\Sigma,P)$, we can define a marked three manifold $(M,N)$ with $M=\Sigma\times[-1,1]$ and 
$N=\cup_{x\in P} \{x\}\times [-1,1]$ and the orientation  of $N$ is given by the positive direction of $[-1,1]$.
And we call $(M,N)$ the thickening of $(\Sigma,P)$.
Then define $S_n(\Sigma, P,v) = \SM$, similarly as the pb surface, $S_n(\Sigma, P,v)$ has an algebra structure given by stacking the stated $n$-webs. Obviously a marked surface has the same skein theory with a certain pb surface.

For a marked three manifold $(M,N)$, Costantino and and L{\^e} defined the so called strict subsurface in
\cite{CL2022TQFT}. Here we recall their definition. A {\bf strict subsurface} $\Sigma$ of $(M,N)$ is a proper  embedding $ \Sigma\rightarrow M$ of a compact surface (so that $\partial \Sigma\subset \partial M$), $\Sigma$ is traversal to $N$ and  every connected component of $\Sigma$ intersects $N$.
Let $U(\Sigma)$ be a small open collar  neighborhood of $\Sigma$. $U(\Sigma)$ is 
diffeomorphic with $\Sigma\times (-1,1)$ (by an orientation preserving diffeomorphism) and we require that $\partial \Sigma \times (-1,1)\subset\partial M$. Define the slit $Sl_{\Sigma}(M,N)$ to be $(M^{'}, N^{'})$, where $M^{'} = M -U(\Sigma)$ and $N^{'} = N-cl(U(\Sigma))$. Clearly $(M^{'}, N^{'})$ is well-defined up to isomorphism. If  $(M^{'},N^{'})$ is isomorphic 
 with the thickening of some marked surface, we call $\Sigma$ a
{\bf good strict subsurface} of $(M,N)$.

\begin{proposition}\label{pp8.14}
Let $(M,N)$ be a compact marked three manifold such that every component of $M$ contains at least one marking. Suppose 
$(M,N)$ contains a good strict subsurface $\Sigma$. Then $S_n(M,N,v)$ is finitely generated as an
$S_n(M,N,1)$-module.
\end{proposition}
\begin{proof}
We still use $(M^{'}, N^{'})$ to denote $Sl_{\Sigma}(M,N)$.
Since every component of $M$ containes at least one marking, we have every component of $M^{'}$ also contains at least one marking. Then we have $(M^{'},N^{'})$  is isomorphic to the thickening of some essentially marked  surface. From Lemma \ref{lmlm8.5}
and (a) in Proposition \ref{finite}, we know $S_n(M^{'},N^{'},v)$ is a finitely generated $S_n(M^{'},N^{'},1)$-module. 

Let $f$ denote the embedding from $(M^{'},N^{'})$ to $(M,N)$. From Lemma \ref{lmlm8.5}, it suffices to show 
the linear map $f_{*}:S_n(M^{'},N^{'},v)\rightarrow S_n(M,N,v)$ is surjective. Let $\alpha$ be any stated $n$-web, we can isotope $\alpha$ such that $\alpha$ is transverse to $\Sigma$,  and no sink or source of $\alpha$ is contained in $\Sigma$. Then we use a technique used in Theorem 5.1 in \cite{CL2022TQFT} to kill all the points in $\alpha\cap \Sigma$. Roughly speaking, we first drag each point in $\alpha\cap \Sigma$ close to $\Sigma\cap N$, then use relation (\ref{wzh.seven}) to kill this point. Thus we get $\alpha = \sum_{1\leq i\leq t}k_i \alpha_i$, where $\alpha_i\cap \Sigma =\emptyset$ for all $1\leq i\leq t$. Then, for each $1\leq i\leq t$, we can isotope $\alpha_i$ such that $\alpha_i\subset M^{'}$. We still use $\alpha_i$ to denote the element in $S_n(M^{'},N^{'},v)$ determined by the stated $n$-web $\alpha_i$ in $(M^{'}, N^{'})$. Then $f_*(\sum_{1\leq i\leq t}k_i \alpha_i)  = \alpha$. Thus $f_*$ is surjective.
\end{proof}



\begin{conjecture}\label{conj8.14}
For  any marked three manifold $(M,N)$, there is a subvector space $S_n(M,N,v)^{(m)}$ such that $S_n(M,N,v)^{(m)}$ has a commutative algebra structure  and any stated $n$-web in $S_n(M,N,v)^{(m)}$ is transparent. And we have the followings.

(a) There exists a surjective algebra homomorphism $\hat{\cF}: S_n(M,N,1)\rightarrow S_n(M,N,v)^{(m)}$.

(b) $\hat{F}$ is an isomorphism and $S_n(M,N,v)^{(m)}$ is contained in the center of $S_n(\Sigma,v)$ when $(M,N)$ is the thickening of a pb surface $\Sigma$.

(c) The  map $\cF$ is commutative with the splitting map, where $\cF$ is the combination of $\hat{\cF}$ and the embedding from $S_n(M,N,v)^{(m)}$ to $S_n(M,N,v)$.

(d) $S_n(M,N,v)$ has  an $S_n(M,N,1)$-module structure. 

(e) Under the module structure in (d), $S_n(M,N,v)$ is a finitely generated   $S_n(M,N,1)$-module if $M$ is compact.
\end{conjecture}

 Bloomquist and L{\^e}
defined $\cF:S_2(M,N,1)\rightarrow S_2(M,N,v)$ for all marked three manifolds \cite{bloomquist2020chebyshev}. They also 
proved transparency for the image of $\cF$. Base on their results, we can define $S_2(M,N,v)^{(m)}$ to be Im$\cF$. Similarly as in Subsections \ref{subb7.2}, \ref{newF}, we can define the action of $S_2(M,N,1)$ on $S_2(M,N,v)$ and the multiplication for  $S_2(M,N,v)^{(m)}$. 

%
%
%
%
%
%

We will  prove Conjecture  \ref{conj8.14}  for $n=2$, and  $S_n(M,N,v)$ is a finitely generated $S_n(M,N,1)$-module when $M$ is compact and every component of $M$ contains at least one marking  in an upcoming paper \cite{wang}. 

\begin{rem}
Detcherry, Kalfagianni and Sikora 
used the 
the reduced skein module   to calculate the dimension of $SL(2)$-skein module of some closed three manifolds \cite{detcherry2023kauffman}. For any character 
$\rho \in Hom(\pi_1(M),SL(2,\mathbb{C}))/\simeq$,
the reduced skein module is defined by sending the image of the Frobenius homomorphism
 to a scalar determined by this character $\rho$.
 When $(M,N)$ is a marked three manifold such that every component of $M$ contains at least one marking,
$S_n(M,N,v)$ has an $S_n(M,N,1)$-module structure.
 Here we can define the {\bf reduced stated $SL(n)$-skein module} using this $S_n(M,N,1)$-module structure. 

For any point $I\in\text{MaxSpec}(S_n(M,N,1))$, we define 
$$S_n(M,N,v)_{I}  \triangleq S_n(M,N,v)/I\cdot S_n(M,N,v)$$
where $I\cdot S_n(M,N,v)$ is the linear span of $x\cdot y$ with $x\in I$ and $y\in S_n(M,N,v)$. Clearly if
$S_n(M,N,v)$ is a finitely generated $S_n(M,N,1)$-module, we have $S_n(M,N,v)_{I}$ is a finite dimensional vector space. 
So when $M$ is compact and every component of $M$ contains at least one marking,  we have  $S_n(M,N,v)_{I}$ is a finite dimensional vector space for any $I\in\text{MaxSpec}(S_n(M,N,1))$; and when  $M$ is compact, we have
$S_2(M,N,v)_{I}$ is a finite dimensional vector space for any $I\in\text{MaxSpec}(S_2(M,N,1))$ \cite{wang}.

Note that from Sections \ref{sec3} and \ref{subb5}, we know there is a one to one correspondence  between $\tilde{\chi}_n(M,N)$ and 
$\text{MaxSpec}(S_n(M,N,1))$.
\end{rem}

\section{Generalized marked three manifold}

Costantino and L{\^e} defined the generalized marked three manifold in \cite{CL2022TQFT}, in which they allow $N$ contains oriented closed circles.

\subsection{Cutting out the closure of a small open interval from $N$}
Let $(M,N)$ be a generalized marked three manifold with $N\neq \emptyset$. Suppose $U$ is a small open interval contained in $e$ such that $cl(U)\subset e$, where $e$ is a component of $N$. Let $N^{'} = (N-e)\cup e^{'}$ where $e^{'} = e-cl(U)$.
Let $l_{U}: S_n(M,N^{'},v)\rightarrow S_n(M,N,v)$ be the linear map induced by the embedding $(M,N^{'})\rightarrow (M,N)$. Clearly $l_U$ is surjective, and is an algebra homomorphism when $v=1$.

\begin{proposition}\label{pro7.1}
The  above map $l_U$ induces an isomorphism 
$$\bar{l}_U:S_n(M,N^{'},v)/\simeq\; \rightarrow S_n(M,N,v),$$
 where $\simeq$ is the equivalence relation  given by  the following picture:

\begin{align}\label{eq10}
\raisebox{-.35in}{
\begin{tikzpicture}
\tikzset{->-/.style=
{decoration={markings,mark=at position #1 with
{\arrow{latex}}},postaction={decorate}}}
\filldraw[draw=white,fill=gray!20] (-0.3,0) rectangle (2, 2);
\draw[line width =1.5pt,decoration={markings, mark=at position 1.0 with {\arrow{>}}},postaction={decorate}](2,0) --(2,0.5);
\draw [color = black, line width =1.5pt](2,0.5) --(2,0.9);
\draw[line width =1.5pt,decoration={markings, mark=at position 1.0 with {\arrow{>}}},postaction={decorate}](2,1.1) --(2,1.75);
\draw [color = black, line width =1.5pt](2,1.75) --(2,2);
\node[right] at(2,0.6) {\small $i$};
\filldraw[draw=black!80,fill=white!20] (1,0.8) circle(0.1);
\draw[color = black, line width =1pt] (0,1)--(0.9,0.82);
\draw[color = black, line width =1pt] (1.1,0.78)--(2,0.6);
\end{tikzpicture}}
\simeq
\raisebox{-.35in}{
\begin{tikzpicture}
\tikzset{->-/.style=
{decoration={markings,mark=at position #1 with
{\arrow{latex}}},postaction={decorate}}}
\filldraw[draw=white,fill=gray!20] (-0.3,0) rectangle (2, 2);
\draw[line width =1.5pt,decoration={markings, mark=at position 1.0 with {\arrow{>}}},postaction={decorate}](2,0) --(2,0.5);
\draw [color = black, line width =1.5pt](2,0.5) --(2,0.9);
\draw[line width =1.5pt,decoration={markings, mark=at position 1.0 with {\arrow{>}}},postaction={decorate}](2,1.1) --(2,1.75);
\draw [color = black, line width =1.5pt](2,1.75) --(2,2);
\node[right] at(2,1.4) {\small $i$};
\filldraw[draw=black!80,fill=white!20] (1,1.2) circle(0.1);
\draw[color = black, line width =1pt] (0,1)--(0.9,1.18);
\draw[color = black, line width =1pt] (1.1,1.22)--(2,1.4);
\end{tikzpicture}}\;.
\end{align}
The missing part between two arrows is $cl(U)$.

\end{proposition}\label{prop7.1}
\begin{proof}
Clearly $l_U$ induces a linear map $\bar{l}_U: S_n(M,N^{'},v)/\simeq\; \rightarrow S_n(M,N,v)$.
For a stated $n$-web $\alpha$ in $(M,N^{'})$, we use cls$(\alpha)$ to denote the element in $S_n(M,N^{'},v)/\simeq$ determined by $\alpha$. Let $\beta$ be any stated $n$-web in $(M,N)$. We can isotope $\beta$ such that $cl(U)\cap  \beta = \emptyset$, then we define $j_U(\beta) =\text{cls}(\beta)\in S_n(M,N^{'},v)/\simeq$.
We have $j_U(\beta)$ is independent of  how we isotope $\beta$ because of relation (\ref{eq10}).
 If $\beta$ and $\beta^{'}$ are isotopic stated $n$-webs in $(M,N)$, clearly we have $j_U(\beta)
=j_U(\beta^{'})$ because of relation (\ref{eq10}). Trivially $j_U$ preserves the defining skein relations for $S_n(M,N,v)$. Thus $j_U$ is a well-defined linear map from $S_n(M,N,v)$ to $S_n(M,N^{'},v)/\simeq$. It is easy to check $l_U$ and $j_U$ are inverse to each other.
\end{proof}

\subsection{Classical limit for stated $SL(n)$-skein module for
 generalized marked three manifolds} In this subsection, we will try to find out the classical limit for generalized marked three manifolds by using results in Section \ref{subb5} and Proposition \ref{pro7.1}. In this subsection, we will  assume the three manifold 
is connected (the corresponding results can be easily  generalized to general  three manifolds).
\begin{rem}\label{rem7.2}
Let $(M,N)$ be a generalized marked three manifold. Suppose $N$ contains $k$ ($k\geq 1$) closed oriented circles, which are denoted as $e_0,e_1,\dots,e_{k-1}$. We denote other oriented open intervals in $N$, if any, as $e_k,\dots,e_{m-1}$. For each $0\leq i\leq k-1$, we pick a small open interval $U_i$ contained in $e_i$ such that $cl(U_i)\subset e_i$, and set $e_i^{'} = e_i- cl(U_i)$.  
Set $e_i^{'} = e_i$ for $k\leq i\leq m-1$.
Let
$N^{'} =\{e_0^{'},e_1^{'},\dots, e_{m-1}^{'}\}$, then $(M,N^{'})$ is a circle free marked three manifold. We choose a relative spin structure $h$ for $(M,N^{'})$. For each $1\leq i\leq m-1$,
 let $\alpha_i$ be a path connecting $e_0^{'}$ and $e_i^{'}$ such that $\alpha_i(0)\in e_0^{'}$ and $\alpha_i(1)\in e_i^{'}$.  We still use $l_U$ to denote the algebra homomorphism from $S_n(M,N^{'},1)$ to $S_n(M,N,1)$ induced by the embedding from $(M,N^{'})$ to $(M,N).$
\end{rem}

 We know there is an isomorphism $L$ from $\Gamma_n(M)\otimes
O(SLn)^{\otimes(m-1)}$ to $S_n(M,N^{'},1)$. For any element $y\in \Gamma_n(M)$, we will use
$y_{\otimes}$ to denote  $y\otimes 1\otimes\dots\otimes 1\in  \Gamma_n(M)\otimes
O(SLn)^{\otimes(m-1)}$. For any $1\leq i,j\leq n, 1\leq t\leq m-1,$ we use $x_{i,j}^{t}$ to denote 
$1\otimes1\otimes\dots \otimes x_{i,j}\otimes \dots\otimes 1\in \Gamma_n(M)\otimes O(SLn)^{\otimes(m-1)}$ where $x_{i,j}$ is in the $t$-th tensor factor for $O(SLn)^{\otimes(m-1)}$. Then the isomorphism $L$ is given by:
$$([\alpha]_{i,j})_{\otimes}\rightarrow (A S^{[\alpha]})_{i,j}\text{\;and\;} x^{t}_{i,j}\rightarrow (A S^{[\alpha_t]})_{i,j}$$
where $[\alpha]\in \pi_1(M,e^{'}_0),1\leq i,j\leq n, 1\leq t\leq m-1.$ For each $0\leq t\leq m-1$, 
set $X_t = (x^{t}_{i,j})_{n\times n}$. For each element $[\alpha]\in\pi_1(M,e_0^{'})$, set 
$ Q_{[\alpha],\otimes} = (([\alpha]_{i,j})_{\otimes})_{n\times n}$. Then $L(X_t) = A S^{[\alpha_t]}, 
L(Q_{[\alpha],\otimes}) = A S^{[\alpha]}$.

As in Subsection \ref{sss3.3}, any component $e\in N$ can be lifted to $\tilde{e}\subset UM$. Note that when $e$ is an oriented closed circle, we have $\tilde{e}$ is also an oriented closed circle in $UM$, thus $\tilde{e}$ is an element in $H_1(UM)$.

\begin{definition}\label{df7.3}
Let $(M,N)$ be a generalized marked three manifold, and $h_{s}$ be a spin structure for $M$. Suppose $N$ contains $k$ oriented closed circles. 

If $k=0$, we define $\Gamma_n(M,N) = \Gamma_n(M)$.

 If $k\geq1$,   we denote all the oriented closed circles in $N$  as $e_0,\dots, e_{k-1}$. For each $e_t$, we  use a path $\beta_t$ ($\beta_t(0)$ is the base point for $\pi_1(M)$ and $\beta_t(1)\in e_t$) to connect the base point for $\pi_1(M)$ and $e_t$ to obtain an element in $\pi_1(M)$, which we denote it as $[e_t^{\beta_t}]$.
Define $\Gamma_n(M,N) = \Gamma_n(M)/(D)$ where
$D = \{(Q_{[e_t^{\beta_t}]} - d_n^{h_s(\widetilde{e_t})}I)_{i,j}\mid 0\leq t\leq k-1, 1\leq i,j\leq n\}$ and $(D)$ is the ideal of $\Gamma_n(M)$ generated by $D$. For any element $x\in \Gamma_n(M)$, we use $\bar{x}$ to denote
$x+(D)\in \Gamma_n(M,N)$.
\end{definition}

Note that the definition of $\Gamma_n(M,N)$ is independent of the choice of $\beta_t,0\leq t\leq k-1$.
Suppose for each 
$0\leq t\leq k-1$, we make another choice $\gamma_t$. Then relation $Q_{[e_t^{\beta_t}]} = d^{h_s(\widetilde{e_t})} I$ becomes $Q_{[e_t^{\gamma_t}]} = d^{h_s(\widetilde{e_t})} I$. Since $Q_{[e_t^{\beta_t}]}$  and $Q_{[e_t^{\gamma_t}]}$ are conjugate to each other, then relation $Q_{[e_t^{\beta_t}]} = d^{h_s(\widetilde{e_t})} I$ is the same with relation $Q_{[e_t^{\gamma_t}]} = d^{h_s(\widetilde{e_t})} I$. 

Note that we do not distinguish between $\pi_1(M)$ and $\pi_1(M,e_0)$ where $e_0$ is an embedded open interval in $\partial M$.
The definition for $\Gamma_n(M,N)$ is related to the spin structure $h_s$ for $M$.
Here we make a convention that the spin structure used for the definition of $\Gamma_n(M,N)$ is obtained by restricting the relative spin structure when the relative spin structure is given.

\begin{lemma}\label{lmm7.4}
With all the conventions and notations in Remark \ref{rem7.2}, we have 
$$\begin{tikzcd}
\Gamma_n(M)\otimes
O(SLn)^{\otimes(m-1)} \arrow[r, "L"] & S_n(M,N^{'},1)\arrow[r, "l_U"] & S_n(M,N,1)
\end{tikzcd}$$
induces a surjective algebra homomorphism $\overline{L}:\Gamma_n(M,N)\otimes
O(SLn)^{\otimes(m-1)} \rightarrow S_n(M,N,1)$. Here we regard $\pi_1(M)$ as $\pi_1(M,e_0^{'})$.
\end{lemma}
\begin{proof}
We have the exact sequence: 
$$\begin{tikzcd}
(D) \arrow[r, tail] & \Gamma_n(M)\arrow[r,two heads] & \Gamma_n(M,N) 
\end{tikzcd}$$
where $D,(D)$ are defined in Definition \ref{df7.3} (the arrow with two heads means the corresponding map is surjective). After using functor $-\otimes O(SLn)^{\otimes(m-1)}$ acting on the above exact sequence, we get the following new exact sequence:
$$\begin{tikzcd}
(D)\otimes O(SLn)^{\otimes(m-1)} \arrow[r, tail] & \Gamma_n(M)\otimes O(SLn)^{\otimes(m-1)}\arrow[r,two heads] & \Gamma_n(M,N)\otimes O(SLn)^{\otimes(m-1)}
\end{tikzcd}.$$
Note that $(D)\otimes O(SLn)^{\otimes(m-1)}$ is the ideal of 
$\Gamma_n(M)\otimes O(SLn)^{\otimes(m-1)}$ generated by $d_{\otimes},d\in D$.
Thus to show $l_U\circ L$ induces $\overline{L}$, it suffices to show $l_U(L(d_{\otimes})) = 0$ for all $d\in D$.

Let $i,j$ be any two integers between $1$ and $n$, and $t$ be an integer between $0$ and $k-1$. Then we have 
$$l_U(L(([e_t^{\beta_t}]_{i,j})_{\otimes}))= (-1)^{i+1}l_U(S^{[e_t^{\beta_t}]}_{\bar{i},j}).$$
Thus we need to show $(-1)^{i+1}l_U(S^{[e_t^{\beta_t}]}_{\bar{i},j}) = d_n^{h(\widetilde{e_t})} \delta_{i,j}$, that is, to show $A\,l_U(S^{[e_t^{\beta_t}]}) =  d_n^{h(\widetilde{e_t})} I$. From the definition of $[e_t^{\beta_t}]$, we know $[e_t^{\beta_t}] = [\beta_t^{-1}*e_t*\beta_t]$. Then we have 
$$A\,l_U(S^{[\beta_t]} )A\,l_U(S^{[e_t^{\beta_t}]}) = l_U(A S^{[\beta_t]} A 
S^{[\beta_t^{-1}*e_t*\beta_t]}) = l_U(A S^{[e_t*\beta_t]}) =A\, l_U(S^{[e_t*\beta_t]}).$$
Note that in $S_n(M,N,1)$, we have $l_U(S^{[e_t*\beta_t]}) = d_n^{h(\widetilde{e_t})} l_U(S^{[\beta_t]})$. Then we get $$l_U(S^{[\beta_t]} )A\,l_U(S^{[e_t^{\beta_t}]}) = l_U(S^{[e_t*\beta_t]}) = d_n^{h(\widetilde{e_t})} l_U(S^{[\beta_t]}).$$
Then we have $A\,l_U(S^{[e_t^{\beta_t}]})= d_n^{h(\widetilde{e_t})} I$ because 
$l_U(S^{[\beta_t]} )$ is invertible.

The above discussion shows $l_U\circ L$ induces $\overline{L}$. We have $\overline{L}$ is a surjective algebra homomorphism since $l_U\circ L$ is  a surjective algebra homomorphism.
\end{proof}

Note that for any $x\in \Gamma_n(M), y\in 
O(SLn)^{\otimes(m-1)}$, we have $\overline{L}(\bar{x}\otimes y) =l_U(L(x\otimes y))$. 
We use $\pi$ to denote the projection from $\Gamma_n(M)\otimes
O(SLn)^{\otimes(m-1)}$ to $\Gamma_n(M,N)\otimes
O(SLn)^{\otimes(m-1)}$. Then $\overline{L}\circ \pi = l_U\circ L$.

\begin{lemma}\label{lmm7.5}
With all the conventions and notations in Remark \ref{rem7.2}, we have 
$$\begin{tikzcd}
S_n(M,N^{'},1)
 \arrow[r, "L^{-1}"] & \Gamma_n(M)\otimes
O(SLn)^{\otimes(m-1)} \arrow[r, "\pi"] & \Gamma_n(M,N)\otimes
O(SLn)^{\otimes(m-1)}
\end{tikzcd}$$
induces a surjective algebra homomorphism $\overline{L^{-1}}:S_n(M,N,1) \rightarrow \Gamma_n(M,N)\otimes
O(SLn)^{\otimes(m-1)}$. Here we regard $\pi_1(M)$ as $\pi_1(M,e_0^{'})$.
\end{lemma}
\begin{proof}
From Proposition \ref{pro7.1}, it suffices to show $\pi\circ L^{-1}$ preserves the equivalence relation (\ref{eq10}) for every $U_t$, $0\leq t\leq k-1$. Let $\alpha$ be any stated $n$-web for $(M,N^{'})$. Suppose there exists $0\leq t\leq k-1$ such that nearby $U_t$ $\alpha$ looks like the left picture in the equivalence relation (\ref{eq10}). Let $\alpha^{'}$ be the same stated $n$-web as $\alpha$ except nearby  $U_t$ $\alpha^{'}$ looks like the right picture in equivalence relation (\ref{eq10}). Then we want to show $\pi(L^{-1}(\alpha)) = \pi(L^{-1}(\alpha^{'}))$.

We can use the same way  
 to kill all the sinks and sources in $\alpha$ and $\alpha^{'}$. Then the resulting two stated $n$-webs only differ on a single stated arc. Since $\pi\circ L^{-1}$ is an algebra homomorphism, we can just assume $\alpha$ is a stated framed oriented boundary arc. Without loss of generality, we assume the white dot in equivalence relation (\ref{eq10}) represents an arrow pointing from left to right, that is, pointing towards the boundary. It is easy to show $h(\widetilde{\alpha^{'}}) = h(\tilde{\alpha})+h(\widetilde{e_t})$. Suppose $s(\alpha(0)) = s(\alpha^{'}(0)) = j$. Then we have $\alpha = d_n^{ h(\tilde{\alpha})} S^{[\alpha]}_{i,j}$ and 
$\alpha^{'} = d_n^{ h(\widetilde{\alpha^{'}}) }S^{[\alpha^{'}]}_{i,j}$.

Suppose $\alpha(0)\in e^{'}_{t_1}$ where $0\leq t_1\leq m-1$.
We have four cases to consider: (1) $t=t_1=0,$ $t=0$ and $t_1\neq 0$, $t\neq 0$ and $t_1=0$, $t\neq 0$ and $t_1\neq 0$.

 Here we only prove the case when $t\neq 0$ and $t_1\neq 0$. Then we have
\begin{align*} 
\pi(L^{-1}(S^{[\alpha]})) &=  A^{-1}\pi(L^{-1}(AS^{[\alpha_t][\alpha_t^{-1}*\alpha*\alpha_{t_1}][\alpha_{t_1}^{-1}]} )) = A^{-1}\pi(L^{-1}(AS^{[\alpha_t]} AS^{[\alpha_t^{-1}*\alpha*\alpha_{t_1}]} AS^{[\alpha_{t_1}^{-1}]} )) \\
&= A^{-1} \pi( X_t\, Q_{[\alpha_t^{-1}*\alpha *\alpha_{t_1}],\otimes} \,X_{t_1}^{-1})
=  A^{-1} \pi( X_t)\pi( Q_{[\alpha_t^{-1}*\alpha *\alpha_{t_1}],\otimes}) \pi( X_{t_1}^{-1}).
\end{align*}
Similarly we have 
$$\pi(L^{-1}(S^{[\alpha^{'}]}))
=  A^{-1} \pi( X_t)\pi( Q_{[\alpha_t^{-1}*\alpha^{'} *\alpha_{t_1}],\otimes})  \pi(X_{t_1}^{-1}).$$
We have 
$$Q_{[\alpha_t^{-1}*\alpha^{'} *\alpha_{t_1}],\otimes}
= Q_{[\alpha_t^{-1}*\alpha *\alpha_{t_1}],\otimes}
Q_{[\alpha_{t_1}^{-1}*\alpha^{-1}*\alpha^{'} *\alpha_{t_1}],\otimes}$$
where $[\alpha_{t_1}^{-1}*\alpha^{-1}*\alpha^{'} *\alpha_{t_1}] = 
[(\alpha^{'} *\alpha_{t_1})^{-1}*e_t*\alpha^{'} *\alpha_{t_1}]$. Thus we have 
$$\pi(Q_{[\alpha_t^{-1}*\alpha^{'} *\alpha_{t_1}],\otimes})=
\pi( Q_{[\alpha_t^{-1}*\alpha *\alpha_{t_1}],\otimes}) \pi(Q_{[\alpha_{t_1}^{-1}*\alpha^{-1}*\alpha^{'} *\alpha_{t_1}],\otimes}) = d_n^{h(\widetilde{e_t})}\pi( Q_{[\alpha_t^{-1}*\alpha *\alpha_{t_1}],\otimes}).$$
Then we have 
\begin{align*}
\pi(L^{-1}(\alpha^{'}))& = d_n^{ h(\widetilde{\alpha^{'}}) } \pi L^{-1}(S^{[\alpha^{'}]}_{i,j})
= d_n^{  h(\tilde{\alpha})+h(\widetilde{e_t})} [A^{-1} \pi (X_t)\pi( Q_{[\alpha_t^{-1}*\alpha^{'} *\alpha_{t_1}],\otimes})\pi( X_{t_1}^{-1})]_{i,j}\\
& = d_n^{  h(\tilde{\alpha})+h(\widetilde{e_t})} d_n^{h(\widetilde{e_t})} [A^{-1} \pi( X_t)\pi( Q_{[\alpha_t^{-1}*\alpha *\alpha_{t_1}],\otimes})\pi( X_{t_1}^{-1})]_{i,j}\\&=
d_n^{  h(\tilde{\alpha})} [A^{-1} \pi( X_t)\pi( Q_{[\alpha_t^{-1}*\alpha *\alpha_{t_1}],\otimes}) \pi(X_{t_1}^{-1})]_{i,j}= \pi(L^{-1}(\alpha)).
\end{align*}

\end{proof}

For any stated $n$-web $\alpha$ in $S_n(M,N,1)$, we can isotope $\alpha$ such that $cl(U_t)\cap \alpha = \emptyset$ for all
$0\leq t\leq k-1$. Then $\alpha$ is also a stated $n$-web $\alpha$ in $S_n(M,N^{'},1)$, we still use $\alpha$ to this element in  $S_n(M,N^{'},1)$. Then $\overline{L^{-1}}(\alpha) = \pi (L^{-1}(\alpha))$, that is, we have $\overline{L^{-1}}\circ l_U
= \pi\circ L^{-1}$.

\begin{lemma}\label{lmm7.6}
The algebra homomorphism $\overline{L}$ obtained in Lemma \ref{lmm7.4} and 
the algebra homomorphism $\overline{L^{-1}}$ obtained in Lemma \ref{lmm7.5} are inverse to each other. Especially for any generalized marked three manifold $(M,N)$ with $N$ containing at least one closed oriented circle, we have $\Gamma_n(M,N)\otimes
O(SLn)^{\otimes(\sharp N-1)} \simeq S_n(M,N,1)$.
\end{lemma}
\begin{proof}
For any stated $n$-web $\alpha$ in $S_n(M,N,1)$, we isotope $\alpha$ such that $cl(U_t)\cap \alpha = \emptyset$ for all
$0\leq t\leq k-1$. Then 
$$\overline{L}(\overline{L^{-1}}(\alpha)) = \overline{L} (\pi (L^{-1}(\alpha)))=
(l_U\circ L)(L^{-1}(\alpha)) = l_U(\alpha) = \alpha .$$

For any $x\in \Gamma_n(M), y\in 
O(SLn)^{\otimes(m-1)}$, we have 
$$\overline{L^{-1}}(\overline{L}(\bar{x}\otimes y)) =
\overline{L^{-1}}(l_U(L(x\otimes y))) = (\pi\circ L^{-1}) (L(x\otimes y))
= \pi(x\otimes y) =\bar{x}\otimes y.$$ 
\end{proof}

\begin{theorem}
Let $(M,N)$ be a generalized marked three manifold with $N\neq \emptyset$.
Then $S_n(M,N,1)\simeq \Gamma_n(M,N)\otimes O(SLn)^{\otimes(\sharp N - 1)}$.
\end{theorem}
\begin{proof}
If $N$ is circle free, then $\Gamma_n(M,N) = \Gamma_n(M)$. From Theorem \ref{thm5.13}, we have 
$S_n(M,N,1)\simeq \Gamma_n(M,N)\otimes O(SLn)^{\otimes(\sharp N - 1)}$.

If $N$ containes at least one oriented closed circle, then Lemma \ref{lmm7.6} shows 
$S_n(M,N,1)\simeq \Gamma_n(M,N)\otimes O(SLn)^{\otimes(\sharp N - 1)}$.
\end{proof}

For generalized marked three manifold $(M,N)$, we can also define the corresponding adding marking map. Suppose $N_{ad}=N\cup e$ where $e$ is an oriented open interval or an oriented closed circle such that there is no intersection between the closure of $N$ and the closure of $e$. We also call $N_{ad}$ is obtained from $N$ by adding one extra marking. The linear map from $S_n(M,N,v)$ to $S_n(M,N_{ad},v)$ induced by the embedding $(M,N)\rightarrow (M,N_{ad})$ is also denoted as $l_{ad}$. Clearly when $v=1$, we have $l_{ad}$ is an algebra homomorphism.

\begin{corollary}
Let $(M,N)$ be a generalized marked three manifold. Suppose $N_{ad}$ is obtained from $N$ by adding one extra oriented open interval. Then $l_{ad} :S_n(M,N,1)\rightarrow S_n(M,N_{ad},1) $ is injective.

\end{corollary}
\begin{proof}
We already proved the injectivity for $l_{ad}$ when $N$ is circle free in Corollary \ref{Cor5.13}.

Then we suppose $N$ contains at least one oriented circle.
When we cut the closure of small open intervals as in Remark \ref{rem7.2}, we choose the same way to cut them for $(M,N)$ and $(M,N_{ad})$,
and the choices for $\alpha_i$ as in Remark \ref{rem7.2} are compatible between $(M,N^{'})$ and $(M,(N_{ad})^{'})$.  The relative spin structure used for $(M,N^{'})$ is the restriction of the relative spin structure used for $(M,(N_{ad})^{'})$.
Since $N_{ad}$ is obtained from $N$ by adding one extra oriented open interval, we have $\Gamma_n(M,N_{ad})
= \Gamma_n(M,N)$. Then it is easy to check we have the following commutative diagram:
$$\begin{tikzcd}
\Gamma_n(M,N)\otimes O(SLn)^{\otimes(\sharp N -1)}  \arrow[r, "\overline{L}_{(M,N)}"]
\arrow[d, "J"]  
&  S_n(M,N,1)  \arrow[d, "l_{ad}"] \\
\Gamma_n(M,N)\otimes O(SLn)^{\otimes(\sharp N )}  \arrow[r, "\overline{L}_{(M,N_{ad})}"] 
&  S_n(M,N_{ad},1)\\
\end{tikzcd}$$
where $J$ is the obvious embedding. Since both $\overline{L}_{(M,N_{ad})}$ and 
$\overline{L}_{(M,N)}$ are isomorphisms, we have $l_{ad}$ is injective.
\end{proof}


%
%
%
%
%
%
%
%

\bibliographystyle{plain}

\bibliography{ref.bib}

\hspace*{\fill} \\

School of Physical and Mathematical Sciences, Nanyang Technological University, 21 Nanyang Link Singapore 637371

$\emph{Email address}$: zhihao003@e.ntu.edu.sg

\end{document}